\documentclass[10pt, a4paper]{amsart}
\usepackage{geometry}

\usepackage[style=alphabetic]{biblatex}
\usepackage[dvipsnames]{xcolor}
\usepackage{accents}
\usepackage{amsmath, amsthm,  amsfonts, etoolbox, changepage,lipsum, bigints, relsize, mathtools, color, tikz-cd, amssymb, float, wrapfig}
\usepackage{aliascnt}
\usepackage{amsmath, amsthm,  amsfonts, etoolbox, changepage,lipsum, bigints, relsize, mathtools, color, tikz-cd, amssymb, tikz, multirow ,adjustbox, tabularx, rotating, booktabs, lscape,xcolor,colortbl}
\usepackage{aliascnt}
\usepackage{hyperref}
\usepackage{cleveref}
\usepackage{url}
\usepackage[english]{babel}
\usepackage[utf8]{inputenc} 
\usepackage[T1]{fontenc}
\input xypic
\xyoption{all}
\usepackage[all,arc]{xy}
\usepackage{enumerate}
\usepackage{mathrsfs}
\usepackage{mathtools} 
\usetikzlibrary{decorations.markings}

\usepackage[OT2,OT1]{fontenc}
\newcommand\cyr
{
\renewcommand\rmdefault{wncyr}
\renewcommand\sfdefault{wncyss}
\renewcommand\encodingdefault{OT2}
\normalfont
\selectfont
}
\DeclareTextFontCommand{\textcyr}{\cyr}

\usepackage{comment}
\usetikzlibrary{calc, decorations.pathreplacing,angles,quotes}

\usepackage{caption}
\usepackage{subcaption}

\pgfkeys{tikz/mymatrixenv/.style={decoration={brace},every left delimiter/.style={xshift=8pt},every right delimiter/.style={xshift=-8pt}}}
\pgfkeys{tikz/mymatrix/.style={matrix of math nodes,nodes in empty cells,left delimiter={[},right delimiter={]},inner sep=1pt,outer sep=1.5pt,column sep=2pt,row sep=2pt,nodes={minimum width=20pt,minimum height=10pt,anchor=center,inner sep=0pt,outer sep=0pt}}}
\pgfkeys{tikz/mymatrixbrace/.style={decorate,thick}}

\hypersetup{
    colorlinks=true,
    linkcolor=blue,
    filecolor=magenta,      
    urlcolor=cyan,
    citecolor = blue
}

  \theoremstyle{plain}
  \newtheorem{theorem}{Theorem}[section]
 \newtheorem{lemma}[theorem]{Lemma}
  
  \newtheorem{question}[theorem]{Question}
 \newtheorem{example}[theorem]{Example}  
 \newtheorem{corollary}[theorem]{Corollary}
 
 \newtheorem{conjecture}[theorem]{Conjecture}

\newtheorem{proposition}[theorem]{Proposition}
\crefname{lemma}{Lemma}{Lemma}
  \crefname{corollary}{Corollary}{Corollary}
  \crefname{theorem}{Theorem}{Theorem}
  \crefname{definition}{Definition}{Definition}
   \crefname{proposition}{Proposition}{Proposition}
 \crefname{section}{Section}{Section} 
   \crefname{construction}{Construction}{Construction}
   \crefname{generalization}{Generalization}{Generalization}
  \crefname{construction}{Construction}{Construction}
  \crefname{notation}{Notation}{Notation}
   \crefname{example}{Example}{Example}
  \crefname{remark}{Remark}{Remark}
  \crefname{fact}{Fact}{Fact}
  \crefname{conjecture}{Conjecture}{Conjecture}
  \crefname{motivation}{Motivation}{Motivation}  
  \crefname{figure}{Figure}{Figure}  
  
  \newtheorem{definition}[theorem]{Definition}
  
  \newtheorem{remark}[theorem]{Remark}
  \newtheorem{note}{Note}[section]

  \numberwithin{equation}{section}
 \numberwithin{figure}{section}
  \numberwithin{figure}{subsection}

  \renewcommand{\cH}{{\mathcal H}}
  
  \newcommand{\cA}{{\mathcal A}}

  \newcommand{\cF}{{\mathcal F}}
  
  \renewcommand{\cD}{{\mathcal D}}
  
  \newcommand{\cC}{{\mathcal C}}
  
  \newcommand{\cO}{{\mathcal O }}
  
  \newcommand{\cK}{{\mathcal K }}
  \newcommand{\cS}{{\mathcal S }}

  \newcommand{\cB}{{\mathcal B }}
    \newcommand{\cZ}{{\mathcal Z }}
        
         \newcommand{\calR}{{\mathcal R }}
       
       \newcommand{\cU}{{\mathcal U }}
  \renewcommand{\cR}{\mathcal{R}}
  \newcommand{\sH}{\mathscr{H}}

  \newcommand{\Hom}{\text{Hom}}
  \newcommand{\HOM}{\text{HOM}}
  \newcommand{\Com}{\text{Com}}
  \newcommand{\Kom}{\text{Kom}}
  \newcommand{\cone}{\text{cone}}
  \newcommand{\id}{\text{id}}
  \newcommand{\Ba}{\mathscr{B}}
  \newcommand{\Aa}{\mathscr{A}}
  \newcommand{\bigrdim}{\text{bigrdim}}
  
  \newcommand{\sgn}{\text{sgn}}

   \newcommand{\ba}{\begin{eqnarray}}
   \newcommand{\na}{\end{eqnarray}}
   \newcommand{\ban}{\begin{eqnarray*}}
   \newcommand{\nan}{\end{eqnarray*}}

 \newcommand{\m}{\mathfrak{m}}
  \newcommand{\fc}{\mathfrak{c}}
  \newcommand{\fC}{\mathfrak{C}}
    \newcommand{\fD}{\mathfrak{D}}

     \newcommand{\fq}{\mathfrak{q}}
     \newcommand{\fp}{\mathfrak{p}}
         \newcommand{\ff}{\mathfrak{f}}
          
            \newcommand{\fB}{\mathfrak{B}}

  \newcommand{\B}{\mathbb B}
  \newcommand{\C}{{\mathbb C}}
  \newcommand{\R}{{\mathbb R}}
  \newcommand{\Z}{{\mathbb Z}}
   \newcommand{\Q}{{\mathbb Q}}

       \newcommand{\bS}{\mathbb S}
        
         \newcommand{\D}{\mathbb D}
          
  \newcommand{\A}{{\mathbb A}}

  \renewcommand{\d}{\delta}
  
  \newcommand{\sB}{\mathscr{B}}
  
\newcommand{\sD}{\mathscr{D}}

  \newcommand{\sF}{\mathscr{F}}

       \newcommand{\sG}{\mathscr{G}}
  
\newcommand{\ra}{\rightarrow}

\newcommand{\xra}{\xrightarrow}

\newcommand{\hra}{\hookrightarrow}

\newcommand{\wt}{\widetilde}

\newcommand\rb[1]{\textcolor{red}{\textbf{#1}}}

\newcommand{\DA}{\D^A_{2n}}
\newcommand{\DAz }{\D^A_{2n} \setminus \Delta_0}
\newcommand{\DB}{\D^B_{n+1}}
\newcommand{\bs}{\backslash}
\newcommand{\bsz}{\backslash \{ 0 \} }
\newcommand{\MCG}{\text{MCG}}
\newcommand{\Diff}{\text{Diff}}
\newcommand{\RP}{\mathbb{R} \text{P}^1}
\newcommand{\udfC}{\undertilde{\widetilde{\mathfrak{C}}}}
\newcommand{\RDA}{\fD^A_{\Delta}}
\newcommand{\RDB}{\fD^B_{\Lambda}}
\newcommand{\RDAz}{\fD^A_{\Delta_{0}}  }
\newcommand{\wRDA}{\wt{\fD}^A_{\Delta}}
\newcommand{\wRDB}{\check{\fD}^B_{\Lambda}}
\newcommand{\wRDAz}{\wt{\fD}^A_{\Delta_{0}}}
\newcommand{\MCGB}{\text{MCG} \left( \D^B_{n+1}, \{ 0 \} \right)}
\newcommand{\MCGA}{\text{MCG} \left( \D^A_{2n} \right)}

\newcommand{\DiffB}{\text{Diff} \left( \D^B_{n+1}, \{0\} \right)}
\newcommand{\ut}{\undertilde}
\newcommand{\wuc}{\undertilde{\widetilde{c}}}

\newcommand{\ol}{\overline}

     \def\uv{\underline{v}}
       \def\uw{\underline{w}}

\newcommand{\bit}[1]{\textit{\textbf{#1}}}

  \newcommand{\<}{\langle}
  \renewcommand{\>}{\rangle}

\usepackage{scalerel,stackengine}
\stackMath
\newcommand\reallywidehat[1]{%
\savestack{\tmpbox}{\stretchto{%
  \scaleto{%
    \scalerel*[\widthof{\ensuremath{#1}}]{\kern-.6pt\bigwedge\kern-.6pt}%
    {\rule[-\textheight/2]{1ex}{\textheight}}
  }{\textheight}%
}{0.5ex}}%
\stackon[1pt]{#1}{\tmpbox}%
}
\newtheorem*{theorem*}{Theorem}

\newtheoremstyle{named}{}{}{\itshape}{}{\bfseries}{.}{.5em}{\thmnote{#3's }#1}
\theoremstyle{named}

\newtheoremstyle{name}{}{}{\itshape}{}{\bfseries}{.}{.5em}{\thmnote{#3}#1}
\theoremstyle{name}

\makeatletter
\newcommand\xrightleftarrows[2][]{\ext@arrow 0099{\longrightleftarrowsfill@}{#1}{#2}}
\def\longrightleftarrowsfill@{\arrowfill@\leftarrow\relbar\rightarrow}
\makeatother

\title{
Curves in the Disc, the Type B Braid Group, and a Type B Zigzag Algebra.}

\author{Edmund Heng}
\author{Kie Seng Nge}

\address{Mathematical Sciences Institute, The Australian National University, Hanna Neumann, Building 145, Science Road, Acton ACT 2601, Australia.}
\email{u5476890@alumni.anu.edu.au}

\address{Mathematical Sciences Institute, The Australian National University, Hanna Neumann, Building 145, Science Road, Acton ACT 2601, Australia.}
\email{kieseng.nge@anu.edu.au}

\address{School of Mathematics and Physics, Xiamen University Malaysia, Block A4, Jalan Sunsuria, Bandar Sunsuria, 43900 Sepang, Selangor Darul Ehsan, Malaysia.}
\email{kieseng.nge@xmu.edu.my}

\calclayout

\begin{document}
\maketitle
\begin{abstract}
We construct a finite dimensional quiver algebra from the non-simply-laced type $B$ Dynkin diagram, which we call type $B$ zigzag algebra.
This leads to a faithful categorical action of the type $B$ Artin (braid) group $\cA(B)$, acting on the  homotopy category of its projective modules.
This categorical action is also closely related to the topological action of $\cA(B)$, viewed as mapping class group of the punctured disc -- hence our exposition can be seen as a type $B$ analogue of Khovanov--Seidel's work.
\end{abstract}

\section{Introduction}
In the seminal work \cite{KhoSei}, Khovanov--Seidel introduce a  categorical action of the type $A_m$ Artin group $\cA(A_m)$, where the group acts faithfully by exact autoequivalences on the bounded homotopy category  $\Kom^b(\Aa_m$-$\text{p$_{r}$g$_{r}$mod})$ of projective (graded) modules over the type $A_m$ zigzag algebra $\Aa_m$.
Moreover, they show that this Artin group action on $\Kom^b(\Aa_m$-$\text{p$_{r}$g$_{r}$mod})$ is deeply related to the mapping class group action on curves on the punctured disc. 
More precisely, they construct a map $L_A$ that associates  complexes of projective $\Aa_m$-modules to isotopy classes of curves in the disc, and show that $L_A$ intertwines the categorical action of $\cA(A_m)$ on complexes with the mapping class action of $\cA(A_m)$ on curves.
Furthermore, the geometric intersection number between two curves $c_1$ and $c_2$ can be computed from the dimension of their corresponding total Hom space $\HOM^*(L_A(c_1), L_A(c_2))$ in $\Kom^b(\Aa_m$-$\text{p$_{r}$g$_{r}$mod})$.
This may be seen as a bridge connecting two appearances of the same group: the former as the Artin group associated to the type $A$ Coxeter group, and the latter as the mapping class group of the punctured disc.

Another family of Artin groups which also appear as mapping class groups are the type $B_n$ Artin groups  $\cA(B_n)$.
To this end, Gadbled--Thiel--Wagner develop a ``type $B$'' analogue of the Khovanov--Seidel story in \cite{GTW}, where they bypass the non-simply-laced structure through viewing the type $B_n$ Artin group $\cA(B_n)$ as the extended Artin group $\widehat{\cA}(\hat{A}_{n-1})$ of affine type $A$.
Although they are both mapping class groups of a $(n+1)$-punctured disc (fixing one of the punctures), their corresponding natural affine configurations of the disc are different (see \cref{fig: diff affine configuration}).

\begin{figure}[h]  
\begin{subfigure}{0.4 \textwidth}
\centering
\begin{tikzpicture}[scale = 0.45]

\draw (5.2,2.5) ellipse (6.5cm and 2.6cm); 
\draw[fill] (1.75,2.5) circle [radius=0.1] ;
\draw[fill] (3.6,2.5) circle [radius=0.1]  ;
\draw[fill] (7.6,2.5) circle [radius=0.1]  ;
\draw[fill] (9.35,2.5) circle [radius=0.1]  ;
\filldraw[color=black!, fill=yellow!, thick]  (0,2.5) circle [radius=0.1];

\node  at (5.6,2.5) {$ \boldsymbol{\cdots}$};
\node[below right] at  (0,2.5) {0};
\node[below right] at  (1.75,2.5) {1};
\node[below right] at  (3.6,2.5) {2};
\node[below] at  (7.6,2.5) {$n-1$};
\node[below right] at  (9.35,2.5) {$n$};

\end{tikzpicture}
\caption{For type $B$.} 
\end{subfigure}
\quad
\begin{subfigure}{0.49 \textwidth}
\centering
\begin{tikzpicture} [scale = 0.35]

\draw (0,0) circle (5.4cm);

\filldraw[color=black!, fill=yellow!, thick] (0,0) circle [radius=0.1] ;


\draw[fill] (3,-.5) circle [radius=0.1] ;
\draw[fill] (-3,-.5) circle [radius=0.1] ;

\draw[fill] (2.2,-1.8) circle [radius=0.1] ;
\draw[fill] (-2.2,-1.8) circle [radius=0.1] ;

\draw[fill] (2.8,1.3) circle [radius=0.1] ;
\draw[fill] (-2.8,1.3) circle [radius=0.1] ;

\draw[fill] (-1.8,2.75) circle [radius=0.1] ;

\draw[fill] (0,-2.9) circle [radius=0.03] ;
\draw[fill] (-0.7,-2.8) circle [radius=0.03] ;
\draw[fill] (0.7,-2.8) circle [radius=0.03] ;

\draw[fill] (0.3,3) circle [radius=0.03] ;
\draw[fill] (0.9,2.9) circle [radius=0.03] ;
\draw[fill] (1.45,2.55) circle [radius=0.03] ;

\node[below] at (0,0) {\scriptsize{$0$}} ;
\node[below right] at (2.2,-1.8) {\scriptsize{$ i$}} ;
\node[below left] at (-2.2,-1.8) {\scriptsize{$2$}} ;
\node[below right] at (3,-.5) {\scriptsize{$ i + 1$}} ;
\node[below left] at (-3,-.5) {\scriptsize{$1$}} ;

\node[below right] at (2.8,1.3) {\scriptsize{$ i + 2$}} ;
\node[below left] at (-2.8,1.3) {\scriptsize{$n$}} ;

\node[above] at (-1.8,2.75) {\scriptsize{$n-1$}} ;

\end{tikzpicture}
\caption{For extended affine type $A$.}
\end{subfigure}
\caption{Two different affine configurations of the $(n+1)$-punctured disc corresponding to the action of $\cA(B_n)$ and $\widehat{\cA}(\hat{A}_{n-1})$, where the puncture labelled `0' is fixed.}
\label{fig: diff affine configuration}
\end{figure}

The goal of the present paper is to develop a proper (non-simply-laced) type $B$ analogue of the stories given by Khovanov--Seidel and Gadbled--Thiel--Wagner.
We introduce a (finite dimensional) quotient of a quiver algebra  $\Ba_n$ over $\R$, which we call type $B_n$ zigzag algebra.
Since the type $B$ root system is no longer simply-laced, the definition of $\Ba_n$ will be somewhat subtle -- the indecomposable projective $\Ba_n$-module whose class in the Grothendieck group is a long simple root will only have the structure of a $\mathbb{R}$-vector space, while all the other indecomposable projective $\Ba_n$-modules whose classes are short simple roots will actually be $\mathbb{C}$-vector spaces.  
Note that this is somewhat reminiscent of the following non-simply-laced extension in quiver theory: by studying representations of $K$-species \cite{GabIndecompII} instead of quivers (the base field are allowed to be different at each vertex), the finite type $K$-species are characterised by \emph{all} (including non-simply-laced) Dynkin diagrams \cite[Theorem B]{DR}.

The relevance of $\Ba_n$ to Coxeter theory is provided by the following theorem:
\begin{theorem}[\cref{Cat B action} and \cref{faithful action}]\label{faithfulintro}
The homotopy category $\Kom^b(\Ba_n$-$\text{p$_{r}$g$_{r}$mod})$ of projective (bigraded) modules  carries a faithful (weak) action of the type $B_n$ Artin group $\cA(B_n).$
\end{theorem}
\noindent
Similar to the works of Khovanov--Seidel and Gadbled--Thiel--Wagner,  we establish the following result that relates the categorical notions to low-dimensional topology:
\begin{theorem}[\cref{L_B equivariant} and \cref{poin poly equals tri int}] \label{L_Bintro} There exists a map $L_B$ that associates complexes in $\Kom^b(\Ba_n$-$\text{p$_{r}$g$_{r}$mod})$ to curves in the $(n+1)$-punctured disc. 
This map $L_B$ is $\cA(B_{n})$-equivariant, intertwining the $\cA(B_n)$-action on curves and the $\cA(B_n)$-action on complexes in  $\Kom^b(\Ba_n$-$\text{p$_{r}$g$_{r}$mod}).$
Moreover, the (trigraded) intersection number between two curves $c_1$ and $c_2$ is given by the Poincar\'e polynomial of the total Hom space between $L_B(c_1)$ and $L_B(c_2)$.
\end{theorem}
%

One main feature of our work contrasting that of Gadbled--Thiel--Wagner is an explicit realisation of a well-known connection between type $B$ and type $A$.
To explain this, recall that the type $B_n$ Artin group is known to be a (proper) subgroup of the type $A_{n-1}$ Artin group: algebraically the embedding is induced from a  folding of Coxeter diagrams \cite{brieskorn_1973};
topologically the embedding is obtained by lifting through the double-branched cover
of a $(n+1)$-punctured disc $\DB$ by a $2n$-punctured disc $\DA$ \cite{BH}.
The topological interpretation induces a $\mathcal{A}(B_n)$-equivariant map $\mathfrak{m}$ that takes curves in the $(n+1)$-punctured disc to \emph{multicurves} in the $2n$-punctured disc, defined by taking the preimage of the covering map.
Our work includes a categorical interpretation of this map $\mathfrak{m}$, given by a scalar extension functor:
\begin{theorem}[\cref{isomorphic algebras} and \cref{fullmaintheorem}]\label{intromaintheorem}
The type $B$ zigzag algebra $\Ba_n$ algebra (over $\R$) is isomorphic to Khovanov--Seidel type $A$ zigzag algebra $\Aa_{2n-1}$ after extending scalars to $\C$, namely 
\[
\C\otimes_\R \Ba_n \cong \Aa_{2n-1} \quad \text{as $\C$-algebras}.
\]
This induces a scalar extension functor $\Aa_{2n-1} \otimes_{\Ba_n} -$, which renders the diagram in \cref{fig: full picture} commutative, with all four maps on the square $\cA(B_n)$-equivariant.
\end{theorem}

\begin{figure}[h]
\centering
\begin{tikzpicture} [scale=0.85]
\node (tbB) at (-3,1.75) 
	{$\mathcal{A}(B_n)$};
\node (cbB) at (-3,-3.5) 
	{$\mathcal{A}(B_n)$};
\node (tbA) at (10.5,1.75) 
	{$\mathcal{A}(A_{2n-1}) \hookleftarrow \mathcal{A}(B_n) $}; 
\node (cbA) at (10.5,-3.5) 
	{$\mathcal{A}(A_{2n-1}) \hookleftarrow \mathcal{A}(B_n) $};

\node[align=center] (cB) at (0,0) 
	{Isotopy classes of trigraded \\ admissible curves $\check{\fC}^{adm}$  in $\DB$};
\node[align=center] (cA) at (7,0) 
	{Isotopy classes of bigraded \\ admissible multicurves $\ddot{\undertilde{\wt{\fC}}}^{adm}$  in $\DA$};
\node (KB) at (0,-2)
	{$\Kom^b(\Ba_n$-$\text{p$_{r}$g$_{r}$mod})$};
\node (KA) at (7,-2) 
	{$\Kom^b(\Aa_{2n-1}$-$\text{p$_{r}$g$_{r}$mod})$};

\coordinate (tbB') at ($(tbB.east) + (0,-1)$);
\coordinate (cbB') at ($(cbB.east) + (0,1)$);
\coordinate (tbA') at ($(tbA.west) + (0,-1)$);
\coordinate (cbA') at ($(cbA.west) + (0,1)$);

\draw [->,shorten >=-1.5pt, dashed] (tbB') arc (245:-70:2.5ex);
\draw [->,shorten >=-1.5pt, dashed] (cbB') arc (-245:70:2.5ex);
\draw [->, shorten >=-1.5pt, dashed] (tbA') arc (-65:250:2.5ex);
\draw [->,shorten >=-1.5pt, dashed] (cbA') arc (65:-250:2.5ex);

\draw[->] (cB) -- (KB) node[midway, left]{$L_B$};
\draw[->] (cB) -- (cA) node[midway,above]{$\mathfrak{m}$}; 
\draw[->] (cA) -- (KA) node[midway,right]{$L_A$};
\draw[->] (KB) -- (KA) node[midway,above]{$\Aa_{2n-1} \otimes_{\Ba_n} -$};
\end{tikzpicture}
\caption{The commutative diagram of \cref{intromaintheorem}. 
The map $\mathfrak{m}$ is obtained by lifting curves in $\DB$ to multicurves in $\DA$ through the double-branched cover $\DA \twoheadrightarrow \DB$.
The map $L_B$ (resp. $L_A$) is as in \cref{L_Bintro} (resp. \cite[Theorem 4.3]{KhoSei}).}
\label{fig: full picture}
\end{figure}

We would like to mention here that the construction of type $B$ zigzag algebra in this paper can be easily modified to allow for other \emph{Lie-type Dynkin diagrams}, particularly for types $C,F_4$ and $G_2$; for the Dynkin diagrams involving edge label 6 one uses a field extension of degree 3 instead\footnote{In the early writing of this paper, we have made the (arbitrary) choice of using the field extension $\R \subset \C$ for edge label 4. To be able to deal with both edge labels 4 and 6, it may be more natural to use $\Q$ as the base field to allow for both field extensions of degree 2 and 3.}.
Together with the simply-laced constructions \cite{HueKho, LicQuef}, this covers all of the (Lie-type) Dynkin diagrams.
To the best of our knowledge, there is no easy generalisation of this construction via finite-dimensional algebras that encapsulates all \emph{Coxeter diagrams}; not even for the finite types $H$ and $I_2(k)$.
A different construction via algebra objects in fusion categories that allows for arbitrary Coxeter diagrams can be found in the first author's thesis \cite{Heng_PhDthesis}.

  Finally, recall that the type $A_n$ zigzag algebra has a geometric origin: it is quasi-isomorphic as a differential graded algebra (dga) with zero differential to the dga associated to an $A_n$-chain of spherical objects \cite{SeiTho}.
   In particular, it is quasi-isomorphic to the Fukaya $A_\infty$-algebra of a distinguished collection of objects in the Fukaya-Seidel category corresponding to the Milnor fibre of type $A$ singularities \cite{Sei}.
   Type $B$ singularities have also been studied; from the symplectic point of view by Arnold  \cite{Arn1,  Arn2} as boundary singularities and from the algebraic geometry point of view by Slodowy \cite{SlowFour, SlowSim} as simple singularities associated with a $\Z /2 \Z$-group action.
  We expect our type $B$ zigzag algebra to have a similar geometric origin as in type $A$ case.
This is an ongoing work of the second author with Shuaige Qiao.

\subsection*{Outline of the paper}
\cref{topology} contains the topological story of this paper -- the top row of the commutative diagram in \cref{fig: full picture}.
We describe the double-branched cover of $\D^B_{n+1}$ by $\D^A_{2n}$, which induces an injection of groups $\Psi: \mathcal{A}(B_n) \hookrightarrow \mathcal{A}(A_{2n-1})$.
We state the precise definition of curves and admissible curves in this section, and also introduce the notion of trigraded curves -- a type $B$ analogue of bigraded curves for type $A$.
The construction of the $\cA(B_n)$-equivariant map $\mathfrak{m}$, which lifts trigraded curves to bigraded multicurves can be found in \cref{lift section}.

\cref{define zigzag} and  \cref{relating categorical b a action} tell the algebraic story instead -- the bottom row of the commutative diagram in \cref{fig: full picture}.
The definition of type $B$ zigzag algebra $\Ba_n$ and the proof of the corresponding (weak) categorical action of $\mathcal{A}(B_n)$ on $\Kom^b(\Ba_n$-$\text{p$_{r}$g$_{r}$mod})$ can be found in \cref{define zigzag}.
We then relate our type $B$ zigzag algebra $\Ba_n$ to the type $A$ zigzag algebra $\Aa_{2n-1}$ in \cref{relating categorical b a action}, which allows us to obtain the scalar extension functor $\Aa_{2n-1}\otimes_{\Ba_n} -$ and also derive the faithfulness of the $\cA(B_n)$ categorical action.

\cref{main theorem} is where we complete the full picture in \cref{fig: full picture} -- connecting the top and bottom rows.
We recall the $\cA(A_{2n-1})$-equivariant map $L_A$ defined in \cite{KhoSei} and construct the analogous map $L_B$ for type $B$.
This section also contains the proofs that $L_B$ is $\cA(B_n)$-equivariant and that the diagram in \cref{intromaintheorem} commutes, where the latter is the most technical proof of this paper.

\cref{categorification hom rep} contains a ``decategorified'' version of the main theorem (see \cref{decat main theorem} for the corresponding diagram).
Just as the $\cA(A_m)$ action on $\Kom^b(\Aa_m$-$\text{p$_{r}$g$_{r}$mod})$ categorifies the Burau representation (which can be described as a representation on the first homology of an explicit covering space of $\DA$), we show that the categorical action of $\cA(B_n)$ on $\Kom^b(\Ba_n$-$\text{p$_{r}$g$_{r}$mod})$ categorifies a representation on (a submodule of) the first homology of an explicit covering space of $\DB$.

In \cref{int num and hom}, we relate the trigraded intersection numbers of (admissible) curves to the Poincar\'e polynomial of the total Hom spaces of their corresponding complexes.

\subsection*{Acknowledgements}  We would like to thank our supervisor, Anthony Licata, for suggesting this problem and guidance throughout. We would like to acknowledge Peter McNamara for suggesting the construction of type $B$ zigzag algebra $\Ba_n$ during the Kiola Conference 2019. 
%
We would also like to thank Hoel Queffelec and Daniel Tubbenhauer for their helpful comments on the early draft(s) of this paper.
Finally, we would like to thank the referees for their patience in reading this lengthy paper and also for their wonderful suggestions.

\section{Artin Groups of Type $B_n$ and Type $A_{2n-1}$ as Mapping Class Groups}\label{topology}
	
In this section, we will first describe type $A$ and type $B$ Artin groups using generators and relations. 
  After that, we associate these two Artin groups to mapping class groups of surfaces.
  We then introduce trigraded curves and trigraded intersection numbers as  trigraded analogues of bigraded curves and bigraded intersection numbers in \cite[Section 3b]{KhoSei}.
  Finally, we construct a $\cA(B_n)$-equivariant lift of the isotopy classes of trigraded curves to the isotopy classes of bigraded multicurves.

\subsection{Artin groups by generators and relations} \label{genbraid}
An \emph{Artin group} associated to a Coxeter graph $\Gamma$ is a group defined by generators and relations according to the data of the graph $\Gamma$.
In this paper, we shall only concern ourselves with the Artin groups associated to the type $A$ and type $B$ Coxeter graphs.
As such, we shall explicitly define them below, and refer the reader to \cite{BraidGroups} and \cite{ComCox} for a more extensive theory on Artin groups.

For $m \geq 2,$ the type $A_m$ Artin group $\cA({A_m})$ associated to the type $A_m$ Coxeter graph
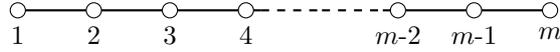
\begin{figure}[H]
\begin{tikzpicture}
\draw[thick] (0,0) -- (1,0) ;
\draw[thick] (1,0) -- (2,0) ;
\draw[thick] (2,0) -- (3.2,0) ;
\draw[thick,dashed] (3.2,0) -- (4,0) ;
\draw[thick,dashed] (4,0) -- (4.8,0) ;
\draw[thick] (4.8,0) -- (6,0) ;
\draw[thick] (6,0) -- (7,0) ;
\filldraw[color=black!, fill=white!]  (0,0) circle [radius=0.1];
\filldraw[color=black!, fill=white!]  (1,0) circle [radius=0.1];
\filldraw[color=black!, fill=white!]  (2,0) circle [radius=0.1];
\filldraw[color=black!, fill=white!]  (3,0) circle [radius=0.1];
\filldraw[color=black!, fill=white!]  (5,0) circle [radius=0.1];
\filldraw[color=black!, fill=white!]  (6,0) circle [radius=0.1];
\filldraw[color=black!, fill=white!]  (7,0) circle [radius=0.1];

\node[below] at (0,-0.1) {1};
\node[below] at (1,-0.1) {2};
\node[below] at (2,-0.1) {3};
\node[below] at (3,-0.1) {4};
\node[below] at (5,-0.1) {$m$-2};
\node[below] at (6,-0.1) {$m$-1};
\node[below] at (7,-0.1) {$m$};
\end{tikzpicture}
\end{figure}
\noindent is the group generated by
$$ \sigma_1^A,\sigma_2^A,\ldots, \sigma_{m}^A $$
\noindent subject to the relations
\begin{align}
      \sigma_j^A \sigma_k^A &= \sigma_k^A \sigma_j^A,  & \text{for} \ |j-k|> 1;\\    
      \sigma_j^A \sigma_{j+1}^A \sigma_j^A &=  \sigma_{j+1}^A \sigma_{j}^A \sigma_{j+1}^A, & \text{for} \ j= 1, 2, \ldots, m-1.
\end{align}
Note that $\cA(A_m)$ is the usual $(m+1)$-strand braid group $\cB r_{m+1}$.

For $n \geq 2,$ the type $B_n$ Artin group $\cA({B_n})$ associated to the type $B_n$ Coxeter graph 

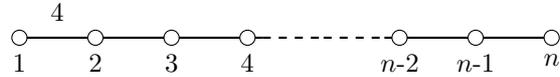
\begin{figure}[H]
\begin{tikzpicture}
\draw[thick] (0,0) -- (1,0) ;
\draw[thick] (1,0) -- (2,0) ;
\draw[thick] (2,0) -- (3.2,0) ;
\draw[thick,dashed] (3.2,0) -- (4,0) ;
\draw[thick,dashed] (4,0) -- (4.8,0) ;
\draw[thick] (4.8,0) -- (6,0) ;
\draw[thick] (6,0) -- (7,0) ;
\filldraw[color=black!, fill=white!]  (0,0) circle [radius=0.1];
\filldraw[color=black!, fill=white!]  (1,0) circle [radius=0.1];
\filldraw[color=black!, fill=white!]  (2,0) circle [radius=0.1];
\filldraw[color=black!, fill=white!]  (3,0) circle [radius=0.1];
\filldraw[color=black!, fill=white!]  (5,0) circle [radius=0.1];
\filldraw[color=black!, fill=white!]  (6,0) circle [radius=0.1];
\filldraw[color=black!, fill=white!]  (7,0) circle [radius=0.1];

\node[below] at (0,-0.1) {1};
\node[below] at (1,-0.1) {2};
\node[below] at (2,-0.1) {3};
\node[below] at (3,-0.1) {4};
\node[below] at (5,-0.1) {$n$-2};
\node[below] at (6,-0.1) {$n$-1};
\node[below] at (7,-0.1) {$n$};
\node[above] at (.5,0.1) {4};
\end{tikzpicture}
\end{figure}
\noindent is the group generated by
$$ \sigma_1^B,\sigma_2^B,\ldots, \sigma_{n}^B $$
\noindent subject to the relations
\begin{align}
\sigma_1^B \sigma_2^B \sigma_1^B \sigma_2^B & = \sigma_2^B \sigma_1^B \sigma_2^B \sigma_1^B;   \\
      \sigma_j^B  \sigma_k^B  &= \sigma_k^B  \sigma_j^B,   & \text{for} \ |j-k|> 1;\\    
      \sigma_j^B  \sigma_{j+1}^B  \sigma_j^B  &=  \sigma_{j+1}^B  \sigma_{j}^B  \sigma_{j+1}^B, & \text{for} \ j= 2,3, \ldots, n-1.
\end{align}

\subsection{Mapping class groups of discs with marked points}  \label{gendefmap}
Suppose $\cS$ is a compact, connected, oriented surface, possibly with boundary $\partial \cS$, and $\Delta \subset \cS \bs \partial \cS $ a finite set of marked points. 
 We denote such a surface as $(\cS,\Delta)$ and we will just write $\cS$ if the associated $\Delta$ is clear from the context. 
 Let $\Delta^{id} \subset \Delta$ be a subset.
 Denote by Diff$(\cS, \partial \cS ; \Delta^{id} )$ as the group of orientation-preserving diffeomorphisms $f:\cS \ra \cS$ with $f|_{\partial \cS \cup \Delta^{id}} = id$ and $f(\Delta) = \Delta$.
 If $\Delta^{id} = \emptyset,$ then we write Diff$(\cS, \partial \cS ) :=$ Diff$(\cS, \partial \cS; \emptyset)$ for simplicity.
 We then define the mapping class group MCG$(\cS, \Delta^{id})$ of the surface $\cS$ with a set $\Delta$ of marked points fixing elements in $\Delta^{id}$ pointwise by 
 $$ \text{MCG} (\cS, \Delta^{id}) := \pi_0 \big( \text{Diff}(\cS, \partial \cS;\Delta^{id}) \big). $$
\noindent In a similar fashion, if $\Delta^{id} = \emptyset,$ we denote the mapping class group of $\cS$ by $ \MCG(\cS) := \text{MCG} (\cS, \emptyset)$.
	We will just write $\MCG(\cS)$ if those conditions are clear from the context.
	The elements of $\MCG(\cS)$ are called \textit{mapping classes}.
%
   We shall see that both Artin groups from \cref{genbraid} appear as mapping class groups, where we refer the reader to \cite{PrimerMCG} for a more detailed exposition on this.
   
\subsubsection{Branched covering of $\D^B_{n+1}$ by $\D^A_{2n}$} \label{brcover}
Consider the following closed disc
$\D^A_{2n} := \{ z \in \C : \|z \| \leq n+1 \}$ embedded in $\C$, equipped with the set $\Delta := \{-n, \hdots,-1,1, \hdots,n\}$  of $2n$ marked points,
as drawn in \cref{Disk}.

Let $r:\D_{2n}^A \ra \D_{2n}^A$ be the half rotation of the disc $\D^A_{2n}$ defined by $r(x) = -x$ for $x \in \D^A_{2n}.$ 
Consider the group $\calR \cong \Z/2\Z$ generated by $r$
and its action on $\D^A_{2n}$.
It is clear that each $x \in \D^A_{2n} \backslash \{0 \}$ has a neighbourhood $U_x$ such that $r(U_x)\cap U_x = \emptyset$.
  In this way, the quotient map $q_{br}: \DA \ra \DA / (\Z/2\Z)$ to its orbit space is a normal branched covering with branched point $\{0\}$ \cite{Piek}.
	From now on, we will denote $\DB$ as the orbit space $\left(\D^A_{2n} \right)/ \left(\Z/2\Z \right)$, and $\Lambda = \{[0], [1], [2], \cdots, [n]\}$ as the set of $n+1$ marked points in $\DB$.
	To simplify notation and to help us picture the orbit space $\DB$, to each equivalence class in $\DB$ we always pick the element with positive real part as the representative whenever possible (i.e. as long as the equivalence class does not contain points on the imaginary line).
This way, we shall abuse notation and denote the set of marked points $\Lambda$ as $\{0, 1, 2, \cdots, n\}$.
\cref{orbitspace} illustrates how we will be picturing $\DB$, where the two oriented green lines are identified.

\begin{figure}[H]  
\begin{subfigure}{0.45 \textwidth}
\centering
\begin{tikzpicture} [scale= 0.5]
\draw (0,0) ellipse (5cm and 2.5cm);
\draw[fill] (1,0) circle [radius=0.045];
\draw[fill] (-1,0) circle [radius=0.045];
\draw[fill] (-4,0) circle [radius=0.045];
\draw[fill] (4,0) circle [radius=0.045];

\node  at (2.5,0) {$ \boldsymbol{\cdots}$};
\node  at (-2.5,0) {$ \boldsymbol{\cdots}$};
\node  [below] at (1,0) {$1$} ;
\node [below] at (-1,0) {$-1$};
\node [below] at (-4,0) {$-n$};
\node [below] at (4,0) {$n$};
\end{tikzpicture}
\caption{{\small The disc $\D^A_{2n}$ with marked points $\Delta$.}} \label{Disk}
\end{subfigure}
\quad
\begin{subfigure}{0.45 \textwidth}
\centering
\begin{tikzpicture}[scale = 0.38]

\draw[thick] plot[smooth, tension=2.25]coordinates {(0,-.5) (12,2.5) (0,5.5)};
\draw[thick,green, dashed]  (0,5.5)--(0,3.875);
\draw[thick, green, dashed]   (0,-.5)--(0,1.125);
\draw[thick, green, dashed, ->]   (0,2.5)--(0,3.875);
\draw[thick,  green, dashed, ->]   (0,2.5)--(0,1.125);
\draw[fill] (1.75,2.5) circle [radius=0.1] ;
\draw[fill] (3.6,2.5) circle [radius=0.1]  ;
\draw[fill] (7.6,2.5) circle [radius=0.1]  ;
\draw[fill] (9.35,2.5) circle [radius=0.1]  ;
\filldraw[color=black!, fill=yellow!, thick]  (0,2.5) circle [radius=0.1];

\node  at (5.6,2.5) {$ \boldsymbol{\cdots}$};
\node[below] at  (0,2.5) {0};
\node[below] at  (1.75,2.5) {1};
\node[below] at  (3.6,2.5) {2};
\node[below] at  (7.6,2.5) {$n-1$};
\node[below] at  (9.35,2.5) {$n$};

\end{tikzpicture}
\caption{{\small The orbit space $\DB := \left(\D^A_{2n}  \right)/ \left(\Z/2\Z \right)$ with marked points $\Lambda$.}} \label{orbitspace}
\end{subfigure}
\caption{The affine configurations of the two discs.}
\end{figure}

\subsubsection{Artin groups as mapping class groups}
By construction, the marked points on $\DA$ and $\DB$ are subsets of $\Z.$  
  Therefore, we enumerate the marked points on the disc by increasing sequences of points.
  Let $\varrho_j$ (resp. $b_j$) be the horizontal curve connecting the $j$-th marked point and $(j+1)$-th marked point in $\DA$ (resp. $\DB$) for $1 \leq j\leq 2n$ (resp. $1 \leq j\leq n+1$).

 The group $\cA(A_{2n-1})$ is isomorphic to the mapping class group MCG($\DA$) of a closed disc $\DA$ with $2n$ marked points.  
The generator $\sigma^A_j$ corresponds to the half twist $[t^A_{\varrho_j}]$ along the arc $\varrho_j.$
 Here, $t^A_{\varrho_j}$ is a diffeomorphism in $\DA$ rotating a small open disc enclosing the $j$-th and $(j+1)$-th marked points anticlockwise by an angle of $\pi$, permuting the two enclosed marked points, whilst leaving all other marked points fixed; see \cref{halftwist}.

Similarly, the group $\cA({B_n})$ is isomorphic to the mapping class group MCG($\D^B_{n+1}, \{0\}$) of a closed disc $\DB$ with $n+1$ marked points, fixing the point $\{0\}$ pointwise.  
The generator $\sigma^B_1$ corresponds to the full twist $[(t^B_{b_1})^2]$ along the arc $b_1$,
 and for $2 \leq j \leq n,$ each generators $\sigma^B_j$ correspond to the half twist $[t^B_{b_j}]$ along the arc $b_j.$
Here, $t^B_{b_j}$ is a diffeomorphism in $\DB$ rotating a small open disc enclosing the $j$-th and $(j+1)$-th marked points by an angle of $\pi$ anticlockwise as illustrated in \cref{halftwist}.  
 As a result, it interchanges the $j$-th and $(j+1)$-th marked points and leaves the other points fixing pointwise.
 	Consequently, $(t^B_{b_1})^2$  is a diffeomorphism rotating a small open disc enclosing the marked points 0 and 1 anticlockwise by an angle of $2 \pi$  leaving all the marked points fixed, as shown in \cref{fulltwist}.

\begin{figure}[H]
\begin{subfigure}{0.45 \textwidth}
\centering
\begin{tikzpicture} [scale= 1] 
\draw (-3,0) circle (1cm);
\draw (1,0) circle (1cm);
\draw [|->] (-1.2,0) -- (-0.6,0);
\draw[->,thick] (-3.5,0)--(-2.5,0);
\draw[fill] (-3.5,0) circle [radius=0.045];
\draw[fill] (-2.5,0) circle [radius=0.045];
\draw[densely dotted, blue] (-4.5,0) -- (-3.5,0);
\draw[densely dotted,red] (-2.5,0) -- (-1.5,0);

\draw [densely dotted, blue] (0.05,0) arc (180:360:0.725);
\draw[densely dotted, blue] (0.05,0) -- (-0.51,0);

\draw [densely dotted, red] (1.95,0) arc (0:180:0.725);
\draw[densely dotted,red] (1.95,0) -- (2.49,0);

\draw [<-|](0.5,0)--(1.5,0);
\draw[fill] (1.5,0) circle [radius=0.045];
\draw[fill] (0.5,0) circle [radius=0.045];

\node [above] at (-3,0) {};
\node [below] at (-3.5,0) { {\scriptsize $j$}};
\node [below] at (-2.5,0) {{\scriptsize $j+1$}};

\node [above] at (1.5,0) { {\scriptsize $j+1$}};
\node [below] at (0.5,0) {{\scriptsize $j$}};

\end{tikzpicture}
\caption{A half twist $t^A_{\varrho_j}$ (similarly $t^B_{b_j}$).} \label{halftwist}
\end{subfigure}
\qquad \quad
\begin{subfigure}{0.45 \textwidth}
\begin{tikzpicture} [scale=0.6]
\draw[thick,green, dashed]  (2.525,7)--(2.525,5.5);
\draw plot[smooth, tension=2.25]coordinates {(2.525,4) (5.25,5.5) (2.525,7)};
\draw[thick, green, dashed]   (2.525,5.5)--(2.525,4);
\draw[thick, green, dashed, ->]   (2.525,5.5)--(2.525,4.5);
\draw[thick,  green, dashed, ->]   (2.525,5.5)--(2.525,6.5);
\draw[->,thick] (2.525,5.5)--(4.25,5.5);
\draw[fill] (4.25,5.5) circle [radius=0.1]  ;
\filldraw[color=black!, fill=yellow!, thick]  (2.525,5.5) circle [radius=0.1];
\draw[densely dotted,red] (4.25,5.5) -- (6,5.5);

\node [above] at (3.3875,5.5) {$b_1$};
\node[below] at  (2.7,5.5) {0};
\node[below] at  (4.4,5.5) {1};

\draw [|->] (7,5.5) -- (8,5.5);

\draw[thick,green, dashed]  (9.525,7)--(9.525,5.5);
\draw plot[smooth, tension=2.25]coordinates {(9.525,4) (12.25,5.5) (9.525,7)};
\draw[thick, green, dashed]   (9.525,5.5)--(9.525,4);
\draw[thick, green, dashed, ->]   (9.525,5.5)--(9.525,4.5);
\draw[thick,  green, dashed, ->]   (9.525,5.5)--(9.525,6.5);
\draw[->,thick] (9.525,5.5)--(11.25,5.5);
\draw[fill] (11.25,5.5) circle [radius=0.1]  ;
\filldraw[color=black!, fill=yellow!, thick]  (9.525,5.5) circle [radius=0.1];
\draw[densely dotted,red] (12,5.5) -- (13,5.5);

\draw[densely dotted,red] plot[smooth, tension=1]coordinates {(12,5.5) (11.4, 6.5) (9.525,6.7)};

\draw[densely dotted,red] plot[smooth, tension=1]coordinates {(9.525,5.5) ( 10.55 ,6.4) (11.75,5.7) (11.1,4.4) (9.525,4.3)};

\node [above] at (10.3875,5.5) {$b_1$};
\node[below] at  (9.7,5.5) {0};
\node[below] at  (11.4,5.5) {1};

\end{tikzpicture}
\caption{{\small A full twist $(t^B_{b_1})^2$.}}\label{fulltwist}
\end{subfigure}
\caption{The twists in $\DA$ and $\DB$.}
\end{figure}

\subsubsection{Injection of $\MCGB$ into $\MCGA$ } \label{InjSec}

A diffeomorphism $f^B$ in $\Diff(\DB,\{0\})$ can be lifted to a unique fiber-preserving diffeomorphism $f^A$ in $\Diff(\DA)$ via the branched covering map $q_{br}.$
Similarly,  an isotopy in $\DB $ can be lifted to an isotopy in $\DA \bsz.$ 
   As such, we have a well-defined map $\Psi$ on the mapping class groups from MCG$\left(\DB, \{0\} \right)\ra$ MCG$ \left(\DA\right)$ defined by lifting the mapping class of $f^B$ to the mapping class of $f^A$. 
   More concretely, using the standard presentation of the groups, $\Psi$ is given by $\sigma^B_1$ mapping to $\sigma^A_n$ and $\sigma^B_{j}$ mapping to $\sigma^A_{n+j-1} \sigma^A_{n-(j-1)}$ for $j \geq 2$.
   In fact, the image of the map $\Psi$ is generated by  fiber-preserving mapping classes in $\MCG_p \left({\DA} \right).$ 
   By \cite[Theorem 1]{BH}, we know that any fibre-preserving diffeomorphism $f^A$ which is isotopic to the identity possesses a fiber-preserving isotopy to the identity, which can then be projected to $\DB$ to get the isotopy $f^B \simeq id$. 
   Therefore, we have the following well-known result:

\begin{proposition} \label{injB}
The homomorphism $\Psi: \MCG \left(\D^B_{n+1}, \{0\} \right)\ra$ MCG$\left(\D^A_{2n} \right)$ defined by
$$ \Psi([t^B_{b_i}]) = \begin{cases} 
      [t^A_{\varrho_n}], & \text{for }i = 1; \\
      \left[t^A_{\varrho_{n+i-1}} t^A_{\varrho_{n-(i-1)}} \right], & \text{for } i \geq 2 \\
   \end{cases}
   $$
is injective.
\end{proposition}

\subsection{Curves and geometric intersection numbers} \label{diskcurves}

Here we collect the definitions of curves and geometric intersection numbers as defined in \cite[Section 3a]{KhoSei}.
Let $(\cS,\Delta)$ be a surface with marked points as in \cref{gendefmap}.
	A \textit{curve} $c$ in $(\cS,\Delta)$ is a subset of $\cS$ that is either a simple closed curve in the interior $\cS^o := \cS \bs (\partial \cS \cup \Delta)$ of $\cS$ and essential (non-nullhomotopic in $\cS^o$), or the image of an embedding $\gamma:[0,1] \ra \cS$ which is transverse to the boundary $\partial \cS $ of $\cS   $ 
with its endpoint lying in $\partial \cS \cup \Delta$, that is, $\gamma^{-1}( \partial \cS \cup \Delta) = \{0,1\}$.
	In this way, our defined curves are smooth and unoriented.
	A \textit{multicurve} in $(\cS,\Delta)$ is the union of a finite collection of disjoint curves in $(\cS,\Delta)$.
	We say two curves $c_0$ and $c_1$ are \emph{isotopic} if there exists an isotopy in $\Diff(\cS, \partial \cS; \Delta)$ deforming one into the other, denoted by $c_0 \simeq c_1.$
	Note that the points on $\partial \cS \cup \Delta$ cannot move during an isotopy.
	Therefore, we can partition all curves in $(\cS,\Delta)$ into isotopy classes of curves. 
	Two multicurves $\fc_0, \fc_1$ are isotopic if they have the same number of disjoint curves, and each curve in $\fc_0$ is isotopic to one and only one curve in $\fc_1$.   
	Two curves $c_0,c_1$ are said to have \emph{minimal intersection} if given two intersection points $z_- \neq z_+$ in $c_0 \cap c_1,$ the two arcs $\alpha_0 \subset c_0$,  $\alpha_1 \subset c_1$  with endpoints $z_- \neq z_+$ such that $\alpha_0 \cap \alpha_1 = \{z_-, z_+\}$ do not form an empty bigon (the bigon contains no marked points) unless $z_-,z_+$ are marked points.
	Two multicurves $\fc_0, \fc_1$ are said to have \emph{minimal intersection} if any two curves $c_0 \subseteq \fc_0$ and  $c_1 \subseteq \fc_1$ have minimal intersection (see \cref{fig: minimal intersection}).

\begin{figure}[H]

\begin{subfigure}{0.2 \textwidth}
\centering
\begin{tikzpicture} [scale= 1] 
\draw[thick] (-.75,-.75) .. controls (-.075, 0) and (.575,0) .. (1.25,-.75) ;
\draw[thick,dotted] (-.5,-.75) .. controls (-.075, -.25) and (.575,-.25) .. (1,-.75)  ;

\draw[thick,dotted] (-1,.75) .. controls (-.5, 0) and (.5,0) .. (1,.75) ;
\draw[thick] (-.5,.75) .. controls (-.5, -.35) and (.55,-.35) .. (.5,.75) ;
\end{tikzpicture}
\caption{}
\end{subfigure}
~
\begin{subfigure}{0.2 \textwidth}
\centering
\begin{tikzpicture} 
\draw[fill] (1,0) circle [radius=0.075];
\draw[thick] (1,0) .. controls (0.05,.65) and (0,.45) .. (-1,-.5) ;


\draw[thick,dotted] (1,0) .. controls (0.5,-.85) and (0,-.65) .. (-1,.5) ;


\end{tikzpicture}
\caption{}
\end{subfigure}
~
\begin{subfigure}{0.2 \textwidth}
\centering
\begin{tikzpicture} 
\draw[fill] (1,0) circle [radius=0.075];
\draw[fill] (-1,0) circle [radius=0.075];
\draw[thick] (1,0) .. controls (.25,.85) and (-.25,.85) .. (-1,0) ;
\draw[thick,dotted] (1,0) .. controls (.25,-.85) and (-.25,-.85) .. (-1,0) ;

\end{tikzpicture}
\caption{}
\end{subfigure}
~
\begin{subfigure}{0.2 \textwidth}
\centering
\begin{tikzpicture} 
\draw[fill] (0,0) circle [radius=0.075];

\draw[thick] (-.75,-0.5) .. controls (0.65,-.25) and (0.65,.25) .. (-.75,0.5) ;
\draw[thick, dotted] (.75,-0.5) .. controls (-0.65,-.25) and (-0.65,.25) .. (.75,0.5) ;


\end{tikzpicture}
\caption{}
\end{subfigure}
\caption{The dotted curves and solid curves belong to different multicurves.
The multicurves in (A) and (B) do not have minimal intersection, whereas the multicurves in (C) and (D) do.}
\label{fig: minimal intersection}
\end{figure}

Let $c_0,c_1$ be curves in $(\cS,\Delta)$ with $c_0 \cap c_1 \cap \partial\cS = \emptyset.$
 Note that we can always find a curve $c'_1 \simeq c_1$ such that $c_0$ and $c'_1$ have minimal intersection.
We define the \emph{geometric intersection number} $I(c_0,c_1) \in \frac{1}{2} \Z$ as follows:

\begin{equation}\label{geoint}
 I(c_0,c_1) =  \left\{
\begin{array}{ll}
      2, & \text{if $c_0,c_1$ are simple closed curves and isotopic;} \\
      |(c_0 \cap c_1')\backslash \Delta | + \frac{1}{2}|(c_0 \cap c_1')\cap \Delta |, & \text{if $c_0 \cap c_1'\cap \partial \cS = \emptyset $.}  
\end{array} 
\right. 
\end{equation} 
\noindent
By \cite[Lemma 3.2]{KhoSei} and \cite[Lemma 3.3]{KhoSei}, the definition is indeed independent of the choice of $c_1'$.
Moreover, note that  the definition above doesn't depend on the orientation of $\cS$ and is symmetric.
    We extend the definition of geometric intersection numbers for multicurves (which do not intersect at $\partial\cS$) by just adding up the geometric intersection numbers of each pair of curves $c_0\subseteq \fc_0$ and $c_1 \subseteq \fc_1.$

      

\subsection{Trigraded curves in $\DB$} \label{tribundle}

In this subsection, we shall extend the notion of bigraded curves and bigraded intersection numbers defined in \cite[Section 3d]{KhoSei} to \emph{trigraded curves} and \emph{trigraded intersection numbers} (see also \cref{bigraded curves}).
 
   Let us remind the reader that we equipped the disc $\DA$ with the set of marked points $\Delta = \{-n, \hdots,-1,1, \hdots,n\}$ and the disc $\DB$ with the set of marked points $\Lambda = \{0,1, \hdots,n\}.$
   Consider another set of marked points $ \Delta_0 = \Delta \cup \{0\}$ in the disc $\DA$.
   Fix the notation as follows: $ \RDB := PT \left(\DB \setminus \Lambda \right),$ and $\RDAz := PT \left(\DA \setminus \Delta_0 \right) $ where $PT(\cdot)$ denotes the real projectivisation of the tangent bundle of the respective discs.
    By taking an oriented trivialisation of its tangent bundle, we can then identify $\RDAz \cong \RP \times \left( \DA\setminus \Delta_0 \right)$.
    In $\DA \setminus \Delta_0$, pick a small loop $\lambda_j$ winding positively around each puncture $j \in \Delta_0$.
  In this way, the classes $[point \times \lambda_j]$ and $[\RP \times point]$ form a basis of $H_1 (\RDAz; \Z).$
 Using the universal coefficient theorem for cohomology \cite[Theorem 3.2]{Hatcher}, we consider the covering space $\wRDAz$ of $\RDAz$ classified by the cohomology class $C_0 \in H^1(\RDAz; \Z \times \Z)$ defined as follows:
	\begin{align}
\label{10} C_0([point \times \lambda_0]) &= (0,0); \\ 
\label{20} C_0([point \times \lambda_j]) &= (-2,1), \quad \text{ for } j= -n, \ldots, -1, 1, \ldots, n; \\
\label{30} C_0([\R \text{P}^1 \times point]) &= (1,0).
    \end{align}

In fact, $\wRDAz$ is a covering for $\RDB$ with group of deck transformation $\Z \times \Z \times \Z / 2\Z$, as explained in the following lemma.

\begin{lemma} \label{finalcovering} ~ 
\begin{enumerate}
\item  Under the action of the rotation group $\calR$ generated by the half rotation $r,$ the quotient map 
$q : \DAz   \ra   \DB \setminus \Lambda  $
 is a normal covering space with deck transformation group $\calR \cong \Z/2 \Z$.
\item The composite $\wRDAz \xra{\fp} \RDAz \xra{\fq} \RDB $
 is a normal covering, where $\fq$ is the normal covering map induced by  the quotient map $q$ on the disc component and the identity map on the $\RP$ component.

\item The group of deck transformations for the covering space $\wRDAz \xra{\fq \circ \fp}  \RDB $ is 
$\Z \times \Z \times \Z / 2\Z.$
\end{enumerate}
\end{lemma}
\begin{proof} The proofs of $(1)$ and $(2)$ are straightforward and we leave them to the reader. 

We will now prove $(3).$
 Since the covering $\fq \circ \fp$ is normal, its deck transformation group $G$ is given by
$G \cong \frac{\pi_1 \left(\RDB \right)} {(\fq_* \circ \fp_*) \left(\pi_1 (\wRDAz ) \right)}.$
Recall $C_0: H_1(\RDAz) \ra \Z \times \Z$ as defined by \eqref{10}--\eqref{30}.
 Let $\overline{C_0} :  \pi_1(\RDAz) \ra \Z \times \Z$ be the map defined by precomposing $C_0$ with the natural quotient map $\pi_1(\RDAz) \twoheadrightarrow H_1(\RDAz).$
 Observe that we have the following commutative diagram of short exact sequences:
 \begin{equation} \label{comsos}
 \begin{tikzcd}[column sep = 7mm,  row sep = 5mm] 
 & 1 \ar[d] &  1 \arrow[d] & 1 \arrow[d] &  \\
 1 \arrow[r] & \pi_1 ( \wRDAz ) \arrow[r, "\fp_*"] \ar[equal]{d}  & \pi_1(\RDAz) \arrow[r, "\overline{C_0}"] \arrow[d, "\fq_*"] & {\Z \times \Z} \arrow[r] \arrow[d, "\wt{\fq_*}"] & 1\\
 1 \arrow[r] & \pi_1 ( \wRDAz ) \arrow[r, "\fq_* \circ \fp_* "] \ar[d] & \pi_1(\RDB) \arrow[r] \arrow[d, "\ell_0 \mapsto 1"'] &  G \arrow[r] \arrow[d] & 1 \\
1 \ar[r] & 0 \ar[r] \ar[d] & \Z / 2 \Z \arrow[d] \ar[equal]{r} & \Z / 2 \Z \arrow[d] \ar[r] & 1\\
& 1 & 1 & 1 &
 \end{tikzcd},
 \end{equation}
where $\Z \times \Z$ is the deck transformation group of the covering $\fp$, $\Z/2\Z$ is the deck transformation group of the covering $\fq$, and $\wt{\fq_*}$ is the map induced by $\fq_*$.
We shall show that the  rightmost column of short exact sequence is left-split; namely we shall construct a map $\wt{\varphi}:G \ra \Z\times \Z$ such that $\wt{\varphi} \circ \wt{\fq_*} = \id$:
\begin{equation} \label{splitsos}
 \begin{tikzcd}[column sep = 7mm,  row sep = 5mm] 
 1 \ar[r] &  \Z\times \Z \ar[r, "\wt{\fq_*}"] & G \ar[l, "\wt{\varphi}", bend left, dashed] \ar[r] & \Z/2\Z \ar[r] & 1
 \end{tikzcd},
\end{equation}
which shows that $G \cong \Z \times \Z \times \Z/2\Z$ as required.

We shall first define a map $\varphi: \pi_1(\RDB) \ra \pi_1(\RDAz)$ and show that
$\overline{C_0}\circ\varphi$ factors uniquely through the quotient $G$, which we will define to be our map $\wt{\varphi}: G \ra \Z\times \Z$.
We pick loops $\lambda_i \subset \DA \setminus \Delta_0$ and $\ell_i \subset \DB \setminus \Lambda$ as in \cref{fig: loops in disk}.

\begin{figure}[H]  
\begin{subfigure}{0.45 \textwidth}
\centering
\begin{tikzpicture} [scale= 0.5]
\draw (0,0) ellipse (7cm and 2.5cm);
\draw[fill] (1.25,0) circle [radius=0.045];
\draw[fill] (-1.25,0) circle [radius=0.045];
\draw[fill] (-5.5,0) circle [radius=0.045];
\draw[fill] (5.5,0) circle [radius=0.045];
\draw[fill] (2.5,0) circle [radius=0.045];
\draw[fill] (-2.5,0) circle [radius=0.045];
\draw[fill] (0,0) circle [radius=0.045];

\node  at (4,0) {$ \boldsymbol{\cdots}$};
\node  at (-4,0) {$ \boldsymbol{\cdots}$};
\node  [above] at (2.5,0.35) {$2$} ;
\node [above] at (-2.5,0.35) {$-2$};
\node  [above] at (1.25,0.35) {$1$} ;
\node  [above] at (0,0.35) {$0$} ;
\node [above] at (-1.25,0.35) {$-1$};
\node [above] at (-5.5,0.35) {$-n$};
\node [above] at (5.5,0.35) {$n$};

\node at (0, -2.5) {$\times$};

\node at (5.85, 0) {{\tiny $\wedge$}};
\node at (2.85, 0) {{\tiny $\wedge$}};
\node at (1.6, 0) {{\tiny $\wedge$}};
\node at (.35, 0) {{\tiny $\wedge$}};
\node at (-.9, 0) {{\tiny $\wedge$}};
\node at (-2.15, 0) {{\tiny $\wedge$}};
\node at (-5.15, 0) {{\tiny $\wedge$}};

\draw (-5.15,0) arc(0:180:0.35);
\draw (-5.85,0) .. controls (-5.7, -1.5) and (-.5,-2.4) .. (0,-2.5);
\draw (-5.15,0) .. controls (-5, -1.5) and (.2,-2.4) .. (0,-2.5) ;

\draw (-2.15,0) arc(0:180:0.35);
\draw (-2.85,0) .. controls (-2.7, -1.5) and (-.5,-2.4) .. (0,-2.5) ;
\draw (-2.15,0) .. controls (-2, -1.5) and (-.5,-2.4) .. (0,-2.5) ;

\draw (-.9,0) arc(0:180:0.35);
\draw (-.9,0) .. controls (-.85, -1.3) and (-.5,-2.4) .. (0,-2.5) ;
\draw (-1.6,0) .. controls (-1.25, -1.5) and (-.5,-2.4) .. (0,-2.5) ;

\draw  (.35,0) arc(0:180:0.35);
\draw (.35,0) .. controls (.25, -1.5) and (.05,-2.4) .. (0,-2.5) ;
\draw (-.35,0) .. controls (-.25, -1.5) and (.05,-2.4) .. (0,-2.5) ;

\draw  (1.6,0) arc(0:180:0.35);
\draw (.9,0) .. controls (.85, -1.3) and (.5,-2.4) .. (0,-2.5) ;
\draw (1.6,0) .. controls (1.25, -1.5) and (.5,-2.4) .. (0,-2.5) ;

\draw  (2.85,0) arc(0:180:0.35);
\draw  (2.85,0) .. controls (2.7, -1.5) and (.5,-2.4) .. (0,-2.5) ;
\draw (2.15,0) .. controls (2, -1.5) and (.5,-2.4) .. (0,-2.5) ;

\draw (5.85,0) arc(0:180:0.35);
\draw (5.85,0) .. controls (5.7, -1.5) and (.5,-2.4) .. (0,-2.5) ;
\draw (5.15,0) .. controls (5, -1.5) and (-.2,-2.4) .. (0,-2.5) ;

\end{tikzpicture}
\end{subfigure}
\quad
\begin{subfigure}{0.45 \textwidth}
\centering
\begin{tikzpicture}[scale = 0.38]

\draw[thick] plot[smooth, tension=2.25]coordinates {(0,-.5) (14,2.5) (0,5.5)};
\draw[thick,green, dashed]  (0,5.5)--(0,3.875);
\draw[thick, green, dashed]   (0,-.5)--(0,1.125);
\draw[thick, green, dashed, ->]   (0,2.5)--(0,3.875);
\draw[thick,  green, dashed, ->]   (0,2.5)--(0,1.125);
\draw[fill] (3,2.5) circle [radius=0.1] ;
\draw[fill] (6,2.5) circle [radius=0.1]  ;

\draw[fill] (11.5,2.5) circle [radius=0.1]  ;
\filldraw[color=black!, fill=yellow!, thick]  (0,2.5) circle [radius=0.1];

\draw [thin] (0,4) .. controls (2, 4.25) and (2,.7) .. (0,-.5)  ;
\draw [thin] (0,1) .. controls (.75, 0.85) and (.5,.15) .. (0,-.5)  ;

\draw (12.25,2.5) arc(0:180:.75);
\draw (12.25,2.5) .. controls (12.1, .15) and (.15,-.45) .. (0,-.5) ;
\draw (10.75,2.5) .. controls (10.6, .15) and (0,-.45) .. (0,-.5) ;

\draw (6.75,2.5) arc(0:180:.75);
\draw (6.75,2.5) .. controls (6.6, .15) and (.15,-.45) .. (0,-.5) ;
\draw (5.25,2.5) .. controls (5.1, .15) and (0,-.45) .. (0,-.5) ;

\draw (3.75,2.5) arc(0:180:.75);
\draw (3.75,2.5) .. controls (3.6, .15) and (.15,-.45) .. (0,-.5) ;
\draw (2.25,2.5) .. controls (2.4, .15) and (0,-.45) .. (0,-.5) ;





\node at (3.75, 2.5) {{\tiny $\wedge$}};
\node at (6.75, 2.5) {{\tiny $\wedge$}};
\node at (12.25, 2.5) {{\tiny $\wedge$}};
\node at (1.5, 2.5) {{\tiny $\wedge$}};

\node  at (9,2.5) {$ \boldsymbol{\cdots}$};
\node[below] at  (0,2.5) {0};
\node[below] at  (3,2.5) {1};
\node[below] at  (6,2.5) {2};
\node[below] at  (11.5,2.5) {$n$};

\end{tikzpicture}
\end{subfigure}
\caption{The loops chosen for the fundamental groups of $\DA \setminus \Delta_0$ (left) and $\DB \setminus \Lambda$ (right).}\label{fig: loops in disk}
\end{figure}

The induced map $q_*: \pi_1(\DA \setminus \Delta_0) \ra \pi_1(\DB \setminus \Lambda)$ on the fundamental groups  satisfies
\[
q_*([\lambda_j]) = 
\begin{cases}
[\ell_0\circ\ell_0], &\text{for } j=0; \\
[(\ell_0 \ell_1 \cdots \ell_{j-1} )\ell_{|j|}(\ell_{j-1}^{-1} \cdots \ell_1^{-1}\ell_0^{-1})], &\text{for } -n \leq j \leq -1; \\
[\ell_j], &\text{for } 1 \leq j \leq n.
\end{cases}
\]
Now define $\varphi: \pi_1(\RDB) \ra \pi_1(\RDAz)$ by sending 
\[
\begin{cases}
[point \times \ell_0] \mapsto [point \times \lambda_0]; \\
[point \times \ell_j] \mapsto [point \times \lambda_j], &\text{ for all } j\in \{1,...,n\}, \text{ and};\\
[\RP \times point] \mapsto [\RP \times point].
\end{cases} 
\]
We claim that $\overline{C_0}\circ\varphi\circ\fq_* = \overline{C_0}$.
Firstly, note that $\overline{C_0}\circ\varphi\circ\fq_*([point \times \lambda_0]) = 0 = \overline{C_0}([point \times \lambda_0])$, and $\varphi\circ\fq_*([\RP \times point]) = [\RP \times point]$ by construction.
Moreover, for $j \in \{-n,...,-1,1,...,n\}$,
\begin{align*}
&(\varphi \circ \fq_* )[point \times \lambda_j] = \\
&\begin{cases}
[point \times \lambda_0\lambda_1\cdots \lambda_{j-1}][point \times \lambda_{|j|}][point \times \lambda_0\lambda_1\cdots \lambda_{j-1}]^{-1}, &\text{for } -n \leq j \leq -1; \\
[point \times \lambda_j], &\text{for } 1 \leq j \leq n.
\end{cases}
\end{align*}
Since $\overline{C_0}$ maps to $\Z \times \Z$, which is abelian, we have that, from \eqref{20},
\[
(\overline{C_0} \circ \varphi \circ \fq_* )[point \times \lambda_j] = \overline{C_0}[point \times \lambda_{|j|}] = \overline{C_0}[point \times \lambda_{j}].
\]
This shows that $\overline{C_0} \circ \varphi \circ \fq_*$ and $\overline{C_0}$ agrees on all generators of $\pi_1 (\RDAz)$, and so are equal.
This implies that
\[
(\overline{C_0} \circ \varphi \circ \fq_*\circ\fp_*) \left(\pi_1(\wRDAz) \right) = (\overline{C_0}\circ\fp_*)\left( \pi_1(\wRDAz) \right) = 0.
\]
As such, $\overline{C_0} \circ \varphi$ factors uniquely through the quotient $G$, and we denote this map by $\wt{\varphi}: G \ra \Z \times \Z$.
By definition, $\wt{\fq_*}$ is uniquely determined by the images of $[point \times \lambda_1],[\R P^1 \times point] \in \pi_1(\RDAz)$ under $\fq_*$.
It is now easy to see that $\wt{\varphi}\circ \wt{\fq_*} = \id$ by the construction of $\wt{\varphi}$.
\end{proof}

	\begin{remark} \label{rmk: classifying cohomology type B}
	Following the proof of \cref{finalcovering} (3), it is easy to see that the covering space  $\wRDAz$ of  $\RDB$ is classified by the cohomology class $C^B \in H^1(\RDB; \Z \times \Z \times \Z / 2\Z)$ defined as follows:
	\begin{align}
 C^B([point \times \ell_0]) &= (0,0,1); \\ 
 C^B([point \times \ell_j]) &= (-2,1,0), \quad \text{ for } 1, \ldots, n; \\
 C^B([\R \text{P}^1 \times point]) &= (1,0,0).
    \end{align}
%
	\end{remark}

   Note that every $f \in$ Diff$ \left(\DB, \{0\} \right)$ preserves the class $C^B$ and therefore can be lifted to a unique equivariant diffeomorphism $\check{f}$ of $\wRDAz$ that acts trivially on the fibre of $\wRDAz$over all points in $T_z\DB$ for $z \in {\partial \DB}$.
   We shall call $\check{f}$ the \emph{preferred lift} of $f$.
   Furthermore, every curve $c$ in $\DB$ admits a canonical section $s_c : c \backslash \Lambda \ra \RDB$ defined by $s_c(z) = T_z c$. 
     We define a \textit{trigrading} of $c$ to be a lift $\check{c}$ of $s_c$ to $\wRDAz$  and a \textit{trigraded curve} to be a pair $(c, \check{c})$ consisting a curve and its trigrading; we will often just write $\check{c}$ instead of $(c, \check{c})$ when the context is clear.
    We denote the $\Z \times \Z \times \Z / 2\Z$-action on $\wRDAz$ by $\chi^B.$
     On top of that, we can easily extend the notion of isotopy to the set of trigraded curves where $\chi^B$ and $ \MCG \left( \DB, \{0\} \right) $ have induced actions on the set of isotopy classes of trigraded curves.
    In particular, for $[f] \in  \MCG\left(\DB, \{0\} \right)$ and a trigraded curve $\check{c}$, $ [\check{f}] (\check{c}) :=  \check{f} \circ \check{c} \circ f^{-1} : f(c) \setminus \Lambda \ra \wRDAz. $

\begin{lemma}\label{freeact} ~
\begin{enumerate}
\item A curve $c$ admits a trigrading if and only if it is not a simple closed curve. 

\item  The $\Z \times \Z \times \Z / 2\Z$-action on the set of isotopy classes of trigraded curves is free. 
	Equivalently, a trigraded curve $\check{c}$ is never isotopic to $\chi^B(r_1,r_2,r_3) \check{c}$ for any $(r_1,r_2,r_3)\neq 0.$
\end{enumerate}

\end{lemma}

 \begin{proof}
    This is essentially the same proof as in \cite[Lemma 3.12 and 3.13]{KhoSei}.
     \end{proof}

\begin{lemma} \label{action} ~
\begin{enumerate}
\item Let $c$ be a curve in $\DB$ which joins two points of $\Lambda \backslash \{0\},$ $t_c \in \MCG \left(\DB, \{0\} \right)$ the half twist along it, and $\check{t}_c$ its preferred lift to $\wRDAz.$
	Then, $\check{t}_c(\check{c}) = \chi^B(-1,1,0) \check{c}$ for any trigrading $\check{c}$ of $c$.
	\item Let $c$ be a curve in $\DB$ which joins two points of $\Lambda$ with one of them being $\{0\},$ $t_c \in \MCG \left(\DB, \{0\} \right)$ the full twist along it, and $\check{t}_c$ its preferred lift to $\wRDAz.$
	Then, $\check{t}_c(\check{c}) = \chi^B(-1,1,1) \check{c}$ for any trigrading $\check{c}$ of $c$.
	
\end{enumerate}     
\end{lemma}

\begin{proof}
The proof of $(1)$ is as in \cite[Lemma 3.14]{KhoSei}.
 We will now prove $(2).$ 
 Let $\beta: [0,1] \ra \DB \setminus \Lambda$ be an embedded smooth path from a point $\beta(0) \in \partial\DB$ to the fixed point $\beta(1) \in c$ of $t_c.$
  Note that we have $\check{t}_c(\check{c}) = \chi(r_1, r_2, r_3) \check{c}$ as $t_c(c)= c.$
  Consider the closed path $\kappa:[0,2] \ra \RDB$ given by 
  $$ \kappa(t) = \begin{cases} 
      Dt_c(\R\beta'(t)), & \text{ if } t \leq 1; \\
      \R\beta'(2-t), & \text{ if } t \geq 1, \\
   \end{cases}
$$
where $\R\beta'(s) \subset T_{\beta(s) }\DB$. The above situation is illustrated in \cref{ft0}.
Let  $C^B \in H^1(\RDB; \Z \times \Z \times \Z / 2\Z)$ be the cohomology class classifying the covering space  $\wRDAz$ of  $\RDB$ as in \cref{rmk: classifying cohomology type B}.
    Then,  we compute 
$(r_1,r_2,r_3) = - C^B([\kappa]) = C^B([\RP \times points]) + C^B([points \times \ell_0] + C^B([points \times \ell_j]  ) = (-1,1,1).$
Note that $[\kappa]$ only picks up one copy of $[\RP\times points]$ due to the disk configuration, or more precisely the oriented trivilisation of $\DB$.
\end{proof}

\begin{figure}[H]  
\begin{tikzpicture} [scale=0.35]

\draw[thick] plot[smooth, tension=2.25]coordinates {(0,-.5) (6.5,2.5) (0,5.5)};
\draw[thick,green, dashed]  (0,5.5)--(0,3.875);
\draw[thick, green, dashed]   (0,-.5)--(0,1.125);
\draw[thick, green, dashed, ->]   (0,2.5)--(0,4);
\draw[thick,  green, dashed, ->]   (0,2.5)--(0,1);
\draw[very thick]   (0,2.5)--(3.6,2.5);
\draw[fill] (3.6,2.5) circle [radius=0.1]  ;
\draw (1.8, -.55) -- (1.8, 2.5);
\draw [dashed] plot[smooth, tension=1.3]coordinates { (3.6,1.35)  (4.8,2.5) (3.6, 3.65)};
\draw [dashed] plot[smooth, tension=1.3]coordinates { (3.6, 3.65) (1.7,3.9) (0, 3.65)};
\draw [dashed] plot[smooth, tension=1.1]coordinates { (0, 1.35) (1.3, 1.65) (1.8, 2.5)};
\draw [dashed] plot[smooth, tension=.85]coordinates { (3.6,1.35)  (2.3,.95) (1.8,-.15)};

\filldraw[color=black!, fill=yellow!, thick]  (0,2.5) circle [radius=0.1];

\node [above] at (0.8, 2.6) {$c$};
\node [left] at (1.8, .9) {{\small $\beta$}};
\node [above] at (1.8,3.9) {{\small $t_c({\beta})$}};
\node [below] at (0.2, 2.5) {$0$};
\node [below] at (3.8, 2.6) {{\small $j$}};
\end{tikzpicture}
\caption{{\small The action of full twist around curve joining $\{0\}$ and another point in $\Lambda.$}} \label{ft0}
\end{figure}

\subsection{Local index and trigraded intersection numbers}\label{subsec: local index trigraded intersect}
Mimicking the definition of local index for bigraded curves in \cite[pg. 225]{KhoSei}, we shall define the local index of an intersection between two trigraded curves. 
 Suppose $(c_0, \check{c_0})$ and $(c_1, \check{c_1})$ are two trigraded curves, and $z \in \DB \setminus \partial \DB$ is a point where $c_0$ and $c_1$ intersect transversally. Take a small circle $\ell \subset \DB \setminus \Lambda$ around $z$ and an embedded arc $\alpha : [0,1] \ra \ell $ which moves clockwise around $\ell$ such that $\alpha(0) \in c_0$ and $\alpha(1) \in c_1$ and $\alpha(t) \notin c_0 \cup c_1$ for all $t \in (0,1).$ 
 If $z \in \Delta,$ then $\alpha$ is unique up to a change of parametrisation, otherwise, there are two choices which can be told apart by their endpoints.
 	Then, take a smooth path $\kappa:[0,1] \ra \RDB$ with $\kappa(t) \in \left(\RDB \right)_{\alpha(t)}$ for all $t,$ going from $\kappa(0)= T_{\alpha(0)}c_0$ to $\kappa(1)= T_{\alpha(1)}c_1$ such that  $\kappa(t) \neq T_{\alpha(t)}\ell$ for every $t$. 
 	One can take $\kappa$ as a family of tangent lines along $\alpha$ which are all transverse to $\ell.$
 	After that, lift $\kappa$ to a path $\check{\kappa}:[0,1] \ra \wRDB$ with $\check{\kappa}(0) = \check{c}_0(\alpha(0));$
 	subsequently, there exists some $(\mu_1, \mu_2, \mu_3) \in \Z \times \Z \times \Z / 2\Z$ such that 
\begin{equation} \label{localpath}
\check{c}_1(\alpha(1)) = \chi^B(\mu_1, \mu_2, \mu_3) \check{\kappa}(1),
\end{equation}
 	 as $\check{c}_1(\alpha(1))$ and $\check{\kappa}(1)$ are the lift of the same point in $\RDB.$ 
 	To this end, we define the local index of $\check{c}_0, \check{c}_1$ at $z$ as 
 	$$ \mu^{trigr}(\check{c}_0, \check{c}_1;z) = (\mu_1, \mu_2, \mu_3) \in \Z \times \Z \times \Z / 2\Z. $$
 	It is easy to see that the definition is independent of all the choices made.
 	
 	The local index has a nice symmetry property similar to \cite[pg. 227]{KhoSei}, namely:

\begin{lemma} \label{locinsym}
If $(c_0,\check{c}_0)$ and $(c_1,\check{c}_1)$ are two trigraded curves such that $c_0$ and $c_1$ have minimal intersection, then 
\[ \mu^{trigr}(\check{c}_1, \check{c}_0;z) = \left\{
\begin{array}{ll}
     (1,0,0) - \mu^{trigr}(\check{c}_0, \check{c}_1;z), & \text{if } z \notin \Delta; \\
    (0,1,0) - \mu^{trigr}(\check{c}_0, \check{c}_1;z), & \text{if } z \in \Delta \backslash \{0\}; \\
     (1,0,1) - \mu^{trigr}(\check{c}_0, \check{c}_1;z), & \text{if } z \in \{0\}.\\
\end{array} 
\right. \]

\end{lemma}
 
\begin{proof}
The first two formulae are essentially the same as in \cite[pg. 227]{KhoSei} and can be proven in a similar fashion, which we omit the details.
 The third formula can be verified using \cref{fig: proof of trigrading lemma}, as the blue path $\ell$ in \cref{fig: proof of trigrading lemma} contributes $[\RP \times point] + [point \times \ell_0].$
\end{proof}
\begin{figure} [H]
\begin{tikzpicture} [scale=0.65]
\draw[thick, green, dashed] (0,1.1)--(0,2);
\draw[ thick, green, dashed, ->] (0,0)--(0,1.1);
\draw [thick, green, dashed, ->] (0,0)--(0,-1.1);
\draw[thick, green, dashed] (0,-1.1)--(0,-2);

\draw [thick] plot[smooth, tension=.5]coordinates { (0, 0) (1.3, 1.15) (2.5, 1.3)};
\draw [thick] plot[smooth, tension=.5]coordinates { (0, 0) (1.3, -1.15) (2.5, -1.3)};
\filldraw[color=black!, fill=yellow!, thick]  (0,0) circle [radius=0.1];
\draw [thick, blue] (0,-1.3) arc (-90:90:1.3);

\draw[thick, brown] (1,0) -- (1.6,0); 
\draw[thick, brown] (0.7,.55) -- (1.3,1); 
\draw[thick, brown] (0.8,.49) -- (1.4,.8); 
\draw[thick, brown] (0.9,.4) -- (1.5,.6); 
\draw[thick, brown] (.95,.3) -- (1.55,.4); 
\draw[thick, brown] (1,.15) -- (1.6,.2); 

\draw[thick, brown] (0.7,-.55) -- (1.3,-1); 
\draw[thick, brown] (0.8,-.49) -- (1.4,-.8); 
\draw[thick, brown] (0.9,-.4) -- (1.5,-.6); 
\draw[thick, brown] (.95,-.3) -- (1.55,-.4); 
\draw[thick, brown] (1,-.15) -- (1.6,-.2); 

\node[right] at (2.5,1.3) {$c_0$};
\node[right] at (2.5,-1.3) {$c_1$};
\node[right] at (1.6,0) {$\kappa$};
\node[above] at (.3,1.2) {$\ell$};
\node[below]  at (-.15,0) {$0$};
\end{tikzpicture}
\caption{{\small Two curves $c_0, c_1$ intersecting at $\{0\}.$}}
\label{fig: proof of trigrading lemma}
\end{figure}

Let $\check{c_0}$ and $\check{c_1}$ be two trigraded curves that do not intersect at $\partial \DB.$
   Pick a curve $c_1' \simeq c_1$ which intersects minimally with $c_0.$  
   Then, by \cref{freeact}, $c_1'$ has a unique trigrading $\check{c}_1'$ of $c_1'$ so that $\check{c}_1' \simeq \check{c}_1.$ 
   We define the \emph{trigraded intersection number} $I^{trigr}(\check{c}_0,\check{c}_1) \in \Z[q_1^{\pm 1}, q_2{\pm 1},q_3]/ \<q_3^2-1\>$  of $\check{c}_0$ and $\check{c}_1$ as follows:
\begin{itemize}
\item if $\check{c}_1 \simeq \chi(r_1, r_2, r_3)\check{c}_0$ with $(r_1,r_2,r_3) \in \Z \times \Z \times \Z / 2\Z $ and $c_0 \cap c_1 \cap \{0 \}$ non-empty, then
\begin{equation}\label{eqn: exceptional trigrad int num}
 I^{trigr}(\check{c}_0,\check{c}_1) = q_1^ {r_1} q_2^ {r_2} q_3^ {r_3} (1 + q_2);
\end{equation}
\item otherwise
\begin{equation} 
\begin{split}
I^{trigr}(\check{c}_0,\check{c}_1) &=
  (1+q_3)(1+q_1^{-1}q_2) \sum_{z \in (c_0 \cap c_1') \backslash \Delta} q_1^ {\mu_1(z)} q_2^ {\mu_2(z)} q_3^ {\mu_3(z) } \\
& + (1+q_3) \sum_{z \in (c_0 \cap c_1') \cap \Delta \backslash \{0\}} q_1^ {\mu_1(z)} q_2^ {\mu_2(z)} q_3^ {\mu_3(z) } \\ 
& + (1+ q_1^ {-1} q_2 q_3) \sum_{z \in (c_0 \cap c_1') \cap\{0\}} q_1^ {\mu_1(z)} q_2^ {\mu_2(z)} q_3^ {\mu_3(z) } .\\
\end{split}
\end{equation}
\end{itemize}
 The fact that this definition is independent of the choice of $c_1'$ and is an invariant of the isotopy classes of  $(\check{c}_0,\check{c}_1)$ follows similarly as in the case of ordinary geometric intersection numbers.
 
\begin{remark}
Note that the exceptional case \eqref{eqn: exceptional trigrad int num} in the definition above is motivated by the algebraic correspondence explored in \cref{int num and hom} (graded HOM space between corresponding irreducible projective modules).
We hope to find a more geometric explanation from symplectic geometry in the near future.
\end{remark}

\begin{lemma}  \label{triintpro}
The trigraded intersection number has the following properties:

\begin{enumerate} [(T1)]
 \item For any $f \in \Diff \left( \DB, \{0 \} \right),$
 $I^{trigr}(\check{f}(\check{c}_0),\check{f}(\check{c}_1)) = I^{trigr}(\check{c}_0,\check{c}_1).$
 \item For any $(r_1,r_2,r_3) \in \Z \times \Z \times \Z / 2\Z,$ \newline 
 $ I^{trigr}(\check{c}_0,\chi(r_1,r_2,r_3)\check{c}_1) = I^{trigr}(\chi(-r_1,-r_2,r_3)\check{c}_0,\check{c}_1) = q_1^ {r_1} q_2^ {r_2} q_3^ {r_3} I^{trigr}(\check{c}_0,\check{c}_1).$
 \item If $c_0, c_1$  are not isotopic curves with $\{0\}$ as one of its endpoints, $c_0 \cap c_1 \cap \partial \DB = \emptyset,$ and $I^{trigr}(\check{c}_0, \check{c}_1) = \sum_{r_1,r_2,r_3 } a_{r_1,r_2,r_3} q_1^{r_1} q_2^{r_2} q_3^{r_3},$ then 
 $I^{trigr}(\check{c}_1, \check{c}_0) = \sum_{r_1,r_2,r_3} a_{r_1,r_2,r_3} q_1^{-r_1} q_2^{1-r_2} q_3^{r_3}.$

\noindent If $c_0, c_1$  are isotopic curves with $\{0\}$ as one of its endpoints, $c_0 \cap c_1 \cap \partial \DB = \emptyset,$ then 
 $I^{trigr}(\check{c}_1, \check{c}_0) = I^{trigr}(\check{c}_0, \check{c}_1).$

\end{enumerate}

\end{lemma}

\begin{proof}
For (T1) and (T2) these can be proven using a simple topological argument which we omit. 
For (T3), this is a consequence of \cref{locinsym}. 
We point out that the term $(1+q_1^{-1}q_2q_3)$ in the definition of trigraded intersection numbers for two curves that intersect at the point $\{0\}$ is essential for property (T3).
\end{proof}

\subsection{Admissible curves and normal form in $\DB$} \label{normal DB}
	Following \cite[Section 3b and 3e]{KhoSei}, we introduce the notion of (trigraded)  admissible curves in $\DB$ and their normal forms.
	Other than the extra consideration of trigradings, the main difference lies in the normal forms of (trigraded) admissible curves in the region containing the marked point $0$; see \Cref{2types0}.
	
	We fix the set of basic curves $b_1, \hdots , b_{n}$ and choose vertical curves $d_1, \hdots, d_{n}$ as in \Cref{bidi}, which divide $\DB$ into regions $D_0, \hdots, D_{n+1}.$
	Note that unlike in \cite{KhoSei}, none of our basic curves touches the boundary of the disc $\DB$.
	
 A curve $c$ is called \bit{admissible} if it is equal to $f(b_j)$ for some $0 \leq j \leq {n}$ and some diffeomorphisms $f \in \Diff \left(\DB, \{0\} \right)$.
 Note that the endpoints of $c$ must then lie in $\{0, \hdots, n\}$; conversely, every curves which start and end at $\{0, \hdots, n\}$ are admissible.
 Moreover, the two (distinct) orbits $\mathcal{O}([b_1])$ and $\mathcal{O}([b_2])$ under the action of $\cA(B_n) \cong MCG(\D^B_{n+1}, \{0\})$ partition the set of isotopy classes of admissible curves.

 If an admissible curve $c$ in its isotopy class has minimal intersection with all the $d_j$'s among its other representatives,  then we say that $c$ is in \bit{normal form}. 
 A normal form of $c$ is always achievable by performing an isotopy.

\begin{figure}[h]  

\begin{tikzpicture} [scale = 0.6]

\draw[thick] plot[smooth, tension=2.25]coordinates {(0,-.5) (12,2.5) (0,5.5)};
\draw[thick,green, dashed]  (0,5.5)--(0,3.875);
\draw[thick, green, dashed]   (0,-.5)--(0,1.125);
\draw[thick, green, dashed, ->]   (0,2.5)--(0,3.875);
\draw[thick,  green, dashed, ->]   (0,2.5)--(0,1.125);
\draw[fill] (1.75,2.5) circle [radius=0.1] ;
\draw[fill] (3.6,2.5) circle [radius=0.1]  ;
\draw[fill] (7.6,2.5) circle [radius=0.1]  ;
\draw[fill] (9.35,2.5) circle [radius=0.1]  ;
\filldraw[color=black!, fill=yellow!, thick]  (0,2.5) circle [radius=0.1];

\draw [thick,blue] (0.875, -.55) -- (0.875, 5.55);
\draw [thick,blue](2.625, -.55) -- (2.625, 5.55);
\draw [thick,blue](4.475, -.5) -- (4.475, 5.5);
\draw [thick, blue] (6.725, -.3) -- (6.725, 5.3);
\draw [thick,blue](8.475, -.05) -- (8.475, 5.05);

\node[right,blue] at (0.9, 4.5) {{\scriptsize $d_1$}};
\node[right,blue] at (2.65, 4.5) {{\scriptsize $d_2$}};
\node[right,blue] at (4.5, 4.5) {{\scriptsize $d_3$}};
\node[right,blue] at (6.75, 4.5) {{\scriptsize $d_{n-1}$}};
\node[right,blue] at (8.5, 4.4) {{\scriptsize $d_n$}};

\draw [thick,red] (1.75,2.5)--(0,2.5);
\draw [thick,red] (1.75,2.5)--(3.6,2.5);
\draw [thick,red] (3.6,2.5)--(4.85,2.5);
\draw [thick,red] (6.35,2.5)--(7.6,2.5);
\draw [thick,red](7.6,2.5) -- (9.35,2.5);

\node[above right,red] at (0.25,2.5) {{\scriptsize $b_1$}};
\node[above right,red] at (2,2.5) {{\scriptsize $b_2$}};
\node[above right,red] at (3.85,2.5) {{\scriptsize $b_3$}};
\node[above left,red] at (7.7,2.5) {{\scriptsize $b_{n-1}$}};
\node[above left,red] at (9.35,2.5) {{\scriptsize $b_n$}};

\node  at (5.6,2.5) {$ \boldsymbol{\cdots}$};
\node[below right] at  (0,2.5) {0};
\node[below] at  (1.75,2.5) {1};
\node[below] at  (3.6,2.5) {2};
\node[below] at  (7.6,2.5) {$n$-1};
\node[below] at  (9.35,2.5) {$n$};

\node at (.4375,3.75) {{\footnotesize $D_0$}};
\node at (.4375,1.25) {{\footnotesize $D_0$}};
\node at (1.75,3.75) {{\footnotesize $D_1$}};
\node at (1.75,1.25) {{\footnotesize $D_1$}};
\node at (3.6,3.75) {{\footnotesize $D_2$}};
\node at (3.6,1.25) {{\footnotesize $D_2$}};
\node at (7.6,3.75) {{\scriptsize $D_{n-1}$}};
\node at (7.6,1.25) {{\scriptsize  $D_{n-1}$}};
\node at (10.5,3.5) {{\footnotesize  $D_n$}};

\end{tikzpicture}

\caption{{\small The curves $b_i$ and $d_i$ in the aligned configuration with regions $D_i.$} } \label{bidi} 
\end{figure}


Let $c$ be an admissible curve in normal form. 
We use the same classification as in \cite[Section 3e]{KhoSei} to group every connected components of $c \cap D_j$ into finitely many types. 
For $1 \leq j \leq n$, the classification is exactly the same as in \cite[Section 3e]{KhoSei}: there are six types for the case $1 \leq j \leq n-1$ as depicted in \cref{sixtype}; whereas for $j=n$, there are two types as shown in \cref{2typesn}.
At $j=0$, we have two possible types as depicted in \cref{2types0}, where they are drawn slightly differently due to the nature of $D_0$; compare Type 2' in \cref{sixtype} and Type 2'' in \cref{2types0}\footnote{Technically, the difference between Type 2' and Type 2'' lies in the their trigradings when we consider trigraded curves later on.}. 
Note that an admissible curve $c$ intersecting all the $d_j$ transversely with each connected component of $c \cap D_j$ belonging to \cref{sixtype}, \cref{2typesn}, and \cref{2types0} is already in normal form.

For the rest of this section, $c$ will be an admissible curve in  normal form. 
  We call the intersections of $c$ with the curves $d_i$ \bit{crossings} and denote them 
$cr(c) = c \cap (d_0 \cup d_1 \cup \hdots \cup d_{n-1}).$
  Those intersections $c \cap d_j$ are called \bit{$j$-crossings of $c$}.
   For $0 \leq j \leq n$, the connected components of $c \cap D_j$ are called \bit{segments} of $c$.
 If the endpoints of a segment are both crossings, then it is \bit{essential}.

\begin{figure}[H] 

\begin{tikzpicture} [scale = 0.75]
\draw[thick] (0,0)--(14,0);
\draw[thick] (0,0)--(0,12);
\draw[thick] (14,0)--(14,12);
\draw[thick] (0,12)--(14,12);
\draw[thick] (7,0)--(7,12);
\draw[thick] (0,4)--(14,4);
\draw[thick] (0,8)--(14,8);
\draw[thick] (0,0.5)--(14,0.5);
\draw[thick] (0,4.5)--(14,4.5);
\draw[thick] (0,8.5)--(14,8.5);

\draw[thick]  (2.25,1)--(2.25,3.5);
\draw[thick]  (4.75,1)--(4.75,3.5);
\draw[fill] (3.50,2.25) circle [radius=0.1]  ;
\draw[thick,red] (2.25,2.25) -- (3.50,2.25);
\node at (3.50, 0.25) {Type 3};

\draw[thick]  (2.25,5)--(2.25,7.5);
\draw[thick]  (4.75,5)--(4.75,7.5);
\draw[fill] (3.50,6.25) circle [radius=0.1]  ;
\draw[thick,red] plot[smooth,tension=1.75] coordinates {(2.25,5.55) (4.15,6.25) (2.25,6.85) };
\node at (3.50, 4.25) {Type 2};

\draw[thick]  (2.25,9)--(2.25,11.5);
\draw[thick]  (4.75,9)--(4.75,11.5);
\draw[fill] (3.50,10.25) circle [radius=0.1]  ;
\draw[thick,red] (2.25,10.85) -- (4.75,10.85);
\node at (3.50, 8.25) {Type 1};

\draw[thick]  (9.25,1)--(9.25,3.5);
\draw[thick]  (11.75,1)--(11.75,3.5);
\draw[fill] (10.50,2.25) circle [radius=0.1]  ;
\draw[thick,red] (11.75,2.25) -- (10.50,2.25);
\node at (10.50, 0.25) {Type 3'};

\draw[thick]  (9.25,5)--(9.25,7.5);
\draw[thick]  (11.75,5)--(11.75,7.5);
\draw[fill] (10.50,6.25) circle [radius=0.1]  ;
\draw[thick,red] plot[smooth,tension=1.75] coordinates {(11.75,5.55) (9.85,6.25) (11.75,6.85) };
\node at (10.50, 4.25) {Type 2'};

\draw[thick]  (9.25,9)--(9.25,11.5);
\draw[thick]  (11.75,9)--(11.75,11.5);
\draw[fill] (10.50,10.25) circle [radius=0.1]  ;
\draw[thick,red] (9.25,9.55) -- (11.75,9.55);
\node at (10.50, 8.25) {Type 1'};

\node[left] at (2.25,2.25) {{\scriptsize (r$_1$,r$_2$,r$_3$)}};
\node[left] at (9.35,9.55) {{\tiny  (r$_1$+1,r$_2$-1,r$_3$)}};
\node[right] at (11.75,9.55) {{\scriptsize (r$_1$,r$_2$,r$_3$)}};
\node[right] at (11.65,5.55) {{\scriptsize (r$_1$,r$_2$,r$_3$)}};
\node[right] at (11.65,6.85) {{\tiny (r$_1$+1,r$_2$-1,r$_3$)}};
\node[right] at  (11.75,2.25) {{\scriptsize (r$_1$,r$_2$,r$_3$)}};
\node[left] at (2.35,6.85) {{\scriptsize (r$_1$,r$_2$,r$_3$)}};
\node[left] at (2.35,5.55) {{\tiny (r$_1$+1,r$_2$-1,r$_3$)}};
\node[left] at (2.25,10.85)  {{\scriptsize (r$_1$,r$_2$,r$_3$)}};
\node[right] at (4.75,10.85) {{\tiny  (r$_1$+1,r$_2$,r$_3$)}};

\end{tikzpicture}
\caption{{\small The six possible types of connected components $c \cap D_j,$ for $c$ in normal form and $1 \leq j < n$.}} 
\label{sixtype}
\end{figure}

\begin{figure}[H] 
\begin{tikzpicture} [scale = 0.75]
\draw[thick] (0,0)--(14,0);
\draw[thick] (0,0)--(0,4);
\draw[thick] (14,0)--(14,4);
\draw[thick] (0,4)--(14,4);
\draw[thick]  (7,0)--(7,4);
\draw[thick] (0,0.5)--(14,0.5);

\draw[thick]  (2.25,1)--(2.25,3.5);
\draw[thick]  (2.25,3.5)--(2,3.5);
\draw[thick]  (2.25,1)--(2,1);
\draw[fill] (3.50,2.25) circle [radius=0.1]  ;
\draw[thick] plot[smooth,tension=2] coordinates {(2.25,3.5) (4.75,2.25) (2.25,1)};
\draw[thick,red] plot[smooth,tension=1.75] coordinates {(2.25,1.55) (4.15,2.25) (2.25,2.85) };
\node at (3.50, 0.25) {Type 2};

\draw[thick]  (9.25,1)--(9.25,3.5);
\draw[thick]  (9.25,3.5)--(9,3.5);
\draw[thick]  (9.25,1)--(9,1);
\draw[thick] plot[smooth,tension=2] coordinates {(9.25,3.5) (11.75,2.25) (9.25,1)};
\draw[fill] (10.50,2.25) circle [radius=0.1]  ;
\draw[thick,red] (9.25,2.25) -- (10.50,2.25);
\node at (10.50, 0.25) {Type 3};

\node [left] at (2.35,1.55) {{\tiny (r$_1$+1,r$_2$-1,r$_3$)}};
\node [left] at (2.35,2.85) {{\scriptsize (r$_1$,r$_2$,r$_3$)}};
\node [left] at (9.25,2.25) {{\scriptsize (r$_1$,r$_2$,r$_3$)}};

\end{tikzpicture}
\caption{{\small The two possible types of connected components $c\cap D_n.$}} \label{2typesn}
\end{figure}

\begin{figure}[H] 
\begin{tikzpicture} [scale = 0.75]
\draw[thick] (0,0)--(14,0);
\draw[thick] (0,0)--(0,4);
\draw[thick] (14,0)--(14,4);
\draw[thick] (0,4)--(14,4);
\draw[thick]  (7,0)--(7,4);
\draw[thick] (0,0.5)--(14,0.5);

\draw[thick]  (4.75,3.5)--(5,3.5);
\draw[thick]  (4.75,1)--(5,1);
\draw[thick]  (3.5,3.5)--(5,3.5);
\draw[thick]  (3.5,1)--(5,1);
\draw[thick,red]  (3.5,2.85)--(4.75,2.85);
\draw[thick,red]  (3.5,1.55)--(4.75,1.55);
\draw[thick]  (4.75,1)--(4.75,3.5);

\draw[thick,green, dashed]  (3.50,3.5)--(3.50,2.875);
\draw[thick, green, dashed]   (3.50,1)--(3.50,1.825);
\draw[thick, green, dashed, ->]   (3.50,2.25)--(3.50,2.9);
\draw[thick,  green, dashed, ->]   (3.50,2.25)--(3.50,1.6);
\filldraw[color=black!, fill=yellow!, very thick] (3.50,2.25) circle [radius=0.1]  ;
\node at (3.50, 0.25) {Type 2''};

\draw[thick]  (11.75,3.5)--(12,3.5);
\draw[thick]  (11.75,1)--(12,1);
\draw[thick]  (11.75,1)--(11.75,3.5);
\draw[thick]  (10.5,3.5)--(12,3.5);
\draw[thick]  (10.5,1)--(12,1);

\draw[thick,green, dashed]  (10.50,3.5)--(10.50,2.875);
\draw[thick, green, dashed]   (10.50,1)--(10.50,1.825);
\draw[thick, green, dashed, ->]   (10.50,2.25)--(10.50,2.9);
\draw[thick,  green, dashed, ->]   (10.50,2.25)--(10.50,1.6);
\filldraw[color=black!, fill=yellow!, very thick] (10.50,2.25) circle [radius=0.1]  ;

\draw[thick,red] (11.75,2.25) -- (10.50,2.25) ;
\node at (10.50, 0.25) {Type 3'};


\node [right] at (11.75,2.25) {{\scriptsize (r$_1$,r$_2$,r$_3$)}};
\node [right] at (4.75,2.85) {{\scriptsize (r$_1$,r$_2$,r$_3$+1)}};
\node[right] at (4.75,1.55) {{\tiny (r$_1$,r$_2$,r$_3$)}};

\end{tikzpicture}
\caption{{\small The two possible types of connected components $c\cap D_0.$ }} \label{2types0}
\end{figure}

  Now, we will study the action of half twist $t_{b_k}^B$ on normal forms.
 In gerenal, $t_{b_k}^B(c)$ won't be in normal form even though $c$ is a normal form.
	Nonetheless,  $t_{b_k}^B(c)$ has minimal intersection with all $d_j$ for $j \neq k.$
In order to get $t_{b_k}^B(c)$ into a normal form, one just need to isotope it so that its intersections with $d_k$ is minimal.
 The same argument used in \cite[Proposition 3.17]{KhoSei} gives us the following analogous result:
\begin{proposition} ~
\begin{enumerate} 
\item The normal form of $t_{b_k}^B(c)$ coincides with $c$ outside of $D_{k-1} \cup D_{k}.$
	The curve $t_{b_k}^B(c)$ can be brought into normal form by an isotopy inside   $D_{k-1} \cup D_{k}.$
	
\item Suppose that $t_{b_k}^B(c)$ is in normal form. 
	There is a natural bijection between $j$-crossings of $c$ and the $j$-crossings of $t_{b_k}^B(c)$ for $j \neq k.$
		There is a natural bijection between connected components of intersections of $c$ and $t_{b_k}^B(c)$ inside $D_{k-1} \cup D_{k}.$
\end{enumerate}
\end{proposition}

A connected component of $c \cap (D_{j-1} \cup D_{j})$ is called \bit{j-string of $c$}.
Denote by $st(c,j)$ the set of $j$-string of $c.$ 
In addition, we define a \bit{j-string} as a curve in  $D_{j-1} \cup D_{j}$ which is a $j$-string of $c$ for some admissible curve $c$ in normal form.

Two $j$-strings are isotopic (equivalently belong to the same isotopy class) if there exists a deformation of one into the other via diffeomorphisms $f$ of $D' = D_{j-1} \cup D_{j}$ which fix $d_{j-1}$ and $d_{j+1}$ as well as preserves the marked points in $D'$ pointwise , that is, $f(d_{j-1})= d_{j-1}, f(d_{j+1})= d_{j+1},$ and $f|_{\Delta \cap D'} = id.$

\begin{figure}[h] 

\begin{tikzpicture} [scale=0.7][yscale=1, xscale=0.92]
\draw[thick] (0,0)--(16,0);
\draw[thick] (0,0)--(0,14);
\draw[thick] (16,0)--(16,14);
\draw[thick] (0,14)--(16,14);
\draw[thick] (8,0)--(8,14);
\draw[thick] (0,4)--(16,4);

\draw[thick] (0,0.5)--(16,0.5);
\draw[thick] (0,3.5)--(16,3.5);
\draw[thick] (0,4)--(16,4);
\draw[thick] (0,7)--(16,7);
\draw[thick] (0,7.5)--(16,7.5);
\draw[thick] (0,10.5)--(16,10.5);
\draw[thick] (0,11)--(16,11);
\draw[thick] (0,14)--(16,14);
\draw[thick] (4,14)--(4,17.5);
\draw[thick] (12,14)--(12,17.5);

\draw[thick] (4,-3.5)--(12,-3.5);
\draw[thick] (4,-3)--(12,-3);
\draw[thick] (4,-3.5)--(4,0);
\draw[thick] (12,-3.5)--(12,0);
\draw[thick] (4,14.5)--(12,14.5);
\draw[thick] (4,17.5)--(12,17.5);

\draw[thick]  (2.525,1)--(2.525,3);
\draw[thick]  (5.475,1)--(5.475,3);
\draw[fill] (3.25,2) circle [radius=0.1]  ;
\draw[fill] (4.75,2) circle [radius=0.1]  ;
\draw[thick,red] plot[smooth,tension=2.25] coordinates {(2.525,1.5)  (5.1125,2) (2.525, 2.5)};
\node at (4, .25) {{\small Type V}};

\draw[thick]  (10.525,1)--(10.525,3);
\draw[thick]  (13.475,1)--(13.475,3);
\draw[fill] (11.25,2) circle [radius=0.1]  ;
\draw[fill] (12.75,2) circle [radius=0.1]  ;
\draw[thick,red] plot[smooth,tension=2.25] coordinates {(13.475,2.5)  (10.8675,2) (13.475, 1.5)};
\node at (12, 0.25) {{\small Type V'}};

\draw[thick]  (2.525,4.5)--(2.525,6.5);
\draw[thick]  (5.475,4.5)--(5.475,6.5);
\draw[fill] (3.25,5.5) circle [radius=0.1]  ;
\draw[fill] (4.75,5.5) circle [radius=0.1]  ;
\draw[thick,red]  (2.525,6)--(5.475,6);;
\node at (4, 3.75) {{\small Type IV}};

\draw[thick]  (10.525,4.5)--(10.525,6.5);
\draw[thick]  (13.475,4.5)--(13.475,6.5);
\draw[fill] (11.25,5.5) circle [radius=0.1]  ;
\draw[fill] (12.75,5.5) circle [radius=0.1]  ;
\draw[thick,red]  (10.525,5)--(13.475,5);;
\node at (12, 3.75) {{\small Type IV'}};

\draw[thick]  (2.525,8)--(2.525,10);
\draw[thick]  (5.475,8)--(5.475,10);
\draw[fill] (3.25,9) circle [radius=0.1]  ;
\draw[fill] (4.75,9) circle [radius=0.1]  ;
\draw[thick,red] (2.525,9) -- (3.25,9);
\node at (4, 7.25) {{\small Type III$_0$}};

\draw[thick]  (10.525,8)--(10.525,10);
\draw[thick]  (13.475,8)--(13.475,10);
\draw[fill] (11.25,9) circle [radius=0.1]  ;
\draw[fill] (12.75,9) circle [radius=0.1]  ;
\draw[thick,red] (12.75,9) -- (13.475,9);
\node at (12, 7.25) {{\small Type III'$_0$}};

\draw[thick]  (2.525,11.5)--(2.525,13.5);
\draw[thick]  (5.475,11.5)--(5.475,13.5);
\draw[fill] (3.25,12.5) circle [radius=0.1]  ;
\draw[fill] (4.75,12.5) circle [radius=0.1]  ;
\draw[thick,red] plot[smooth,tension=1.7] coordinates {(2.525,13)  (4.1,12.5) (2.525, 12)};
\node at (4, 10.75) {{\small Type II$_0$}};

\draw[thick]  (10.525,11.5)--(10.525,13.5);
\draw[thick]  (13.475,11.5)--(13.475,13.5);
\draw[fill] (11.25,12.5) circle [radius=0.1]  ;
\draw[fill] (12.75,12.5) circle [radius=0.1]  ;
\draw[thick,red] plot[smooth,tension=1.7] coordinates {(13.475,13)  (12.1,12.5) (13.475, 12)};
\node at (12, 10.75) {{\small Type II'$_0$}};

\draw[thick]  (6.525,15)--(6.525,17);
\draw[thick]  (9.475,15)--(9.475,17);
\draw[fill] (7.25,16) circle [radius=0.1]  ;
\draw[fill] (8.75,16) circle [radius=0.1]  ;
\draw[thick,red] plot[smooth,tension=.7] coordinates {(6.525,16)  (7.4,16.4)  (8.6, 15.6) (9.475,16)};
\node at (8, 14.25) {{\small Type I$_0$}};

\draw[thick]  (6.525,-2.5)--(6.525,-0.5);
\draw[thick]  (9.475,-2.5)--(9.475,-0.5);
\draw[fill] (7.25,-1.5) circle [radius=0.1]  ;
\draw[fill] (8.75,-1.5) circle [radius=0.1]  ;
\draw[thick,red] (7.25,-1.5)--(8.75,-1.5);
\node at (8, -3.25) {{\small Type VI}};

\node [left] at (6.525,16) {{\scriptsize (r$_1$,r$_2$,r$_3$)}};
\node [right] at (9.475,16) {{\tiny (r$_1$,r$_2$+1,r$_3$)}};
\node [left] at (2.525,13) {{\scriptsize (r$_1$,r$_2$,r$_3$)}};
\node [left] at (2.525, 12) {{\tiny (r$_1$+1,r$_2$-1,r$_3$)}};
\node [right] at (13.475,13) {{\tiny (r$_1$+1,r$_2$-1,r$_3$)}};
\node [right] at (13.475, 12) {{\scriptsize (r$_1$,r$_2$,r$_3$)}};
\node [left] at (2.525,9) {{\tiny (r$_1$,r$_2$,r$_3$)}};
\node [right] at (13.475,9) {{\scriptsize (r$_1$,r$_2$,r$_3$)}};
\node [left] at (2.525,6) {{\scriptsize (r$_1$,r$_2$,r$_3$)}};
\node [right] at (5.475,6) {{\tiny (r$_1$+2,r$_2$,r$_3$)}};
\node [right] at (13.475,5) {{\scriptsize (r$_1$,r$_2$,r$_3$)}};
\node [left] at (10.525,5) {{\tiny (r$_1$+2,r$_2$-2,r$_3$)}};
\node [left] at (2.525, 2.5) {{\scriptsize (r$_1$,r$_2$,r$_3$)}};
\node [left] at (2.525,1.5)  {{\tiny (r$_1$+3,r$_2$-2,r$_3$)}};
\node [right] at (13.475, 1.5)  {{\scriptsize (r$_1$,r$_2$,r$_3$)}};
\node [right] at (13.475,2.5)  {{\tiny (r$_1$+3,r$_2$-2,r$_3$)}};

\end{tikzpicture}
\caption{{\small The isotopy classes of $j$-strings, for $1 < j < n$.}}
\label{B j-string}
\end{figure}

For $1 < j < n,$ isotopy classes of $j$-strings can be divided into types as follows: there are five infinite families $I_w, II_w, II'_w, III_w, III'_w(w \in \Z)$ and five exceptional types $IV, IV',V,V'$ and $VI$ (see \cref{B j-string}). 
	When $j = n$, there is a similar list, with two infinite families and two exceptional types (see \cref{B n-string}).
    The rule for obtaining the $(w+1)$-th from the $w$-th is by applying $t^B_{b_j}.$ 
	For $1$-string, there are instead four infinite-family types: $II'_w, II'_{w + \frac{1}{2}},  III'_w,  III'_{w+ \frac{1}{2}}(w \in \Z)$ and two exceptional types $V''$ and $VI$ (see \cref{B 1-string}).
	As for segments of curves, these are drawn slightly different due to the nature of the disc; compare Type $V'$ in \cref{B j-string} and type $V''$ in \cref{B 1-string}.
	Note that for 1-strings, the rule for obtaining the $(w+1)$-th from the $w$-th is instead by applying $(t^B_{b_1})^2.$

Based on our definition, $j$-strings are assumed to be in normal form. 
	As before, we can define \emph{crossings} and \emph{essential segments} of $j$-string as in the case for admissible curves in normal form and denote the set of crossings of a $j$-string $g$ by $cr(g).$

\begin{figure}[h] 

\begin{tikzpicture} [scale= 0.7][yscale=1, xscale=0.92]
\draw[thick] (0,0)--(16,0);
\draw[thick] (0,0)--(0,7);
\draw[thick] (16,0)--(16,7);
\draw[thick] (8,0)--(8,7);
\draw[thick] (0,4)--(16,4);

\draw[thick] (0,0.5)--(16,0.5);
\draw[thick] (0,3.5)--(16,3.5);
\draw[thick] (0,4)--(16,4);
\draw[thick] (0,7)--(16,7);

\draw[thick] plot[smooth, tension=2.25]coordinates {(2.525,1) (5.475,2) (2.525,3)};
\draw[thick]  (2.525,1)--(2.3,1);
\draw[thick]  (2.525,3)--(2.3,3);
\draw[thick]  (2.525,1)--(2.525,3);
\draw[fill] (3.25,2) circle [radius=0.1]  ;
\draw[fill] (4.75,2) circle [radius=0.1]  ;
\draw[thick,red] plot[smooth,tension=2.25] coordinates {(2.525,1.5)  (5.1125,2) (2.525, 2.5)};
\node at (4, .25) {{\small Type V}};
\node [left] at (2.525, 2.5) {{\scriptsize (r$_1$,r$_2$,r$_3$)}};
\node [left] at (2.525,1.5)  {{\tiny (r$_1$+3,r$_2$-2,r$_3$)}};

\draw[thick]  (10.525,4.5)--(10.525,6.5);
\draw[thick] plot[smooth, tension=2.25]coordinates {(10.525,4.5) (13.475,5.5) (10.525,6.5)};
\draw[thick]  (10.525,4.5)--(10.3,4.5);
\draw[thick]  (10.525,6.5)--(10.3,6.5);
\draw[fill] (11.25,5.5) circle [radius=0.1]  ;
\draw[fill] (12.75,5.5) circle [radius=0.1]  ;
\draw[thick,red] (10.525,5.5) -- (11.25,5.5);
\node at (12, 3.75) {{\small Type III$_0$}};
\node [left] at (10.525,5.5) {{\scriptsize (r$_1$,r$_2$,r$_3$)}};

\draw[thick]  (2.525,4.5)--(2.525,6.5);
\draw[thick] plot[smooth, tension=2.25]coordinates {(2.525,4.5) (5.475,5.5) (2.525,6.5)};
\draw[thick]  (2.525,4.5)--(2.3,4.5);
\draw[thick]  (2.525,6.5)--(2.3,6.5);
\draw[fill] (3.25,5.5) circle [radius=0.1]  ;
\draw[fill] (4.75,5.5) circle [radius=0.1]  ;
\draw[thick,red] plot[smooth,tension=1.7] coordinates {(2.525,6)  (4.1,5.5) (2.525, 5)};
\node at (4, 3.75) {{\small Type II$_0$}};
\node [left] at (2.525,6) {{\scriptsize (r$_1$,r$_2$,r$_3$)}};
\node [left] at (2.525, 5) {{\tiny (r$_1$+1,r$_2$-1,r$_3$)}};

\draw[thick]  (10.525,1)--(10.525,3);
\draw[thick] plot[smooth, tension=2.25]coordinates {(10.525,1) (13.475,2) (10.525,3)};
\draw[thick]  (10.525,1)--(10.3,1);
\draw[thick]  (10.525,3)--(10.3,3);
\draw[fill] (11.25,2) circle [radius=0.1]  ;
\draw[fill] (12.75,2) circle [radius=0.1]  ;
\draw[thick,red] (11.25,2)--(12.75,2);
\node at (12, 0.25) {{\small Type VI}};

\end{tikzpicture}
\caption{{\small The isotopy classes of $n$-strings.}}  \label{B n-string}
\end{figure}

\begin{figure}[h] 

\begin{tikzpicture} [scale=0.7] [yscale=1, xscale=0.92]

\draw[thick] (0,0)--(16,0);
\draw[thick] (0,0)--(0,10.5);
\draw[thick] (16,0)--(16,10.5);
\draw[thick] (8,0)--(8,10.5);

\draw[thick] (0,0.55)--(16,0.55);
\draw[thick] (0,3.5)--(16,3.5);
\draw[thick] (0,4.05)--(16,4.05);
\draw[thick] (0,7)--(16,7);
\draw[thick] (0,7.55)--(16,7.55);
\draw[thick] (0,10.5)--(16,10.5);

\draw[thick]  (5.475,1)--(5.7,1);
\draw[thick]  (5.475,3)--(5.7,3);  
\draw[thick]  (5.475,1)--(5.475,3);
\draw[thick,red]  (3.25,2.5)--(5.475,2.5);
\draw[thick,red]  (3.25,1.5)--(5.475,1.5);
\draw[thick]  (3.25,3)--(5.7,3);  
\draw[thick]  (3.25,1)--(5.7,1);
\draw[thick,green, dashed]  (3.25,3)--(3.25,2.5);
\draw[thick, green, dashed]   (3.25,1)--(3.25,1.5);
\draw[thick, green, dashed, ->]   (3.25,2)--(3.25,2.6);
\draw[thick,  green, dashed, ->]   (3.25,2)--(3.25,1.4);\filldraw[color=black!, fill=yellow!, very thick] (3.25,2) circle [radius=0.1]  ;
\draw[fill] (4.75,2) circle [radius=0.1]  ;
\node [right] at (5.475, 2.5) {{\tiny (r$_1$+2,r$_2$-1,r$_3$+1)}};
\node [right] at (5.475, 1.5) {{\scriptsize (r$_1$,r$_2$,r$_3$)}};
\node at (4, .25) {{\small Type V''}};

\draw[thick]  (13.475,1)--(13.7,1);
\draw[thick]  (13.475,3)--(13.7,3);  
\draw[thick]  (13.475,1)--(13.475,3);
\draw[thick]  (11.25,3)--(13.7,3);  
\draw[thick]  (11.25,1)--(13.7,1);
\draw[thick,green, dashed]  (11.25,3)--(11.25,2.5);
\draw[thick, green, dashed]   (11.25,1)--(11.25,1.5);
\draw[thick, green, dashed, ->]   (11.25,2)--(11.25,2.6);
\draw[thick,  green, dashed, ->]   (11.25,2)--(11.25,1.4);
\filldraw[color=black!, fill=yellow!, very thick] (11.25,2) circle [radius=0.1]  ;
\draw[fill] (12.75,2) circle [radius=0.1]  ;
\draw[thick,red] (11.25,2) -- (12.75,2);
\node at (12, 0.25) {{\small Type VI}};

\draw[thick]  (5.475,4.5)--(5.7,4.5);
\draw[thick]  (5.475,6.5)--(5.7,6.5);  
\draw[thick]  (5.475,4.5)--(5.475,6.5);

\draw[thick]  (3.25,6.5)--(5.7,6.5);  
\draw[thick]  (3.25,4.5)--(5.7,4.5);
\draw[thick,green, dashed]  (3.25,6.5)--(3.25,6);
\draw[thick, green, dashed]   (3.25,4.5)--(3.25,5);
\draw[thick, green, dashed, ->]   (3.25,5.5)--(3.25,6.1);
\draw[thick,  green, dashed, ->]   (3.25,5.5)--(3.25,4.9);

\filldraw[color=black!, fill=yellow!, very thick]  (3.25,5.5) circle [radius=0.1]  ;
\draw[fill] (4.75,5.5) circle [radius=0.1]  ;
\draw[thick,red] (4.75,5.5) -- (5.475,5.5);
\node [right] at (5.475,5.5) {{\scriptsize (r$_1$,r$_2$,r$_3$)}};
\node at (4, 3.75) {{\small Type III'$_0$}};

\draw[thick]  (13.475,4.5)--(13.7,4.5);
\draw[thick]  (13.475,6.5)--(13.7,6.5);  
\draw[thick]  (13.475,4.5)--(13.475,6.5);
\draw[thick]  (11.25,6.5)--(13.7,6.5);  
\draw[thick]  (11.25,4.5)--(13.7,4.5);
\draw[thick,green, dashed]  (11.25,6.5)--(11.25,6);
\draw[thick, green, dashed]   (11.25,4.5)--(11.25,5);
\draw[thick, green, dashed, ->]   (11.25,5.5)--(11.25,6.1);
\draw[thick,  green, dashed, ->]   (11.25,5.5)--(11.25,4.9);
\filldraw[color=black!, fill=yellow!, very thick]  (11.25,5.5) circle [radius=0.1]  ;
\draw[fill] (12.75,5.5) circle [radius=0.1]  ;
\draw[thick,red] plot[smooth,tension=1] coordinates {(11.25,5.5)  (12.225,5.95) (13.475, 6)};
\node [right] at (13.475, 6) {{\scriptsize (r$_1$,r$_2$,r$_3$)}};
\node at (12, 3.75) {{\small Type III'$_\frac{1}{2}$}};

\draw[thick]  (5.475,8)--(5.7,8);
\draw[thick]  (5.475,10)--(5.7,10);  
\draw[thick]  (5.475,8)--(5.475,10);

\draw[thick]  (3.25,10)--(5.7,10);  
\draw[thick]  (3.25,8)--(5.7,8);
\draw[thick,green, dashed]  (3.25,10)--(3.25,9.5);
\draw[thick, green, dashed]   (3.25,8)--(3.25,8.5);
\draw[thick, green, dashed, ->]   (3.25,9)--(3.25,9.6);
\draw[thick,  green, dashed, ->]   (3.25,9)--(3.25,8.4);

\filldraw[color=black!, fill=yellow!, very thick]  (3.25,9) circle [radius=0.1]  ;
\draw[fill] (4.75,9) circle [radius=0.1]  ;
\draw[thick,red] plot[smooth,tension=1.7] coordinates {(5.475,9.5)  (4.1,9) (5.475, 8.5)};

\node [right] at (5.475, 9.5) {{\tiny (r$_1$+1,r$_2$-1,r$_3$)}};
\node [right] at (5.475, 8.5) {{\scriptsize  (r$_1$,r$_2$,r$_3$)}};
\node at (4, 7.25) {{\small Type II'$_0$}};

\draw[thick]  (13.475,8)--(13.7,8);
\draw[thick]  (13.475,10)--(13.7,10); 
\draw[thick]  (13.475,8)--(13.475,10);
\draw[thick,red]  (11.25,9.65)--(13.475,9.65);
\draw[thick]  (11.25,10)--(13.7,10);  
\draw[thick]  (11.25,8)--(13.7,8);
\draw[thick,green, dashed]  (11.25,10)--(11.25,9.5);
\draw[thick, green, dashed]   (11.25,8)--(11.25,8.5);
\draw[thick, green, dashed, ->]   (11.25,9)--(11.25,9.6);
\draw[thick,  green, dashed, ->]   (11.25,9)--(11.25,8.4);
\filldraw[color=black!, fill=yellow!, very thick]  (11.25,9) circle [radius=0.1]  ;
\draw[fill] (12.75,9) circle [radius=0.1]  ;

\draw[thick,red] plot[smooth,tension=1] coordinates { (11.25,8.35) (11.7,8.5) (12.5,9.2) (13.475,9.35)};

\node [right] at (13.475,9.65) {{\tiny (r$_1$,r$_2$,r$_3$+1)}};
\node [right] at (13.475,9.3) {{\scriptsize (r$_1$,r$_2$,r$_3$)}};
\node at (12, 7.25) {{\small Type II'$_\frac{1}{2}$}};

\end{tikzpicture}
\caption{{\small The isotopy classes of $1$-strings.}}
\label{B 1-string}
\end{figure}

Now, let us adapt the discussion to trigraded curves. Choose trigradings $\check{b}_j, \check{d}_j$ of $b_j,d_j$ for $1 \leq j \leq n,$ such that
\begin{equation} \label{basic trigrading condition}
I^{trigr}(\check{d}_j,\check{b}_j) = (1+ q_3)(1+q_1^{-1}q_2), \qquad
 I^{trigr}(\check{b}_j,\check{b}_{j+1}) = 1 + q_3.
\end{equation}
These conditions determine the trigradings uniquely up to an overall shift $\chi^B(r_1,r_2,r_3).$

Suppose $\check{c}$ is a trigrading of an admisible curve $c$ in normal form. 
	If $a \subset c$ is  a connected component of $c \cap D_j$ for some $j$, and $\check{a}$ is $\check{c}|_{a\setminus \Lambda}$, then $\check{a}$ is evidently determined by $a$ together with the local index $\mu^{trigr}(\check{d}_{j-1}, \check{a};z)$ or $\mu^{trigr}(\check{d}_{j}, \check{a};z)$ at any point $z \in (d_{j-1} \cup d_j) \cap a.$
	Moreover, if there is more than one such point, the local indices determine each other.
	
	In \cref{sixtype}, \cref{2typesn}, and \cref{2types0}, we classify the types of pair $(a,\check{a})$ with the local indices. 
	For instance, consider the Type $1(r_1, r_2, r_3)$ with $(k-1)$-crossing $z_0 \in d_{k-1} \cap a$ and $k$-crossing $z_1 \in d_{k} \cap a$.
	We have that the local index at $z_0$ and at $z_1$ are $\mu^{trigr}(\check{d}_{k-1}, \check{a};z_0) = (r_1, r_2, r_3)$ and $\mu^{trigr}(\check{d}_{k}, \check{a};z_1) = (r_1 +1, r_2, r_3)$ respectively.

We recall from \cref{tribundle} that there is a preferred lift $\check{f} \in \Diff(\wRDAz)$ of $f \in \DiffB$ which acts as the identity on the boundary.
   Denote by $\check{t}_{b_j}$ the preferred lift of the twist $t_{b_j}$ along the curve $b_j$ in $\DB$ (see \Cref{bidi}).
\begin{proposition}
The diffeomorphisms $\check{t}_{b_j}$ induce a type $B_n$ braid group action on the set of isotopy classes of admissible trigraded curves. Namely, if $\check{c}$ is an admissible trigraded curve, we have the following isotopy relations:
\begin{align*}
\check{t}_{b_1} \check{t}_{b_2} \check{t}_{b_1} \check{t}_{b_2}(\check{c}) &\simeq \check{t}_{b_2} \check{t}_{b_1} \check{t}_{b_2} \check{t}_{b_1} (\check{c})  \\
     \check{t}_{b_j} \check{t}_{b_k}(\check{c})  &\simeq \check{t}_{b_k} \check{t}_{b_j} (\check{c})  & \text{for} \ |j-k|> 1,\\    
     \check{t}_{b_j} \check{t}_{b_{j+1}} \check{t}_{b_j} ( \check{c})  &\simeq \check{t}_{b_{j+1}} \check{t}_{b_j} \check{t}_{b_{j+1}} ( \check{c})  & \text{for} \ j= 2,3, \ldots, n.
\end{align*} 
\end{proposition}

A crossing of $c$ will be also a \emph{crossing of $\check{c}$}, and we denote the set of crossings of $\check{c}$ by $cr(\check{c}).$
 Note that as set, $cr(\check{c}) = cr(c).$ 
 However, a crossing of $\check{c}$ comes with a local index in $\Z \times \Z \times \Z/ 2\Z.$
 
	Moreover, to each crossing $y$ of $\check{c}$ we assign a $4$-tuple $(y_0, y_1, y_2, y_3)$ where $y_0$ denotes the index of the vertical curve which contains the crossing $y \in d_{y_0} \cap c$, and $(y_1, y_2, y_3)$ is the local index $(\mu_1, \mu_2,\mu_3)$ of the crossing $y.$  \label{local index function}
	
	We define the \emph{essential segments} of $\check{c}$ as the essential segments of $c$ together with the trigradings which can be obtained from  local indices assigned to the ends of the segments.
	
	We also define a \emph{$j$-string of $\check{c}$} as a connected component of $\check{c} \cap \left( D_{j-1} \cup D_{j} \right)$ together with the trigrading induced from $\check{c}.$ 
	Denote the set of $j$-string of $\check{c}$ by $st(\check{c},j).$
	
	On top of that, we define a trigraded $j$-string as a trigraded curve in $D_{j-1} \cup D_{j}$ that is a connected component of $\check{c} \cap (D_{j-1} \cup D_{j})$ for some trigraded curve $\check{c}.$ 
	
	In \cref{B j-string}, \cref{B n-string}, and \cref{B 1-string} we depict the isotopy classes of trigraded $j$-strings. 
	Since $j$-strings of type $VI$ do not intersect with $d_{j-1} \cup d_{j+1},$ we say that a trigraded $j$-string $\check{g}$ with the underlying $j$-string $g$ of type $VI$ has type $VI(r_1, r_2, r_3)$ if $\check{g}= \chi^B(r_1, r_2, r_3) \check{b}_j.$

     The next crucial lemma is the type $B$ analogue of \cite[Lemma 3.20]{KhoSei}, allowing the computation of trigraded intersection numbers between $\check{b}_j$ and any given trigraded curve. 

\begin{lemma} \label{compute tri int}
Let $(c, \check{c})$ be a trigraded curve. 
Then, $I^{trigr}(\check{b}_j, \check{c})$ can be computed by adding up contributions from each trigraded $j$-string of $\check{c}.$
 For $j > 1,$ the contributions are listed in the following table:
\begin{center}
\begin{tabular}{ c|c|c|c} 

 $I_0(0,0,0)$ & $II_0(0,0,0)$ & $II'_0(0,0,0)$ & $III_0(0,0,0)$  \\ 
 \hline 
$q_1 + q_2 + q_2q_3 + q_1q_3 $ & $q_1 + q_2 + q_2q_3 + q_1q_3 $ & $1 + q_1q_2^{-1} + q_3 + q_1q_2^{-1} q_1q_3 $  & $q_2 + q_2q_3$  \\

\end{tabular}
\end{center}
\begin{center}
\begin{tabular}{ c|c|c|c|c|c} 
 $III'_0(0,0,0)$ & $IV$  & $IV'$ & $V$ &  $V'$  & $VI(0,0,0)$\\
 \hline
 $1 + q_3$ & 0 & 0 & 0 & 0 & $1+q_2 +q_3 + q_2q_3$

\end{tabular}
\end{center}
and the remaining ones can be computed as follows: to determine the contribution of a component of type, say, $I_u(r_1,r_2,r_3)$, one takes the contribution of $I_0(0,0,0)$ and multiplies it by $q_1^{r_1}q_2^{r_2}q_3^{r_3}(q_1^{-1}q_{2})^u.$
For $j=1,$ the relevant contributions are
\begin{center}
\begin{tabular}{ c|c|c|c|c|c} 
  $II'_0(0,0,0)$ &  $II'_\frac{1}{2}(0,0,0)$ & $III'_0(0,0,0)$ & $III'_\frac{1}{2}(0,0,0)$ & $V''$  & $VI(0,0,0)$\\ 
 \hline 
$ 1+ q_3 + q_1q_2^{-1} + q_1q_2^{-1}q_3$ & $ 1+ q_3 + q_1^{-1}q_2 + q_1^{-1}q_2q_3$ & $1 + q_3 $  & $q_1^{-1}q_2 + q_3$ & 0 & $1 + q_2$ 
\end{tabular}
\end{center}
and the remaining ones can be computed as follows: to determine the contribution of a component of type, say, $II'_u(r_1,r_2,r_3)$, one takes the contribution of $II'_0(0,0,0)$ and multiplies it by $q_1^{r_1}q_2^{r_2}q_2^{r_2}(q_1^{-1}q_2q_3)^u.$

\end{lemma}

\begin{proof}
Apply \cref{action} as well as $(T2)$ and $(T3)$ of \cref{triintpro}.
\end{proof}

\subsection{Bigraded curves and bigraded multicurves  in $\wRDAz$} \label{bigraded curves}
We briefly remind the reader the definition of a bigraded curve in $\wRDAz$; refer to \cite[Section 3d]{KhoSei} for a more detailed construction.
	Consider the projectivisation  $\RDA := PT \left(\DA \setminus \Delta \right)$ of the tangent bundle of $\DA \setminus \Delta.$ 
	 The covering $\wRDA$ of $\RDA$ is classified by
	the cohomology class $C^A \in H^1(\RDA; \Z \times \Z)$ defined as follows:
	\begin{align}
 C^A([point \times \lambda_i]) &= (-2,1), \quad \text{for } i= -n, \cdots, -1, 1, \cdots, n; \\
 C^A([\R \text{P}^1 \times point]) &= (1,0).
\end{align}
A bigrading of a curve $c \in \DA$ is a lift $\ddot{c}$ of $s^A_c $ to $\wRDA$ where  $s^A_c: c \setminus \Delta \ra \RDA$ is the canonical section given by $s^A_c(z) = T_zc.$ 
  A \emph{bigraded curve} is a pair $(c, \ddot{c}),$ where sometimes we  abbreviate as $\ddot{c}.$  
  
  A \emph{bigraded multicurve} $\ddot{\fc}$ is a union of a finite collection of disjoint bigraded curves.
  There is an obvious notion of isotopy for bigraded multicurves.

\subsection{Lifting of trigraded curves to bigraded multicurves}\label{lift section}
Our goal is to define a map  $\m: \check{\fC} \ra \ddot{\undertilde{\wt{\fC}}}$ from the set $\check{\fC}$ of isotopy classes of trigraded curves to the set $\ddot{\undertilde{\wt{\fC}}}$ of isotopy classes of bigraded multicurves.
Let $c $ be a curve in $\D^B_{n+1}$ with trigrading $\check{c}$.
First consider the case when $c \cap \{0\} = \emptyset.$
Recall the map $q_{br}: \DA \ra \DB $ as defined in \cref{brcover}. 
Then,  $q_{br}^{-1}{(c)}$ has two connected components in $\DA$; denote them as $\widetilde{c}, \undertilde{c}$, such that  $\wt{c} \setminus \Delta $ agrees with the curve component of  $ \fp \circ \check{c}(c \setminus \Lambda)$ and $\ut{c} \setminus \Delta $ agrees with the curve component of  $\fp \circ \chi^B(0,0,1)\check{c}(c \setminus \Lambda).$
Define $\ddot{\wt{c}}: \wt{c}\setminus \Delta \ra \wRDA$ as 
$\ddot{\widetilde{c}} := \wt{\sF} \circ \check{c}  \circ q_{br}|_{\wt{c} \setminus \Delta};$
similarly $\ddot{\ut{c}} : \ut{c}\setminus \Delta \ra \wRDA $ is defined by 
$\ddot{\widetilde{c}} := \wt{\sF} \circ \chi^B(0,0,1) \check{c} \circ q_{br}|_{\wt{c} \setminus \Delta}$
where $\wt{\sF}: \wRDAz \ra \wRDA$ is the unique map induced by the inclusion $\sF: \RDAz \ra \RDA.$
It is easy to check that these are indeed bigradings of the respective curves.
     On the other hand,  if $c$ contains $0$ as one of its endpoints, we define $\widetilde{\undertilde{c}} := \wt{ c \setminus \{0\}} \amalg \{0\} \amalg \ut{c \setminus \{0\}} $, which is just a single connected component. 
     Furthermore, $\ddot{\wuc}$  is defined to be the unique continuous extension of  $\ddot{\widetilde{c \setminus \{0\}}} \amalg \ddot{\ut{c \setminus \{0\}}}$, which is again an easy verification that it is a bigrading of $\wuc.$

	 In total, we define the map $\m: \check{\fC} \ra \ddot{\undertilde{\wt{\fC}}}$  as follows: for a trigraded curve $\left(c, \check{c} \right)$ in $\wRDAz,$
    \[ \m(\left(c, \check{c} \right))  :=  \begin{cases} 
      (\wuc, \ddot{\tilde{\undertilde{c}}} ),  & \text{ if $c $ has $\{0\}$ as one of its endpoints;} \\
      (\wt{c}, \ddot{\tilde{c}}) \amalg (\undertilde{c},\ddot{\undertilde{c}}), & \text{ otherwise.}\\
       
   \end{cases}
\]
  Due to the isotopy lifting property of the space, $\m$ is well-defined on the isotopy classes of trigraded curves.

Recall the natural induced action of $\cA(B_n) \cong$ MCG$(\DB, \{0\} )$ on $\check{\fC}$ given in the paragraph before \cref{freeact}.
 Since $\mathcal{A}(A_{2n-1}) \cong $ MCG$( \DA )$ acts on $\ddot{\undertilde{\wt{\fC}}} $ \cite[Proposition 3.19]{KhoSei},  there exists an induced action of $\cA(B_n)$ on $\ddot{\undertilde{\wt{\fC}}} $ through the injection $\Psi$ as given in \cref{injB}.
\begin{proposition} \label{topological equivariant}

The map 
$\m : \check{\fC} \ra  \ddot{\undertilde{\wt{\fC}}}$ from isotopy classes of trigraded curves in $\wRDAz$  to  isotopy classes of bigraded multicurves in $\wRDA$
is $\cA({B_{n}})$-equivariant:
\begin{center}
\begin{tikzpicture} [scale=0.8]
\node (tbB) at (-3,1.75) 
	{$\mathcal{A}(B_n)$};
\node (tbA) at (10.5,1.75) 
	{$\mathcal{A}(A_{2n-1}) \overset{\Psi}{\hookleftarrow} \mathcal{A}(B_n) $};

\node[align=center] (cB) at (0,0) 
	{Isotopy classes $\check{\fC}$ of\\ trigraded curves  in $\DB$};
\node[align=center] (cA) at (6.5,0) 
	{Isotopy classes $\ddot{\undertilde{\wt{\fC}}}$ of\\ bigraded multicurves  in $\DA$};

\coordinate (tbB') at ($(tbB.east) + (0,-1)$);

\coordinate (tbA') at ($(tbA.west) + (0,-1)$);

\draw [->,shorten >=-1.5pt, dashed] (tbB') arc (245:-70:2.5ex);

\draw [->, shorten >=-1.5pt, dashed] (tbA') arc (-65:250:2.5ex);

\draw[->] (cB) -- (cA) node[midway,above]{$\mathfrak{m}$}; 

\end{tikzpicture}.
\end{center}
\end{proposition} 
\begin{proof} This follows directly from the definition $\mathfrak{m}$ and the actions.
\end{proof}

\subsection{Bigraded intersection number and bigraded admissible multicurves in $\wRDA$}\label{sect: bigraded intersection number}

The local index for bigraded curves in $\wRDA$ is defined in the same spirit as the local index for trigraded curves in $\wRDAz.$ 
 For a more detailed explanation, we refer the reader to \cite[Section 3d]{KhoSei}. 

We recall from \cite[Section 3d]{KhoSei} that the bigraded intersection number of two bigraded curves $\ddot{c}_0, \ddot{c}_1$ that do not intersect at $\partial \DA$  is defined by 
\begin{equation} \label{bigrint}
I^{bigr} (\ddot{c}_0, \ddot{c}_1) = 
  (1+q_1^{-1}q_2) \left( \ \sum_{z \in (c_0 \cap c_1') \backslash \Delta} q_1^ {\mu_1(z)} q_2^ {\mu_2(z)} \ \right) 
 + \left( \ \sum_{z \in (c_0 \cap c_1') \cap \Delta } q_1^ {\mu_1(z)} q_2^ {\mu_2(z)} \ \right)  .
\end{equation}
We extend the definition of bigraded intersection number of bigraded curves to bigraded multicurves by adding up the bigraded intersection numbers of each pair of bigraded curves.

 To talk about bigraded admissible curves in $\DA$ and their normal forms,
we need to fix a set of basic bigraded curves.
To do so, first recall the set of trigraded basic curves $(b_j, \check{b}_j)$ and the set of trigraded vertical curves $(d_j, \check{d}_j)$ as defined in \cref{normal DB}.
Denote this set of basic trigraded curves as $\check{\fB}$.
Consider,  for each $(c,\check{c}) \in \check{\fB}$, its lift to bigraded multicurves in $\DA$:
\[
\m(c, \check{c})  =  \begin{cases} 
      (\wuc, \ddot{\tilde{\undertilde{c}}} ),  & \text{ if $(c, \check{c}) = (b_1, \check{b}_1) $;} \\
      (\wt{c}, \ddot{\tilde{c}}) \amalg (\undertilde{c},\ddot{\undertilde{c}}), & \text{ otherwise},\\
   \end{cases}
\]
	where $\wt{c}$ denotes the curve whose points have positive real parts; so points in $\ut{c}$ have negative real parts. 
We shall fix the set of bigraded basic curves $(\varrho, \ddot{\varrho}_j)$ and bigraded vertical curves $(\theta_j, \ddot{\theta}_j)$ as follows:
	\begin{itemize}
	\item choose $(\theta_n, \ddot{\theta}_n) := (\wt{d}_1, \ddot{\wt{d}_1} ) ;$
	\item choose $(\theta_{n+j-1}, \ddot{\theta}_{n+j-1}) := (\wt{d}_j, \ddot{\wt{d}_j} ) $ and $(\theta_{n-j+1}, \ddot{\theta}_{n-j+1}) := (\ut{d_j}, \ddot{\ut{d_j}} ),  $ \quad for $2 \leq j \leq n;$
	\item choose $( \varrho_{n+j-1}, \ddot{\varrho}_{n+j-1}) := (\wt{b}_j, \ddot{\wt{b}}_j ) $ and $(\varrho_{n-j+1}, \ddot{\varrho}_{n-j+1}) := (\ut{b_j}, \ut{\ddot{b}_j} ),  $ \quad for $2 \leq j \leq n;$
	\item choose $( \varrho_{n}, \ddot{\varrho}_{n}) := (\ut{\wt{b}_1}, \ut{\ddot{\wt{b}}_1}).$
	\end{itemize}
\Cref{NormalformA} illustrates the (underlying) basic curves $\varrho_j$ and vertical curves $\theta_j$ chosen.
\begin{figure}[H]  
\begin{tikzpicture}  [xscale=1.5, yscale=1]
\draw[thick] (0,0) ellipse (4cm and 2cm);
\draw[fill] (-3.3,0) circle [radius=0.045];
\node  at (-2.4,0) {$ \boldsymbol{\cdots}$};
\draw[fill] (-1.5,0) circle [radius=0.045];
\draw[fill] (-0.75,0) circle [radius=0.045];
\draw[fill] (.75,0) circle [radius=0.045];
\draw[fill] (1.5,0) circle [radius=0.045];
\node  at (2.4,0) {$ \boldsymbol{\cdots}$};
\draw[fill] (3.3,0) circle [radius=0.045];

\node [below] at (-3.3,0) {-$n$};
\node  [below] at (-1.5,0) {-$2$} ;
\node [below] at (-0.75,0) {-$1$};
\node [below] at (.75,0) {$1$};
\node  [below] at (1.5,0) {$2$} ;
\node [below] at (3.3,0) {$n$};

\draw [thick,blue] (-2.925, 1.35) -- (-2.925, -1.35);
\draw [thick,blue] (-1.875, 1.75) -- (-1.875, -1.75);
\draw [thick,blue](-1.125, 1.93) -- (-1.125, -1.93);
\draw [thick, blue] (0.375, 2) -- (0.375, -2);
\draw [thick,blue](1.125, 1.93) -- (1.125, -1.93);
\draw [thick,blue] (1.875, 1.75) -- (1.875, -1.75);
\draw [thick,blue] (2.925, 1.35) -- (2.925, -1.35);

\node at (-3.4,0.55) {$\cD_0$};
\node at (3.4,0.55) {$\cD_{2n}$};
\node at (-.6,0.75) {$\cD_{n-1}$};
\node at (-.6,-0.75) {$\cD_{n-1}$};
\node at (.6,0.75) {$\cD_{n}$};
\node at (.6,-0.75) {$\cD_{n}$};
\node at (1.5,0.7) {$\cD_{n+2}$};
\node at (1.5,-0.75) {$\cD_{n+2}$};
\node at (-1.5,0.7) {$\cD_{n-2}$};
\node at (-1.5,-0.75) {$\cD_{n-2}$};

\node[left,blue] at (-2.925, 1) {$\theta_1$};
\node[left,blue] at (-1.875, 1.3) {$\theta_{n-2}$};
\node[left,blue] at (-1.125, 1.4) {$\theta_{n-1}$};
\node[left,blue] at (0.375, 1.5) {$\theta_n$};
\node[left,blue] at (1.125, 1.4) {$\theta_{n+1}$};
\node[left,blue] at (1.875, 1.3) {$\theta_{n+2}$};
\node[left,blue] at (2.925, 1.1) {$\theta_{2n-1}$};

\draw [thick,red] (-3.3,0)--(-2.8,0);
\draw [thick,red] (-1.5,0)--(-2.1,0);
\draw [thick,red] (-1.5,0)--(-0.75,0);
\draw [thick,red](-0.75,0) -- (.75,0);
\draw [thick,red] (3.3,0)--(2.8,0);
\draw [thick,red] (1.5,0)--(2.1,0);
\draw [thick,red] (1.5,0)--(.75,0);

\node[above right,red] at (-3.3,0) {$\varrho_1$};
\node[above left,red] at (-1.5,0) {$\varrho_{n-2}$};
\node[above left,red] at (-0.75,0) {$\varrho_{n-1}$};
\node[above right,red] at (1.6,0) {$\varrho_{n+2}$};
\node[above left,red] at (0,0) {$\varrho_{n}$};
\node[above left,red] at (1.5,0) {$\varrho_{n+1}$};
\node[above left,red] at (3.3,0) {$\varrho_{2n-1}$};

\end{tikzpicture}

\caption{The basic curves $\theta_i$ and $\varrho_i$ in the aligned configuration with regions $\cD_i$ for $\DA.$ } \label{NormalformA}

\end{figure}

\begin{remark} \label{basic curves A}
We warn the reader of two slight differences here in comparison to \cite{KhoSei}: our underlying set of basic curves $\{\varrho_j\}$ chosen does not include the one curve connected to the boundary of the disc; moreover, the bigradings of the basic curves and vertical curves are slightly different.
Compare our equations \eqref{bigrading defn eqn 1} and \eqref{bigrading defn eqn 2} to the defining equations in  \cite[pg. 232]{KhoSei}.
%
\end{remark}
\begin{lemma}
The bigradings we choose for the set of basic curves and vertical curves in $\DA$ satisfy the following by properties:
	\begin{align}
	I^{bigr}(\ddot{\theta}_j, \ddot{\varrho}_j) &= 1 + q_1^{-1}q_2, & \text{ for $ 1 \leq j \leq 2n-1;$}\\
	I^{bigr}(\ddot{\varrho}_{j}, \ddot{\varrho}_{j+1}) &= 1,   & \text{ for $ n \leq j \leq 2n-2;$} \label{bigrading defn eqn 1} \\
	I^{bigr}(\ddot{\varrho}_{j}, \ddot{\varrho}_{j-1}) &=1, & \text{ for $ 2 \leq j \leq n$. \label{bigrading defn eqn 2}}
	\end{align}
\end{lemma}
\begin{proof}
This follows immediately from the construction.
\end{proof}

Similar to curves in $\DB$, a curve $c$ in $\DA$ is called \bit{admissible} if $c = f(\varrho_j)$ for some $ -n \leq j \leq n$ and some $f \in \Diff \left(\DA, \partial \DA \right)$.
Note that unlike in \cite{KhoSei}, admissible curves in $\DA$ will not touch the boundary of the disc (none of our basic curves $\varrho_j$ does).
An admissible curve $c$ in its isotopy class that has minimal intersection with all the $\theta_j$'s among its other representatives is said to be in \bit{normal form}. 
We define crossings, essential segments, $j$-strings and bigraded $j$-strings in a similar fashion to the trigraded case (see \cref{local index function}).
In particular, given a $j$-crossing $x$ of a bigraded curve $\ddot{c}$, we fix $x_0 := j$ and $(x_1, x_2)$ is the local index $(\mu_1, \mu_2)$ of the crossing $x$.
	All these notions can be extended to that for multicurves and bigraded multicurves naturally.
    
    Suppose $c_0$ and $c_1$ are two admissible curves in $\DB$ intersecting at $z \in \DB$.
    If $z = 0,$ we require $c_0 \not\simeq c_1.$
    Their preimage  $q_{br}^{-1}(c_0)$ and $q_{br}^{-1}(c_1)$  in $\DA$ under the map $q_{br}: \DA \ra \DB$  would then intersect minimally. 
    However,  if $c_0 \cap c_1 \cap \{0\} \neq \emptyset$ and $c_0 \simeq c_1$, they will not intersect minimally as illustrated below:

\begin{figure}[H]
\begin{tikzpicture}[scale = 0.7]
\draw (-3,0) circle (2cm);
\draw[fill] (-4.5,0) circle [radius=0.07];
\draw[fill] (-1.5,0) circle [radius=0.07];
\draw [thick, blue] (-4.5,0) arc (180:360:.75);
\draw [thick, blue] (-1.5,0) arc (0:180:.75);
\draw [thick, red] (-4.5, 0) -- (-1.5,0);

\draw (3,0) circle (2cm);
\node at (0,0) {$\simeq$};

\draw[fill] (4.5,0) circle [radius=0.07];
\draw[fill] (1.5,0) circle [radius=0.07];
\draw [thick, blue] (4.5,0) arc (0:180:1.5);
\draw [thick, red] (4.5, 0) -- (1.5,0);

\node [above] at (-2.25,.7) {{\scriptsize $q_{br}^{-1}(c_1)$}};
\node [above] at (-3.6,0) {{{\scriptsize $q_{br}^{-1}(c_0)$}}};
\node [below] at (-1.5,0) {{{\scriptsize $j$}}};
\node [below] at (-4.75,0) {{{\scriptsize $-j$}}};
\node [below] at (3,1.4) {{{\scriptsize $(q_{br}^{-1}(c_1))'$}}};
\node [below] at (3,0) {{{\scriptsize $q_{br}^{-1}(c_0)$}}};
\node [below] at (1.5,0) {{{\scriptsize $-j$}}};
\node [below] at (4.5,0) {{{\scriptsize $j$}}};

\end{tikzpicture}
\caption{{\small The preimages $q_{br}^{-1}(c_0), q_{br}^{-1}(c_1)$ in $\DA$.}} 
\end{figure}
As such, we obtain the following proposition:
\begin{proposition} \label{tribilocin}
Let $\check{c}_0$ and $\check{c}_1$ be two trigraded curves intersecting at $z \in \DB$, with local index $\mu^{trigr}(\check{c}_0, \check{c}_1, z) = (r_1, r_2, r_3)$. If $c_0 \cap c_1 \cap \{0\} \neq \emptyset,$ we require $c_0 \not\simeq c_1$.
	If $z \neq 0$, further suppose that $\m({c_0 , \check{c}_0}) = (\wt{c_0}, \check{\wt{c}}_0) \amalg  (\ut{c_0}, \ut{\check{c}_0})$ and  $\m({c_1 , \check{c}_1}) = (\wt{c_1}, \check{\wt{c_1}}) \amalg  (\ut{c_1}, \ut{\check{c}_1})$
	such that $\wt{c}_0 \cap \wt{c}_1 = \wt{z}$ and $\ut{c_0} \cap \ut{c_1} = \ut{z}.$ 
Then
$$
\begin{cases}
\mu^{bigr}(\ddot{\wt{c}}_0, \ddot{\wt{c}}_1, \wt{z}) = (r_1, r_2) 
= \mu^{bigr}({\undertilde{\ddot{c}_0}},{\undertilde {\ddot{c}_1}}, \undertilde{z}), &\text{for } z \neq 0; \\
\mu^{bigr}(\ddot{\ut{\wt{c_0}}}, \ddot{\ut{\wt{c_1}}}, 0  ) = (r_1, r_2), &\text{for } z=0.
\end{cases}
$$
\end{proposition}

Furthermore, this proposition allows us to relate trigraded intersection numbers and bigraded intersection number in the following way. 
    \begin{corollary}For any trigraded admissible curves $(c_0, \check{c}_0)$ and $(c_1, \check{c}_1)$	
	   \[ I^{trigr}(\check{c}_0,\check{c}_1) |_{q_3=1}  =   I^{bigr}\left(\m({\check{c}_0}),\m({\check{c}_1}) \right).\]
In particular,	
	   $ \frac{1}{2} I^{trigr}(\check{c}_0,\check{c}_1) |_{q_1 = q_2 = q_3=1}  = I(\m(c_0), \m(c_1)),
	   $
i.e. $ \frac{1}{2} I^{trigr}(\check{c}_0,\check{c}_1) |_{q_1 = q_2 = q_3=1}$ counts the geometric intersection number of the lifts of $c_0$ and $c_1$ in $\DA$ under the map $\m.$
\end{corollary} 
\begin{proof}
The case when $c_0 \not\simeq c_1$ in $\DB$ or when at least one of $c_0$ and $c_1$ does not have its endpoint at $\{0\}$ follows directly from \cref{tribilocin}.
 The other case follows from a direct computation.
 The last statement relating trigraded intersection number and geometric intersection number follows from the property of bigraded intersection number (see \cite[pg.227, property (B1)]{KhoSei}).
\end{proof}

We shall abuse notation and also allow $\m$ to lift crossings of a trigraded admissible curve $(c,\check{c})$ to crossings of the bigraded admissible multicurves $\m((c,\check{c})) = (\wt{c}, \ddot{\wt{c}})\amalg (\ut{c}, \ddot{\ut{c}}) .$
    Suppose $z$ is a $j$-crossing of $c$, for $j >1$. 
    Then, $q_{br}^{-1}(z) = \{ \wt{z}, \ut{z} \}$ where $\wt{z} \in \wt{c}$ and $\ut{z} \in \ut{c}.$
    If $z$ is a $1$-crossing of $c$,  then we also have $q_{br}^{-1}(z) = \{ \wt{z}, \ut{z} \}$; in this case we shall pick $\wt{\ut{z}}$ to be the unique element in $ \{ \wt{z}, \ut{z} \} \cap \theta_n.$
    So, if $z$ is a $j$-crossing of $c$ with $\mu^{trigr}(\check{d}_k, \check{c}, z)=(r_1,r_2,r_3)$,
    We define $\m(z) = \{ \wt{z} , \ut{z} \}$, both $\wt{z}$ and $\ut{z}$ with local index $(r_1, r_2)$ for $j >1$; otherwise $\m(z) = \{\wt{\ut{z}}\}$ for $j =1$, with local index $(r_1, r_2)$ by \cref{tribilocin}. 
	Let $\check{h}$ be a connected subset of $\check{c}$ within some connected region given by unions of $D_j$'s, equipped with the trigrading given by local indices of crossings of $\check{h}$ induced from crossings of $\check{c}$.
	We define $\m(\check{h})$ to consist of $q_{br}^{-1}(\check{h})$, with bigradings given by local indices of crossings of $q_{br}^{-1}(\check{h})$ induced from crossings of $\m(\check{c})$.



\section{Type $A_{2n-1}$ and Type $B_n$ Zigzag Algebras}\label{define zigzag}

In this section, we recall the construction of the type $A_{m}$ zigzag algebra $\Aa_{m}$ as given in \cite{KhoSei} (with slight change in gradings) and recall the $\cA(A_{m})$ action on the bounded homotopy category $\Kom^b(\Aa_{m}$-$\text{p$_{r}$g$_{r}$mod})$ of complexes of projective graded modules over $\Aa_{m}$. 
We then construct a type $B_n$ zigzag algebra $\Ba_n$, following a similar construction, and show that $\mathcal{A}(B_n)$ acts on $\Kom^b(\Ba_n$-$\text{p$_{r}$g$_{r}$mod})$.
We shall assume that the reader is familiar with projective modules over finite-dimensional (graded) algebras, and refer the unfamiliar readers to \cite[Chapter 1]{Frob_alg}.
Note that throughout this whole paper, all complexes are bounded.

\subsection{Type $A_{m}$ zigzag algebra $\Aa_{m}$} \label{A zigzag} 
Consider the following quiver $\Gamma_{m}$:
\begin{figure}[H]
\begin{tikzcd}[column sep = 1.5cm]
1				\arrow[r,bend left,"1|2"]  														 &		 
2 			\arrow[l,bend left,"2|1"] \arrow[r,bend left, "2|3"] 
				&
3 			\arrow[l,bend left,"3|2"] \arrow[r,bend left, "3|4"] 
				 &
\cdots \arrow[l,bend left,"4|3"] \arrow[r,bend left, "m-1|m"] &
m 		.	\arrow[l,bend left,"m|m-1"]				
\end{tikzcd}
\caption{{\small The quiver $\Gamma_{m}$.}}
\label{A quiver}
\end{figure}
We can take its path algebra $\C \Gamma_{m}$ over $\C$;
$\C \Gamma_{m}$ is the $\C$-vector space spanned by the set of all paths in $\Gamma_{m}$, with multiplication given by concatenation of paths (the multiplication is zero if the endpoints do not agree).
We denote the constant path at each vertex $j$ by $e_j$. 

In this paper, we will only consider $m$ odd, and the grading we put on $\C\Gamma_{m}$ will be slightly different to \cite{KhoSei}; we set 
 \begin{itemize}
 \item the degree of $(j|j-1)$ is $1$  for $j > \frac{m+1}{2}$ and $0$ for  $j \leq \frac{m+1}{2}$,
 \item the degree of $(j|j+1)$ is $1$ for  $j < \frac{m+1}{2}$ and $0 $ for $j \geq \frac{m+1}{2},$ and
 \item the degree of $e_j = 0$, for all $j$,
 \end{itemize}
where the grading is extended to all paths via multiplication.
 In this way, this path algebra is $\Z$-graded with the grading shift denoted by $\{-\}$ and unital with a family of pairwise orthogonal primitive central idempotent $e_j$ summing up to the unit element.

	Let $\Aa_{m}$ be the quotient of the path algebra of the quiver $\Gamma_{m}$ by the relations:
	\begin{align*}
	(j|j+1|j) &= (j| j-1|j), \\
		(j-1|j|j+1) &= 0 = (j+1| j|j-1),
	\end{align*}
for all $2 \leq j \leq m-1$.
 It is easy to see that these relations are homogenous with respect to the above grading, so that $\Aa_{m}$ is a $\Z$-graded algebra. 
As a $\C$-vector space, it has dimension $4m-2$ with the following basis $\{ e_1 , \ldots, e_{m},  (1|2), \ldots, (m-1 |m), (2|1), \ldots (m | m-1), (1|2|1), \ldots, (m|m-1|m) \}$.

The indecomposable projective, $\Z$-graded $\Aa_{m}$-modules are denoted by $P^A_j := \Aa_{m}e_j$, and we denote the (additive) category of projective, graded $\Aa_m$-modules by $\Aa_m$-$p_rg_r$mod.
We recall the following results from \cite{KhoSei}:
\begin{theorem}
For each $j$, the following complex of graded $(\Aa_{m}, \Aa_{m})$-bimodule
\[
\mathcal{R}_j := 0 \ra P^A_j \otimes_\C {}_jP^A \xra{\beta_j} \Aa_{m} \ra 0,
\]
with $\Aa_{m}$ in cohomological degree 0 is invertible in the homotopy category $\Kom^b((\Aa_{m}, \Aa_{m})\text{-bimod})$ of graded $(\Aa_m,\Aa_m)$-bimodules and satisfies the following relations:
\begin{align*}
\cR_j \otimes \cR_k &\cong \cR_k \otimes \cR_j, \quad \text{ for } |j-k|>1; \\
\cR_j \otimes \cR_{j+1} \otimes \cR_j &\cong \cR_{j+1} \otimes \cR_j \otimes \cR_{j+1}.
\end{align*}
\end{theorem}
\begin{proof}
See \cite[Proposition 2.4 and Theorem 2.5]{KhoSei}.
\end{proof}
\begin{proposition}
There is a (weak) $\cA(A_{m})$-action on $\Kom^b(\Aa_{m}$-$\text{p$_{r}$g$_{r}$mod})$, where each standard generator $\sigma_j^A$ of $\cA(A_{m})$ acts on a complex $M \in \Kom^b(\Aa_{m}$-$\text{p$_{r}$g$_{r}$mod})$ via $\mathcal{R}_j$:
\[
\sigma^A_j(M) := \mathcal{R}_j\otimes_{\Aa_{m}}M.
\]
\end{proposition}
\begin{proof}
See \cite[Proposition 2.7]{KhoSei}.
\end{proof}
We will abuse notation and use $\sigma^A_j$ in place of $\cR_j$ whenever the context is clear.

\subsection{Type $B_n$ zigzag algebra $\Ba_n$} \label{B zigzag}
Consider the following quiver $\Omega_n$:
\begin{figure}[H]
\begin{tikzcd}[column sep = 1.5cm]
1				\arrow[r,bend left,"1|2"]  														 &		 
2 			\arrow[l,bend left,"2|1"] \arrow[r,bend left, "2|3"] 
				\arrow[color=blue,out=70,in=110,loop,swap,"ie_2"] &
3 			\arrow[l,bend left,"3|2"] \arrow[r,bend left, "3|4"] 
				\arrow[color=blue,out=70,in=110,loop,swap,"ie_3"] &
\cdots \arrow[l,bend left,"4|3"] \arrow[r,bend left, "n-1|n"] &
n 		.	\arrow[l,bend left,"n|n-1"]
				\arrow[color=blue,out=70,in=110,loop,swap,"ie_n"]
\end{tikzcd}
\caption{{\small The quiver $\Omega_n$.}}
\label{B quiver}
\end{figure}
%
Take its path algebra $\R \Omega_n$ over $\R$ and consider the two gradings on $\R \Omega_n$ given as follows:
	\begin{enumerate} [(i)]
	\item the first grading is defined following the convention in  \cite{KhoSei} where we set \begin{itemize}
	\item 	the degree of $(j+1|j)$ to be $1$ and $(j|j+1)$ to be $0$ for all $1 \leq j \leq n-1$, and
	\item   the degree of $e_j$, $ie_j$ (blue paths in \cref{B quiver}) to be $0$ for all $1 \leq j \leq n$,
\end{itemize}
	extending to all paths.
	\item the second grading is a $\Z/2\Z$-grading defined by setting 
	\begin{itemize}
	\item  the degree of $ie_j$ as 1 for all $1 \leq j \leq n$, and
	\item the degree of all other paths in \cref{B quiver} and the constant paths as zero,
	\end{itemize}
	extending again to all paths.
	\end{enumerate}

\noindent We denote  a shift in the $\Z$-grading by  $\{-\}$ and a shift in the $\Z/2\Z$-grading by $\< - \>.$

We are now ready to define the zigzag algebra of type $B_n$:
\begin{definition} \label{defn: B zigzag}
The zigzag path algebra of $B_n$, denoted by $\Ba_n$, is the quotient algebra of the path algebra $\R \Omega_n$ modulo the usual zigzag relations given by
\begin{align}
(j|j-1)(j-1|j) &= (j|j+1)(j+1|j), \label{B-zizgzag1}\\
(j-1|j)(j|j+1) = & 0 = (j+1|j)(j|j-1), \label{B-zigzag2}
\end{align}
for $2\leq j \leq n-1$; in addition to the relations
\begin{align}
(ie_j)(ie_j) &= -e_j, \qquad \text{for } j \geq 2 \label{imaginary};\\
(ie_{j-1})(j-1|j) &= (j-1|j)(ie_j), \qquad \text{for } j\geq 3; \label{complex symmetry 1}\\
(ie_{j})(j|j-1) &= (j|j-1)(ie_{j-1}), \qquad \text{for } j\geq 3; \label{complex symmetry 2}\\
(1|2)(ie_2)(2|1) &= 0; \\
(ie_2) (2|1|2) &= (2|1|2) (ie_2).
\end{align}
We shall also denote the (non-trivial) loop on vertex $j$ by $X_j := (j|j\pm 1)(j\pm 1|j)$.
The relations above are homogeneous with respect to both the $\Z$ and $\Z/2\Z$ gradings, so $\Ba_n$ is a bigraded algebra.
\end{definition}

\begin{proposition}\label{prop: dim of B-zigzag}
As a $\R$-vector space, $\Ba_n$ has dimension $8n-6$.
\end{proposition}            
\begin{proof}
Using the relations \eqref{B-zizgzag1} to \eqref{complex symmetry 2}, the vector subspace (over $\R$) spanned by paths that do not pass through the vertex $1$ is isomorphic to $\Aa_{n-1}$ viewed as a $\R$-vector space, which has $\R$-dimension $2(4(n-1)-2) = 8n-12$.
Combined with the remaining relations, one can check that the remaining paths that passes through the vertex 1 (modulo relations) are exactly $e_1, (1|2), (2|1), (1|2)(ie_2), (ie_2)(2|1)$ and $X_1$.
   Hence, $\dim_\R(\Ba_n) =  8n-12 + 6 = 8n-6$; in particular,
   the set $\{ e_1 , \ldots, e_{n}, ie_2 , \ldots, ie_{n}, (1|2), \ldots, (n-1 | n), (2|1), \ldots, (n | n-1), (ie_2)(2|1), (1|2)(ie_2),$ $(ie_2)(2|3), \ldots, (ie_{n-1})(n-1|n), (3|2)(ie_2), \ldots, (n|n-1)(ie_{n-1}), X_1, \ldots, X_n, (ie_2)X_2, \ldots, (ie_n)X_n \}$ forms a $\R$-basis of $\Ba_n.$
\end{proof}

The indecomposable (left) projective, bigraded $\Ba_n$-modules are given by $P^B_j := \Ba_n e_j$, and we denote the (additive) category of projective, bigraded $\Ba_n$-modules by $\Ba_n$-$p_rg_r$mod.
For $j=1$, $P^B_j$ is naturally a $(\Ba_n, \R)$-bimodule; there is a natural left $\Ba_n$-action given by multiplication of the algebra and the right $\R$-action induced by the left (commutative) $\R$-action.
But for $j\geq 2$, we shall endow $P^B_j$ with a right $\C$-action.

To this end, let us view $\C$ as a $\Z/2\Z$-graded algebra over $\R$ by endowing the reals with degree $0$ and the complex imaginary $i$ with degree $1$ over $\Z/2\Z$ and extend linearly.
Note also that \eqref{imaginary} in \cref{defn: B zigzag} is analogous to the relation satisfied by the complex imaginary number $i$.
We define a right $\C$-action on $P^B_j$ by $p * (a+ib) = ap + bp(ie_j)$ for $p \in P^B_j, a+ib\in \C$.
It follows from the definition that this right action restricted to $\R$ agrees with the natural right (and left) $\R$-action.
This makes $P^B_j$ into a bigraded $(\Ba_n,\C)$-bimodule for $j\geq 2$.
Dually, we shall define ${}_jP^B := e_j\Ba_n$, where we similarly consider it as a bigraded ($\R,\Ba_n)$-bimodule for $j=1$ and as a bigraded $(\C,\Ba_n)$-bimodule for $j\geq 2$.

\begin{proposition}
Denote ${}_jP^B_k := {}_jP^B\otimes_{\Ba_n}P^B_k$.
We have that
\[
  {}_jP^B_k \cong 
  \begin{cases}
  		\vspace{1mm}
  		\ _{\C}\C_{\C}, & \text{as bigraded } (\C,\C)\text{-bimodules, for } j,k \in \{2,\hdots, n\} \text{ and }  k-j=1;\\
  		\vspace{1mm}
  		\ _{\C}\C_{\C}\{1\}, & \text{as bigraded } (\C,\C)\text{-bimodules, for } j,k \in \{2,\hdots, n\} \text{ and }  j-k=1;\\
  		\vspace{1mm}
  		\ _{\C}\C_{\C} \oplus \ _{\C}\C_{\C}\{1\}, & \text{as bigraded }  (\C,\C)\text{-bimodules, for } j=k=2,3,\hdots,n; \\
  		\vspace{1mm}
        \ _{\R}\C_{\C}, & \text{as bigraded } (\R,\C)\text{-bimodules, for } j=1 \text{ and } k=2; \\
        \vspace{1mm}
        \ _{\C}\C_{\R}\{1\}, & \text{as bigraded }  (\C,\R)\text{-bimodules, for } j=2 \text{ and } k=1; \\ 
        \vspace{1mm}
        \ _{\R}\R_{\R} \oplus \ _{\R}\R_{\R}\{1\}, & \text{as bigraded }  (\R,\R)\text{-bimodules, for } j=k=1. \\    
  \end{cases}
\]
\label{bimodule isomorphism}
\end{proposition}
\begin{proof}
The case where $j,k \in \{2,...,n\}$ follows as in type $A$.
By identifying ${}_jP_k$ as the $\R$-vector subspace of $\Ba_n$ spanned by paths starting at vertex $j$ and ending at vertex $k$, we see that ${}_1P_2$ has basis $\{(1|2), (1|2)(ie_2)\}$; ${}_2P_1$ has basis $\{(2|1), (ie_2)(2|1)\}$ and ${}_1P_1$ has basis $\{e_1, X_1\}$.
The fact that the bimodule and bigrading structures agree follows from the definition and is left as a simple exercise to the reader.
\end{proof}

\begin{remark}\label{rmk: Z/2Z grading bimodule iso}
Note that all the bigraded bimodules in \cref{bimodule isomorphism} can be restricted to bigraded $(\R, \R)$-bimodules by identifying ${}_\R \C_\R \cong \R\oplus \R\<1\>$.
For example, ${}_1 P_2^B$ restricted to an $(\R, \R)$-bimodule is generated by $(1|2)$ and $(1|2)ie_2$, and it is isomorphic to $\R \oplus \R\<1\> \cong {}_\R \C_\R$.
\end{remark}

\begin{lemma}
Denote $\mathbb{K}_1 := \R$ and $\mathbb{K}_j := \C$ when $j \geq 2$.
The maps \[\beta_j: P^B_j \otimes_{\mathbb{K}_j} {}_jP^B  \to \Ba_n \text{ and } \gamma_j: \Ba_n  \to  P^B_j \otimes_{\mathbb{K}_j} {}_jP^B  \{-1\}\] defined by:
\begin{align*} 
\beta_j(x\otimes y) &:= xy, \\
\gamma_j(1) &:= 
\begin{cases}
X_j \otimes e_j + e_j \otimes X_j + (j+1|j) \otimes (j|j+1) \\
 \hspace{8mm} + (-ie_{j+1})(j+1|j) \otimes (j|j+1)(ie_{j+1}), &\text{for } j=1;\\
X_j \otimes e_j + e_j \otimes X_j + (j-1|j) \otimes (j|j-1) + (j+1|j) \otimes (j|j+1), &\text{for } 1<j < n; \\
X_j \otimes e_j + e_j \otimes X_j + (j-1|j) \otimes (j|j-1), &\text{for } j = n
\end{cases}
\end{align*}
are $(\Ba_n,\Ba_n)$-bimodule maps.
\end{lemma}

\begin{proof}
It is obvious from the definition that the $\beta_j$ are maps of $(\Ba_n,\Ba_n)$-bimodules for all $j$. 
The fact that $\gamma_j$ is a $(\Ba_n,\Ba_n)$-bimodule map also follows from a tedious check on each basis elements, which we shall omit and leave it to the reader.
 \end{proof}

\begin{definition}
Define the following complexes of bigraded $(\Ba_n,\Ba_n)$-bimodules:
\begin{align*}
R_j &:= (0 \to P^B_j \otimes_{\mathbb{K}_j} {}_jP^B \xra{\beta_j} \Ba_n \to 0), \text{and} \\
R_j' &:= (0 \to \Ba_n  \xra{\gamma_j}  P^B_j \otimes_{\mathbb{K}_j} {}_jP^B\{-1 \} \to 0).
\end{align*}
for each $1 \leq j \leq n,$ with both $\Ba_n$ in cohomological degree 0, $\mathbb{K}_1 = \R$ and $\mathbb{K}_j = \C$ for $j \geq 2$.
\end{definition}

\begin{proposition} \label{standard braid relations}
We have the following isomorphisms in the homotopy category $\Kom^b ((\Ba_n, \Ba_n )$-$\text{bimod})$ of complexes of projective graded $(\Ba_n,\Ba_n)$-bimodules:
\begin{align}
R_j \otimes R_j^{'} \cong & \Ba_n \cong R_j^{'} \otimes R_j; \\
R_j \otimes R_k & \cong R_k \otimes R_j, \quad \text{for } |k-j| > 1;\\
R_j \otimes R_{j+1} \otimes R_j &\cong R_{j+1} \otimes R_j \otimes R_{j+1}, \quad \text{for } j \geq 2.
\end{align}
\end{proposition}

\begin{proof} 
These relations can be verified similarly as in \cite[Theorem 2.5]{KhoSei}.
\end{proof}

\subsection{Adjunctions and Dehn Twist}
To show the last type $B_n$ relation (the 4-braiding relation), we shall introduce a larger family of invertible complexes that will aid us in our calculation. 
This construction mirrors the notion of Dehn twists in topology and uses the theory on adjunctions (we highly recommend \cite{KhoBiadj} for an amazing exposition on expressing adjunctions using planar diagrammatics).
Throughout this section we will denote $\mathbb{K}_1 := \R$ and $\mathbb{K}_j := \C$ for $j \geq 2$.

\begin{definition}
Let $X\in \Kom^{b}((\Ba_n,\mathbb{K}_j)$-$bimod)$ and $X^\ell, X^r \in \Kom^{b}((\mathbb{K}_j,\Ba_n)\text{-bimod})$ such that $X^\ell \otimes_{\Ba_n} -$ and $X^r \otimes_{\Ba_n} -$ are left and right adjoints of $X \otimes_{\mathbb{K}_j} -$ respectively.
We define the twist of $X$ as the complex of $(\Ba_n,\Ba_n)$-bimodule
\[
\sigma_X:= \cone\left( X\otimes_{\mathbb{K}_j} X^r \xrightarrow{\varepsilon} \Ba_n \right),
\]
with $\varepsilon$ the counit of the adjunction $X \dashv X^r$.
Similarly, the dual twist of $X$ is given by
\[
\sigma_X':= \cone \left( \Ba_n \xrightarrow{\eta} X\otimes_{\mathbb{K}_j} X^\ell \right)
\]
with $\eta$ the unit of the adjunction $X \vdash X^\ell$.
The twist $\sigma_X$ is said to be spherical if the twist and dual twist are inverses of each other, namely $\sigma_X \otimes_{\Ba_n} \sigma_X' \cong \Ba_n \cong \sigma_X' \otimes_{\Ba_n} \sigma_X$.
\end{definition}
One can verify from the definition of the adjunctions that the twist (resp. dual twist) is uniquely defined up to isomorphism, i.e. $X \cong Y$ implies $\sigma_X \cong \sigma_{Y}$ (resp. $\sigma_X' \cong \sigma_{Y}'$).
On the other hand, the shift functors $[1], \{1\}$ and $\<1\>$ are autoequivalences, so we also have that $\sigma_{X} = \sigma_{X[r]\{s\}\<t\>}$.
More generally, given a pair of adjunctions $(X, X^\ell, X^r)$ on $X$ with $Y \cong \Sigma \otimes_{\Ba_n} X$ where $\Sigma$ an invertible object in $\Kom^b((\Ba_n, \Ba_n)\text{-mod})$, we also have a pair of adjunctions $(Y, Y^\ell, Y^r)$ on $Y$ given by
\[
Y^\ell := X^\ell \otimes_{\Ba_n} \Sigma^{-1}, \quad Y^r := X^r \otimes_{\Ba_n} \Sigma^{-1}.
\]
Furthermore, the twists and dual twists are related by:
\[
\sigma_Y \cong \Sigma \otimes_{\Ba_n} \sigma_X \otimes_{\Ba_n} \Sigma^{-1}, \quad
\sigma_Y' \cong \Sigma \otimes_{\Ba_n} \sigma_X' \otimes_{\Ba_n} \Sigma^{-1}.
\]

\begin{lemma} \label{P_i biadjoint}
The functor $P_j^B \otimes_{\mathbb{K}_j} - $ is a left adjoint of ${}_{j}{P}^B \otimes_{\Ba_n} - $ and a right adjoint of ${}_{j}{P}^B \{-1\} \otimes_{\Ba_n} - $.
\end{lemma}
\begin{proof}
To show that the $P_j^B \otimes_{\mathbb{K}_j} - $ is a left adjoint to ${}_{j}{P}^B \otimes_{\Ba_n} - $, take the counit to be the functor induced by $P_j^B\otimes_{\mathbb{K}_j} {}_{j}{P}^B \xrightarrow{\beta_j} \Ba_n$ and the unit is instead induced by $\mathbb{K}_j \xrightarrow{\varphi} {}_{j}{P}^B\otimes_{\Ba_n} P_j^B$, where $\varphi$ is defined by $\varphi(1) = e_j\otimes_{\Ba_n} e_j$.
To show ${}_{j}{P}^B\{-1\} \otimes_{\Ba_n} - $ is a left adjoint to $P_j^B \otimes_{\mathbb{K}_j} - $, take the counit to be the functor induced by ${}_{j}{P}^B \otimes_{\Ba_n} P_j^B \{-1\} \xrightarrow{\varphi'} \mathbb{K}_j$ where $\varphi'$ is defined by $\varphi'(e_j\otimes e_j) = 0, \varphi'(X_j\otimes e_j) = 1$ (note that $X_j \otimes e_j = e_j \otimes X_j$), and the unit is instead induced by ${\Ba_n} \xrightarrow{\gamma_j} P_j^B\otimes_{\mathbb{K}_j} {}_{j}{P}^B \{-1\}$.
We leave the verification of the conditions required to the reader.
\end{proof}

Using this, we shall now prove the last type $B_n$ relation required:
\begin{proposition} \label{type B relation}
We have the following isomorphism of $\Kom^b((\Ba_n,\Ba_n)\text{-bimod})$:
\[
R_2 \otimes_{\Ba_n} R_1 \otimes_{\Ba_n} R_2 \otimes_{\Ba_n} R_1 \cong 
R_1 \otimes_{\Ba_n} R_2 \otimes_{\Ba_n} R_1 \otimes_{\Ba_n} R_2 .
\]
\end{proposition}
\begin{proof}
We shall drop the tensor products for the sake of readability:
$
 R_2 R_1 R_2 R_1 \cong R_1 R_2 R_1 R_2. 
 $
Using \cref{standard braid relations}, note that this relation is equivalent to
$
(R_2 R_1 R_2) R_1 (R_2' R_1' R_2') \cong 
R_1.
$
By the adjunctions shown in \cref{P_i biadjoint}, note that $R_1$ and $R_1'$ are by definition the same as the twist $\sigma_{P_1^B}$ and dual twist $\sigma_{P_1^B}'$ of $P_1^B$ respectively.
As such, we get that
$$
\sigma_{R_2 R_1 R_2(P_1^B)} \cong (R_2 R_1 R_2) \sigma_{P_1^B} (R_2' R_1' R_2') \cong 
\sigma_{P_1^B}.
$$
It is now sufficient to show $R_2 R_1 R_2(P_1^B)$ and $P_1^B$ are isomorphic in $\Kom^b((\Ba_n, \R)$-bimod) up to cohomological or internal gradings shifts.
This is shown in the following series computation (note that we have omitted the cohomological grading since it does not matter).
\begin{align*}
R_2(P_1^B) &= 0 \ra P_2^B\otimes_\C {}_2P_1^B \ra P_1^B \ra 0 
\cong 0 \ra P_2^B \{1\} \xra{(2|1)} P_1^B \ra 0 \qquad (\text{by } \cref{bimodule isomorphism})
\end{align*}
\begin{align*}
R_1R_2(P_1^B) &\cong R_1 \left( 0 \ra P_2^B \{1\} \xra{(2|1)} P_1^B \ra 0 \right)\\
&= \text{cone } \left\{
\begin{tikzcd}[row sep = large, column sep = large, ampersand replacement=\&]
P_1^B\{1\}\otimes_\R {}_1P_2^B \ar{r}{\id\otimes (2|1)} \ar{d}{}\& 
P_1^B \otimes_\R {}_1P_1^B \ar{d}{} \\
P_2^B \{1\} \ar{r}{(2|1)} \& 
P_1^B
\end{tikzcd} \right\}\\
&\cong \text{cone }\left\{
\begin{tikzcd}[row sep = large, column sep = large, ampersand replacement=\&]
P_1^B\{1\}\oplus P_1^B\{1\}\<1\> 
\ar{r}
	{
	\begin{bsmallmatrix}
	0 & 0 \\
	\id & 0
	\end{bsmallmatrix}
} 
\ar{d}
	{
	\begin{bsmallmatrix}
	(1|2) & (1|2)i
	\end{bsmallmatrix}
}\& 
P_1^B \oplus P_1^B\{1\} 
\ar{d}{
	\begin{bsmallmatrix}
	\id & X_1
	\end{bsmallmatrix}
} \\
P_2^B \{1\} \ar{r}{(2|1)} \& 
P_1^B
\end{tikzcd} \right\} \qquad (\text{by } \cref{bimodule isomorphism})\\
&\cong 0 \ra P_1^B\{1\}\<1\> \xra{(1|2)i} P_2^B\{1\} \ra 0 \\
R_2R_1R_2(P_1^B) &\cong \text{cone }\left\{
\begin{tikzcd} [row sep = large, column sep = large, ampersand replacement=\&]
P_2^B\otimes_\C {}_2P_1^B\{1\} \<1\> 
\ar{r}{} 
\ar{d}{}
\&
P_2^B \otimes_\C {}_2P_2^B\{1\}
\ar{d}{} \\
P_1^B\{1\} \<1\> \ar{r}{(1|2)i} \&
P_2^B\{1\}
\end{tikzcd}
\right\}\\
&\cong \text{cone }\left\{
\begin{tikzcd} [row sep = large, column sep = large, ampersand replacement=\&]
P_2^B\{2\} \<1\> 
\ar{r}{
	\begin{bsmallmatrix}
	0 \\
	\id
	\end{bsmallmatrix}
} 
\ar{d}{(2|1)}
\&
P_2^B\{1\} \oplus P_2^B\{2\}\<1\>
\ar{d}{
	\begin{bsmallmatrix}
	\id & X_2i
	\end{bsmallmatrix}
} \\
P_1^B\{1\} \<1\> \ar{r}{(1|2)i} \&
P_2^B\{1\}.
\end{tikzcd}
\right\} \qquad (\text{by } \cref{bimodule isomorphism})\\
&\cong P_1^B\{1\}\<1\>
\end{align*}
\end{proof}
\begin{theorem} \label{Cat B action}
We have a (weak) $\mathcal{A}(B_n)$-action on $\Kom^b(\Ba_n$-$p_r g_r mod)$, where each standard generator $\sigma^B_j$ for $j \geq 2$ of $\mathcal{A}(B_n)$ acts on a complex $M \in \Kom^b(\Ba_n$-$p_rg_rmod)$ via $R_j$, and $\sigma^B_1$ acts via $R_1 \<1\>$:
$$\sigma^B_j(M):= R_j \otimes_{\Ba_n} M, \text{ and } (\sigma^B_j)^{-1}(M):= R_j' \otimes_{\Ba_n} M,$$
\[
\sigma^B_1(M):= R_1 \<1\> \otimes_{\Ba_n} M, \text{ and } (\sigma^B_1)^{-1}(M):= R_1' \<1\> \otimes_{\Ba_n} M.
\]
\end{theorem}
\begin{proof}
This follows directly from \cref{standard braid relations} and \cref{type B relation}, where the required relations still hold with the extra third grading shift $\<1\>$ on $R_1$ and $R_1'$.
\end{proof}
\noindent From now on, we will abuse notation and use $\sigma^B_j$ and $(\sigma^B_j)^{-1}$ in place of $R_j$ and $R_j'$ (with an extra grading shift $\<1\>$ for $j=1$) respectively  whenever it is clear from the context what we mean.

\section{Relating Categorical Type $B_n$ and Type $ {A_{2n-1}}$ actions} \label{relating categorical b a action}

In \cref{topology}, we have defined $\m$ that lifts isotopy classes of trigraded curves in $\D^B_{n+1}$ to isotopy classes of bigraded multicurves in $\D^A_{2n}$. Furthermore, we showed that the map $\mathfrak{m}$ is equivariant under the $\cA(B_n)$-action.
In this section, we shall develop the algebraic version of this story.
We will first relate our type $B_n$ zigzag algebra $\Ba_n$ to the type $A_{2n-1}$ zigzag algebra $\Aa_{2n-1}$ by showing that $\C\otimes_\R \Ba_n \cong \Aa_{2n-1}$ as graded $\C$-algebras (forgetting the $\Z/2\Z$ grading in $\Ba_n$).
Through this, we have an injection $\Ba_n \hookrightarrow \C\otimes_\R \Ba_n \cong \Aa_{2n-1}$ as graded $\R$-algebras.
Thus, we can relate the two categories $\Kom^b(\Ba_n$-$\text{p$_{r}$g$_{r}$mod})$ and $\Kom^b(\Aa_{2n-1}$-$\text{p$_{r}$g$_{r}$mod})$  through an extension of scalar $\Aa_{2n-1} \otimes_{\Ba_n} -$. 
We end  this section by showing that the functor $\Aa_{2n-1} \otimes_{\Ba_n} -$ is $\cA(B_n)$-equivariant, which also allows us to deduce that the $\cA(B_n)$-action on $\Kom^b(\Ba_n\text{-}p_rg_rmod)$ is faithful.

Let $Q$ be a left $\C$-module. 
Throughout this section, we shall denote ${}_{\bar{\C}} Q$ to be the left $\C$-module with a deformed left action, given by multiplication its with complex conjugate:
\begin{equation} \label{eqn: deformed C-action}
a(c) = \bar{a}c.
\end{equation}
Similarly for $Q$ a right $\C$-module, we use $Q_{\bar{\C}}$ to denote the right $\C$-module with the deformed action.

\begin{proposition} \label{isomorphic algebras}
Consider the graded $\R$-algebra $\ddot{\Ba}_n$, where $\ddot{\Ba}_n$ is just $\Ba_n$ without the $\Z/2\Z$-grading $\<-\>$.
The $\Z$-graded $\C$-algebras $\C \otimes_{\R} \ddot{\Ba}_n $ and $\Aa_{2n-1}$ are isomorphic.
\end{proposition}
\begin{proof}
Note that as $\C$-vector space, we have the following decomposition:
\begin{align*}
\C \otimes_\R \ddot{\Ba}_n 
& \cong \bigoplus_{j=2}^{n} \C \otimes_\R \left(  {}_jP^B_j \ \oplus {}_jP^B_{(j-1)} \oplus  {}_{(j-1)}P^B_{j} \right) \oplus (\C \otimes_\R {}_1P^B_1) \\ 
& \cong 
\bigoplus_{j=2}^{n} \big(\C \otimes_\R  {}_jP^B_j \big) \oplus  \bigoplus_{j=2}^{n} \big( \C \otimes_\R {}_jP^B_{(j-1)} \big) \oplus \bigoplus_{j=2}^{n} \big(\C \otimes_\R {}_{(j-1)}P^B_{j} \big) \oplus 
(\C \otimes_\R {}_1P^B_1). 
\end{align*}

Firstly, for each $j\geq 2$, note that ${}_jP^B_j$ is itself a $\C$-algebra with unit $e_j\otimes 1$.
Moreover, after tensoring with $\C$ over $\R$, $\C \otimes_\R {}_jP^B_j$ has idempotent $\nu_j := \frac{1}{2}(1 \otimes e_j + i \otimes ie_j)$.
We shall define a $\C$-linear morphism $\Phi: \C\otimes_\R \ddot{\Ba}_n \ra \Aa_{2n-1}$ by specifying the images of the basis elements of $\C\otimes_\R \ddot{\Ba}_n$.
It will be easy to see that $\Phi$ is grading preserving and we leave the routine check that $\Phi$ is an algebra morphism to the reader.
For $j=1$,
\begin{align*}
  \C \otimes_\R {}_1P^B_1 \to {}_n P^A_n 
  \\
\begin{cases}
 1 \otimes \frac{1}{2} X_1 &\mapsto X_n;
\\
  1 \otimes e_1 &\mapsto e_n.
\end{cases}
\end{align*}
For $2 \leq j \leq n$,
\begin{enumerate}
\item 
$
\C \otimes_\R {}_jP^B_j \to {}_{n-j+1}P^A_{n-j+1} \oplus   {}_{n+j-1}P^A_{n+j-1}
$
\\
$
\begin{cases}
 (1 \otimes X_j ) \nu_j & \mapsto X_{n-j+1};\\
 (1 \otimes X_j )(1\otimes e_j - \nu_j) & \mapsto X_{n+j-1};\\
  \nu_j & \mapsto e_{n-j+1}; \\
  (1\otimes e_j - \nu_j)& \mapsto e_{n+j-1},
\end{cases}
$
\vspace{5pt}
\item
$
 \C \otimes_\R {}_{j-1}P^B_{j} \to {}_{n-j+2}P^A_{n-j+1} \oplus {}_{n+j-2}P^A_{n+j-1}
$
 \\
$
\begin{cases}
 \big(1 \otimes (j-1|j)\big) \nu_j & \mapsto ((n-j+2) \mid \big(n-j+1)\big); \\
   \big(1 \otimes (j-1|j)\big)(1 \otimes e_j - \nu_j) & \mapsto \big((n+j-2) \mid (n+j-1) \big),
\end{cases}
$
\vspace{5pt}
\item
$
   \C \otimes_\R {}_jP^B_{j-1} \to {}_{n-j+1}P^A_{n-j+2} \oplus {}_{n+j-1}P^A_{n+j-2}
$
\\
$
\begin{cases}
  v_j( 1 \otimes  (j|j-1) ) & \mapsto \big((n-j+1) \mid (n-j+2) \big); \\
   ( 1 \otimes  e_j -v_j)( 1 \otimes  (j|j-1)) & \mapsto \big((n+j-1) \mid (n+j-2) \big).
\end{cases}
$
\end{enumerate}
It is easy to see that this map is surjective and $\dim_\C ( \C \otimes_\R \ddot{\Ba}_n ) = \dim_\C \left( \Aa_{2n-1} \right)$.
\end{proof} 

Let $i:\ddot{\Ba}_n \hookrightarrow \C \otimes_\R \ddot{\Ba}_n$ be the canonical injection of $\Z$-graded $\R$-algebras.
\cref{isomorphic algebras} allows us to view $\ddot{\Ba}_n$ as a $\Z$-graded subalgebra over $\R$ of $\Aa_{2n-1}$ through $\Phi\circ i$.
Thus, every $\Aa_{2n-1}$-module can be restricted to a $\ddot{\Ba}_n$-module.
In particular, $\Aa_{2n-1}$ as a $\Z$-graded $(\Aa_{2n-1}, \Aa_{2n-1})$-bimodule can be restricted to a $\Z$-graded $(\Aa_{2n-1}, \ddot{\Ba}_n)$-bimodule.
This gives us an extension of scalar functor $\Aa_{2n-1} \otimes_{\ddot{\Ba}_n} -$, sending $\Z$-graded $\ddot{\Ba}_n$-modules to $\Z$-graded $\Aa_{2n-1}$-modules. 

Let $\mathfrak{F}$ denote the functor which forgets the $\Z/2\Z$-grading of the bigraded $\Ba_n$-modules.
We define
$$
\Aa_{2n-1} \otimes_{\Ba_n} - := \Aa_{2n-1} \otimes_{\ddot{\Ba}_n} (\mathfrak{F}(-)).
$$
The proposition below identifies the indecomposable projectives under the functor $\Aa_{2n-1} \otimes_{\Ba_n} -$.
\begin{proposition} \label{B tensor to A}
Recall the deformed action of $\C$ given in \eqref{eqn: deformed C-action}.
We have the following isomorphisms of $\Z$-graded bimodules:
\begin{align*}
\Aa_{2n-1} \otimes_{\Ba_n} P^B_1 &\cong (P^A_n)_\R, &\text{ as } \Z\text{-graded  $(\Aa_{2n-1}, \R)$-bimodules}; \text{ and } \\
\Aa_{2n-1} \otimes_{\Ba_n} P^B_j &\cong 
		\left( P^A_{n-(j-1)} \right)_{\bar{\C}} \oplus 
		P^A_{n+(j-1)}, &\text{ as } \Z\text{-graded $(\Aa_{2n-1}, \C)$-bimodules}.
\end{align*}
\end{proposition}
\begin{proof}
Define $\Phi_1: \Aa_{2n-1} \otimes_{\Ba_n} P^B_1 \ra (P^A_n)_\R$ and $\Phi_j: \Aa_{2n-1} \otimes_{\Ba_n} P^B_j \ra 
		\left( P^A_{n-(j-1)} \right)_{\bar{\C}} \oplus 
		P^A_{n+(j-1)}$
as the maps given on the basis elements by $a\otimes b \mapsto a \Phi(1\otimes b)$ and extend linearly.
It is easy to check that $\Phi_1$ is a graded $(\Aa_{2n-1},\R)$-bimodule morphism and $\Phi_j$ is a graded $(\Aa_{2n-1},\C)$-bimodule morphism; the only detail that one should be aware of is that $\Phi_j$ maps into $\left( P^A_{n-(j-1)} \right)_{\bar{\C}} \oplus 
		P^A_{n+(j-1)}$ instead of $P^A_{n-(j-1)}\oplus P^A_{n+(j-1)}$.
The fact that they are isomorphisms follows easily from looking at the dimensions.
\end{proof}

It follows from the above proposition that $\Aa_{2n-1} \otimes_{\Ba_n} -$ sends projectives to projectives.
Therefore $\Aa_{2n-1} \otimes_{\Ba_n} -$ extends to a functor from $\Kom^b(\ddot{\Ba}_n$-$\text{p$_{r}$g$_{r}$mod})$ to $\Kom^b(\Aa_{2n-1}$-$\text{p$_{r}$g$_{r}$mod})$.
We will denote the functor
$$
\Aa_{2n-1} \otimes_{\Ba_n} - := \Aa_{2n-1} \otimes_{\ddot{\Ba}_n} (\mathfrak{F}(-)) : \Kom^b(\Ba_n\text{-p$_{r}$g$_{r}$mod})\ra \Kom^b(\Aa_{2n-1}\text{-p$_{r}$g$_{r}$mod}).
$$

Recall the injection $\Psi: \mathcal{A}(B_n) \ra \mathcal{A}(A_{2n-1}) $  as defined in \cref{injB}; the image of standard generators are explicitly given by:
\begin{align*}
\Psi(\sigma^B_j) &= 
\begin{cases}
\sigma^A_n, & \text{for } j = 1; \\
\sigma^A_{n-(j-1)} \sigma^A_{n+(j-1)}, & \text{otherwise}.
\end{cases}
\end{align*}
We have previously shown that $\mathcal{A}(B_n)$ acts on $\Kom^b(\Ba_n$-$\text{p$_{r}$g$_{r}$mod})$ and similarly $\mathcal{A}(A_{2n-1})$ acts on $\Kom^b(\Aa_{2n-1}$-$\text{p$_{r}$g$_{r}$mod})$.
Through $\Psi$, we have an induced action of $\mathcal{A}(B_n)$ on   $\Kom^b(\Aa_{2n-1}$-$\text{p$_{r}$g$_{r}$mod})$.
We will now prove the algebraic version of \cref{topological equivariant}.

\begin{theorem}\label{tensor equivariant}
For all $1 \leq j \leq n$, we have an isomorphism $
\Aa_{2n-1} \otimes_{\Ba_n} \sigma^B_j \cong \Psi(\sigma^B_j)_{\Ba_n}$
in $\Kom^b((\Aa_{2n-1},\Ba_n)\text{-bimod})$.
In particular, the functor $\Aa_{2n-1} \otimes_{\Ba_n} -$ is $\mathcal{A}(B_n)$-equivariant:
\begin{center}
\begin{tikzpicture} [scale=0.8]
\node (tbB) at (-2.5,1.5) 
	{$\mathcal{A}(B_n)$};

\node (tbA) at (10,1.5) 
	{$\mathcal{A}(A_{2n-1}) \overset{\Psi}{\hookleftarrow} \mathcal{A}(B_n) $};

\node[align=center] (cB) at (0,0) 
	{$\Kom^b(\Ba_n$-$p_rg_rmod)$};
\node[align=center] (cA) at (7,0) 
	{$\Kom^b(\Aa_{2n-1}$-$p_rg_rmod),$};

\coordinate (tbB') at ($(tbB.east) + (0,-1)$);

\coordinate (tbA') at ($(tbA.west) + (0,-1)$);

\draw [->,shorten >=-1.5pt, dashed] (tbB') arc (245:-70:2.5ex);

\draw [->, shorten >=-1.5pt, dashed] (tbA') arc (-65:250:2.5ex);

\draw[->] (cB) -- (cA) node[midway,above]{$\Aa_{2n-1}\otimes_{\Ba_n} -$}; 

\end{tikzpicture}
\end{center}
i.e. for any $\sigma \in \mathcal{A}(B_n)$ and any complex $C \in \Kom^b(\Ba_n$-$p_rg_rmod)$,
$$
\Aa_{2n-1} \otimes_{\Ba_n} (\sigma \otimes_{\Ba_n} C) \cong \Psi(\sigma)\otimes_{\Aa_{2n-1}} (\Aa_{2n-1} \otimes_{\Ba_n} C).
$$
\end{theorem}

Before we start with the proof, we will need the following lemma:
\begin{lemma} \label{lemma bimod iso}
Recall the deformed $\C$-action given in \eqref{eqn: deformed C-action}. We have that
\begin{align}
\C\otimes_\R {}_1P^B &\cong ({}_nP^A)_{\Ba_n}, \quad &\text{as $\Z$-graded $(\C, \Ba_n)$-bimodules}; 
	\label{lemma 1}\\
{}_jP^B &\cong \left({}_{n+ (j-1)}P^A \right)_{\Ba_n},
	\quad &\text{as $\Z$-graded $(\C, \Ba_n)$-bimodules;}
	\label{lemma 2}\\
{}_jP^B &\cong 
	{}_{\bar{\C}} \left( {}_{n- (j-1)}P^A
		\right)_{\Ba_n},
	\quad &\text{as $\Z$-graded $(\C, \Ba_n)$-bimodules;}
	\label{lemma 3} \\
{}_\C \C \otimes_{\bar{\C}} \C_\C &\cong {}_\C \C_\C,
	\quad &\text{as $\Z$-graded $(\C,\C)$-bimodules.}
	\label{lemma 4}
\end{align}
\end{lemma}
\begin{proof}
We will only define the maps; the proofs that they are isomorphisms respecting the required structures follows from a simple verification.

For \eqref{lemma 1} take the morphism $\phi_1 : \C \otimes_\R {}_1P^B \ra ({}_nP^A)_{\Ba_n}$ as the restriction of $\Phi$ constructed in the proof of \cref{isomorphic algebras}.
Note that $\phi_1$ does indeed map into $({}_nP^A)_{\Ba_n}$ since
\[
\Phi(c\otimes b) = \Phi(c\otimes e_1b)= \Phi((1\otimes e_1)(c\otimes b)) =  \Phi(1\otimes e_1)\Phi(c\otimes b) = e_n\Phi(c\otimes b).
\]

For \eqref{lemma 2} and \eqref{lemma 3}, consider the morphisms
\begin{align*}
\phi_{+j} : {}_jP^B  &\ra 
	\left({}_{n+ (j-1)}P^A \right)_{\Ba_n} &\text{and}  &&\phi_{-j} : {}_jP^B  &\ra 
	{}_{\bar{\C}}\left({}_{n- (j-1)}P^A \right)_{\Ba_n} \\
b &\mapsto 
	e_{n+ (j-1)}\Phi(1\otimes b),  &    &&b &\mapsto 
	e_{n-(j-1)}\Phi(1\otimes b). 
\end{align*}

Finally for \eqref{lemma 4}, consider the morphism $c: {}_\C \C \otimes_{\bar{\C}} \C_\C \ra \C$ uniquely defined by $
1\otimes 1 \mapsto 1.$
\end{proof}


\begin{proof}[Proof of \cref{tensor equivariant}]
We shall show the required statement by showing it for each generator, namely $\Psi(\sigma_j^B)_\Ba \cong \Aa \otimes_\Ba \sigma_j^B$ as complexes of $(\Aa, \Ba)$-bimodules for each $j$.

Let $j\geq 2$.
Using the relevant isomorphisms in \cref{lemma bimod iso}, we have the following chain of bimodule isomorphisms:
\begin{align*}
&\left(P_{n-(j-1)}^A \otimes_\C \left({}_{n-(j-1)}P^A\right)_\Ba \right) \oplus \left(P_{n+(j-1)}^A \otimes_\C \left( {}_{n+(j-1)}P^A\right)_\Ba \right)  \\
&\cong \left(P_{n-(j-1)}^A \otimes_{\bar{\C}} \left( {}_{n-(j-1)}P^A \right)_\Ba \right) \oplus \left(P_{n+(j-1)}^A \otimes_\C \left( {}_{n+(j-1)}P^A \right)_\Ba \right) \quad (\text{by }\eqref{lemma 4}) \\
&\cong \left( \left( P_{n-(j-1)}^A \right)_{\bar{\C}} \otimes_\C {}_{j}P^B \right) \oplus \left(P_{n+(j-1)}^A \otimes_\C {}_{j}P^B \right) \qquad \qquad (\text{by } \eqref{lemma 2} \text{ and } \eqref{lemma 3}) \\
&\cong \left( \left( P_{n-(j-1)}^A \right)_{\bar{\C}} \oplus P_{n+(j-1)}^A \right) \otimes_\C {}_{j}P^B.
\end{align*}
Using \cref{B tensor to A}, we have that 
\[
\left( \left( P_{n-(j-1)}^A \right)_{\bar{\C}} \oplus P_{n+(j-1)}^A \right) \otimes_\C {}_{j}P^B 
\cong \Aa \otimes_\Ba P_j^B \otimes_\C {}_jP^B.
\]
Let us denote the composition of this chain of isomorphisms by
\[
\Xi: \left(P_{n-(j-1)}^A \otimes_\C \left({}_{n-(j-1)}P^A\right)_\Ba \right) \oplus \left(P_{n+(j-1)}^A \otimes_\C \left( {}_{n+(j-1)}P^A\right)_\Ba \right)
\xra{\cong}
\Aa \otimes_\Ba P_j^B \otimes_\C {}_jP^B.
\]
Since we have that
\begin{align*}
\Psi(\sigma_j^B)_\Ba 
&= \left(\sigma_{n-(j-1)}^A\sigma_{n+(j-1)}^A \right)_\Ba \\
&= \left(P_{n-(j-1)}^A \otimes_\C \left({}_{n-(j-1)}P^A\right)_\Ba \right) \oplus \left(P_{n+(j-1)}^A \otimes_\C \left( {}_{n+(j-1)}P^A\right)_\Ba \right) 
	\xra{
		\begin{bsmallmatrix}
		\beta_{n-(j-1)}^A & \beta_{n+(j-1)}^A
		\end{bsmallmatrix}
	} 
	\Aa_\Ba
\end{align*}
and
\begin{align*}
\Aa \otimes_\Ba \sigma_j^B
&= \Aa \otimes_\Ba \left( P_j^B \otimes_\C {}_jP^B 
	\xra{\beta_j^B} \Ba \right) \\
&= \Aa \otimes_\Ba P_j^B \otimes_\C {}_jP^B 
	\xra{\id \otimes_\Ba \beta_j^B} 
\Aa \otimes_\Ba \Ba,
\end{align*}
all that is left is to show that the following diagram commutes:
\[
\begin{tikzcd}[column sep=80, ampersand replacement=\&]
\left(P_{n-(j-1)}^A \otimes_\C \left({}_{n-(j-1)}P^A\right)_\Ba \right) \oplus \left(P_{n+(j-1)}^A \otimes_\C \left( {}_{n+(j-1)}P^A\right)_\Ba \right)
	\arrow{r}{ 
		\begin{bmatrix} 
		\beta_{n-(j-1)}^A & \beta_{n+(j-1)}^A 
		\end{bmatrix}}
	\ar[d, "\Xi"] \& 
\Aa_\Ba 
	\ar[d, "\cong"] \\
\Aa \otimes_\Ba P_j^B \otimes_\C {}_jP^B  
	\ar[r, "\id \otimes_\Ba \beta_j^B"] \& 
\Aa \otimes_\Ba \Ba,
\end{tikzcd}
\]
which we leave for the reader to verify.

The proof for $j=1$ is simpler and follows from a similar argument.
\end{proof}
\begin{remark}
Define $U_j := P^B_j \otimes {}_jP^B$ and $\cU_j := P^A_j \otimes {}_jP^A$.
\cref{B tensor to A} implies that 
\[
\Aa_{2n-1} \otimes_{\Ba_n} U_j \cong
\begin{cases}
\left( \mathcal{U}_n \right)_{\Ba_n}, &\text{ for } j = 1; \\
\left( \mathcal{U}_{n-(j-1)} \oplus \mathcal{U}_{n+(j-1)} \right)_{\Ba_n}, &\text{ otherwise.}
\end{cases}
\]
\noindent 
When $n=2$, this also relates our bimodules to the ones given in \cite{MackTubb}, where $\cU_2 = \Theta_t$ and $\cU_1 \oplus \cU_3 = \Theta_s$ in \cite[Example 2.12]{MackTubb} for the $A_3$ graph (up to a difference in grading).
\end{remark}
We may now use this relation to deduce that the categorical action of $\cA(B_n)$ is faithful:
\begin{theorem}\label{faithful action}
The (weak) action of $\cA(B_n)$ on the category $\Kom^b(\Ba_n$-$\text{p$_{r}$g$_{r}$mod})$ given in \cref{Cat B action} is faithful.
\end{theorem}
\begin{proof}
Assume that we are given $\sigma \in \cA(B_n)$ such that
$ \sigma\left(C\right) \cong C $
for all $C \in \Kom^b(\Ba_n$-$\text{p$_{r}$g$_{r}$mod})$.
We will show that this implies $\sigma$ is the identity.
In particular, take $C = \oplus_{j=1}^n P^B_j$ so that we have \[\sigma\left(\oplus_{j=1}^n P^B_j \right) \cong \oplus_{j=1}^n P^B_j.\]
Applying the functor $\Aa_{2n-1} \otimes_{\Ba_n} -$, we obtain
\begin{equation}\label{faithful 1}
\Aa_{2n-1}\otimes_{\Ba_n} \sigma\left(\oplus_{j=1}^n P^B_j \right) \cong 
	\Aa_{2n-1}\otimes_{\Ba_n} \left(\oplus_{j=1}^n P^B_j \right) \cong 
	\oplus_{j=1}^{2n-1} P^A_j.
\end{equation}
Applying \cref{tensor equivariant} to the LHS of \eqref{faithful 1}, we get 
\begin{equation}\label{faithful 2}
\Aa_{2n-1}\otimes_{\Ba_n} \sigma \left(\oplus_{j=1}^n P^B_j \right)
	\cong \Psi(\sigma) \left( \Aa_{2n-1} \otimes_{\Ba_n} \left(\oplus_{j=1}^n P^B_j \right) \right) \cong \Psi(\sigma) \left( \oplus_{j=1}^{2n-1} P^A_j \right).
\end{equation}
Combining the two equations \eqref{faithful 1} and \eqref{faithful 2} above we deduce that
\[
\Psi(\sigma) \left( \oplus_{j=1}^{2n-1} P^A_j \right) \cong \oplus_{j=1}^{2n-1} P^A_j.
\]
Since it was shown in \cite[Corollary 1.2]{KhoSei} that the type $A$ categorical action is faithful,  we conclude that $\Psi(\sigma) = \id$.
But $\Psi$ is injective, so we must have that $\sigma = \id$ as required.
\end{proof}


\section{Main Theorem} \label{main theorem}
Let us state the main theorem that we aim to prove:
\begin{theorem} \label{fullmaintheorem}
The diagram in \cref{full picture} is commutative, where the maps $\mathfrak{m}, L_B, L_A$ and $\Aa_{2n-1} \otimes_{\Ba_n} -$ are all $\mathcal{A}(B_n)$-equivariant.
\end{theorem}
%
%
\begin{figure}[H] 
\begin{tikzpicture} [scale = 0.9]
\node (tbB) at (-3,1.75) 
	{$\mathcal{A}(B_n)$};
\node (cbB) at (-3,-3.5) 
	{$\mathcal{A}(B_n)$};
\node (tbA) at (10.5,1.75) 
	{$\mathcal{A}(A_{2n-1}) \overset{\Psi}{\hookleftarrow} \mathcal{A}(B_n) $}; 
\node (cbA) at (10.5,-3.5) 
	{$\mathcal{A}(A_{2n-1}) \overset{\Psi}{\hookleftarrow} \mathcal{A}(B_n) $};

\node[align=center] (cB) at (0,0) 
	{Isotopy classes of admissible \\ trigraded curves $\check{\fC}^{adm}$  in $\D^A_{n+1}$};
\node[align=center] (cA) at (7,0) 
	{Isotopy classes of admissible \\ bigraded multicurves $\ddot{\udfC}^{adm}$  in $\D^B_{2n}$};
\node (KB) at (0,-2)
	{$\Kom^b(\Ba_n$-$\text{p$_{r}$g$_{r}$mod})$};
\node (KA) at (7,-2) 
	{$\Kom^b(\Aa_{2n-1}$-$\text{p$_{r}$g$_{r}$mod})$};

\coordinate (tbB') at ($(tbB.east) + (0,-1)$);
\coordinate (cbB') at ($(cbB.east) + (0,1)$);
\coordinate (tbA') at ($(tbA.west) + (0,-1)$);
\coordinate (cbA') at ($(cbA.west) + (0,1)$);

\draw [->,shorten >=-1.5pt, dashed] (tbB') arc (245:-70:2.5ex);
\draw [->,shorten >=-1.5pt, dashed] (cbB') arc (-245:70:2.5ex);
\draw [->, shorten >=-1.5pt, dashed] (tbA') arc (-65:250:2.5ex);
\draw [->,shorten >=-1.5pt, dashed] (cbA') arc (65:-250:2.5ex);

\draw[->] (cB) -- (KB) node[midway, left]{$L_B$};
\draw[->] (cB) -- (cA) node[midway,above]{$\mathfrak{m}$}; 
\draw[->] (cA) -- (KA) node[midway,right]{$L_A$};
\draw[->] (KB) -- (KA) node[midway,above]{$\Aa_{2n-1} \otimes_{\Ba_n} -$};
\end{tikzpicture}
\caption{The commutative diagram of \cref{fullmaintheorem}. The first row is from \cref{topological equivariant} and the second row is from \cref{tensor equivariant}.}
\label{full picture}
\end{figure}
\noindent In \cref{topology}, we introduced and showed that $\mathfrak{m}$ is $\cA(B_n)$-equivariant.
In \cref{relating categorical b a action}, we showed that the functor $\Aa_{2n-1} \otimes_{\Ba_n} -$ is also $\cA(B_n)$-equivariant.
Thus, the missing pieces are:
\begin{enumerate}
\item the definitions of the maps $L_B$ and $L_A$;
\item the fact that the maps $L_B$ and $L_A$ are $\cA(B_n)$-equivariant; and
\item the commutativity of the diagram.
\end{enumerate}
We shall start by recalling the definition of $L_A$ from \cite[Section 4a]{KhoSei} and their result that $L_A$ is $\cA(A_{2n-1})$-equivariant.
We then define $L_B$ in \cref{subsect: L_B}, where the (technical) proof that the diagram commutes and that $L_B$ is $\cA(B_n)$-equivariant are given in \cref{subsect: technical results}.

\subsection{Complexes associated to admissible multi-curves (Type $A$)}

Here we state the constructions and results shown in \cite[Section 4]{KhoSei}, which can be easily extended to admissible multicurves.
Note that we added a subscript $L_A$ instead of $L$ used in \cite{KhoSei} to differentiate between type $A$ and type $B$ later on.
Let $\ddot{c}$ be a bigraded admissible curve in normal form.
We associate to $\ddot{c}$ an object $L_A(\ddot{c})$ in the category $\Kom^b(\Aa_{2n-1}$-$\text{p$_{r}$g$_{r}$mod})$.
Start by defining $L_A(\ddot{c})$ as a bigraded $\Aa_{2n-1}$-module:
\begin{align}
L_A(\ddot{c}) = \bigoplus_{x \in cr(\ddot{c})}  P(x),
\end{align}
where $P(x) = P^A_{x_0}[-x_1]\{x_2\}$, with $[-]$ denoting the cohomological degree shift; see the paragraph after \cref{basic curves A} for the definition of $(x_0, x_1, x_2)$.
For every $x,y \in cr(\ddot{c})$ define
$
\partial_{yx}: P(x) \ra P(y)
$
by the following rules:
\begin{itemize}
\item If $x$ and $y$ are the endpoints of an essential segment and $y_1 = x_1 + 1$, then
\begin{enumerate}
\item If $x_0 = y_0$ (then it must be that $x_2 = y_2 + 1$) then
\[
\partial_{yx}: P(x) \ra P(y) \cong P(x)[-1]\{1\}
\]
is the multiplication on the right by $X_{x_0} \in \Aa_{2n-1}$.
\item If $x_0 = y_0 \pm 1$, then $\partial_{yx}$ is the right multiplication by $(x_0|y_0) \in \Aa_{2n-1}$;
\end{enumerate}
\item otherwise $\partial_{yx} = 0$.
\end{itemize}
We define the differential $\partial$ as
$
\partial := \sum_{x,y} \partial_{yx}.
$
See \cite[Lemma 4.1]{KhoSei} for a proof that this defines a complex.
Moreover, it follows easily that 
\begin{equation} \label{action shift comm A}
L_A(\chi_A(r_1, r_2)\ddot{c}) \cong L_A(\ddot{c})[-r_1]\{r_2\}.
\end{equation}
For $\ddot{g}$ a bigraded $j$-string of $\ddot{c}$, we can also assign a complex $L_A(\ddot{g})$ to $\ddot{g}$, where as a bigraded abelian group, $L_A(\ddot{g}) = \bigoplus_{x \in cr{(g)}} P^A(x)$ and the differentials are obtained from essential segments of $\ddot{g}$ the same way as for admissible curves.
We can easily extend this to define $L_A(\ddot{h})$ for $h \subseteq c$ a connected subset of $c$ such that $h = \cup g_{\alpha, j}$ with each $g_{\alpha, j}$ some bigraded $j$-string of $c$.
The following theorem is proven in \cite[Theorem 4.3]{KhoSei}:
\begin{theorem} \label{L_A equivariant}
For a braid $\sigma \in \cA(A_{m})$ and a bigraded admissible curve $\ddot{c}$ in $\D^A_{m+1}$, we have $\sigma L_A(\ddot{c}) \cong L_A(\sigma(\ddot{c}))$ in the category $\Kom^b(\Aa_m$-$p_rg_rmod)$, i.e. $L_A$ is $\cA(A_m)$-equivariant.
\end{theorem}
We extend $L_A$ to admissible multicurves as follows:
given bigraded multicurves $\coprod_j \ddot{c}_j$,
\[
L_A \bigg(  \coprod_j \ddot{c}_j \bigg) := \bigoplus_j L_A(\ddot{c}_j).
\]
It follows easily that this defines a complex, and both \eqref{action shift comm A} and \cref{L_A equivariant} still hold for admissible multicurves.

\subsection{Complexes associated to admissible curves (Type $B$)} \label{subsect: L_B}
Consider a trigraded admissible  curve $\check{c}$.
We associate to $\check{c}$ an object $L_B(\check{c})$ in the category $\Kom^b(\Ba_n$-$p_rg_rmod)$.
Start by defining $L_B(\check{c})$ as a trigraded $\Ba_n$-module:
\[
L_B(\check{c}) = \bigoplus_{y \in cr(\check{c})}  P(y)
\]
where $P(y) = P^B_{y_0}[-y_1]\{y_2\}\<y_3\>$ (see the second paragraph after \cref{injBA} in \cref{local index function} for definition of $y_0, y_1, y_2$ and $y_3$).

We now define maps $\partial_{zy}: P(y) \ra P(z)$ for each $y,z \in cr(\check{c})$ using the following rules (note that these are \emph{\textbf{not}} the differentials yet):
\begin{itemize}
\item If $y$ and $z$ are the endpoints of an essential segment in $D_j$ for $j \geq 1$ and $z_1 = y_1 + 1$, then
	\begin{enumerate}[1.]
 	\item If $y_0 = z_0 $ (then also $y_2 = z_2 + 1$ and $y_3 = z_3$), then
 	$$ \partial_{zy}: P(y) \ra P(z) \cong P(y)[-1]\{1\} $$
 	is the right multiplication by the element $X_{y_0} \in \Ba_n.$
 	\item If  $y_0 = z_0 \pm 1$, then $\partial_{zy}$
  	is the right multiplication by $(y_0 | z_0) \in \Ba_n$.
\end{enumerate}
\item Otherwise, $\partial_{yz} = 0.$
\end{itemize}
We will modify some of these maps before using them as differentials.
Define the following equivalence relation on the set of 1-crossings:
\[
y \sim y' \iff y \text{ and } y' \text{ are connected by an essential segment in } D_0.
\]
Consider the partitioning of the set of 1-crossings using the equivalence relation above.
Referring to the possible normal forms in $D_0$ given by \Cref{2types0}, every equivalence classes under this relation consists of either one element (Type 3')  or two elements (Type 2'').
For each equivalence class $[y]$ of 1-crossings, we modify the some of the maps given previously by the following rule:
\begin{itemize}
\item If $[y] = \{y \}$, we modify nothing;
\item otherwise, $[y] = \{y,y'\}$ has two distinct 1-crossings. 
Note that at least one of the 1-crossings must be an endpoint of some essential segment in $D_1$ (otherwise $c$ is clearly not be admissible).
Up to relabelling, we may assume that $y$ is always a 1-crossing that is an endpoint of an essential segment $\gamma$ in $D_1$.
Note that $\gamma$ cannot have both $y$ and $y'$ as endpoints (this will imply that $c$ is a simple closed curve, which is not admissible).
As such, let us label the other endpoint of $\gamma$ connecting $y$ as $z \neq y'$.
Consider the two possible cases for $z$:
	\begin{enumerate}[1.]
	\item $z$ is a 2-crossing:
		\begin{enumerate}[(a)]
		\item if $y_1 = z_1 + 1$, then we have that $\partial_{yz}: P(z) \ra P(y)$ is given by the right multiplication by $(2|1)$. 
		We modify 
		$\partial_{y'z}: P(z) \ra P(y')\cong P(y)\<1\>$ (which was necessarily 0 previously) so that it is now the right multiplication by $-i(2|1)$;
		\item otherwise, we have instead $z_1 = y_1 + 1$.
		In this case, $\partial_{zy}: P(y) \ra P(z)$ is given by the right multiplication by $(1|2)$. 
		We modify $\partial_{zy'}: P(y') \cong P(y)\<1\> \ra P(z)$ (which was necessarily 0 previously) so that it is now the right multiplication by $(1|2)i$.
		\end{enumerate}
	\item $z$ is a 1-crossing: 
		\begin{enumerate}[(a)]
		\item if $y_1 = z_1 + 1$, we modify nothing;
		\item otherwise, we have instead $z_1 = y_1 + 1$.
		In this case $\partial_{zy}: P(y) \ra P(z)$ is given by the right multiplication by $X_1$.
		Once again consider the two possible cases of the equivalence class $[z]$:  
  			\begin{enumerate}[(i)] 
 			\item If $[z] = \{z \}$, we modify nothing;
			\item otherwise $[z] = \{z, z'\}$ with $z\neq z'$.
 			We then modify $\partial_{z'y'}: P(y)\<1\> \cong P(y') \ra P(z') \cong P(z)\<1\>$ (which was necessarily $0$ previously) so that it is now the right multiplication by $X_1$;
			\end{enumerate}
		\end{enumerate}
	\end{enumerate} 
	Similarly, if the other 1-crossing $y'$ in $[y]$ is also an endpoint of an essential segment in $D_1$ (distinct from $\gamma$), we shall repeat the same process above for $y'$.
\end{itemize}
Finally, we define the differential as
$
\partial = \sum_{x,y \in cr(\check{c})} \partial_{xy},
$
where $\partial_{xy}$ are the \emph{modified} version above.

\begin{lemma}\label{check complex}
$(L_B(\check{c}), \partial)$ is a complex of projective  graded $\Ba_n$-modules with a grading-preserving differential.
\end{lemma}
\begin{proof}
For $x, y, z \in cr(\check{c})$ with $x_0,y_0,z_0 \geq 2$, the same argument as in the type $A$ shows that the product of $\partial_{zy} \partial_{yx}:P_x \ra P_z$ is always 0.
%
%
The only occurrence of $\partial_{zy} \partial_{yx} \neq 0$ is when $\partial_{zy} = \partial_{ac}, \partial_{ad}$ and $\partial_{yx} = \partial_{db}, \partial_{cb}$ with $a,b,c,d$ the crossings of the following type of 1-string labelled below:
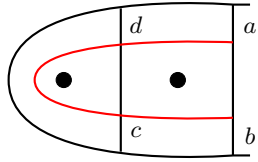
\begin{figure}[H]
\begin{tikzpicture}  [scale=1]
\draw[thick]  (5.475,1)--(5.7,1);
\draw[thick]  (5.475,3)--(5.7,3);  
\draw[thick]  (5.475,1)--(5.475,3);
\draw[thick,red]  (3.25,2.5)--(5.475,2.5);
\draw[thick,red]  (3.25,1.5)--(5.475,1.5);
\draw[thick]  (3.25,3)--(5.7,3);  
\draw[thick]  (3.25,1)--(5.7,1);
\draw[thick]  (4,1)--(4,3);
\draw[thick,green, dashed]  (3.25,3)--(3.25,2.5);
\draw[thick, green, dashed]   (3.25,1)--(3.25,1.5);
\draw[thick, green, dashed, ->]   (3.25,2)--(3.25,2.6);
\draw[thick,  green, dashed, ->]   (3.25,2)--(3.25,1.4);\filldraw[color=black!, fill=yellow!, very thick] (3.25,2) circle [radius=0.1]  ;
\draw[fill] (4.75,2) circle [radius=0.1]  ;

\node [above right] at (5.475,2.5) {{\small $a$}};
\node [below right] at (5.475,1.5) {{\small $b$}};
\node [below right] at (4,1.5) {{\small $c$}};
\node [above right] at (4,2.5) {{\small $d$}};
\end{tikzpicture}.
\end{figure}
Note that the two non-zero composition $\partial_{ad}\partial_{db}$ and $\partial_{ac}\partial_{cb}$ always occur as a pair.
Moreover, we see that their sum is equal to 0: $\partial_{ad}\partial_{db} + \partial_{ac}\partial_{cb} = X_2i - X_2i = 0$, thus showing that $\partial^2 = 0$ as required.
%
%
%
%
%
%
%
%
\end{proof}

\begin{lemma}

 For any triple $(r_1,r_2,r_3)$ of integers and any  trigraded admissible curve $\check{c}$ we have:
 $$ L_B(\chi(r_1,r_2,r_3)\check{c}) \cong L_B(\check{c})[-r_1]\{r_2\}\<r_3\>. $$

\end{lemma}
\begin{proof}
This follows directly from the definition.
\end{proof}

\subsection{Some rather technical results} \label{subsect: technical results}
This subsection is where we complete the proof of \cref{fullmaintheorem} by proving that the diagram commutes (\cref{diagram commutes}) and that $L_B$ is $\cA(B_n)$-equivariant (\cref{L_B equivariant}).
The proof of these two statements, first of which is technical, will occupy the rest of this section.

\begin{proposition} \label{diagram commutes}
The diagram in \cref{full picture} commutes, i.e. for each trigraded admissible curve $\check{c}$ in $\D^B_{n+1}$ we have that 
$
\Aa_{2n-1}\otimes_{\Ba_n} L_B(\check{c})\cong L_A (\mathfrak{m}(\check{c}))
$
in $\Kom^b(\Aa_{2n-1}$-$p_rg_rmod)$.
\end{proposition}
\begin{remark}
The theorem is stated in the context of homotopy category since the group actions of $\cA(A_{2n-1})$ and $\cA(B_n)$ live there. 
In fact, we will show that
$
\Aa_{2n-1}\otimes_{\Ba_n} L_B(\check{c})\cong L_A (\mathfrak{m}(\check{c}))
$
in the category $\Com^b(\Ba_n$-$\text{p$_{r}$g$_{r}$mod})$ of bounded complexes  (no homotopy will be required).
\end{remark}
\begin{proof}[Proof of \cref{diagram commutes}]
Let $x$ be any $j$-crossing of $\check{c}$.
If $j \geq 2$, we have that $\mathfrak{m}(x)$ consists of a $(n+(j-1))$-crossing $\wt{x}$ and a $(n-(j-1))$-crossing $\undertilde{x}$ of $\mathfrak{m}(\check{c})$; 
if $j=1$, $\mathfrak{m}(x)$ consists of a single $n$-crossing $\undertilde{\wt{x}}$ of $\mathfrak{m}(\check{c})$.
In either cases, we have isomorphisms $\Phi_j: \Aa_{2n-1} \otimes_{\Ba_n} P^B(x) \ra P^A(\wt{x}) \oplus P^A(\undertilde{x}) = P^A_{n+(j-1)} \oplus P^A_{n-(j-1)}$ or $\Phi_1: \Aa_{2n-1} \otimes_{\Ba_n} P^B(x) \ra P^A(\undertilde{\wt{x}}) = P^A_n$ given in the proof of \cref{B tensor to A}.
Putting together these isomorphisms for each crossing $x$ of $\check{c}$, we obtain a cohomological and internal grading preserving isomorphism of $\Aa_{2n-1}$-modules between the underlying modules of $\Aa_{2n-1} \otimes_{\Ba_n} L_B(\check{c})$ and $L_A (\mathfrak{m}(\check{c}))$; denote this isomorphism by $\eta$.
Denoting the complexes $\Aa_{2n-1} \otimes_{\Ba_n} L_B(\check{c})$ as $(Q, \d)$ and $L_A (\mathfrak{m}(\check{c}))$ as $(Q', \d')$ (so $\eta$ is an isomorphism from $Q$ to $Q'$), it follows that $\eta$ induces an isomorphism of complexes:
\[
(Q, \d) \cong (Q',\d_0),
\]
with $\d_0 = \eta \d \eta^{-1}$.
We now aim to show that $(Q', \d_0) \cong (Q', \d')$ in $\Kom^b(\Aa_{2n-1}$-$\text{p$_{r}$g$_{r}$mod})$.
In fact, we will show that they are isomorphic in the ordinary category $\Com^b(\Aa_{2n-1}$-$\text{p$_{r}$g$_{r}$mod})$ of complexes in $\Aa_{2n-1}$-$\text{p$_{r}$g$_{r}$mod}$.
Before we proceed with the proof, we will need to introduce some (substantial amount of) notation.

\paragraph{(\textbf{Slicing $c$})}\label{slicing}
Recall that $c$, being admissible, must have both of its end points at two \emph{distinct} marked points, so at least one of its end points is not $0$ and let $m$ be such an end point.
Orient the curve $c$ so that it starts from $m$ and ends at its other end point.
Following this orientation, we can slice $c$ into distinct connected components $c_j \subset c \cap \left( \cup_{j \geq 2} D_j \right)$ and $g_{j'} \subset c \cap \left( D_0 \cup D_1 \right)$ (note that $g_{j'}$ are the 1-strings of $c$) enumerated as follows:
following the orientation of $c$, $g_j$ denotes the $(j-1)$-th 1-string component of $c$; whereas $c_i$ is the component of $c$ that has a total of $i$ 1-strings components before $c_i$.
In other words, if $m \neq 1$, we start from $c_0$ to $g_0$ to $c_1$ and so on;
otherwise $m = 1$ and we start from $g_0$ to $c_0$ to $g_1$ and so on. 
Note that if $c$ has no 1-string component, $c=c_0$.
Following the same orientation, we shall also enumerate the $2$-crossings $r_t$ of $c$ (if any), starting the enumeration with $t=0$ if $m=1$; otherwise we start with $t=1$.
The figure below illustrates two examples with different starting point $m$.
\begin{figure}[H]  
\begin{subfigure}{.45 \textwidth}
\centering
\begin{tikzpicture} [scale=.6]

\draw[thick,green, dashed]  (0,5.5)--(0,3.875);
\draw[thick, green, dashed]   (0,-.5)--(0,1.125);
\draw[thick, green, dashed, ->]   (0,2.5)--(0,4);
\draw[thick,  green, dashed, ->]   (0,2.5)--(0,1);
\draw[fill] (3,2.5) circle [radius=0.1]  ;
\draw[fill] (6,2.5) circle [radius=0.1]  ;
\draw[fill] (9,2.5) circle [radius=0.1]  ;
\draw[thick, blue] (4.5,-.5) -- (4.5,5.5);

\draw  plot[smooth, tension=1]coordinates { (0,4.8) (2.25, 4.9) (4.5,4.8)  };
\draw  plot[smooth, tension=1.3]coordinates { (0, 3.65) (2.25, 3.75) (4.5,3.65) };
\draw  plot[smooth, tension=1]coordinates { (0, 1.35) (1.5, 1.65) (3, 2.5)};
\draw  plot[smooth, tension=1]coordinates { (0, .2) (3.1, .65) (6, 2.5)};
\draw   plot[smooth, tension=3]coordinates { (8, 2.5) (9.8, 1.9)  (9, 3.5)};
\draw   plot[smooth, tension=1]coordinates { (9, 3.5) (7, 4.3) (4.5,4.8) };
\draw   plot[smooth, tension=1]coordinates { (8, 2.5) (6.75, 3.2) (4.5,3.65)};

\filldraw[color=black!, fill=yellow!, thick]  (0,2.5) circle [radius=0.1];

\node [above, blue] at (7, 4.3) {$c_1$};
\node [below, red] at (2.25, .45) {$g_0$};
\node [above, red] at (2.25,3.75) {$g_1$};
\node [below] at (0.2, 2.5) {$0$};
\node [below, blue] at (5.5, 2) {$c_0$};
\node [right] at (4.5, 5) {$r_1$};
\node [right] at (4.5, 3.9) {$r_2$};
\node [right] at (4.5, 1.1) {$r_0$};

  \node[rotate=135] at (5.5,2.075) {{\tiny $\wedge$}};
\node[rotate=90] at (0.75, .175){{\tiny $\wedge$}};
\node[rotate=-90] at (0.75, 4.875){{\tiny $\wedge$}};
\node[rotate=90] at (0.75, 3.725){{\tiny $\wedge$}};
\node[rotate=-75] at (0.75, 1.425){{\tiny $\wedge$}};
  \node[rotate=135] at (9.8,1.9) {{\tiny $\wedge$}};

\end{tikzpicture}
\caption{{\small A slicing with starting point $m=2\neq 1$.}} \label{notend01}
\end{subfigure}
\quad
\begin{subfigure}{.45 \textwidth}
\centering
\begin{tikzpicture} [scale=.6]

\draw[thick,green, dashed]  (12,5.5)--(12,3.875);
\draw[thick, green, dashed]   (12,-.5)--(12,1.125);
\draw[thick, green, dashed, ->]   (12,2.5)--(12,4);
\draw[thick,  green, dashed, ->]   (12,2.5)--(12,1);
\draw[fill] (15,2.5) circle [radius=0.1]  ;
\draw[fill] (18,2.5) circle [radius=0.1]  ;
\draw[fill] (21,2.5) circle [radius=0.1]  ;
\draw[thick, blue] (16.5,-.5) -- (16.5,5.5);

\draw  plot[smooth, tension=1]coordinates { (12,4.5) (14.25, 4.9) (16.5,5)  };
\draw  plot[smooth, tension=1.3]coordinates { (12, 3.65) (14.25, 4.1) (16.5,4.2) };
\draw  plot[smooth, tension=1]coordinates { (12, 1.35) (13.5, 1.65) (15, 2.5)};
\draw  plot[smooth, tension=1]coordinates { (12,.5 ) (13.8, .7) (16.5, 1.2)};
\draw   plot[smooth, tension=3]coordinates { (20.25, 2.5) (21.8, 1.9)  (21, 3.5)};
\draw   plot[smooth, tension=1]coordinates { (21, 3.5) (19, 4.5) (16.5,5) };
\draw   plot[smooth, tension=1]coordinates { (20.25, 2.5) (18.75, 3.7) (16.5,4.2)};
\draw  plot[smooth, tension=1.3]coordinates { (12, 2.5) (13.8, 3.1) (16.5,3.4) };
\draw  plot[smooth, tension=1.5]coordinates { (16.5, 3.4) (19, 2.5) (16.5,1.2) };

\filldraw[color=black!, fill=yellow!, thick]  (12,2.5) circle [radius=0.1];

\node [above, blue] at (19, 3.7) {$c_1$};
\node [above, red] at (13.8, 1.8) {$g_0$};
\node [above, red] at (14.25,4.9) {$g_1$};
\node [above, red] at (14.25,3.2) {$g_2$};
\node [below] at (12.2, 2.5) {$0$};
\node [below, blue] at (18, 1.5) {$c_2$};
\node [right] at (16.5, 5.2) {$r_2$};
\node [right] at (16.5, 4.4) {$r_1$};
\node [right] at (16.5, 3.6) {$r_4$};
\node [right] at (16.5, 1) {$r_3$};

  \node[rotate=0] at (19,2.5) {{\tiny $\wedge$}};

\node[rotate=-90] at (12.75, .55){{\tiny $\wedge$}};
\node[rotate=105] at (12.75, 4.675){{\tiny $\wedge$}};

\node[rotate=105] at (12.75, 2.865){{\tiny $\wedge$}};

\node[rotate=-85] at (12.75, 3.85){{\tiny $\wedge$}};

\node[rotate=105] at (12.75, 1.425){{\tiny $\wedge$}};
  \node[rotate=315] at (21.8,1.9) {{\tiny $\wedge$}};

\end{tikzpicture}
\caption{{\small A slicing with starting point $m=1$.}} \label{end01}
\end{subfigure}
\caption{Examples of slicings of curves.}
\end{figure}
Now consider the following subsets of (graded) crossings of $\mathfrak{m}(\check{c})$:
\begin{equation} \label{subset crossing}
\begin{cases}
C_j := \mathfrak{m}(\check{c}_j) \cap (\bigcup_i \ddot{\theta}_i); \\
\bar{C}_j := \mathfrak{m}(\check{c}_j) \cap (\bigcup_{i\neq n-1, n+1} \ddot{\theta}_i);\\
G_{j'} := \mathfrak{m}(\check{g}_{j'}) \cap (\ddot{\theta}_{n-1} \cup \ddot{\theta}_n \cup \ddot{\theta}_{n+1}); \\
\bar{G}_{j'} := \mathfrak{m}(\check{g}_{j'}) \cap \ddot{\theta}_n; \text{ and} \\
R_j := \mathfrak{m}(\check{r}_j).
\end{cases}
\end{equation}
Note that by definition, the subsets of crossings $\bar{C}_i, R_j$ and $\bar{G}_k$ are pairwise disjoint, and  $\left( \coprod_j \bar{C}_j\right) \amalg \left( \coprod_j R_j \right) \amalg \left( \coprod_j \bar{G}_j \right)$ is the set of all crossings of $\mathfrak{m}(\check{c})$.
On the other hand, $C_j$ and $G_{j'}$ contains all the crossings of $\mathfrak{m}(c_j)$ and $\mathfrak{m}(g_{j'})$ respectively, which may contain common crossings $R_i$.

For $K$ a subset of crossings of $\mathfrak{m}(\check{c})$, we define 
\[
Q'_K := \bigoplus_{x \in K} P^A(x) \subseteq Q'.
\]
If $K$ is empty, $Q'_K$ will be the $0$ module by convention. 
Using this, we can decompose $Q'$ as follow:
when $m \in \Lambda \setminus \{0, 1\}$, we have
\begin{equation}\label{R decomp}
Q' = 
\rlap{$\underbrace{\phantom{Q'_{\bar{C}_0} \oplus Q'_{R_0}}}_{Q'_{C_0}}$} 
	Q'_{\bar{C}_0} \oplus 
\rlap{$\overbrace{\phantom{Q'_{R_0} \oplus Q'_{\bar{G}_0} \oplus Q'_{R_1}}}^{Q'_{G_0}}$}
	Q'_{R_0} \oplus Q'_{\bar{G}_0} \oplus 
\rlap{$\underbrace{\phantom{Q'_{R_1} \oplus Q'_{\bar{C}_1} \oplus Q'_{R_2}}}_{Q'_{C_1}}$}
	Q'_{R_1} \oplus Q'_{\bar{C}_1} \oplus Q'_{R_2} \oplus \dots \quad ;
\end{equation}
whereas when $m =1$, we have instead
\begin{equation}\label{R' decomp}
Q' = 
\rlap{$\underbrace{\phantom{Q'_{\bar{G}_0} \oplus Q'_{R_1}}}_{Q'_{G_0}}$} 
	Q'_{\bar{G}_0} \oplus 
\rlap{$\overbrace{\phantom{Q'_{R_1} \oplus Q'_{\bar{C}_1} \oplus Q'_{R_2}}}^{Q'_{C_1}}$}
	Q'_{R_1} \oplus Q'_{\bar{C}_1} \oplus 
\rlap{$\underbrace{\phantom{Q'_{R_2} \oplus Q'_{\bar{G}_2} \oplus Q'_{R_3}}}_{Q'_{G_2}}$}
	Q'_{R_2} \oplus Q'_{\bar{G}_2} \oplus Q'_{R_3} \oplus \dots \quad .
\end{equation}

In general, given a decomposition of modules $M = \oplus_{i\in Y} M_{K_i}$ and a complex $(M, \partial)$, we can then write $\partial$ as a block matrix.
We will use the notation $\partial_{K_i K_j}$ to denote the block of $\partial$ that maps from $M_{K_j}$ to $M_{K_i}$, where we use the shorthand notation $\partial_{K_i}$ for the block of $\partial$ that maps from $M_{K_i}$ to itself.
We will also use the notation $\partial_{\oplus_{i \in X \subseteq Y} K_i}$ for the block of $\partial$ that maps from $\oplus_{i \in X} M_{K_i}$ to itself.
To illustrate, consider the decomposition of $Q'$ as in \eqref{R decomp} and let $(Q', \partial)$ be a cochain complex with differential $\partial$.
We will then write the differential $\partial$ as the matrix
\begin{center}
\begin{tikzpicture}
\matrix (m) [mymatrix, inner sep = 2pt, column sep = 7pt, row sep = 5pt] {
\partial_{\bar{C}_0} & \partial_{R_0 \bar{C}_0} & \partial_{\bar{G}_0 \bar{C}_0} & \partial_{R_1 \bar{C}_0}  & \partial_{\bar{C}_1 \bar{C}_0} & \dots \\
 \partial_{\bar{C}_0 R_0}& \partial_{R_0} & \partial_{\bar{G}_0 R_0} & \partial_{R_1 R_0} &  \partial_{\bar{C}_1 R_0} & \dots \\
\partial_{\bar{C}_0 \bar{G}_0}  & \partial_{R_0 \bar{G}_0}  & \partial_{\bar{G}_0} & \partial_{R_1 \bar{G}_0} &  \partial_{\bar{C}_1 \bar{G}_0} & \dots \\
\partial_{\bar{C}_0 R_1} & \partial_{R_0 R_1}  & \partial_{\bar{G}_0 R_1}   & \partial_{R_1} & \partial_{\bar{C}_1 R_1} & \dots \\
\partial_{\bar{C}_0 \bar{C}_1}  & \partial_{R_0 \bar{C}_1}&\partial_{\bar{G}_0 \bar{C}_1} &\partial_{R_1 \bar{C}_1}  & \partial_{\bar{C}_1} & \dots \\
\vdots & \vdots & \vdots & \vdots & \vdots & \ddots \\
};

\draw[red] (m-2-2.east|-m-2-2.south)++(3pt,-2pt) rectangle (m-1-1.north-|m-2-1.west);
\draw[blue] (m-4-4.east|-m-4-4.south)++(3pt,-2pt) rectangle (m-2-2.north-|m-4-2.west);
\draw[red] (m-4-5.north east)-|(m-5-4.south west);
\end{tikzpicture},
\end{center}
where the blocks $\partial_{C_j}$ corresponding to the summands $Q'_{C_j}$ are the blocks in {\color{red} red} and similarly blocks $\partial_{G_j}$ corresponding to the summands $Q'_{G_j}$ are the blocks in {\color{blue} blue}.

After our lengthy notational digression, let us return to the proof and analyse the difference between the matrices of the two differentials $\d_0$ and $\d'$ corresponding to the two possible decomposition of $Q'$ as in \eqref{R decomp} and \eqref{R' decomp}.
By looking at how the components $c_i$ and $g_j$ are connected, it follows that the non-zero entries of the matrix of $\d_0$ are all contained in the block matrices $(\d_0)_{C_k}$ and $(\d_0)_{G_k}$ for all $k$;
similarly the connection between the components $\mathfrak{m}(c_i)$ and $\mathfrak{m}(g_j)$ dictates that the non-zero entries of the matrix of $\d'$ are all contained in the block matrices $\d'_{C_k}$ and $\d'_{G_k}$ for all $k$.
One can show from a direct computation that
\begin{equation} \label{C diff component}
(\d_0)_{C_k} = \d'_{C_k}, \quad \text{for all } k.
\end{equation}
So between the differentials $\d_0$ and $\d'$, only the block matrices $(\d_0)_{G_k}$ and $\d'_{G_k}$ may differ.
As such, if we were in the case where $c$ has no 1-string, i.e. $c = c_0$, then we are done.
In the rest of the proof, we shall assume otherwise and treat the rest of the cases.

As we will only need to focus on the module summand $Q'_{G_j}$ later on, for each $j$ let us simplify both the decompositions of $Q'$ in \eqref{R decomp} and \eqref{R' decomp} into
\begin{equation} \label{new R decomp}
Q' = Q'_{V_j} \oplus 
	\rlap{$\underbrace{
		\phantom{
			Q'_{R_j} \oplus Q'_{\bar{G_j}} \oplus Q'_{R_{j+1}}
		}
	}_{Q'_{G_j}}$}
	Q'_{R_j} \oplus Q'_{\bar{G_j}} \oplus Q'_{R_{j+1}} \oplus Q'_{W_j},
\end{equation}
where $Q'_{V_j}$ (resp. $Q'_{W_j}$) consists of all the module summands before $Q'_{R_j}$ (resp. after $Q'_{R_{j+1}}$) for both decompositions \eqref{R decomp} and \eqref{R' decomp}.
From here on we will use this simplified decomposition for the matrix of any differential on $Q'$.

Let $g_0, g_1,... , g_{s-1}$ be the 1-strings in $c$.
To show $(Q', \d_0) \cong (Q', \d')$, by \eqref{C diff component} and \eqref{new R decomp} it is sufficient to construct a cohomological and internal grading preserving isomorphism of modules $\mu_j: Q' \ra Q'$ for each $ 0 \leq j \leq s-1$, so that we have an induced chain of isomorphisms in $\Com^b(\Aa_{2n-1}\text{-p$_{r}$g$_{r}$mod})$:
$$
(Q', \d_0)  \cong (Q', \d_1) \cong \cdots  \cong (Q', \d_{s-1}) \cong (Q', \d_s) = (Q', \d');
$$
with $\d_{j+1} := \mu_j \d_j \mu_j^{-1}$, where each $\d_j$ for $1 \leq j \leq s$ satisfies the following property:

\begin{equation} \label{induction} \tag{*}
\begin{cases}
(\d_{j})_{G_{j-1}} = \d'_{G_{j-1}}, &\text{{\small and}} \\
(\d_{j})_{XY} = (\d_{j-1})_{XY}, &\text{{\small for all $X,Y\in \{V_{j-1},G_{j-1},W_{j-1}\}$ such that $(X,Y) \neq (G_{j-1},G_{j-1})$}}.
\end{cases}
\end{equation}

In other words, each $\mu_j$ will be constructed in a way that the conjugation of $\d_j$ by $\mu_j^{-1}$ \emph{only} alters the differential component $(\d_j)_{G_j}$, so that $(\d_{j+1})_{G_j} = \mu_j(\d_j)_{G_j}\mu_j^{-1} = \d'_{G_j}$ without affecting the rest of the differential components.
In particular, if $\d_\ell$ satisfy \eqref{induction} for all $1 \leq \ell \leq j-1$, then we have that $(\d_j)_{G_j} = (\d_0)_{G_j}$.
Moreover, property \eqref{induction} will guarantee that $\d_s = \d'$.

What remains is to define the required $\mu_j : Q' \ra Q'$. 
We will define $\mu_j$ according to the type of 1-string $\check{g}_j$.
Within each possible types of 1-string $\check{g}_j$, we will show the following:
\begin{enumerate}
\item For $j=0$, we show that we can always construct $\mu_0$ to get $(Q'_0, \d_0) \cong (Q'_1, \d_1)$ such that $\d_1 = \mu_0\d_0 \mu_0^{-1}$ satisfies \eqref{induction}.
\item For $j\geq 1$, we show that given $(Q', \d_0) \cong \cdots \cong (Q'_j,\d_j)$ with $\d_1, ..., \d_j$ satisfying \eqref{induction} for $j \geq 1$, we can construct $ \mu_j: Q' \ra Q' $ such that $\d_{j+1} = \mu_j \d_j \mu_j^{-1}$ satisfies \eqref{induction}.
\end{enumerate}
By an induction on the total number of 1-strings components (over $j$), we can always construct a chain of isomorphisms $\mu_j$ with property \eqref{induction} and hence complete the proof of this theorem.

The rest of the proof is an extensive case-by-case analysis.
The list below shows that given each type of $\check{g}_j$, we can construct $ \mu_j: Q' \ra Q' $  satisfying \eqref{induction}.
We shall start with the simple cases: types IV, III$'_k$ and II$'_k$, followed by the types V$''$, III$'_{k + \frac{1}{2}}$ and II$'_{k + \frac{1}{2}}$ that requires some further analysis; refer to \cref{B 1-string} for the list of possible types of 1-strings.
Within the list, we will omit the gradings when writing out the modules in $Q'$ as it will be obvious from construction that $\mu_j$ preserves both the cohomological and internal gradings.
We will also use solid arrows for differentials and dashed arrows for the isomorphism $\mu_j : Q' \ra Q'$.
\begin{enumerate} 
\item $\check{g}_j$ is of Type VI: \\
This case is only possible when $j=0$ and $g_j = c$.
But in this case $(Q', \d_0)$ is given by $0 \ra P_n^A \ra 0$, which is already the complex $(Q', \d')$ as required ($\d_0 = 0 = \d')$.

\item $\check{g}_j$ is of Type III$'_k$ for $k \in \Z$: \\
The case when $k=0$ is straightforward, where we just pick $\mu_j$ to be the identity.

Now consider the case when $k > 0$.
If $j=0$, we get that $(\d_j)_{G_j} = (\d_0)_{G_j}$.
For $j \geq 1$, $\d_j$ satisfies \eqref{induction} by our inductive assumption, so we have that $(\d_j)_{G_j} = (\d_0)_{G_j}$.
Thus for all $j$, we can draw the part of $(Q', \d_j)$ that contains $(Q'_{G_i}, (\d_j)_{G_j}) = (Q'_{G_j}, (\d_0)_{G_j})$ as follows:
\begin{center}
\begin{tikzpicture}[>=stealth, baseline]
\matrix (M) [matrix of math nodes, column sep=7mm]
{
P_n^A  & P_n^A  & \cdots & P_n^A  & P_n^A  & P_{n-1}^A & \nabla\\
\oplus & \oplus &        & \oplus & \oplus & \oplus    &       & = (Q', \d_j),\\
P_n^A  & P_n^A  & \cdots & P_n^A  & P_n^A  & P_{n+1}^A &       \\
};

\draw[-> ,font=\small](M-1-1.east |- M-1-2) -- (M-1-2) 
	node[midway,above] {$2X_n$};
\draw[-> ,font=\small](M-1-2.east |- M-1-3) -- (M-1-3)
	;
\draw[-> ,font=\small](M-1-3.east |- M-1-4) -- (M-1-4)
	;
\draw[-> ,font=\small](M-1-4.east |- M-1-5) -- (M-1-5)
	node[midway,above] {$2X_n$}; 
\draw[-> ,font=\small](M-1-5) -- (M-1-6)
	node[midway,above] {$F$}; 
\draw[-> ,font=\small](M-1-5) -- (M-3-6)
	;
\draw[<->](M-1-6.east |- M-1-7) -- (M-1-7)
	;
\draw[<->](M-3-6) -- (M-1-7)
	;

\draw[-> ,font=\small](M-3-1.east |- M-3-2) -- (M-3-2) 
	node[midway,above] {$2X_n$};
\draw[-> ,font=\small](M-3-2.east |- M-3-3) -- (M-3-3)
	;
\draw[-> ,font=\small](M-3-3.east |- M-3-4) -- (M-3-4)
	;
\draw[-> ,font=\small](M-3-4.east |- M-3-5) -- (M-3-5)
	node[midway,above] {$2X_n$};
\draw[-> ,font=\small](M-3-5) -- (M-1-6)
	; 
\draw[-> ,font=\small](M-3-5) -- (M-3-6)
	;
\end{tikzpicture}
\end{center}
where $ F =
\begin{bsmallmatrix}
(n|n-1) & -(n|n-1)i \\
(n|n+1) &  (n|n+1)i
\end{bsmallmatrix}
=
\begin{bsmallmatrix}
(n|n-1) & 0 \\
0       & (n|n+1)
\end{bsmallmatrix}
\begin{bsmallmatrix}
1 & -i \\
1 &  i
\end{bsmallmatrix}
$
and where $\nabla$ denotes the rest of the complex $(Q', \d_j)$ containing the modules complement to $Q'_{G_j}$.
In particular, for $j=0$, $\nabla$ is the part of $(Q', \d_0)$ that contains the module $Q'_{W_0}$;
if $j\geq 1$, then this case is only possible when $j=s-1$ and $\nabla$ is the part of $(Q', \d_0)$ that contains the module $Q'_{V_{s-1}}$.
Nevertheless, the construction of $\mu_j$ below depends only on the form above, so the construction will work for all $j$.
Denote 
\begin{equation}\label{eqn: matrix M}
M :=
\begin{bmatrix}
1 & -i \\
1 &  i
\end{bmatrix}
\end{equation}
and $I$ as the $2\times 2$ identity matrix.
We define $\mu_j|_{Q'_{G_j}}$ to be the following dashed map, with $\mu_j$ acting as the identity map on the rest of the modules in $\nabla$:
\begin{center}
\begin{tikzpicture}[scale = 0.8][>=stealth, baseline]
\matrix (M) [matrix of math nodes, column sep=7mm]
{
P_n^A  & P_n^A  & \cdots & P_n^A  & P_n^A  & P_{n-1}^A & \nabla\\
\oplus & \oplus &        & \oplus & \oplus & \oplus    &       & = (Q', \d_j)\\
P_n^A  & P_n^A  & \cdots & P_n^A  & P_n^A  & P_{n+1}^A &       \\
{}     &        &        &        &        &           &       \\
{}     &        &        &        &        &           &       \\
{}     &        &        &        &        &           &       \\
P_n^A  & P_n^A  & \cdots & P_n^A  & P_n^A  & P_{n-1}^A & \nabla\\
\oplus & \oplus &        & \oplus & \oplus & \oplus    &       & =: (Q', \d_{j+1}).\\
P_n^A  & P_n^A  & \cdots & P_n^A  & P_n^A  & P_{n+1}^A &       \\
};

\draw[-> ,font=\small](M-1-1.east |- M-1-2) -- (M-1-2) 
	node[midway,above] {$2X_n$};
\draw[-> ,font=\small](M-1-2.east |- M-1-3) -- (M-1-3)
	;
\draw[-> ,font=\small](M-1-3.east |- M-1-4) -- (M-1-4)
	;
\draw[-> ,font=\small](M-1-4.east |- M-1-5) -- (M-1-5)
	node[midway,above] {$2X_n$}; 
\draw[-> ,font=\small](M-1-5) -- (M-1-6)
	node[midway,above] {$F$}; 
\draw[-> ,font=\small](M-1-5) -- (M-3-6)
	;
\draw[<->](M-1-6.east |- M-1-7) -- (M-1-7)
	;
\draw[<->](M-3-6) -- (M-1-7)
	;

\draw[-> ,font=\small](M-3-1.east |- M-3-2) -- (M-3-2) 
	node[midway,above] {$2X_n$};
\draw[-> ,font=\small](M-3-2.east |- M-3-3) -- (M-3-3)
	;
\draw[-> ,font=\small](M-3-3.east |- M-3-4) -- (M-3-4)
	;
\draw[-> ,font=\small](M-3-4.east |- M-3-5) -- (M-3-5)
	node[midway,above] {$2X_n$};
\draw[-> ,font=\small](M-3-5) -- (M-1-6)
	; 
\draw[-> ,font=\small](M-3-5) -- (M-3-6)
	;

\draw[-> ,font=\small](M-7-1.east |- M-7-2) -- (M-7-2) 
	node[midway,above] {$X_n$};
\draw[-> ,font=\small](M-7-2.east |- M-7-3) -- (M-7-3)
	;
\draw[-> ,font=\small](M-7-3.east |- M-7-4) -- (M-7-4)
	;
\draw[-> ,font=\small](M-7-4.east |- M-7-5) -- (M-7-5)
	node[midway,above] {$X_n$}; 
\draw[-> ,font=\small](M-7-5) -- (M-7-6)
	node[midway,above, scale=0.7] {$(n|n-1)$};
\draw[<->](M-7-6.east |- M-7-7) -- (M-7-7)
	;
\draw[<->](M-9-6) -- (M-7-7)
	;

\draw[-> ,font=\small](M-9-1.east |- M-9-2) -- (M-9-2) 
	node[midway,above] {$X_n$};
\draw[-> ,font=\small](M-9-2.east |- M-9-3) -- (M-9-3)
	;
\draw[-> ,font=\small](M-9-3.east |- M-9-4) -- (M-9-4)
	;
\draw[-> ,font=\small](M-9-4.east |- M-9-5) -- (M-9-5)
	node[midway,above] {$X_n$};
\draw[-> ,font=\small](M-9-5) -- (M-9-6)
	node[midway,above, scale=0.7] {$(n|n+1)$};

\draw[-> ,font=\small, dashed](M-3-1) -- (M-7-1) 
	node[midway,left,scale=0.8] {$2^{k-1}M$};
\draw[-> ,font=\small, dashed](M-3-2) -- (M-7-2) 
	node[midway,right,scale=0.8] {$2^{k-2}M$};
\draw[-> ,font=\small, dashed](M-3-4) -- (M-7-4) 
	node[midway,left,scale=0.8] {$2M$};
\draw[-> ,font=\small, dashed](M-3-5) -- (M-7-5) 
	node[midway,left,scale=0.8] {$M$};
\draw[-> ,font=\small, dashed](M-3-6) -- (M-7-6) 
	node[midway,left,scale=0.8] {$I$};
\end{tikzpicture}
\end{center}
The arrows in the last two rows shows the differential component $(\d_{j+1})_{G_j}$ in $\d_{j+1}$, induced by the conjugation of $\mu_j^{-1}$.
Hence, the required condition \eqref{induction} follows directly.

Now consider when $k<0$.
As before, we have that $(\d_j)_{G_j} = (\d_0)_{G_j}$ for all $j$, so the analysis of the part of $(Q', \d_j)$ that contains $(Q'_{G_j}, (\d_j)_{G_j})$ will be the same. 
Similarly the construction of $\mu_i$ below will work for both cases.
We draw the part of $(Q', \d_j)$ that contains $(Q'_{G_j}, (\d_j)_{G_j}) = (Q'_{G_j}, (\d_0)_{G_j})$ as follows:
\begin{center}
\begin{tikzpicture} [scale= 0.8][>=stealth, baseline]
\matrix (M) [matrix of math nodes, column sep=7mm]
{
       & P_{n+1}^A & P_n^A  & P_n^A  & \cdots & P_n^A  & P_n^A  \\
       & \oplus    & \oplus & \oplus &        & \oplus & \oplus & =(Q', \d_j)\\
\nabla & P_{n-1}^A & P_n^A  & P_n^A  & \cdots & P_n^A  & P_n^A  \\
};

\draw[-> ,font=\small](M-1-2.east |- M-1-3) -- (M-1-3) 
	node[midway,above] {$E$};
\draw[-> ,font=\small](M-1-2) -- (M-3-3) 
	;
\draw[-> ,font=\small](M-1-3.east |- M-1-4) -- (M-1-4) 
	node[midway,above] {$2X_n$};
\draw[-> ,font=\small](M-1-4.east |- M-1-5) -- (M-1-5)
	;
\draw[-> ,font=\small](M-1-5.east |- M-1-6) -- (M-1-6)
	;
\draw[-> ,font=\small](M-1-6.east |- M-1-7) -- (M-1-7)
	node[midway,above] {$2X_n$};
\draw[<->](M-3-1) -- (M-1-2)
	;

\draw[-> ,font=\small](M-3-2) -- (M-1-3)
	;
\draw[-> ,font=\small](M-3-2) -- (M-3-3)
	;
\draw[-> ,font=\small](M-3-3.east |- M-3-4) -- (M-3-4) 
	node[midway,above] {$2X_n$};
\draw[-> ,font=\small](M-3-4.east |- M-3-5) -- (M-3-5)
	;
\draw[-> ,font=\small](M-3-5.east |- M-3-6) -- (M-3-6)
	;
\draw[-> ,font=\small](M-3-6.east |- M-3-7) -- (M-3-7)
	node[midway,above] {$2X_n$};
\draw[<->](M-3-1) -- (M-3-2)
	;

\end{tikzpicture}
\end{center}
with $
E = 
\begin{bsmallmatrix}
-(n+1|n)i &  (n-1|n)i \\
(n+1|n)   &  (n-1|n)
\end{bsmallmatrix}
$ 
and where $\nabla$ denotes the rest of the complex $(Q', \d_j)$ containing the modules complement to $Q'_{G_j}$.
Denote 
\begin{equation} \label{eqn: matrix N}
N :=
\begin{bmatrix}
i  & 1 \\
-i & 1
\end{bmatrix}
\end{equation} 
and $I$ as the $2\times 2$ identity matrix.
Note that 
$
N
\begin{bsmallmatrix}
-(n+1|n)i &  (n-1|n)i \\
(n+1|n)   &  (n-1|n)
\end{bsmallmatrix}
=
2
\begin{bsmallmatrix}
(n+1|n) &   0      \\
0       &  (n-1|n)
\end{bsmallmatrix}.
$
We define $\mu_j|_{Q'_{G_j}}$ to be the following dashed map, with $\mu_j$ the identity map on all modules in $\nabla$:
\begin{center}
\begin{tikzpicture}[scale=0.8][>=stealth, baseline]
\matrix (M) [matrix of math nodes, column sep=7mm]
{
       & P_{n+1}^A & P_n^A  & P_n^A  & \cdots & P_n^A  & P_n^A  \\
       & \oplus    & \oplus & \oplus &        & \oplus & \oplus & =(Q', \d_j)\\
\nabla & P_{n-1}^A & P_n^A  & P_n^A  & \cdots & P_n^A  & P_n^A  \\
{}     &           &        &        &        &        &        \\
{}     &           &        &        &        &        &        \\
{}     &           &        &        &        &        &        \\
       & P_{n+1}^A & P_n^A  & P_n^A  & \cdots & P_n^A  & P_n^A  \\
       & \oplus    & \oplus & \oplus &        & \oplus & \oplus & =:(Q', \d_{j+1}).\\
\nabla & P_{n-1}^A & P_n^A  & P_n^A  & \cdots & P_n^A  & P_n^A \\
};

\draw[-> ,font=\small](M-1-2.east |- M-1-3) -- (M-1-3) 
	node[midway,above] {$E$};
\draw[-> ,font=\small](M-1-2) -- (M-3-3) 
	;
\draw[-> ,font=\small](M-1-3.east |- M-1-4) -- (M-1-4) 
	node[midway,above] {$2X_n$};
\draw[-> ,font=\small](M-1-4.east |- M-1-5) -- (M-1-5)
	;
\draw[-> ,font=\small](M-1-5.east |- M-1-6) -- (M-1-6)
	;
\draw[-> ,font=\small](M-1-6.east |- M-1-7) -- (M-1-7)
	node[midway,above] {$2X_n$};

\draw[-> ,font=\small](M-3-2) -- (M-1-3)
	;
\draw[-> ,font=\small](M-3-2) -- (M-3-3)
	;
\draw[-> ,font=\small](M-3-3.east |- M-3-4) -- (M-3-4) 
	node[midway,above] {$2X_n$};
\draw[-> ,font=\small](M-3-4.east |- M-3-5) -- (M-3-5)
	;
\draw[-> ,font=\small](M-3-5.east |- M-3-6) -- (M-3-6)
	;
\draw[-> ,font=\small](M-3-6.east |- M-3-7) -- (M-3-7)
	node[midway,above] {$2X_n$};
\draw[<->](M-3-1.east |- M-3-2) -- (M-3-2);
\draw[<->](M-3-1) -- (M-1-2); 

\draw[-> ,font=\small](M-7-2.east |- M-7-3) -- (M-7-3) 
	node[midway,above,scale=0.7] {$(n+1|n)$};
\draw[-> ,font=\small](M-7-3.east |- M-7-4) -- (M-7-4) 
	node[midway,above] {$X_n$};
\draw[-> ,font=\small](M-7-4.east |- M-7-5) -- (M-7-5)
	;
\draw[-> ,font=\small](M-7-5.east |- M-7-6) -- (M-7-6)
	;
\draw[-> ,font=\small](M-7-6.east |- M-7-7) -- (M-7-7)
	node[midway,above] {$X_n$};

\draw[-> ,font=\small](M-9-2) -- (M-9-3)
	node[midway,above,scale=0.7] {$(n-1|n)$};
\draw[-> ,font=\small](M-9-3.east |- M-9-4) -- (M-9-4) 
	node[midway,above] {$X_n$};
\draw[-> ,font=\small](M-9-4.east |- M-9-5) -- (M-9-5)
	;
\draw[-> ,font=\small](M-9-5.east |- M-9-6) -- (M-9-6)
	;
\draw[-> ,font=\small](M-9-6.east |- M-9-7) -- (M-9-7)
	node[midway,above] {$X_n$};
\draw[<->](M-9-1.east |- M-9-2) -- (M-9-2);
\draw[<->](M-9-1) -- (M-7-2); 

\draw[-> ,font=\small, dashed](M-3-2) -- (M-7-2) 
	node[midway,left,scale=0.8] {$I$};
\draw[-> ,font=\small, dashed](M-3-3) -- (M-7-3) 
	node[midway,right,scale=0.8] {$2^{-1}N$};
\draw[-> ,font=\small, dashed](M-3-4) -- (M-7-4) 
	node[midway,right,scale=0.8] {$2^{-2}N$};
\draw[-> ,font=\small, dashed](M-3-6) -- (M-7-6) 
	node[midway,left,scale=0.8] {$2^{k+1}N$};
\draw[-> ,font=\small, dashed](M-3-7) -- (M-7-7) 
	node[midway,left,scale=0.8] {$2^{k}N$};

\end{tikzpicture}
\end{center}
The arrows in the last two rows shows the differential component $(\d_{j+1})_{G_j}$ in $\d_{j+1}$, induced by the conjugation of $\mu_j^{-1}$.
It follows directly that the required condition \eqref{induction} is satisfied.

\item $\check{g}_j$ is of Type II$'_k$ for $k \in \Z$: \\
As before, we have that $(\d_j)_{G_j} = (\d_0)_{G_j}$ for all $j$ by our inductive assumption, so the part of $(Q', \d_j)$ that contains $(Q'_{G_j}, (\d_j)_{G_j})$ will be of the same form. 
As such, the construction of $\mu_j$ below will work for all $j$.

We shall start with $k=0$.
We draw the part of $(Q', \d_j)$ that contains $(Q'_{G_j}, (\d_j)_{G_j}) = (Q'_{G_j}, (\d_0)_{G_j})$ as follows:
\begin{center}
\begin{tikzpicture}[>=stealth, baseline]
\matrix (M) [matrix of math nodes, column sep=7mm]
{
\nabla & P_{n+1}^A & P_{n+1}^A & \nabla' \\
       & \oplus    & \oplus    &       & = (Q', \d_j), \\
       & P_{n-1}^A & P_{n-1}^A &        \\
};

\draw[<->](M-1-1) -- (M-1-2); 
\draw[-> ,font=\small](M-1-2.east |- M-1-3) -- (M-1-3)
	node[midway,above] {$X_{n+1}$};
\draw[<->](M-1-3) -- (M-1-4); 

\draw[<->](M-1-1) -- (M-3-2); 
\draw[-> ,font=\small](M-1-2.east |- M-3-3) -- (M-3-3)
	node[midway,above] {$X_{n-1}$};
\draw[<->](M-3-3) -- (M-1-4); 
\end{tikzpicture}
\end{center}
where either $\nabla$ or $\nabla'$ is the part of $(Q', \d_j)$ that contains the module $Q'_{V_j}$, whereas the other contains $Q'_{W_j}$.
However, in this case we already have that $(\d_j)_{G_j} = \d'_{G_j}$, thus we just choose $\mu_j$ to be the identity map.

For $k > 0$, the part of $(Q', \d_j)$ containing $(Q'_{G_j}, (\d_j)_{G_j}) = (Q'_{G_j}, (\d_0)_{G_j})$ is as follows:
\begin{center}
\begin{tikzpicture}[>=stealth, baseline]
\matrix (M) [matrix of math nodes, column sep=7mm]
{
P_n^A  & P_n^A  & P_n^A  & \cdots & P_n^A  & P_n^A  & P_{n-1}^A & \nabla \\
\oplus & \oplus & \oplus &        & \oplus & \oplus & \oplus    & \\
P_n^A  & P_n^A  & P_n^A  & \cdots & P_n^A  & P_n^A  & P_{n+1}^A &        \\
       & \oplus & \oplus &        & \oplus & \oplus &       &  &=(Q', \d_j),\\ 
       & P_n^A  & P_n^A  & \cdots & P_n^A  & P_{n-1}^A & \nabla' \\
       & \oplus & \oplus &        & \oplus & \oplus    &        \\
       & P_n^A  & P_n^A  & \cdots & P_n^A  & P_{n+1}^A &        \\
};

\draw[-> ,font=\small](M-1-1) -- (M-1-2) 
	node[midway,above] {$2X_n$};
\draw[-> ,font=\small](M-1-1) -- (M-5-2) 
	node[midway,below] {$2X_n$};
\draw[-> ,font=\small](M-1-2) -- (M-1-3) 
	node[midway,above] {$2X_n$};
\draw[-> ,font=\small](M-1-3) -- (M-1-4) 
	;
\draw[-> ,font=\small](M-1-4) -- (M-1-5)
	;
\draw[-> ,font=\small](M-1-5) -- (M-1-6)
	node[midway,above] {$2X_n$};
\draw[-> ,font=\small](M-1-6) -- (M-1-7)
	node[midway, above] {$F$};
\draw[-> ,font=\small](M-1-6) -- (M-3-7)
	;
\draw[<-> ,font=\small](M-1-7) -- (M-1-8)
	;
\draw[<-> ,font=\small](M-3-7) -- (M-1-8)
	;

\draw[-> ,font=\small](M-3-1) -- (M-3-2)
	node[midway,above] {$2X_n$};
\draw[-> ,font=\small](M-3-2) -- (M-3-3)
	node[midway,above] {$2X_n$};
\draw[-> ,font=\small](M-3-1) -- (M-7-2)
	node[midway,above] {$2X_n$};
\draw[-> ,font=\small](M-3-3) -- (M-3-4) 
	;
\draw[-> ,font=\small](M-3-4) -- (M-3-5)
	;
\draw[-> ,font=\small](M-3-5) -- (M-3-6)
	node[midway,above] {$2X_n$};
\draw[-> ,font=\small](M-3-6) -- (M-1-7)
	;
\draw[-> ,font=\small](M-3-6) -- (M-3-7)
	;

\draw[-> ,font=\small](M-5-2) -- (M-5-3)
	node[midway,above] {$2X_n$};
\draw[-> ,font=\small](M-5-3) -- (M-5-4) 
	;
\draw[-> ,font=\small](M-5-4) -- (M-5-5)
	;
\draw[-> ,font=\small](M-5-5) -- (M-5-6)
	node[midway, above] {$F$};
\draw[-> ,font=\small](M-5-5) -- (M-7-6)
	;
\draw[<-> ,font=\small](M-5-6) -- (M-5-7)
	;
\draw[<-> ,font=\small](M-7-6) -- (M-5-7)
	;

\draw[-> ,font=\small](M-7-2) -- (M-7-3)
	node[midway,above] {$2X_n$};
\draw[-> ,font=\small](M-7-3) -- (M-7-4) 
	;
\draw[-> ,font=\small](M-7-4) -- (M-7-5)
	;
\draw[-> ,font=\small](M-7-5) -- (M-5-6)
	;
\draw[-> ,font=\small](M-7-5) -- (M-7-6)
	;
\end{tikzpicture}
\end{center}
where either $\nabla$ or $\nabla'$ is the part of $(Q', \d_j)$ that contains the module $Q'_{V_j}$ and the other contains $Q'_{W_j}$.
We define $\mu_j|_{Q'_{G_j}}$ to be the following dashed map:
\begin{center}
\begin{tikzpicture}[>=stealth, baseline]
\matrix (M) [matrix of math nodes, column sep=7mm]
{
P_n^A  & P_n^A  & P_n^A  & \cdots & P_n^A  & P_n^A  & P_{n-1}^A & \nabla \\
\oplus & \oplus & \oplus &        & \oplus & \oplus & \oplus    \\
P_n^A  & P_n^A  & P_n^A  & \cdots & P_n^A  & P_n^A  & P_{n+1}^A &        \\
       & \oplus & \oplus &        & \oplus & \oplus &       & & =(Q', \d_j)\\ 
       & P_n^A  & P_n^A  & \cdots & P_n^A  & P_{n-1}^A & \nabla'\\
       & \oplus & \oplus &        & \oplus & \oplus    &       \\
       & P_n^A  & P_n^A  & \cdots & P_n^A  & P_{n+1}^A &       \\
{}     &        &        &        &        &        &          \\
{}     &        &        &        &        &        &          \\ {}     &        &        &        &        &        &          \\
P_n^A  & P_n^A  & P_n^A  & \cdots & P_n^A  & P_n^A  & P_{n-1}^A & \nabla\\
\oplus & \oplus & \oplus &        & \oplus & \oplus & \oplus    \\
P_n^A  & P_n^A  & P_n^A  & \cdots & P_n^A  & P_n^A  & P_{n+1}^A &       \\
       & \oplus & \oplus &        & \oplus & \oplus &       & & =: (Q', \d_{j+1}),\\ 
       & P_n^A  & P_n^A  & \cdots & P_n^A  & P_{n-1}^A & \nabla'\\
       & \oplus & \oplus &        & \oplus & \oplus    &       \\
       & P_n^A  & P_n^A  & \cdots & P_n^A  & P_{n+1}^A &       \\
};

\draw[-> ,font=\small](M-1-1) -- (M-1-2) 
	node[midway,above] {$2X_n$};
\draw[-> ,font=\small](M-1-1) -- (M-5-2) 
	node[midway,below] {$2X_n$};
\draw[-> ,font=\small](M-1-2) -- (M-1-3) 
	node[midway,above] {$2X_n$};
\draw[-> ,font=\small](M-1-3) -- (M-1-4) 
	;
\draw[-> ,font=\small](M-1-4) -- (M-1-5)
	;
\draw[-> ,font=\small](M-1-5) -- (M-1-6)
	node[midway,above] {$2X_n$};
\draw[-> ,font=\small](M-1-6) -- (M-1-7)
	node[midway, above] {$F$};
\draw[-> ,font=\small](M-1-6) -- (M-3-7)
	;
\draw[<-> ,font=\small](M-1-7) -- (M-1-8)
	;
\draw[<-> ,font=\small](M-3-7) -- (M-1-8)
	;

\draw[-> ,font=\small](M-3-1) -- (M-3-2)
	node[midway,above] {$2X_n$};
\draw[-> ,font=\small](M-3-2) -- (M-3-3)
	node[midway,above] {$2X_n$};
\draw[-> ,font=\small](M-3-1) -- (M-7-2)
	node[midway,above] {$2X_n$};
\draw[-> ,font=\small](M-3-3) -- (M-3-4) 
	;
\draw[-> ,font=\small](M-3-4) -- (M-3-5)
	;
\draw[-> ,font=\small](M-3-5) -- (M-3-6)
	node[midway,above] {$2X_n$};
\draw[-> ,font=\small](M-3-6) -- (M-1-7)
	;
\draw[-> ,font=\small](M-3-6) -- (M-3-7)
	;

\draw[-> ,font=\small](M-5-2) -- (M-5-3)
	node[midway,above] {$2X_n$};
\draw[-> ,font=\small](M-5-3) -- (M-5-4) 
	;
\draw[-> ,font=\small](M-5-4) -- (M-5-5)
	;
\draw[-> ,font=\small](M-5-5) -- (M-5-6)
	node[midway, above] {$F$};
\draw[-> ,font=\small](M-5-5) -- (M-7-6)
	;
\draw[<-> ,font=\small](M-5-6) -- (M-5-7)
	;
\draw[<-> ,font=\small](M-7-6) -- (M-5-7)
	;

\draw[-> ,font=\small](M-7-2) -- (M-7-3)
	node[midway,above] {$2X_n$};
\draw[-> ,font=\small](M-7-3) -- (M-7-4) 
	;
\draw[-> ,font=\small](M-7-4) -- (M-7-5)
	;
\draw[-> ,font=\small](M-7-5) -- (M-5-6)
	;
\draw[-> ,font=\small](M-7-5) -- (M-7-6)
	;

\draw[-> ,font=\small](M-11-1) -- (M-11-2) 
	node[midway,above] {$X_n$};
\draw[-> ,font=\small](M-11-1) -- (M-15-2) 
	node[midway,below] {$X_n$};
\draw[-> ,font=\small](M-11-2) -- (M-11-3) 
	node[midway,above] {$X_n$};
\draw[-> ,font=\small](M-11-3) -- (M-11-4) 
	;
\draw[-> ,font=\small](M-11-4) -- (M-11-5)
	;
\draw[-> ,font=\small](M-11-5) -- (M-11-6)
	node[midway,above] {$X_n$};
\draw[-> ,font=\small](M-11-6) -- (M-11-7)
	node[midway,above,scale=0.7] {$(n|n-1)$};
\draw[<-> ,font=\small](M-11-7) -- (M-11-8)
	;
\draw[<-> ,font=\small](M-13-7) -- (M-11-8)
	;

\draw[-> ,font=\small](M-13-1) -- (M-13-2)
	node[midway,above] {$X_n$};
\draw[-> ,font=\small](M-13-2) -- (M-13-3)
	node[midway,above] {$X_n$};
\draw[-> ,font=\small](M-13-1) -- (M-17-2)
	node[midway,left] {$X_n$};
\draw[-> ,font=\small](M-13-3) -- (M-13-4) 
	;
\draw[-> ,font=\small](M-13-4) -- (M-13-5)
	;
\draw[-> ,font=\small](M-13-5) -- (M-13-6)
	node[midway,above] {$X_n$};
\draw[-> ,font=\small](M-13-6) -- (M-13-7)
	node[midway,above,scale=0.7] {$(n|n+1)$};

\draw[-> ,font=\small](M-15-2) -- (M-15-3)
	node[midway,above] {$X_n$};
\draw[-> ,font=\small](M-15-3) -- (M-15-4) 
	;
\draw[-> ,font=\small](M-15-4) -- (M-15-5)
	;
\draw[-> ,font=\small](M-15-5) -- (M-15-6)
	node[midway,above,scale=0.7] {$(n|n-1)$};
\draw[<-> ,font=\small](M-15-6) -- (M-15-7)
	;
\draw[<-> ,font=\small](M-17-6) -- (M-15-7)
	;

\draw[-> ,font=\small](M-17-2) -- (M-17-3)
	node[midway,above] {$X_n$};
\draw[-> ,font=\small](M-17-3) -- (M-17-4) 
	;
\draw[-> ,font=\small](M-17-4) -- (M-17-5)
	;
\draw[-> ,font=\small](M-17-5) -- (M-17-6)
	node[midway,above,scale=0.7] {$(n|n+1)$};

\draw[-> ,font=\small, dashed](M-3-1) -- (M-11-1)
	node[midway,left] {$2M$};
\draw[-> ,font=\small, dashed](M-7-2) -- (M-11-2)
	node[midway,left,scale=0.6] {$
		\begin{bmatrix}
		M & 0 \\
		0 & M \\
		\end{bmatrix}
		$};
\draw[-> ,font=\small, dashed](M-7-3) -- (M-11-3)
	node[midway, left,scale=0.6] {$ 2^{-1}
		\begin{bmatrix}
		M & 0 \\
		0 & M \\
		\end{bmatrix}
		$};
\draw[-> ,font=\small, dashed](M-7-5) -- (M-11-5)
	node[midway, left,scale=0.6] {$2^{-(k-2)}
		\begin{bmatrix}
		M & 0 \\
		0 & M \\
		\end{bmatrix}
		$};
\draw[-> ,font=\small, dashed](M-7-6) -- (M-11-6)
	node[midway, right,scale=0.6] {$
		\begin{bmatrix}
		2^{-(k-1)}M & 0 \\
		0 & 2^{-(k-2)}I \\
		\end{bmatrix}
		$};
\draw[-> ,font=\small, dashed](M-3-7) to [out=-30,in=30]node[midway,right] {$2^{-(k-1)}I$}
	(M-11-7);

\end{tikzpicture}
\end{center}
where $M$ is as in \eqref{eqn: matrix M}.
For the rest of the modules in $Q'$, $\mu_j$ sends $v$ to $2^{-(k-1)}v$ (resp. $v$ to $2^{-(k-2)}v$) for any $v$ belonging to the modules in $\nabla$ (resp. $\nabla'$).
The black arrows in the last four rows show the differential component $(\d_{j+1})_{G_j}$ in $\d_{j+1}$, induced by the conjugation of $\mu_j^{-1}$.
It is easy to see that the required condition \eqref{induction} is satisfied.

The construction for $k < 0$ is similar, using the map $N$ (from \eqref{eqn: matrix N}) in place of $M$.

\item $\check{g}_j$ is of Type V'': \\
Recall the definitions of $g_j, c_j$ and $r_j$ from the paragraph  \nameref{slicing} and recall the subsets of crossings of $\mathfrak{m}(\check{c})$ as defined in \eqref{subset crossing}.
Let $h_j$ be the connected component of $(c\setminus g_j) \cup (g_j \cap d_2)$ that contains the point $m$, so that $g_j$ intersects $h_j$ at the point $r_j$.
To illustrate, $h_0 = c_0$ and $h_1 = c_0 \cup g_0 \cup c_1$ in \cref{notend01}; whereas $h_0 = \emptyset,$ $h_1 = g_0 \cup c_1,$ and $h_2 = g_0 \cup c_1 \cup g_1 \cup c_2$ in \cref{end01}. 
Let $\mathfrak{m}(\check{h}_j) =  \wt{h}_j \coprod \undertilde{h}_j$ and $\mathfrak{m}(r_j) = \wt{r}_j \coprod \undertilde{r}_j$ so that the curves of $\mathfrak{m}(\check{g}_j)$ and $ \wt{h_j}$ intersect at the point $\wt{r}_j \in \theta_{n+1}$ and the curves of $\mathfrak{m}(\check{g}_j)$ and $\undertilde{h_j}$ intersect at the point $\undertilde{r_j}\in \theta_{n-1}$.

Now recall the decomposition of $Q'$ given in \eqref{new R decomp}.
By definition, $\{ \wt{r}_j, \undertilde{r}_j \}$ is the subset of crossings $R_j \subseteq G_j$.
We get that
\[
P^A(\wt{r_j}) \oplus P^A(\undertilde{r_j}) = Q'_{R_j}.
\]
First consider the case when $j=0$.
Then $g_j=g_0$ is of this type only when $m \in \Lambda\setminus \{0, 1\}$ since $g_0 \cap \{1 \} = \emptyset$, so we have $Q'_{V_0 \oplus R_0} = Q'_{C_0}$.
Furthermore, \eqref{C diff component} implies that $(\d_0)_{C_0} = \d'_{C_0}$, giving us
\[
(Q'_{V_0}\oplus Q'_{R_0}, (\d_0)_{V_0 \oplus R_0}) 
= (Q'_{C_0}, (\d_0)_{C_0}) 
= (Q'_{C_0}, \d'_{C_0}) 
= L_A(\mathfrak{m}(\check{c}_0))
\]
with the last equality following from the definition of $(Q', \d') = L_A(\mathfrak{m}(\check{c}))$.
By definition of $h_j$, it follows that $c_0 = h_0$, so we can conclude that
\[
(Q'_{V_0}\oplus Q'_{R_0}, (\d_j)_{V_0 \oplus R_0}) = L_A(\mathfrak{m}(\check{h}_0))
= L_A(\wt{h_0}) \oplus L_A(\undertilde{h_0}),
\]
where the last equality follows from the fact that $\mathfrak{m}(\check{h}_0) =  \wt{h_0} \coprod \undertilde{h_0}$.

Now consider when $j \geq 1$.
In this case $\d_j$ satisfies property \eqref{induction} by our induction hypothesis, allowing us to conclude that
\begin{align*}
(\d_j)_{V_j} = \d'_{V_j}, \quad
(\d_j)_{V_j R_j} = \d'_{V_j R_j}, \text{ and }
(\d_j)_{R_j V_j} = \d'_{R_j V_j}.
\end{align*}
Thus we have that $(\d_j)_{V_j \oplus R_j} = \d'_{V_j \oplus R_j}$, giving us
$$
(Q'_{V_j}\oplus Q'_{R_j}, (\d_j)_{V_j \oplus R_j}) 
= (Q'_{V_j}\oplus Q'_{R_j}, \d'_{V_j \oplus R_j}) 
= L_A(\mathfrak{m}(\check{h}_j)) 
= L_A(\wt{h_j}) \oplus L_A(\undertilde{h_j}),
$$
where the second equality follows from the definition of $(Q', \d') = L_A(\mathfrak{m}(\check{c}))$ and the third equality follows from the fact that $\mathfrak{m}(\check{h}_j) =  \wt{h_j} \coprod \undertilde{h_j}$.

Thus for all $j$, we obtain
\begin{equation} \label{VR splits}
(Q'_{V_j}\oplus Q'_{R_j}, (\d_j)_{V_j \oplus R_j}) = L_A(\wt{h_j}) \oplus L_A(\undertilde{h_j}).
\end{equation}
Furthermore, note that $\wt{h_j}$ and $\undertilde{h_j}$ contain the points $\wt{r}_j$ and $\undertilde{r_j}$ respectively, so $L_A(\wt{h}_j)$ contains $P^A(\wt{r_j})$ as a submodule and $L_A(\undertilde{h_j})$ contains $P^A(\undertilde{r_j})$ as a submodule.
Let us now understand the relation between $(Q'_{G_j}, (\d_j)_{G_j})$, $(Q'_{V_j}\oplus Q'_{R_j}, (\d_j)_{V_j \oplus R_j})$ and $(Q', \d_j)$.
As before our inductive hypothesis gives us $(\d_j)_{G_j} = (\d_0)_{G_j}$ for all $j$.
So the part of $(Q', \d_j)$ that contains $(Q'_{G_j}, (\d_j)_{G_j})$ will be the same for all $j$, and it has either of the following two forms:
\begin{center}
\begin{tikzpicture}[>=stealth, baseline]
\matrix (M) [matrix of math nodes, column sep=7mm]
{
       & P^A(\wt{r_j}) & P_n^A  & P_{n-1}^A & \nabla' \\
       & \oplus    & \oplus & \oplus    &        & =(Q', \d_j)\\
\nabla & P^A(\undertilde{r_j}) & P_n^A  & P_{n+1}^A &        \\
};

\draw[-> ,font=\small](M-1-2) -- (M-1-3) 
	node[midway,above] {$E$};
\draw[-> ,font=\small](M-1-2) -- (M-3-3) 
	;
\draw[-> ,font=\small](M-1-3) -- (M-1-4) 
	node[midway,above] {$F$};
\draw[-> ,font=\small](M-1-3) -- (M-3-4) 
	;
\draw[<->](M-1-4) -- (M-1-5)
	;
\draw[<-> ,font=\small](M-3-4) -- (M-1-5)
	;

\draw[<->](M-3-1) -- (M-1-2)
	;
\draw[<->](M-3-1) -- (M-3-2)
	;
\draw[-> ,font=\small](M-3-2) -- (M-3-3) 
	;
\draw[-> ,font=\small](M-3-2) -- (M-1-3) 
	;
\draw[-> ,font=\small](M-3-3) -- (M-3-4)
	;
\draw[-> ,font=\small](M-3-3) -- (M-1-4)
	;

\draw[BurntOrange] (M-3-1.south west)++(0,-5pt) rectangle (M-1-2.north east)
	node[midway, xshift = -3.5cm] 
  	{$(Q'_{V_j}\oplus Q'_{R_j}, (\d_j)_{V_j \oplus R_j})$};
  	
\draw[OliveGreen] (M-1-2.north west)++(0,5pt) rectangle 
	(M-3-4.south east)
		node[midway, xshift=3.5cm , yshift = -22pt] 
  		{$(Q'_{G_j}, (\d_j)_{G_j})$};
\end{tikzpicture}

\vspace*{1mm}
{\LARGE or}
\vspace*{1mm}

\begin{tikzpicture}[>=stealth, baseline]
\matrix (M) [matrix of math nodes, column sep=7mm]
{
       & P^A_{n+1} & P_n^A  & P^A(\undertilde{r_j}) & \nabla \\
       & \oplus    & \oplus & \oplus    &   &  & & &   & =(Q', \d_j),\\
\nabla' & P^A_{n-1} & P_n^A  & P^A(\wt{r_j}) &        \\
};

\draw[-> ,font=\small](M-1-2) -- (M-1-3) 
	node[midway,above] {$E$};
\draw[-> ,font=\small](M-1-2) -- (M-3-3) 
	;
\draw[-> ,font=\small](M-1-3) -- (M-1-4) 
	node[midway,above] {$F$};
\draw[-> ,font=\small](M-1-3) -- (M-3-4) 
	;
\draw[<->](M-1-4) -- (M-1-5)
	;
\draw[<-> ,font=\small](M-3-4) -- (M-1-5)
	;

\draw[<->](M-3-1) -- (M-1-2)
	;
\draw[<->](M-3-1) -- (M-3-2)
	;
\draw[-> ,font=\small](M-3-2) -- (M-3-3) 
	;
\draw[-> ,font=\small](M-3-2) -- (M-1-3) 
	;
\draw[-> ,font=\small](M-3-3) -- (M-3-4)
	;
\draw[-> ,font=\small](M-3-3) -- (M-1-4)
	;

\draw[BurntOrange] (M-3-4.south west) rectangle (M-1-5.north east)
	node[midway, xshift = 3.2cm, yshift = 26pt] 
  	{$(Q'_{V_j}\oplus Q'_{R_j}, (\d_j)_{V_j \oplus R_j})$};
  	
\draw[OliveGreen] (M-1-2.north west)++(0,5pt) rectangle 
	(M-3-4.south east)
		node[midway, xshift=-3.5cm, yshift = 22pt] 
  		{$(Q'_{G_j}, (\d_j)_{G_j})$};
\end{tikzpicture}
\end{center}
where $\nabla$ denotes the rest of the complex $(Q', \d_j)$ that contains the module $Q'_{V_j}$ and $\nabla'$ denotes the rest of the complex $(Q', \d_j)$ that contains the module $Q'_{W_j}$.
%
%
%
%
%

Using \eqref{VR splits}, the part of $(Q', \d_j)$ that contains $(Q'_{G_j}, (\d_j)_{G_j})$, $L_A(\wt{h_j})$ and $L_A(\undertilde{h_j})$ has either of the following two forms:
\begin{center}
\begin{tikzpicture}[>=stealth, baseline]
\matrix (M) [matrix of math nodes, column sep=7mm]
{
\wt{\nabla} & P^A(\wt{r_j}) & P_n^A  & P_{n-1}^A & \nabla' \\
       & \oplus    & \oplus & \oplus    &        & =(Q', \d_j)\\
\undertilde{\nabla} & P^A(\undertilde{r_j}) & P_n^A  & P_{n+1}^A &        \\
};

\draw[<->](M-1-1) -- (M-1-2)
	;
\draw[-> ,font=\small](M-1-2) -- (M-1-3) 
	node[midway,above] {$E$};
\draw[-> ,font=\small](M-1-2) -- (M-3-3) 
	;
\draw[-> ,font=\small](M-1-3) -- (M-1-4) 
	node[midway,above] {$F$};
\draw[-> ,font=\small](M-1-3) -- (M-3-4) 
	;
\draw[<->](M-1-4) -- (M-1-5)
	;
\draw[<-> ,font=\small](M-3-4) -- (M-1-5)
	;

\draw[<->](M-3-1) -- (M-3-2)
	;
\draw[-> ,font=\small](M-3-2) -- (M-3-3) 
	;
\draw[-> ,font=\small](M-3-2) -- (M-1-3) 
	;
\draw[-> ,font=\small](M-3-3) -- (M-3-4)
	;
\draw[-> ,font=\small](M-3-3) -- (M-1-4)
	;
	
\draw[red] (M-1-1.north west) rectangle (M-1-2.south east)
	node[midway, xshift = -2cm] 
  	{$L_A\left(\wt{h_j}\right)$};
\draw[Brown] (M-3-1.north west) rectangle (M-3-2.south east)
	node[midway, xshift = -2cm] 
  	{$L_A\left(\undertilde{h_j}\right)$};
  	
\draw[OliveGreen] (M-1-2.north west)++(0,4pt) rectangle 
	(M-3-4.south east)
		node[midway, xshift= 3.5cm, yshift = -23pt] 
  		{$(Q'_{G_j}, (\d_j)_{G_j})$};
\end{tikzpicture}

\vspace*{1mm}
{\LARGE or}
\vspace*{1mm}

\begin{tikzpicture}[>=stealth, baseline]
\matrix (M) [matrix of math nodes, column sep=7mm]
{
       & P^A_{n+1} & P_n^A  & P^A(\undertilde{r_j}) & \undertilde{\nabla} \\
       & \oplus    & \oplus & \oplus    &    &    & =(Q', \d_j).\\
\nabla' & P^A_{n-1} & P_n^A  & P^A(\wt{r_j}) & \wt{\nabla} \\
};

\draw[<->](M-3-1) -- (M-1-2)
	;
\draw[-> ,font=\small](M-1-2) -- (M-1-3) 
	node[midway,above] {$E$};
\draw[-> ,font=\small](M-1-2) -- (M-3-3) 
	;
\draw[-> ,font=\small](M-1-3) -- (M-1-4) 
	node[midway,above] {$F$};
\draw[-> ,font=\small](M-1-3) -- (M-3-4) 
	;
\draw[<->](M-1-4) -- (M-1-5)
	;
\draw[<-> ,font=\small](M-3-4) -- (M-3-5)
	;

\draw[<->](M-3-1) -- (M-3-2)
	;
\draw[-> ,font=\small](M-3-2) -- (M-3-3) 
	;
\draw[-> ,font=\small](M-3-2) -- (M-1-3) 
	;
\draw[-> ,font=\small](M-3-3) -- (M-3-4)
	;
\draw[-> ,font=\small](M-3-3) -- (M-1-4)
	;

\draw[Brown] (M-1-4.south west) rectangle (M-1-5.north east)
	node[midway, xshift = 2cm] 
  	{$L_A\left(\undertilde{h_j}\right)$};
\draw[red] (M-3-4.north west)++(0,1pt) rectangle (M-3-5.south east)
	node[midway, xshift = 2cm] 
  	{$L_A\left(\wt{h_j}\right)$};
  	
\draw[OliveGreen] (M-1-2.north west)++(0,4pt) rectangle 
	(M-3-4.south east)
		node[midway, xshift = -3.3cm, yshift= 0.8cm] 
  		{$(Q'_{G_j}, (\d_j)_{G_j})$};
\end{tikzpicture}
\end{center}
Thus for all $j$, we conclude that $(Q', \d_j)$ must be of the above two possible forms.
It is now sufficient to give a construction of $\mu_j$ for each of the two possible forms.

We begin with the construction of $\mu_j$ for the first possible form of $(Q', \d_j)$.
Firstly, recall $M$ from \eqref{eqn: matrix M} and note that 
$ME = 2i
	\begin{bsmallmatrix}
	(n+1|n) & 0 \\
	0       & (n-1|n)
	\end{bsmallmatrix}
	\begin{bsmallmatrix}
	-1 & 0 \\
	0  & 1
	\end{bsmallmatrix}.
$
We define $\mu_j|_{Q'_{G_j}}$ to be the following dashed map:
\begin{center}
\begin{tikzpicture}[>=stealth, baseline]
\matrix (M) [matrix of math nodes, column sep=7mm]
{
\wt{\nabla} & P^A(\wt{r_j}) & P_n^A  & P_{n-1}^A & \nabla' \\
       & \oplus       & \oplus & \oplus    &        & =(Q', \d_j) \\
\undertilde{\nabla} & P^A(\undertilde{r_j}) & P_n^A  & P_{n+1}^A &        \\
{}     &           &        &           &        \\
{}     &           &        &           &        \\
{}     &           &        &           &        \\
\wt{\nabla} & P^A(\wt{r_j}) & P_n^A  & P_{n-1}^A & \nabla' \\
       & \oplus       & \oplus & \oplus    &        & =:(Q', \d_{j+1}).\\
\undertilde{\nabla} & P^A(\undertilde{r_j}) & P_n^A  & P_{n+1}^A &        \\
};

\draw[<->](M-1-1) -- (M-1-2)
	;
\draw[-> ,font=\small](M-1-2) -- (M-1-3) 
	node[midway,above] {$E$};
\draw[-> ,font=\small](M-1-2) -- (M-3-3) 
	;
\draw[-> ,font=\small](M-1-3) -- (M-1-4) 
	node[midway,above] {$F$};
\draw[-> ,font=\small](M-1-3) -- (M-3-4) 
	;
\draw[<->](M-1-4) -- (M-1-5)
	;
\draw[<-> ,font=\small](M-3-4) -- (M-1-5)
	;

\draw[<->](M-3-1) -- (M-3-2)
	;
\draw[-> ,font=\small](M-3-2) -- (M-3-3) 
	;
\draw[-> ,font=\small](M-3-2) -- (M-1-3) 
	;
\draw[-> ,font=\small](M-3-3) -- (M-3-4)
	;
\draw[-> ,font=\small](M-3-3) -- (M-1-4)
	;
	
\draw[<->](M-7-1) -- (M-7-2)
	;
\draw[-> ,font=\small](M-7-2) -- (M-7-3) 
	node[midway,above,scale=0.7] {$(n+1|n)$};
\draw[-> ,font=\small](M-7-3) -- (M-7-4) 
	node[midway,above,scale=0.7] {$(n|n-1)$};
\draw[<->](M-7-4) -- (M-7-5)
	;
\draw[<-> ,font=\small](M-9-4) -- (M-7-5)
	;

\draw[<->](M-9-1) -- (M-9-2)
	;
\draw[-> ,font=\small](M-9-2) -- (M-9-3) 
	node[midway,above,scale=0.7] {$(n-1|n)$};
\draw[-> ,font=\small](M-9-3) -- (M-9-4)
	node[midway,above,scale=0.7] {$(n|n+1)$};
	
\draw[-> ,font=\small, dashed](M-3-2) -- (M-7-2)
	node[midway,left,scale=0.7] {$
		2i
		\begin{bmatrix}
		-1 & 0 \\
		0  & 1
		\end{bmatrix}
		$};
\draw[-> ,font=\small, dashed](M-3-3) -- (M-7-3)
	node[midway,left,scale=0.7] {$M$};
\draw[-> ,font=\small, dashed](M-3-4) -- (M-7-4)
	node[midway,left,scale=0.7] {$I$};
\end{tikzpicture}
\end{center}
For the rest of the modules in $Q'$, we define $\mu_j$ as the identity for modules contained in $\nabla'$, and $\mu_j$ sends $v$ to $-2iv$ (resp. $v$ to $2iv$) for any $v$ belonging to the modules in $\wt{\nabla}$ (resp. $\undertilde{\nabla}$).
The black arrows in the last two rows shows the differential component $(\d_{j+1})_{G_j}$ in $\d_{j+1}$, induced by the conjugation of $\mu_j^{-1}$.
It is easy to see that $\d_{j+1}$ does indeed satisfy the required property \eqref{induction}.

The construction of $\mu_j$ for the second form is similar, changing $\mu_j|_{P^A_n \oplus P^A_n}$ to $N$ from \eqref{eqn: matrix N} instead of $M$.

\item $\check{g}_j$ is of Type III$'_{k+\frac{1}{2}}$ for $k\in Z$:\\
Note that the case for $k=0$ is straightforward: $\mu_j$ is just the identity.
We will provide the analysis and construction of $\mu_j$ for $k>0$ and $k <0$.

As in other types, the equations $(\d_j)_{G_j} = (\d_0)_{G_j}$ holds for all $j$ by the induction hypothesis.
Using the same argument in Type V$''$, one can show that \eqref{VR splits} hold for all $j$ in this type as well.
So the part of $(Q', \d_j)$ that contains $(Q'_{G_j}, (\d_j)_{G_j})$, $L_A(\wt{h_j})$ and $L_A(\undertilde{h_j})$ will be of the same form for all $j$.
It follows that the construction of $\mu_j$ below will work for all $j$.

Let us start with $k > 0$.
Using the same notation as in the analysis of Type V$''$ we can draw the part of $(Q', \d_j)$ that contains $(Q'_{G_j}, (\d_j)_{G_j}) = (Q'_{G_j}, (\d_0)_{G_j})$, $L_A(\wt{h}_j)$ and $L_A(\undertilde{h_j})$ as either of the two forms:
\begin{figure}[H]
\begin{tikzpicture}[>=stealth, baseline]
\matrix (M) [matrix of math nodes, column sep=7mm]
{
       & P_n^A  & \cdots & P_n^A  & P_n^A  & P^A(\undertilde{r_j}) & \undertilde{\nabla}\\
       & \oplus &        & \oplus & \oplus & \oplus  &   & & =(Q', \d_j)\\
P_n^A  & P_n^A  & \cdots & P_n^A  & P_n^A  & P^A(\wt{r_j}) & \wt{\nabla}\\
};

\draw[-> ,font=\small](M-1-2.east |- M-1-3) -- (M-1-3)
	;
\draw[-> ,font=\small](M-1-3.east |- M-1-4) -- (M-1-4)
	;
\draw[-> ,font=\small](M-1-4.east |- M-1-5) -- (M-1-5)
	node[midway,above] {$2X_n$}; 
\draw[-> ,font=\small](M-1-5) -- (M-1-6)
	node[midway,above] {$F$}; 
\draw[-> ,font=\small](M-1-5) -- (M-3-6)
	;
\draw[<->](M-1-6.east |- M-1-7) -- (M-1-7)
	;

\draw[-> ,font=\small](M-3-1.east |- M-3-2) -- (M-3-2) 
	node[midway,above] {$2X_n$};
\draw[-> ,font=\small](M-3-2.east |- M-3-3) -- (M-3-3)
	;
\draw[-> ,font=\small](M-3-3.east |- M-3-4) -- (M-3-4)
	;
\draw[-> ,font=\small](M-3-4.east |- M-3-5) -- (M-3-5)
	node[midway,above] {$2X_n$};
\draw[-> ,font=\small](M-3-5) -- (M-1-6)
	; 
\draw[-> ,font=\small](M-3-5) -- (M-3-6)
	;
\draw[<->](M-3-6.east |- M-3-7) -- (M-3-7)
	;

\draw[Brown] (M-1-6.north west) rectangle (M-1-7.south east)
	node[midway, xshift = 2cm] 
  	{$L_A\left(\undertilde{h_j}\right)$};
\draw[red] (M-3-6.north west) rectangle (M-3-7.south east)
	node[midway, xshift = 2cm] 
  	{$L_A\left(\wt{h_j}\right)$};
  	
\draw[OliveGreen] (M-1-6.north east)++(0,4pt) rectangle (M-3-1.south west)
		node[midway,  xshift = -5.4cm] 
  		{$(Q'_{G_j}, \; (\d_j)_{G_j})$};
\end{tikzpicture}

\vspace*{1mm}
{\LARGE or}
\vspace*{1mm}

\begin{tikzpicture}[>=stealth, baseline]
\matrix (M) [matrix of math nodes, column sep=7mm]
{
P_n^A  & P_n^A  & \cdots & P_n^A  & P_n^A  & P^A(\undertilde{r_j}) & \undertilde{\nabla}\\
       & \oplus &        & \oplus & \oplus & \oplus   & & & =(Q', \d_j).\\
       & P_n^A  & \cdots & P_n^A  & P_n^A  & P^A(\wt{r_j}) & \wt{\nabla}\\
};

\draw[-> ,font=\small](M-1-1.east |- M-1-2) -- (M-1-2) 
	node[midway,above] {$2X_n$};
\draw[-> ,font=\small](M-1-2.east |- M-1-3) -- (M-1-3)
	;
\draw[-> ,font=\small](M-1-3.east |- M-1-4) -- (M-1-4)
	;
\draw[-> ,font=\small](M-1-4.east |- M-1-5) -- (M-1-5)
	node[midway,above] {$2X_n$}; 
\draw[-> ,font=\small](M-1-5) -- (M-1-6)
	node[midway,above] {$F$}; 
\draw[-> ,font=\small](M-1-5) -- (M-3-6)
	;
\draw[<->](M-1-6.east |- M-1-7) -- (M-1-7)
	;

\draw[-> ,font=\small](M-3-2.east |- M-3-3) -- (M-3-3)
	;
\draw[-> ,font=\small](M-3-3.east |- M-3-4) -- (M-3-4)
	;
\draw[-> ,font=\small](M-3-4.east |- M-3-5) -- (M-3-5)
	node[midway,above] {$2X_n$};
\draw[-> ,font=\small](M-3-5) -- (M-1-6)
	; 
\draw[-> ,font=\small](M-3-5) -- (M-3-6)
	;
\draw[<->](M-3-6.east |- M-3-7) -- (M-3-7)
	;
	
\draw[Brown] (M-1-6.north west) rectangle (M-1-7.south east)
	node[midway, xshift = 2cm] 
  	{$L_A\left(\undertilde{h_j}\right)$};
\draw[red] (M-3-6.north west) rectangle (M-3-7.south east)
	node[midway, xshift = 2cm] 
  	{$L_A\left(\wt{h_j}\right)$};
  	
\draw[OliveGreen] (M-1-1.north west)++(0,4pt) rectangle (M-3-6.south east)
		node[midway, xshift = -5.4cm] 
  		{$(Q'_{G_j}, \; (\d_j)_{G_j})$};
\end{tikzpicture}
\end{figure}
We will construct $\mu_j$ for the first form; the second form is a mirrored construction.
Define $\mu_j|_{Q'_{G_j}}$ to be the following dashed map:
\begin{center}
\begin{tikzpicture}[>=stealth, baseline]
\matrix (M) [matrix of math nodes, column sep=7mm]
{
       & P_n^A  & \cdots & P_n^A  & P_n^A  & P^A(\undertilde{r_j}) & \undertilde{\nabla}\\ [-5pt]
       & \oplus &        & \oplus & \oplus & \oplus    & & =(Q', \d_j)\\[-5pt]
P_n^A  & P_n^A  & \cdots & P_n^A  & P_n^A  & P^A(\wt{r_j}) & \wt{\nabla}\\
{}     &        &        &        &        &           &       \\
{}     &        &        &        &        &           &       \\
{}     &        &        &        &        &           &       \\
       & P_n^A  & \cdots & P_n^A  & P_n^A  & P^A(\undertilde{r_j}) & \undertilde{\nabla}\\[-5pt]
       & \oplus &        & \oplus & \oplus & \oplus    & & =:(Q', \d_{j+1})\\[-5pt]
P_n^A  & P_n^A  & \cdots & P_n^A  & P_n^A  & P^A(\wt{r_j}) & \wt{\nabla}\\
};

\draw[-> ,font=\small](M-1-2.east |- M-1-3) -- (M-1-3)
	;
\draw[-> ,font=\small](M-1-3.east |- M-1-4) -- (M-1-4)
	;
\draw[-> ,font=\small](M-1-4.east |- M-1-5) -- (M-1-5)
	node[midway,above] {$2X_n$}; 
\draw[-> ,font=\small](M-1-5) -- (M-1-6)
	node[midway,above] {$F$}; 
\draw[-> ,font=\small](M-1-5) -- (M-3-6)
	;
\draw[<->](M-1-6.east |- M-1-7) -- (M-1-7);

\draw[-> ,font=\small](M-3-1.east |- M-3-2) -- (M-3-2) 
	node[midway,above] {$2X_n$};
\draw[-> ,font=\small](M-3-2.east |- M-3-3) -- (M-3-3)
	;
\draw[-> ,font=\small](M-3-3.east |- M-3-4) -- (M-3-4)
	;
\draw[-> ,font=\small](M-3-4.east |- M-3-5) -- (M-3-5)
	node[midway,above] {$2X_n$};
\draw[-> ,font=\small](M-3-5) -- (M-1-6)
	; 
\draw[-> ,font=\small](M-3-5) -- (M-3-6)
	;
\draw[<->](M-3-6.east |- M-3-7) -- (M-3-7)
	;

\draw[-> ,font=\small](M-7-2.east |- M-7-3) -- (M-7-3)
	node[midway,above] {$X_n$};
\draw[-> ,font=\small](M-7-3.east |- M-7-4) -- (M-7-4)
	;
\draw[-> ,font=\small](M-7-4.east |- M-7-5) -- (M-7-5)
	node[midway,above] {$X_n$}; 
\draw[-> ,font=\small](M-7-5) -- (M-7-6)
	;
\draw[<->](M-7-6.east |- M-7-7) -- (M-7-7)
	;

\draw[-> ,font=\small](M-9-1) -- (M-7-2) 
	node[midway,above,transform canvas={xshift=-2mm}] {$X_n$};
\draw[-> ,font=\small](M-9-1.east |- M-9-2) -- (M-9-2) 
	node[midway,below] {$X_n$};
\draw[-> ,font=\small](M-9-2.east |- M-9-3) -- (M-9-3)
	node[midway,above] {$X_n$};
\draw[-> ,font=\small](M-9-3.east |- M-9-4) -- (M-9-4)
	;
\draw[-> ,font=\small](M-9-4.east |- M-9-5) -- (M-9-5)
	node[midway,above] {$X_n$};
\draw[-> ,font=\small](M-9-5) -- (M-9-6)
	; 
\draw[<->](M-9-6.east |- M-9-7) -- (M-9-7)
	;

\draw[-> ,font=\small, dashed](M-3-1) -- (M-9-1) 
	node[midway,left,scale=0.8, dashed] {$2^{k-1}i$};
\draw[-> ,font=\small, dashed](M-3-2) -- (M-7-2) 
	node[midway,left,scale=0.8, dashed] 
		{$2^{k-2}JM$};
\draw[-> ,font=\small, dashed](M-3-4) -- (M-7-4) 
	node[midway,left,scale=0.8, dashed] 
		{$2JM$};
\draw[-> ,font=\small, dashed](M-3-5) -- (M-7-5) 
	node[midway,left,scale=0.8, dashed] {$JM$};
\draw[-> ,font=\small, dashed](M-3-6) -- (M-7-6) 
	node[midway,left,scale=0.8, dashed] {$J$};
\end{tikzpicture}
\end{center}
with $J = \begin{bsmallmatrix}
			-1 & 0 \\
			0  & 1
			\end{bsmallmatrix}$.
For the rest of the modules in $Q'$, $\mu_j$ sends $v$ to $-v$ (resp. $v$ to $v$) for any $v$ belonging to the modules in $\undertilde{\nabla}$ (resp. $\wt{\nabla}$).
The black arrows in the last two rows shows the differential component $(\d_{j+1})_{G_j}$ in $\d_{j+1}$, induced by the conjugation of $\mu_j^{-1}$.
It is easy to see that $\mu_j$ does indeed satisfy the required property \eqref{induction}.

Similarly for $k < 0$, we draw the part of $(Q', \d_j)$ that contains $(Q'_{G_j}, (\d_j)_{G_j}) = (Q'_{G_j}, (\d_0)_{G_j})$, $L_A(\wt{h}_j)$ and $L_A(\undertilde{h_j})$ as either of the two forms:
\begin{figure}[H]
\begin{tikzpicture}[>=stealth, baseline]
\matrix (M) [matrix of math nodes, column sep=7mm]
{
\wt{\nabla} & P^A(\wt{r_j}) & P_n^A  & P_n^A  & \cdots & P_n^A  &        \\
       & \oplus    & \oplus & \oplus &        & \oplus &  & &  &=(Q', \d_j)\\
\undertilde{\nabla} & P^A(\undertilde{r_j}) & P_n^A  & P_n^A  & \cdots & P_n^A  & P_n^A  \\
};

\draw[-> ,font=\small](M-1-2.east |- M-1-3) -- (M-1-3) 
	node[midway,above] {$E$};
\draw[-> ,font=\small](M-1-2) -- (M-3-3) 
	;
\draw[-> ,font=\small](M-1-3.east |- M-1-4) -- (M-1-4) 
	node[midway,above] {$2X_n$};
\draw[-> ,font=\small](M-1-4.east |- M-1-5) -- (M-1-5)
	;
\draw[-> ,font=\small](M-1-5.east |- M-1-6) -- (M-1-6)
	;

\draw[-> ,font=\small](M-3-2) -- (M-1-3)
	;
\draw[-> ,font=\small](M-3-2) -- (M-3-3)
	;
\draw[-> ,font=\small](M-3-3.east |- M-3-4) -- (M-3-4) 
	node[midway,above] {$2X_n$};
\draw[-> ,font=\small](M-3-4.east |- M-3-5) -- (M-3-5)
	;
\draw[-> ,font=\small](M-3-5.east |- M-3-6) -- (M-3-6)
	;
\draw[-> ,font=\small](M-3-6.east |- M-3-7) -- (M-3-7)
	node[midway,above] {$2X_n$};
\draw[<->](M-3-1.east |- M-3-2) -- (M-3-2);
\draw[<->](M-1-1) -- (M-1-2); 

\draw[red] (M-1-1.north west) rectangle (M-1-2.south east)
	node[midway, xshift = -2cm] 
  	{$L_A\left(\wt{h_j}\right)$};
\draw[Brown] (M-3-1.north west) rectangle (M-3-2.south east)
	node[midway, xshift = -2cm] 
  	{$L_A\left(\undertilde{h_j}\right)$};
  	
\draw[OliveGreen] (M-1-2.north west)++(0,4pt) rectangle 
	(M-3-7.south east)
		node[midway, xshift=5.3cm, yshift = 22pt] 
  		{$(Q'_{G_j}, (\d_j)_{G_j})$};
\end{tikzpicture}

\vspace*{1mm}
{\LARGE or}
\vspace*{1mm}

\begin{tikzpicture}[>=stealth, baseline]
\matrix (M) [matrix of math nodes, column sep=7mm]
{
\wt{\nabla} & P^A(\wt{r_j}) & P_n^A  & P_n^A  & \cdots & P_n^A  & P_n^A \\
       & \oplus    & \oplus & \oplus &        & \oplus & & & & = (Q', \d_j).\\
\undertilde{\nabla} & P^A(\undertilde{r_j}) & P_n^A  & P_n^A  & \cdots & P_n^A  &       \\
};

\draw[-> ,font=\small](M-1-2.east |- M-1-3) -- (M-1-3) 
	node[midway,above] {$E$};
\draw[-> ,font=\small](M-1-2) -- (M-3-3) 
	;
\draw[-> ,font=\small](M-1-3.east |- M-1-4) -- (M-1-4) 
	node[midway,above] {$2X_n$};
\draw[-> ,font=\small](M-1-4.east |- M-1-5) -- (M-1-5)
	;
\draw[-> ,font=\small](M-1-5.east |- M-1-6) -- (M-1-6)
	;
\draw[-> ,font=\small](M-1-6.east |- M-1-7) -- (M-1-7)
	node[midway,above] {$2X_n$};

\draw[-> ,font=\small](M-3-2) -- (M-1-3)
	;
\draw[-> ,font=\small](M-3-2) -- (M-3-3)
	;
\draw[-> ,font=\small](M-3-3.east |- M-3-4) -- (M-3-4) 
	node[midway,above] {$2X_n$};
\draw[-> ,font=\small](M-3-4.east |- M-3-5) -- (M-3-5)
	;
\draw[-> ,font=\small](M-3-5.east |- M-3-6) -- (M-3-6)
	;
\draw[<->](M-3-1.east |- M-3-2) -- (M-3-2);
\draw[<->](M-1-1) -- (M-1-2); 

\draw[red] (M-1-1.north west) rectangle (M-1-2.south east)
	node[midway, xshift = -2cm] 
  	{$L_A\left(\wt{h_j}\right)$};
\draw[Brown] (M-3-1.north west) rectangle (M-3-2.south east)
	node[midway, xshift = -2cm] 
  	{$L_A\left(\undertilde{h_j}\right)$};
  	
\draw[OliveGreen] (M-1-2.north west)++(0,4pt) rectangle 
	(M-3-7.south east)
		node[midway, xshift=5.3cm, yshift = -.8cm] 
  		{$(Q'_{G_j}, (\d_j)_{G_j})$};
\end{tikzpicture}
\end{figure}
Once again we construct $\mu_j$ for the first form; the second form is a mirrored construction.
Note that 
$
NE =
2
\begin{bsmallmatrix}
(n+1|n) &   0      \\
0       &  (n-1|n)
\end{bsmallmatrix}.
$
We define $\mu_j|_{Q'_{G_j}}$ to be the following dashed map, with $\mu_j$ acting as the identity on the rest of the modules in both $\wt{\nabla}$ and $\undertilde{\nabla}$:
\begin{center}
\begin{tikzpicture}[>=stealth, baseline]
\matrix (M) [matrix of math nodes, column sep=7mm]
{
\wt{\nabla} & P^A(\wt{r_j}) & P_n^A  & P_n^A  & \cdots & P_n^A  &        \\
       & \oplus    & \oplus & \oplus &        & \oplus & & =(Q', \d_j) \\
\undertilde{\nabla} & P^A(\undertilde{r_j}) & P_n^A  & P_n^A  & \cdots & P_n^A  & P_n^A  \\
{}     &           &        &        &        &        &        \\
{}     &           &        &        &        &        &        \\
{}     &           &        &        &        &        &        \\
\wt{\nabla} & P^A(\wt{r_j}) & P_n^A  & P_n^A  & \cdots & P_n^A  &        \\
       & \oplus    & \oplus & \oplus &        & \oplus & & =:(Q', \d_{j+1}). \\
\undertilde{\nabla} & P^A(\undertilde{r_j}) & P_n^A  & P_n^A  & \cdots & P_n^A  & P_n^A. \\
};

\draw[-> ,font=\small](M-1-2.east |- M-1-3) -- (M-1-3) 
	node[midway,above] {$E$};
\draw[-> ,font=\small](M-1-2) -- (M-3-3) 
	;
\draw[-> ,font=\small](M-1-3.east |- M-1-4) -- (M-1-4) 
	node[midway,above] {$2X_n$};
\draw[-> ,font=\small](M-1-4.east |- M-1-5) -- (M-1-5)
	;
\draw[-> ,font=\small](M-1-5.east |- M-1-6) -- (M-1-6)
	;

\draw[-> ,font=\small](M-3-2) -- (M-1-3)
	;
\draw[-> ,font=\small](M-3-2) -- (M-3-3)
	;
\draw[-> ,font=\small](M-3-3.east |- M-3-4) -- (M-3-4) 
	node[midway,above] {$2X_n$};
\draw[-> ,font=\small](M-3-4.east |- M-3-5) -- (M-3-5)
	;
\draw[-> ,font=\small](M-3-5.east |- M-3-6) -- (M-3-6)
	;
\draw[-> ,font=\small](M-3-6.east |- M-3-7) -- (M-3-7)
	node[midway,above] {$2X_n$};
\draw[<->](M-3-1.east |- M-3-2) -- (M-3-2);
\draw[<->](M-1-1) -- (M-1-2); 

\draw[-> ,font=\small](M-7-2.east |- M-7-3) -- (M-7-3) 
	node[midway,above,scale=0.7] {$(n+1|n)$};
\draw[-> ,font=\small](M-7-3.east |- M-7-4) -- (M-7-4) 
	node[midway,above] {$X_n$};
\draw[-> ,font=\small](M-7-4.east |- M-7-5) -- (M-7-5)
	;
\draw[-> ,font=\small](M-7-5.east |- M-7-6) -- (M-7-6)
	;
\draw[-> ,font=\small](M-7-6) -- (M-9-7)
	node[midway, above, transform canvas={xshift=2mm}] {$X_n$};

\draw[-> ,font=\small](M-9-2) -- (M-9-3)
	node[midway,above,scale=0.7] {$(n-1|n)$};
\draw[-> ,font=\small](M-9-3.east |- M-9-4) -- (M-9-4) 
	node[midway,above] {$X_n$};
\draw[-> ,font=\small](M-9-4.east |- M-9-5) -- (M-9-5)
	;
\draw[-> ,font=\small](M-9-5.east |- M-9-6) -- (M-9-6)
	;
\draw[-> ,font=\small](M-9-6.east |- M-9-7) -- (M-9-7)
	node[midway,above] {$X_n$};
\draw[<->](M-9-1.east |- M-9-2) -- (M-9-2);
\draw[<->](M-7-1) -- (M-7-2); 

\draw[-> ,font=\small, dashed](M-3-2) -- (M-7-2) 
	node[midway,left,scale=0.8] {$I$};
\draw[-> ,font=\small, dashed](M-3-3) -- (M-7-3) 
	node[midway,right,scale=0.8] {$2^{-1}N$};
\draw[-> ,font=\small, dashed](M-3-4) -- (M-7-4) 
	node[midway,right,scale=0.8] {$2^{-2}N$};
\draw[-> ,font=\small, dashed](M-3-6) -- (M-7-6) 
	node[midway,left,scale=0.8] {$2^{k+1}N$};
\draw[-> ,font=\small, dashed](M-3-7) -- (M-9-7) 
	node[midway,right,scale=0.8] {$2^{k+1}$};
\end{tikzpicture}
\end{center}
The black arrows in the last two rows shows the differential component $(\d_{j+1})_{G_j}$ in $\d_{j+1}$, induced by the conjugation of $\mu_j^{-1}$.
The required condition \eqref{induction} follows directly.

\item $\check{g}_j$ is of Type II$'_{k+\frac{1}{2}}$ for $k \in \Z$: \\
As in other types, the equations $(\d_j)_{G_j} = (\d_0)_{G_j}$ holds for all $j$ by the induction hypothesis.
Using the same argument in Type V$''$, one can show that \eqref{VR splits} hold for all $j$ in this type as well.
So the part of $(Q', \d_j)$ that contains $(Q'_{G_j}, (\d_j)_{G_j})$, $L_A(\wt{h_j})$ and $L_A(\undertilde{h_j})$ will be of the same form for all $j$.
It follows that the construction of $\mu_j$ below will work for all $j$.

Let us start with $k=0$.
With the same notation as in the analysis of Type V$''$, we can draw the part of $(Q', \d_j)$ that contains $(Q'_{G_j}, (\d_j)_{G_j}) = (Q'_{G_j}, (\d_0)_{G_j})$, $L_A(\wt{h}_j)$ and $L_A(\undertilde{h_j})$ as follows:
\begin{center}
\begin{tikzpicture}[>=stealth, baseline]
\matrix (M) [matrix of math nodes, column sep=7mm]
{
{}       & P_{n-1}^A & \nabla' \\
         & \oplus    &        \\
 P_{n}^A & P_{n+1}^A &        \\
 \oplus  & \oplus    & & & &  =(Q', \d_j), \\
 P_{n}^A & P^A(\wt{r_j}) & \wt{\nabla} \\
         & \oplus    &        \\
         & P^A(\undertilde{r_j}) & \undertilde{\nabla} \\
};

\draw[<->](M-1-2) -- (M-1-3)
	;
\draw[<->](M-3-2) -- (M-1-3)
	;
\draw[<->](M-5-2) -- (M-5-3)
	; 
\draw[<->](M-7-2) -- (M-7-3)
	; 

\draw[-> ,font=\small](M-3-1) -- (M-1-2.west)
	;
\draw[-> ,font=\small](M-3-1) -- (M-3-2)
	;
\draw[-> ,font=\small](M-5-1) -- (M-1-2)
	;
\draw[-> ,font=\small](M-5-1) -- (M-3-2)
	;
	
\draw[-> ,font=\small](M-3-1) -- (M-5-2)
	;
\draw[-> ,font=\small](M-3-1) -- (M-7-2)
	;
\draw[-> ,font=\small](M-5-1) -- (M-5-2)
	;
\draw[-> ,font=\small](M-5-1) -- (M-7-2.west)
	;

\draw[red] (M-5-2.north west) rectangle (M-5-3.south east)
	node[midway, xshift = 2.3cm] 
  	{$L_A\left(\wt{h_j}\right)$};
\draw[Brown] (M-7-2.north west) rectangle (M-7-3.south east)
	node[midway, xshift = 2.3cm] 
  	{$L_A\left(\undertilde{h_j}\right)$};
  	
\draw[OliveGreen] (M-1-1.north west) ++(-10pt,10pt) rectangle 
	(M-7-2.south east)
		node[midway, xshift =  -2.5cm] 
  		{$(Q'_{G_j}, (\d_j)_{G_j})$};
\end{tikzpicture}
\end{center}
where the first map is given by $
	\begin{bsmallmatrix}
	F \\
	F'
	\end{bsmallmatrix}
	$,
with $F'$ defined by $\begin{bsmallmatrix}
	-i & 0 \\
	 0 & i 
	\end{bsmallmatrix}
	F' =
	\begin{bsmallmatrix}
	(n|n+1) & 0 \\
	 0      & (n|n-1) 
	\end{bsmallmatrix}
	M
	$.
We define $\mu_j|_{Q'_{G_j}}$ to be the following dashed map:
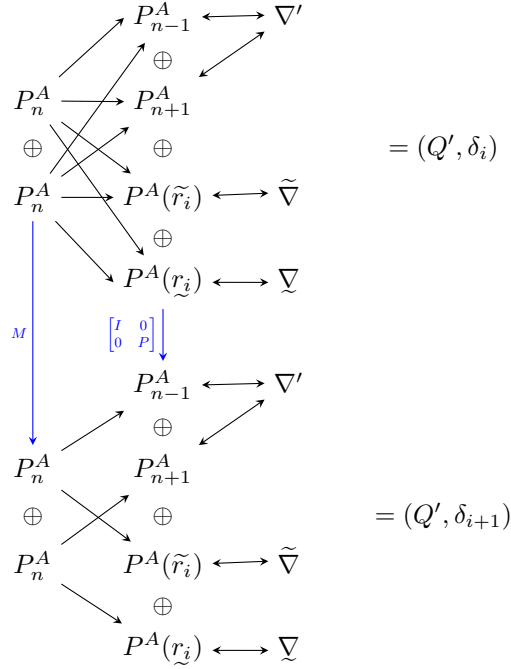
\begin{center}
\begin{tikzpicture}[>=stealth, baseline]
\matrix (M) [matrix of math nodes, column sep=7mm]
{
         & P_{n-1}^A & \nabla' \\
         & \oplus    &        \\
 P_{n}^A & P_{n+1}^A &        \\
 \oplus  & \oplus    & & =(Q', \d_j)  \\
 P_{n}^A & P^A(\wt{r_j}) & \wt{\nabla} \\
         & \oplus    &        \\
         & P^A(\undertilde{r_j}) & \undertilde{\nabla} \\
{}       &           &        \\
{}       &           &        \\
{}       &           &        \\         
         & P_{n-1}^A & \nabla' \\
         & \oplus    &        \\
 P_{n}^A & P_{n+1}^A &        \\
 \oplus  & \oplus    & & =: (Q', \d_{j+1}), \\
 P_{n}^A & P^A(\wt{r_j}) & \wt{\nabla} \\
         & \oplus    &        \\
         & P^A(\undertilde{r_j}) & \undertilde{\nabla} \\
};

\draw[<->](M-1-2) -- (M-1-3)
	;
\draw[<->](M-3-2) -- (M-1-3)
	;
\draw[<->](M-5-2) -- (M-5-3)
	; 
\draw[<->](M-7-2) -- (M-7-3)
	; 

\draw[-> ,font=\small](M-3-1) -- (M-1-2.west)
	;
\draw[-> ,font=\small](M-3-1) -- (M-3-2)
	;
\draw[-> ,font=\small](M-5-1) -- (M-1-2)
	;
\draw[-> ,font=\small](M-5-1) -- (M-3-2)
	;
	
\draw[-> ,font=\small](M-3-1) -- (M-5-2)
	;
\draw[-> ,font=\small](M-3-1) -- (M-7-2)
	;
\draw[-> ,font=\small](M-5-1) -- (M-5-2)
	;
\draw[-> ,font=\small](M-5-1) -- (M-7-2.west)
	;

\draw[<->](M-11-2) -- (M-11-3)
	;
\draw[<->](M-13-2) -- (M-11-3)
	;
\draw[<->](M-15-2) -- (M-15-3)
	; 
\draw[<->](M-17-2) -- (M-17-3)
	; 

\draw[-> ,font=\small](M-13-1) -- (M-11-2)
	;
\draw[-> ,font=\small](M-15-1) -- (M-13-2)
	;
\draw[-> ,font=\small](M-13-1) -- (M-15-2)
	;
\draw[-> ,font=\small](M-15-1) -- (M-17-2)
	;
	
\draw[-> ,font=\small, dashed](M-5-1) -- (M-13-1)
	node[midway, left,scale=0.6] {$M$};
\draw[-> ,font=\small, dashed](M-7-2) -- (M-11-2)
	node[midway, left,scale=0.6] {$
	\begin{bmatrix}
	I & 0 \\
	0 & iJ
	\end{bmatrix}		
	$};	
\end{tikzpicture}
\end{center}
where $J = 
	\begin{bsmallmatrix}
	-1 & 0 \\
	 0 & 1 
	\end{bsmallmatrix}
	$, $\mu_j$ is the identity on all modules contained in $\nabla'$ and $\mu_j$ sends $v$ to $-iv$ (resp. $iv$) for any $v$ belonging to the modules in $\wt{\nabla}$ (resp. $\undertilde{\nabla}$).
The black arrows in the last four rows shows the differential component $(\d_{j+1})_{G_j}$ in $\d_{j+1}$, induced by the conjugation of $\mu_j^{-1}$.
Thus, the required condition \eqref{induction} is easily satisfied.

For $k > 0$, we can draw the part of $(Q', \d_j)$ that contains $(Q'_{G_j}, (\d_j)_{G_j}) = (Q'_{G_j}, (\d_0)_{G_j})$, $L_A(\wt{h}_j)$ and $L_A(\undertilde{h_j})$ as follows:
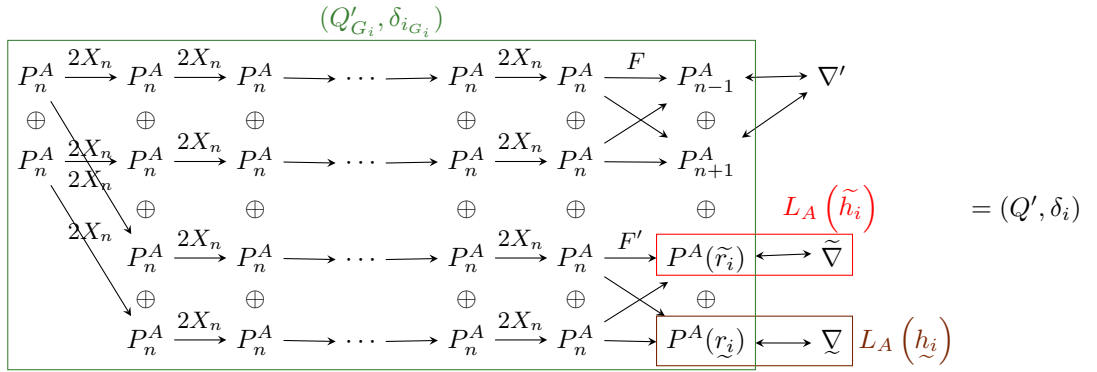
\begin{center}
\begin{tikzpicture}[>=stealth, baseline]
\matrix (M) [matrix of math nodes, column sep=7mm]
{
P_n^A  & P_n^A  & P_n^A  & \cdots & P_n^A  & P_n^A  & P_{n-1}^A & \nabla' \\
\oplus & \oplus & \oplus &        & \oplus & \oplus & \oplus    & \\
P_n^A  & P_n^A  & P_n^A  & \cdots & P_n^A  & P_n^A  & P_{n+1}^A &        \\
       & \oplus & \oplus &        & \oplus & \oplus & \oplus & & & =(Q', \d_j).\\ 
       & P_n^A  & P_n^A  & \cdots & P_n^A  & P_n^A  & P^A(\wt{r_j})& \wt{\nabla} \\
       & \oplus & \oplus &        & \oplus & \oplus & \oplus    & \\
       & P_n^A  & P_n^A  & \cdots & P_n^A  & P_n^A  & P^A(\undertilde{r_j})& \undertilde{\nabla} \\
};

\draw[-> ,font=\small](M-1-1) -- (M-1-2) 
	node[midway,above] {$2X_n$};
\draw[-> ,font=\small](M-1-1) -- (M-5-2) 
	node[midway,above] {$2X_n$};
\draw[-> ,font=\small](M-1-2) -- (M-1-3) 
	node[midway,above] {$2X_n$};
\draw[-> ,font=\small](M-1-3) -- (M-1-4) 
	;
\draw[-> ,font=\small](M-1-4) -- (M-1-5)
	;
\draw[-> ,font=\small](M-1-5) -- (M-1-6)
	node[midway,above] {$2X_n$};
\draw[-> ,font=\small](M-1-6) -- (M-1-7)
	node[midway,above] {$F$};
\draw[-> ,font=\small](M-1-6) -- (M-3-7)
	;
\draw[<-> ,font=\small](M-1-7) -- (M-1-8)
	;
\draw[<-> ,font=\small](M-3-7) -- (M-1-8)
	;

\draw[-> ,font=\small](M-3-1) -- (M-3-2)
	node[midway,below] {$2X_n$};
\draw[-> ,font=\small](M-3-2) -- (M-3-3)
	node[midway,above] {$2X_n$};
\draw[-> ,font=\small](M-3-1) -- (M-7-2)
	node[midway,above] {$2X_n$};
\draw[-> ,font=\small](M-3-3) -- (M-3-4) 
	;
\draw[-> ,font=\small](M-3-4) -- (M-3-5)
	;
\draw[-> ,font=\small](M-3-5) -- (M-3-6)
	node[midway,above] {$2X_n$};
\draw[-> ,font=\small](M-3-6) -- (M-1-7)
	;
\draw[-> ,font=\small](M-3-6) -- (M-3-7)
	;

\draw[-> ,font=\small](M-5-2) -- (M-5-3)
	node[midway,above] {$2X_n$};
\draw[-> ,font=\small](M-5-3) -- (M-5-4) 
	;
\draw[-> ,font=\small](M-5-4) -- (M-5-5) 
	;
\draw[-> ,font=\small](M-5-5) -- (M-5-6)
	node[midway,above] {$2X_n$};
\draw[-> ,font=\small](M-5-6) -- (M-5-7)
	node[midway,above] {$F'$};
\draw[-> ,font=\small](M-5-6) -- (M-7-7)
	;
\draw[<-> ,font=\small](M-5-7) -- (M-5-8)
	;

\draw[-> ,font=\small](M-7-2) -- (M-7-3)
	node[midway,above] {$2X_n$};
\draw[-> ,font=\small](M-7-3) -- (M-7-4) 
	;
\draw[-> ,font=\small](M-7-4) -- (M-7-5) 
	;
\draw[-> ,font=\small](M-7-5) -- (M-7-6)
	node[midway,above] {$2X_n$};
\draw[-> ,font=\small](M-7-6) -- (M-5-7)
	;
\draw[-> ,font=\small](M-7-6) -- (M-7-7)
	;
\draw[<-> ,font=\small](M-7-7) -- (M-7-8)
	;
	
\draw[red] (M-5-7.north west) rectangle (M-5-8.south east)
	node[midway, above, xshift = 1cm, yshift = 5pt] 
  	{$L_A\left(\wt{h_j}\right)$};
\draw[Brown] (M-7-7.north west) rectangle (M-7-8.south east)
	node[midway, xshift = 2cm ] 
  	{$L_A\left(\undertilde{h_j}\right)$};
  	
\draw[OliveGreen] (M-1-1.north west)++(0,5pt) rectangle 
	(M-7-7.south east)
		node[midway, above, yshift = 2.1cm] 
  		{$(Q'_{G_j}, (\d_j)_{G_j})$};
\end{tikzpicture}
\end{center}
Similarly, we define $\mu_j|_{Q'_{G_j}}$ to be the following dashed map:
\begin{center}
\begin{tikzpicture}[>=stealth, baseline]
\matrix (M) [matrix of math nodes, column sep=7mm]
{
P_n^A  & P_n^A  & P_n^A  & \cdots & P_n^A  & P_n^A  & P_{n-1}^A & \nabla' \\
\oplus & \oplus & \oplus &        & \oplus & \oplus & \oplus    & \\
P_n^A  & P_n^A  & P_n^A  & \cdots & P_n^A  & P_n^A  & P_{n+1}^A &        \\
       & \oplus & \oplus &        & \oplus & \oplus & \oplus & & =(Q', \d_j)\\ 
       & P_n^A  & P_n^A  & \cdots & P_n^A  & P_n^A  & P^A(\wt{r_j})& \wt{\nabla} \\
       & \oplus & \oplus &        & \oplus & \oplus & \oplus    & \\
       & P_n^A  & P_n^A  & \cdots & P_n^A  & P_n^A  & P^A(\undertilde{r_j})&\undertilde{\nabla} \\
{}     &        &        &        &        &        &          \\
{}     &        &        &        &        &        &          \\ 
{}     &        &        &        &        &        &          \\
P_n^A  & P_n^A  & P_n^A  & \cdots & P_n^A  & P_n^A  & P_{n-1}^A & \nabla' \\
\oplus & \oplus & \oplus &        & \oplus & \oplus & \oplus    & \\
P_n^A  & P_n^A  & P_n^A  & \cdots & P_n^A  & P_n^A  & P_{n+1}^A &        \\
       & \oplus & \oplus &        & \oplus & \oplus & \oplus & & =:(Q', \d_{j+1}),\\ 
       & P_n^A  & P_n^A  & \cdots & P_n^A  & P_n^A  & P^A(\wt{r_j})& \wt{\nabla} \\
       & \oplus & \oplus &        & \oplus & \oplus & \oplus    & \\
       & P_n^A  & P_n^A  & \cdots & P_n^A  & P_n^A  & P^A(\undertilde{r_j})&\undertilde{\nabla} \\
};

\draw[-> ,font=\small](M-1-1) -- (M-1-2) 
	node[midway,above] {$2X_n$};
\draw[-> ,font=\small](M-1-1) -- (M-5-2) 
	node[midway,above] {$2X_n$};
\draw[-> ,font=\small](M-1-2) -- (M-1-3) 
	node[midway,above] {$2X_n$};
\draw[-> ,font=\small](M-1-3) -- (M-1-4) 
	;
\draw[-> ,font=\small](M-1-4) -- (M-1-5)
	;
\draw[-> ,font=\small](M-1-5) -- (M-1-6)
	node[midway,above] {$2X_n$};
\draw[-> ,font=\small](M-1-6) -- (M-1-7)
	node[midway,above] {$F$};
\draw[-> ,font=\small](M-1-6) -- (M-3-7)
	;
\draw[<-> ,font=\small](M-1-7) -- (M-1-8)
	;
\draw[<-> ,font=\small](M-3-7) -- (M-1-8)
	;

\draw[-> ,font=\small](M-3-1) -- (M-3-2)
	node[midway,below] {$2X_n$};
\draw[-> ,font=\small](M-3-2) -- (M-3-3)
	node[midway,above] {$2X_n$};
\draw[-> ,font=\small](M-3-1) -- (M-7-2)
	node[midway,above] {$2X_n$};
\draw[-> ,font=\small](M-3-3) -- (M-3-4) 
	;
\draw[-> ,font=\small](M-3-4) -- (M-3-5)
	;
\draw[-> ,font=\small](M-3-5) -- (M-3-6)
	node[midway,above] {$2X_n$};
\draw[-> ,font=\small](M-3-6) -- (M-1-7)
	;
\draw[-> ,font=\small](M-3-6) -- (M-3-7)
	;

\draw[-> ,font=\small](M-5-2) -- (M-5-3)
	node[midway,above] {$2X_n$};
\draw[-> ,font=\small](M-5-3) -- (M-5-4) 
	;
\draw[-> ,font=\small](M-5-4) -- (M-5-5) 
	;
\draw[-> ,font=\small](M-5-5) -- (M-5-6)
	node[midway,above] {$2X_n$};
\draw[-> ,font=\small](M-5-6) -- (M-5-7)
	node[midway,above] {$F'$};
\draw[-> ,font=\small](M-5-6) -- (M-7-7)
	;
\draw[<-> ,font=\small](M-5-7) -- (M-5-8)
	;

\draw[-> ,font=\small](M-7-2) -- (M-7-3)
	node[midway,above] {$2X_n$};
\draw[-> ,font=\small](M-7-3) -- (M-7-4) 
	;
\draw[-> ,font=\small](M-7-4) -- (M-7-5) 
	;
\draw[-> ,font=\small](M-7-5) -- (M-7-6)
	node[midway,above] {$2X_n$};
\draw[-> ,font=\small](M-7-6) -- (M-5-7)
	;
\draw[-> ,font=\small](M-7-6) -- (M-7-7)
	;
\draw[<-> ,font=\small](M-7-7) -- (M-7-8)
	;
\draw[-> ,font=\small](M-11-1) -- (M-11-2) 
	node[midway,above] {$X_n$};
\draw[-> ,font=\small](M-11-1) -- (M-15-2) 
	node[midway,below] {$X_n$};
\draw[-> ,font=\small](M-11-2) -- (M-11-3) 
	node[midway,above] {$X_n$};
\draw[-> ,font=\small](M-11-3) -- (M-11-4) 
	;
\draw[-> ,font=\small](M-11-4) -- (M-11-5)
	;
\draw[-> ,font=\small](M-11-5) -- (M-11-6)
	node[midway,above] {$X_n$};
\draw[-> ,font=\small](M-11-6) -- (M-11-7)
	node[midway,above,scale=0.7] {$(n|n-1)$};
\draw[<-> ,font=\small](M-11-7) -- (M-11-8)
	;
\draw[<-> ,font=\small](M-13-7) -- (M-11-8)
	;

\draw[-> ,font=\small](M-13-1) -- (M-13-2)
	node[midway,above] {$X_n$};
\draw[-> ,font=\small](M-13-2) -- (M-13-3)
	node[midway,above] {$X_n$};
\draw[-> ,font=\small](M-13-1) -- (M-17-2)
	node[midway,left] {$X_n$};
\draw[-> ,font=\small](M-13-3) -- (M-13-4) 
	;
\draw[-> ,font=\small](M-13-4) -- (M-13-5)
	;
\draw[-> ,font=\small](M-13-5) -- (M-13-6)
	node[midway,above] {$X_n$};
\draw[-> ,font=\small](M-13-6) -- (M-13-7)
	node[midway,above,scale=0.7] {$(n|n+1)$};

\draw[-> ,font=\small](M-15-2) -- (M-15-3)
	node[midway,above] {$X_n$};
\draw[-> ,font=\small](M-15-3) -- (M-15-4) 
	;
\draw[-> ,font=\small](M-15-4) -- (M-15-5)
	;
\draw[-> ,font=\small](M-15-5) -- (M-15-6)
	node[midway,above] {$X_n$};
\draw[-> ,font=\small](M-15-6) -- (M-15-7)
	node[midway,above,scale=0.7] {$(n|n+1)$};
\draw[<-> ,font=\small](M-15-7) -- (M-15-8)
	;

\draw[-> ,font=\small](M-17-2) -- (M-17-3)
	node[midway,above] {$X_n$};
\draw[-> ,font=\small](M-17-3) -- (M-17-4) 
	;
\draw[-> ,font=\small](M-17-4) -- (M-17-5)
	;
\draw[-> ,font=\small](M-17-5) -- (M-17-6)
	node[midway,above] {$X_n$};
\draw[-> ,font=\small](M-17-6) -- (M-17-7)
	node[midway,above,scale=0.7] {$(n|n-1)$};
\draw[<-> ,font=\small](M-17-7) -- (M-17-8)
	;

\draw[-> ,font=\small, dashed](M-3-1) -- (M-11-1)
	node[midway,left] {$2^{k}M$};
\draw[-> ,font=\small, dashed](M-7-2) -- (M-11-2)
	node[midway,left,scale=0.6, dashed] {$2^{k-1}
		\begin{bmatrix}
		M & 0 \\
		0 & M \\
		\end{bmatrix}
		$};
\draw[-> ,font=\small, dashed](M-7-3) -- (M-11-3)
	node[midway, right,scale=0.6] {$ 2^{k-2}
		\begin{bmatrix}
		M & 0 \\
		0 & M \\
		\end{bmatrix}
		$};
\draw[-> ,font=\small, dashed](M-7-5) -- (M-11-5)
	node[midway, right,scale=0.6] {$2
		\begin{bmatrix}
		M & 0 \\
		0 & M \\
		\end{bmatrix}
		$};
\draw[-> ,font=\small, dashed](M-7-6) -- (M-11-6)
	node[midway, right,scale=0.6] {$
		\begin{bmatrix}
		M & 0 \\
		0 & M \\
		\end{bmatrix}
		$};
\draw[-> ,font=\small, dashed](M-7-7) -- (M-11-7)
	node[midway, right,scale=0.6]{$
	\begin{bmatrix}
	I & 0 \\
	0 & iJ 
	\end{bmatrix}		
	$};
\end{tikzpicture}
\end{center}
where $\mu_j$ is the identity on all modules contained in $\nabla'$ and $\mu_j$ sends $v$ to $-iv$ (resp. $iv$) for any $v$ belonging to the modules in $\wt{\nabla}$ (resp. $\undertilde{\nabla}$).
The black arrows in the last four rows shows the differential component $(\d_{j+1})_{G_j}$ in $\d_{j+1}$, induced by the conjugation of $\mu_j^{-1}$.
Thus, the required condition \eqref{induction} is easily satisfied.
The construction for $k < 0$ is similar, using the map $N$ in place of $M$.
\end{enumerate}
This completes the list of $\mu_j$ for all possible types of 1-string $\check{g}_j$ for all $j$, and thus completing the proof.
\end{proof}
Using this result, we can now deduce a type $B_n$ version of \cref{L_A equivariant}:
\begin{theorem} \label{L_B equivariant}
For $\sigma^B \in \cA({B_n})$ and an admissible trigraded curve $\check{c}$ in $\D^B_{n+1}$, we have that $$\sigma^B(L_B(\check{c})) \cong L_B(\sigma^B(\check{c})),$$ in the category of $\Kom^b(\Ba_n$-$p_rg_rmod)$, i.e. the map $L_B$ is $\mathcal{A}(B_n)$-equivariant.
\end{theorem}
\begin{proof}
Let $\check{c}$ be a trigraded curve in $\D^B_{n+1}$ and $\sigma^B$ be an element of $\mathcal{A}(B_n)$.
By \cref{diagram commutes}, the diagram in \cref{full picture} commutes.
Together with the three other maps being equivariant, we can conclude that
\begin{equation} \label{commute iso}
\Aa_{2n-1} \otimes_{\Ba_n} \left( L_B(\sigma^B(\check{c}))\right) \cong \Aa_{2n-1} \otimes_{\Ba_n} (\sigma^B(L_B(\check{c}))).
\end{equation}
Recall that the functor $\Aa_{2n-1} \otimes_{\Ba_n} -$ was defined as $\Aa_{2n-1} \otimes_{\ddot{\Ba}_n} \mathfrak{F}(-)$ (see paragraph after \cref{isomorphic algebras} for the definitions).
Since $\Aa_{2n-1} \cong \C\otimes_\R \ddot{\Ba}_n$ as graded $\C$-algebras, \eqref{commute iso} together with the fact that both categories $\Kom^b(\ddot{\Ba}_n\text{-}p_rg_rmod)$ and $\Kom^b(\Aa_{2n-1}\text{-}p_rg_rmod)$ are Krull-Schmidt implies that
\begin{equation}\label{iso shift}
L_B(\sigma^B(\check{c})) \cong \sigma^B (L_B(\check{c}))\<s\>,
\end{equation}
with $s = 0$ or $1$.

We aim to show that $s = 0$ for all cases.
First consider the case when $c \cap \{0\} = \emptyset$.
As $\check{c} \simeq \chi(r_1,r_2,r_3)\beta(\check{b}_2)$ and $P_2^B \cong P_2^B \<1\>$, it follows easily that $(L_B(\check{c}))\<0\> \cong (L_B(\check{c}))\<1\>$, and so we are done.
Now consider when $c$ has one of its end points at $0$.
Note that it is sufficient to prove the statement for $\sigma^B = \sigma^B_j$ for each $j$.
We assign to each complex $C$ in $\Kom^b(\Ba_n$-$\text{p$_{r}$g$_{r}$mod})$ an element of $\Z/2\Z$ denoted by $\sgn(C)$, by taking the sum of the third grading $\<-\>$ over all modules $P_1$ in $C$.
One can easily show that $\sgn$ is invariant under isomorphisms in  $\Kom^b(\Ba_n$-$\text{p$_{r}$g$_{r}$mod})$, where using \cref{gradinginvariant} below, we get that $s$ must be 0 as required.
   \end{proof}
%
\begin{lemma} \label{gradinginvariant}
For any trigraded curve $\check{c}$ with one of its endpoint at $\{0\}$ and any generating braid $\sigma^B_j,$  we have $\sgn \left( \sigma^B_j (L_B(\check{c})) \right) = \sgn (L_B(\sigma^B_j(\check{c}))).$
\end{lemma}

\begin{proof}
For $j \geq 2,$ it is clear that
$\sgn\left( \sigma^B_j ( L_B(\check{c}) ) \right) = \sgn \left( L_B(\check{c}) \right) = \sgn (L_B(\sigma^B_j(\check{c}))).$

  Now fix $j=1.$ 
  First consider the case when $\check{c}$ is of type VI, i.e.  $\check{c} = \chi^B(r_1, r_2, r_3)\check{b}_1$.
  Then $\sigma^B_1(\check{c}) = \chi^B(r_1-1, r_2+1, r_3+1)\check{b}_1 $ by \cref{action}(2). 
  So, $L_B(\sigma^B_1(\check{c})) = P_1^B [-r_1+1] \{r_2 +1\} \< r_3+1\>$.
  On the other hand, 
\begin{align*}
\sigma^B_1 (L_B((\check{c}))) 
&\cong  \left(P_1^B [1] \{1 \} \oplus P_1^B [1]  \xra{[X_1 \ id]} P_1^B \right) [-r_1] \{r_2 \} \< r_3 + 1\> \\
&\cong  P_1^B [-r_1+1] \{r_2 +1 \} \< r_3+1\>. 
\end{align*}  
\noindent Thus, $\sgn \left( \sigma^B_j (L_B(\check{c})) \right) = r_3 +1 = \sgn (L_B(\sigma^B_j(\check{c}))).$
  
  Otherwise, we analyse $\sgn$ based on the number of $2$-crossing in $\check{c}.$
  Note that,  for $1$-strings $\check{g},$
 \[\sgn \left( L_B(\sigma^B_1 (\check{g})) \right)  =  \begin{cases} 
     \sgn \left( L_B(\check{g}) \right),  & \text{when $\check{g}$ is of type $II'_w$, $II'_{w+ \frac{1}{2}}$, $III'_{w+ \frac{1}{2}}$ and $V''$;} \\
     \sgn \left( L_B(\check{g}) \right) + 1, & \text{when $\check{g}$ is of type $III'_w$.}\\       
   \end{cases}
    \]  
   Note that $\sigma^B_1$ won't change the number of $2$-crossings of $\check{c}$ and as $\check{c}\cap \{0\} = \{0\}$, $\check{c}$ contains $1$-string of type  $III'_w$ if and only if $\check{c}$ has even number of 2-crossings.
   Since $\sgn \left(  L_B(\sigma^B_1 (\check{c})) \right)$ can be computed by summing over all 1-strings of $\check{c}$, we conclude that
    \[\sgn \left(  L_B(\sigma^B_1 (\check{c})) \right)  =  \begin{cases} 
     \sgn \left( L_B(\check{c}) \right),  & \text{if $\check{c}$ has odd number of $2$-crossings;} \\
     \sgn \left( L_B(\check{c}) \right) + 1, & \text{if $\check{c}$ has even number of $2$-crossings.}\\       
   \end{cases}
    \]
 \noindent 
 On the other hand, note that $\check{c}$ and $\sigma_1^B(\check{c})$ both have an odd number of $1$-crossings. 
 Moreover, $\sgn(C\<1\>)= \sgn(C)+1$ if and only if $C$ has an odd number of underlying $P_1^B$'s. As such, we have that 
  \[\sgn \left( \sigma^B_1 ( L_B(\check{c}) ) \right)  =  \begin{cases} 
     \sgn \left( L_B(\check{c}) \right),  & \text{if $L_B(\check{c}$) has odd number of modules $P^B_2$;} \\
     \sgn \left( L_B(\check{c}) \right) + 1, & \text{if $L_B(\check{c}$) has even number of modules $P^B_2$,}\\       
   \end{cases}
    \]
    since $\sigma_1^B = R_1\<1\>$, ${}_1P^B\otimes_{\Ba_n} P^B_1 \cong \R \oplus \R\{1\}$, ${}_1P^B\otimes_{\Ba_n} P^B_2 \cong \R \oplus \R \<1\>$ and ${}_1P^B \otimes_{\Ba_n} P^B_j = 0$ for all $j \geq 3$ (see \cref{bimodule isomorphism} and \cref{rmk: Z/2Z grading bimodule iso}).
 Thus, we get $\sgn \left( \sigma^B_1 ( L_B(\check{c}) )\right) = \sgn \left(   L_B(\sigma^B_1(\check{c})) \right). $                                                                                                                                                                                                                                                                                                                                                                                                                                                                                                                                                                                                                                                                                                                                                                                                                                                                                                                                                           
  \end{proof}

\section{Categorification of Homological Representations}\label{categorification hom rep}
In this section, we shall relate the categorical representations of type $A_{2n-1}$ and type $B_n$ Artin groups to their representations on the first homology of surfaces.

Throughout this section we will use $\cK_A := \Kom^b(\Aa_{2n-1}$-$\text{p$_{r}$g$_{r}$mod})$ and $\cK_B := \Kom^b(\Ba_n$-$\text{p$_{r}$g$_{r}$mod})$ as shorthand notations.
We will also use $\cZ_A$, $\cZ_{B,s}$ and $\cZ_{B,r}$ to denote the rings $\Z[q^{\pm 1}]$, $\Z[q^{\pm 1}, s]/\<s^2 - 1\>$ and $\Z[q^{\pm 1}, r]/\<r^2 - 1\>$ respectively.
We denote the Grothendieck group of $\cK_A$ and $\cK_B$ as $K_0(\cK_A)$ and $K_0(\cK_B)$ respectively; recall that they are the abelian groups freely generated by the isomorphism classes of objects, quotient by the relation
$
[\cone(A \xra{f} B)] = [B] - [A]
$.

\subsection{Representations on Grothendieck groups}

First, consider the Grothendieck group $K_0(\cK_A)$.
The functor $\{1\}$ makes $K_0(\cK_A)$ into a $\mathcal{Z}_A$-module defined by $[X\{1\}] = q[X]$.
Note that $K_0(\cK_A)$ is isomorphic to the split Grothendieck group $K^\oplus_0(\Aa_{2n-1}$-$p_rg_r$mod) (see \cite[Theorem 1.1]{Rose_Gro}), so $K_0(\cK_A)$ is a free module over $\mathcal{Z}_A$ of rank $2n-1$, generated by $\{[P^A_j] \mid 1 \leq j \leq 2n-1\}$.
The action of $\cA(A_{2n-1})$ on $\cK_A$ preserves cones; namely for all $\sigma \in \cA(A_{2n-1})$,
\[
\sigma(\cone(A \xra{f} B)) \cong \cone(\sigma(A) \xra{\sigma(f)} \sigma(B)).
\]
Moreover, the action commutes with the grading shift functor, so we have an induced $\cZ_A$-linear representation of $\cA(A_{2n-1})$ on $K_0(\cK_A)$, which we denote by $\rho_{KA}: \cA(A_{2n-1}) \ra \text{Aut}_{\cZ_A}(K_0(\cK_A))$.

Now consider the Grothendieck group $K_0(\cK_B)$.
The functors $\{1\}$ and $\<1\>$ make $K_0(\cK_B)$ into a module over $\mathcal{Z}_{B,s}$.
As $P^B_j\<1\> \cong P^B_j$ for all $j \geq 2$, we have that $s[P_j] = [P_j]$ for all $j\geq 2$.
It is easy to see now that
\[
K_0(\cK_B) \cong \mathcal{Z}_{B,s} \oplus ( \mathcal{Z}_{B,s}/\<s-1\> )^{\oplus n-1}
\]
as $\mathcal{Z}_B$-modules, generated by $\{[P^B_j] \ | \ 1 \leq j \leq n\}$.
As before, the action of $\cA(B_n)$ on $\cK_B$ preserves cones and it commutes with the grading shift functors, so we have an induced $\cZ_{B,s}$-linear representation of $\cA(B_n)$ on $K_0(\cK_B)$, which we denote by $\rho_{KB}: \cA(B_n) \ra \text{Aut}_{\cZ_{B,s}}(K_0(\cK_B))$.

\subsection{Homological representations}
It is well-known that the reduced Burau representation of $\cA(A_{2n-1})$ can be realised as a homological representation (see for example \cite[Theorem 3.7 and Theorem 3.9]{BraidGroups}).
Nonetheless, we shall spell out the construction here as it will shed some light on the construction of the homological representation for type $B_n$ and also clarify the relationship between them.

Consider the covering space $\cD_{2n}$ classified by the cohomology class $C_{\cD} \in H^1(\DA \setminus \Delta, \Z)$ defined by
$
[\lambda_k] \mapsto 1, \text{for all } k \in \Delta = \{-n, ..., -1, 1, ..., n\},
$
where each $\lambda_k$ is a closed loop around the puncture $k$.
It is easy to see that the space $\cD_{2n}$ is homotopy equivalent to the infinite graph $\Gamma_{\Z}$ given below:
\begin{figure}[H]  

\begin{tikzpicture} [scale=0.7]

\draw[fill] (3,2.5) circle [radius=0.08] ;
\draw[fill] (6,2.5) circle [radius=0.08]  ;
\draw[fill] (9,2.5) circle [radius=0.08]  ;
\draw[fill] (12,2.5) circle [radius=0.08]  ;
\draw[fill] (0,2.5) circle [radius=0.08];

\draw  [cyan, dashed] plot[smooth, tension=1]coordinates { (0,2.5) (-.55, 3) (-1.25, 3.2)  };
\draw  [cyan, dashed] plot[smooth, tension=1]coordinates { (0,2.5) (-.55, 2.85) (-1.25, 3)  };
\draw  [cyan, dashed] plot[smooth, tension=1]coordinates { (0,2.5) (-.55, 2.75) (-1.25, 2.8)  };
\draw  [cyan, dashed] plot[smooth, tension=1]coordinates { (0,2.5) (-.55, 2) (-1.25, 1.8)  };
\draw  [cyan, dashed] plot[smooth, tension=1]coordinates { (0,2.5) (-.55, 2.15) (-1.25, 2)  };
\draw  [cyan, dashed] plot[smooth, tension=1]coordinates { (0,2.5) (-.55, 2.25) (-1.25, 2.2)  };

\draw  [cyan, dashed] plot[smooth, tension=1]coordinates { (12,2.5) (12.55, 3) (13.25, 3.2)  };
\draw  [cyan, dashed] plot[smooth, tension=1]coordinates { (12,2.5) (12.55, 2.85) (13.25, 3)  };
\draw  [cyan, dashed] plot[smooth, tension=1]coordinates { (12,2.5) (12.55, 2.75) (13.25, 2.8)  };
\draw  [cyan, dashed] plot[smooth, tension=1]coordinates { (12,2.5) (12.55, 2) (13.25, 1.8)  };
\draw  [cyan, dashed] plot[smooth, tension=1]coordinates { (12,2.5) (12.55, 2.15) (13.25, 2)  };
\draw  [cyan, dashed] plot[smooth, tension=1]coordinates { (12,2.5) (12.55, 2.25) (13.25, 2.2)  };

\draw  [cyan] plot[smooth, tension=1]coordinates { (0,2.5) (1.5, 3.2) (3,2.5)};
\draw  [cyan] plot[smooth, tension=1]coordinates { (0,2.5) (1.5, 3) (3,2.5)};
\draw  [cyan] plot[smooth, tension=1]coordinates { (0,2.5) (1.5, 2.8) (3,2.5)};
\draw  [cyan] plot[smooth, tension=1]coordinates { (0,2.5) (1.5, 2) (3,2.5)};
\draw  [cyan] plot[smooth, tension=1]coordinates { (0,2.5) (1.5, 2.2) (3,2.5)};
\draw  [cyan] plot[smooth, tension=1]coordinates { (0,2.5) (1.5, 1.8) (3,2.5)};

\draw  [cyan] plot[smooth, tension=1]coordinates { (3,2.5) (4.5, 3.2) (6,2.5)};
\draw  [cyan] plot[smooth, tension=1]coordinates { (3,2.5) (4.5, 3) (6,2.5)};
\draw  [cyan] plot[smooth, tension=1]coordinates { (3,2.5) (4.5, 2.8) (6,2.5)};
\draw  [cyan] plot[smooth, tension=1]coordinates { (3,2.5) (4.5, 2) (6,2.5)};
\draw  [cyan] plot[smooth, tension=1]coordinates { (3,2.5) (4.5, 2.2) (6,2.5)};
\draw  [cyan] plot[smooth, tension=1]coordinates { (3,2.5) (4.5, 1.8) (6,2.5)};

\draw  [cyan] plot[smooth, tension=1]coordinates { (6,2.5) (7.5, 3.2) (9,2.5)};
\draw  [cyan] plot[smooth, tension=1]coordinates { (6,2.5) (7.5, 3) (9,2.5)};
\draw  [cyan] plot[smooth, tension=1]coordinates { (6,2.5) (7.5, 2.8) (9,2.5)};
\draw  [cyan] plot[smooth, tension=1]coordinates { (6,2.5) (7.5, 2) (9,2.5)};
\draw  [cyan] plot[smooth, tension=1]coordinates { (6,2.5) (7.5, 2.2) (9,2.5)};
\draw  [cyan] plot[smooth, tension=1]coordinates { (6,2.5) (7.5, 1.8) (9,2.5)};

\draw  [cyan] plot[smooth, tension=1]coordinates { (9,2.5) (10.5, 3.2) (12,2.5)};
\draw  [cyan] plot[smooth, tension=1]coordinates { (9,2.5) (10.5, 3) (12,2.5)};
\draw  [cyan] plot[smooth, tension=1]coordinates { (9,2.5) (10.5, 2.8) (12,2.5)};
\draw  [cyan] plot[smooth, tension=1]coordinates { (9,2.5) (10.5, 2) (12,2.5)};
\draw  [cyan] plot[smooth, tension=1]coordinates { (9,2.5) (10.5, 2.2) (12,2.5)};
\draw  [cyan] plot[smooth, tension=1]coordinates { (9,2.5) (10.5, 1.8) (12,2.5)};

\node [cyan] at (1.5,3.2) {$>$};
\node [cyan] at (1.5,3) {$>$};
\node [cyan] at (1.5,2.8) {$>$};
\node [cyan] at (1.5,2.2) {$>$};
\node [cyan] at (1.5,2) {$>$};
\node [cyan] at (1.5,1.8) {$>$};

\node [cyan] at (4.5,3.2) {$>$};
\node [cyan] at (4.5,3) {$>$};
\node [cyan] at (4.5,2.8) {$>$};
\node [cyan] at (4.5,2.2) {$>$};
\node [cyan] at (4.5,2) {$>$};
\node [cyan] at (4.5,1.8) {$>$};

\node [cyan] at (7.5,3.2) {$>$};
\node [cyan] at (7.5,3) {$>$};
\node [cyan] at (7.5,2.8) {$>$};
\node [cyan] at (7.5,2.2) {$>$};
\node [cyan] at (7.5,2) {$>$};
\node [cyan] at (7.5,1.8) {$>$};

\node [cyan] at (10.5,3.2) {$>$};
\node [cyan] at (10.5,3) {$>$};
\node [cyan] at (10.5,2.8) {$>$};
\node [cyan] at (10.5,2.2) {$>$};
\node [cyan] at (10.5,2) {$>$};
\node [cyan] at (10.5,1.8) {$>$};

\node  at (1.5,2.6) {$ \boldsymbol{\vdots}$};
\node  at (4.5,2.6) {$ \boldsymbol{\vdots}$};
\node  at (10.5,2.6) {$ \boldsymbol{\vdots}$};
\node  at (7.5,2.6) {$ \boldsymbol{\vdots}$};

\node[above,cyan] at (1.5,3.2) {\scriptsize {$q^{-1}\lambda_n$}};
\node[above,cyan] at (4.5,3.2) {\scriptsize {$\lambda_n$}};
\node[above,cyan] at (7.5,3.2) {\scriptsize {$q\lambda_n$}};
\node[above,cyan] at (10.5,3.2) {\scriptsize {$q^{2}\lambda_n$}};
\node[below,cyan] at (10.5,1.8) {\scriptsize {$q^{2}\lambda_{-n}$}};
\node[below,cyan] at (1.5,1.8) {\scriptsize {$q^{-1}\lambda_{-n}$}};
\node[below,cyan] at (4.5,1.8) {\scriptsize {$\lambda_{-n}$}};
\node[below,cyan] at (7.5,1.8) {\scriptsize {$q\lambda_{-n}$}};

\node[below] at  (0,2.5) {\scriptsize {$q^{-1}p$}};
\node[below] at  (3,2.5) {\scriptsize {$p$}};
\node[below] at  (6,2.5) {\scriptsize {$qp$}};
\node[below] at  (9,2.5) {\scriptsize {$q^2p$}};
\node[below] at  (12,2.5) {\scriptsize {$q^3p$}};

\end{tikzpicture}

\caption{{\small The infinite graph $\Gamma_\Z$ homotopy equivalent to $\cD_{2n}$.} } 
\end{figure}
\noindent The action of $\cA(A_{2n-1})$ on $\DA \setminus \Delta$ lifts to an action on $\cD_{2n}$ that commutes with the deck transformation group $\Z$, so it induces a $\Z[\Z]$-linear action of $\cA(A_{2n-1})$ on $H_1(\cD_{2n}, \Z)$, which we denote by
$
\rho_{RHA} : \cA(A_{2n-1}) \ra \text{Aut}_{\Z[\Z]}(H_1(\cD_{2n}, \Z)).
$

Now let $\Delta_0 := \Delta \cup \{0\}$ and consider the covering space $\sD_{2n+1}$ of $\DA \setminus \Delta_0$ classified by the cohomology class $C_\sD \in H^1(\DA \setminus \Delta_0, \Z)$ defined by
\[
[\lambda_j] \mapsto 
\begin{cases}
1, &\text{ for } j \neq 0; \\
0, &\text{ for } j = 0,
\end{cases}
\]
where each $\lambda_j$ is a closed loop around the puncture $j$.
Note that the composition of coverings
\[
\sD_{2n+1} \ra \DA \setminus \Delta_0 \ra \DB \setminus \Lambda
\]
is a normal covering space of $\DB \setminus \Lambda$, with its group of deck transformation isomorphic to $\Z \times \Z/2\Z$ (cf. \cref{finalcovering}).

Let $l_j$ be a closed loop around each puncture $j \in \Lambda$ of $\DB$.
Note that in $H_1(\DA \setminus \Delta_0, \Z)$ we have that
\begin{equation} \label{loops in B and A}
\begin{cases}
[\lambda_0] = [l_0] + r[l_0]; \\
[\lambda_{-j}] = r[l_j], &\text{for } j > 0; \\
[\lambda_{j}]  =  [l_j], &\text{for } j > 0.
\end{cases}
\end{equation}
As such, 
the space $\sD_{2n+1}$ is homotopy equivalent to the space given below:

\begin{figure}[H] 

\begin{tikzpicture}  [scale=0.8]

\draw[fill] (3,2.5) circle [radius=0.08] ;
\draw[fill] (6,2.5) circle [radius=0.08]  ;
\draw[fill] (9,2.5) circle [radius=0.08]  ;
\draw[fill] (12,2.5) circle [radius=0.08]  ;
\draw[fill] (0,2.5) circle [radius=0.08];

\draw[fill] (3,5) circle [radius=0.08] ;
\draw[fill] (6,5) circle [radius=0.08]  ;
\draw[fill] (9,5) circle [radius=0.08]  ;
\draw[fill] (12,5) circle [radius=0.08]  ;
\draw[fill] (0,5) circle [radius=0.08];

\draw  [green] plot[smooth, tension=1]coordinates { (0,2.5) (0.25, 3.75) (0,5)};
\draw  [green] plot[smooth, tension=1]coordinates { (0,2.5) (-0.25, 3.75) (0,5)};
\node [green] at  (-0.25, 3.75) {$\wedge$};
\node [green] at  (0.25, 3.75) {$\vee$};

\draw  [green] plot[smooth, tension=1]coordinates { (3,2.5) (2.75, 3.75) (3,5)};
\draw  [green] plot[smooth, tension=1]coordinates { (3,2.5) (3.25, 3.75) (3,5)};
\node [green] at  (2.75, 3.75) {$\wedge$};
\node [green] at  (3.25, 3.75) {$\vee$};

\draw  [green] plot[smooth, tension=1]coordinates { (6,2.5) (5.75, 3.75) (6,5)};
\draw  [green] plot[smooth, tension=1]coordinates { (6,2.5) (6.25, 3.75) (6,5)};
\node [green] at  (5.75, 3.75) {$\wedge$};
\node [green] at  (6.25, 3.75) {$\vee$};

\draw  [green] plot[smooth, tension=1]coordinates { (9,2.5) (8.75, 3.75) (9,5)};
\draw  [green] plot[smooth, tension=1]coordinates { (9,2.5) (9.25, 3.75) (9,5)};
\node [green] at  (8.75, 3.75) {$\wedge$};
\node [green] at  (9.25, 3.75) {$\vee$};

\draw  [green] plot[smooth, tension=1]coordinates { (12,2.5) (11.75, 3.75) (12,5)};
\draw  [green] plot[smooth, tension=1]coordinates { (12,2.5) (12.25, 3.75) (12,5)};
\node [green] at  (11.75, 3.75) {$\wedge$};
\node [green] at  (12.25, 3.75) {$\vee$};

\node [green, left] at  (-0.25, 3.75) {\scriptsize {$q^{-1}l_0$}};
\node [green, right] at  (0.25, 3.75) {\scriptsize {$rq^{-1}l_0$}};
\node [green, left] at  (2.75, 3.75) {\scriptsize {$l_0$}};
\node [green, right] at  (3.25, 3.75) {\scriptsize {$rl_0$}};
\node [green, left] at  (5.75, 3.75) {\scriptsize {$ql_0$}};
\node [green, right] at  (6.25, 3.75) {\scriptsize {$rql_0$}};
\node [green, left] at  (8.75, 3.75) {\scriptsize {$q^2l_0$}};
\node [green, right] at  (9.25, 3.75) {\scriptsize {$rq^2l_0$}};
\node [green, left] at  (11.75, 3.75) {\scriptsize {$q^3l_0$}};
\node [green, right] at  (12.25, 3.75) {\scriptsize {$rq^3l_0$}};

\node [cyan] at (1.5,5.3) {$>$};
\node [cyan] at (1.5,4.7) {$>$};

\node [cyan] at (4.5,5.3) {$>$};
\node [cyan] at (4.5,4.7) {$>$};

\node [cyan] at (7.5,5.3) {$>$};
\node [cyan] at (7.5,4.7) {$>$};

\node [cyan] at (10.5,5.3) {$>$};
\node [cyan] at (10.5,4.7) {$>$};

\draw  [cyan, dashed] plot[smooth, tension=1]coordinates { (0,5) (-.55, 5.2) (-1.25, 5.3)  };
\draw  [cyan, dashed] plot[smooth, tension=1]coordinates { (0,5) (-.55, 4.8) (-1.25, 4.7)  };

\draw  [cyan, dashed] plot[smooth, tension=1]coordinates { (0,2.5) (-.55, 2.7) (-1.25, 2.8)  };
\draw  [cyan, dashed] plot[smooth, tension=1]coordinates { (0,2.5) (-.55, 2.3) (-1.25, 2.2)  };

\draw  [cyan, dashed] plot[smooth, tension=1]coordinates { (12,5) (12.55, 5.2) (13.25, 5.3)  };
\draw  [cyan, dashed] plot[smooth, tension=1]coordinates { (12,5) (12.55, 4.8) (13.25, 4.7)  };

\draw  [cyan, dashed] plot[smooth, tension=1]coordinates { (12,2.5) (12.55, 2.7) (13.25, 2.8)  };
\draw  [cyan, dashed] plot[smooth, tension=1]coordinates { (12,2.5) (12.55, 2.3) (13.25, 2.2)  };

\draw  [cyan] plot[smooth, tension=1]coordinates { (0,5) (1.5, 5.3) (3,5)};
\draw  [cyan] plot[smooth, tension=1]coordinates { (0,5) (1.5, 4.7) (3,5)};

\draw  [cyan] plot[smooth, tension=1]coordinates { (3,5) (4.5, 5.3) (6,5)};
\draw  [cyan] plot[smooth, tension=1]coordinates { (3,5) (4.5, 4.7) (6,5)};

\draw  [cyan] plot[smooth, tension=1]coordinates { (6,5) (7.5, 5.3) (9,5)};
\draw  [cyan] plot[smooth, tension=1]coordinates { (6,5) (7.5, 4.7) (9,5)};

\draw  [cyan] plot[smooth, tension=1]coordinates { (9,5) (10.5, 5.3) (12,5)};
\draw  [cyan] plot[smooth, tension=1]coordinates { (9,5) (10.5, 4.7) (12,5)};

\draw  [cyan] plot[smooth, tension=1]coordinates { (0,2.5) (1.5, 2.8) (3,2.5)};
\draw  [cyan] plot[smooth, tension=1]coordinates { (0,2.5) (1.5, 2.2) (3,2.5)};

\draw  [cyan] plot[smooth, tension=1]coordinates { (3,2.5) (4.5, 2.8) (6,2.5)};
\draw  [cyan] plot[smooth, tension=1]coordinates { (3,2.5) (4.5, 2.2) (6,2.5)};

\draw  [cyan] plot[smooth, tension=1]coordinates { (6,2.5) (7.5, 2.8) (9,2.5)};
\draw  [cyan] plot[smooth, tension=1]coordinates { (6,2.5) (7.5, 2.2) (9,2.5)};

\draw  [cyan] plot[smooth, tension=1]coordinates { (9,2.5) (10.5, 2.8) (12,2.5)};
\draw  [cyan] plot[smooth, tension=1]coordinates { (9,2.5) (10.5, 2.2) (12,2.5)};

\node [cyan] at (1.5,2.8) {$>$};
\node [cyan] at (1.5,2.2) {$>$};

\node [cyan] at (4.5,2.8) {$>$};
\node [cyan] at (4.5,2.2) {$>$};

\node [cyan] at (7.5,2.8) {$>$};
\node [cyan] at (7.5,2.2) {$>$};

\node [cyan] at (10.5,2.8) {$>$};
\node [cyan] at (10.5,2.2) {$>$};

\node  at (1.5,2.6) {$ \boldsymbol{\vdots}$};
\node  at (4.5,2.6) {$ \boldsymbol{\vdots}$};
\node  at (10.5,2.6) {$ \boldsymbol{\vdots}$};
\node  at (7.5,2.6) {$ \boldsymbol{\vdots}$};

\node  at (1.5,5.1) {$ \boldsymbol{\vdots}$};
\node  at (4.5,5.1) {$ \boldsymbol{\vdots}$};
\node  at (10.5,5.1) {$ \boldsymbol{\vdots}$};
\node  at (7.5,5.1) {$ \boldsymbol{\vdots}$};

\node[above,cyan] at (1.5,5.3) {\scriptsize {$rq^{-1}l_1$}};
\node[above,cyan] at (4.5,5.3) {\scriptsize {$rl_1$}};
\node[above,cyan] at (7.5,5.3) {\scriptsize {$rql_1$}};
\node[above,cyan] at (10.5,5.3) {\scriptsize {$rq^{2}l_1$}};
\node[below,cyan] at (10.5,4.7) {\scriptsize {$rq^{2}l_{n}$}};
\node[below,cyan] at (1.5,4.7) {\scriptsize {$rq^{-1}l_{n}$}};
\node[below,cyan] at (4.5,4.7) {\scriptsize {$rl_{n}$}};
\node[below,cyan] at (7.5,4.7) {\scriptsize {$rql_{n}$}};

\node[above,cyan] at (1.5,2.8) {\scriptsize {$q^{-1}l_1$}};
\node[above,cyan] at (4.5,2.8) {\scriptsize {$l_1$}};
\node[above,cyan] at (7.5,2.8) {\scriptsize {$ql_1$}};
\node[above,cyan] at (10.5,2.8) {\scriptsize {$q^{2}l_1$}};
\node[below,cyan] at (10.5,2.2) {\scriptsize {$q^{2}l_{n}$}};
\node[below,cyan] at (1.5,2.2) {\scriptsize {$q^{-1}l_{n}$}};
\node[below,cyan] at (4.5,2.2) {\scriptsize {$l_{n}$}};
\node[below,cyan] at (7.5,2.2) {\scriptsize {$ql_{n}$}};

\node[below] at  (0,2.5) {\scriptsize {$q^{-1}p$}};
\node[below] at  (3,2.5) {\scriptsize {$p$}};
\node[below] at  (6,2.5) {\scriptsize {$qp$}};
\node[below] at  (9,2.5) {\scriptsize {$q^2p$}};
\node[below] at  (12,2.5) {\scriptsize {$q^3p$}};

\node[above] at  (0,5) {\scriptsize {$rq^{-1}p$}};
\node[above] at  (3,5) {\scriptsize {$rp$}};
\node[above] at  (6,5) {\scriptsize {$rqp$}};
\node[above] at  (9,5) {\scriptsize {$rq^2p$}};
\node[above] at  (12,5) {\scriptsize {$rq^3p$}};

\end{tikzpicture}

\caption{{\small The infinite graph $\Gamma_{\Z \times \Z/ 2\Z}$ homotopy equivalent to $\sD_{2n+1}$.} }
\end{figure}
\noindent The action of $\cA(B_n)$ on $\DB \setminus \Lambda$ lifts to an action on $\sD_{2n+1}$ that commutes with deck transformation group $\Z[\Z\times \Z/2\Z] \cong \cZ_{B,r}$, so it induces a $\cZ_{B,r}$-linear action on $H_1(\sD_{2n+1}, \Z)$, which we denote by
$
\wt{\rho_{RHB}}: \cA(B_n) \ra \text{Aut}_{\mathcal{Z}_{B,r}} ( H_1(\sD_{2n+1}, \Z) ).
$

\subsection{Relating the representations on Grothedieck groups and homology groups}
It has been shown in \cite[Section 2e.1]{KhoSei} that the induced action of $\cA(A_{2n-1})$ on the Grothendieck group $K_0(\cK_A)$ is isomorphic to the reduced Burau representation, but we shall spell it out here before dealing with the type $B$ case.
Recall that $\sH_n$ and $K_0(\cK_B)$ are modules over $\cZ_{B,r}$ and $\mathcal{Z}_{B,s}$ respectively, whereas $H_1(\cD_{2n}, \Z)$ and $K_0(\cK_A)$ are modules over $\mathcal{Z}_A$.

\begin{proposition}\label{Hom and cat rep A}
The two $\cZ_A$-linear representations 
$
\rho_{KA}: \cA(A_{2n-1}) \ra \text{Aut}_{\cZ_A}(K_0(\cK_A))
$ 
and 
$
\rho_{RHA} : \cA(A_{2n-1}) \ra \text{Aut}_{\Z[\Z]}(H_1(\cD_{2n}, \Z))
$ 
are isomorphic.
In particular, the categorical action of $\cA(A_{2n-1})$ on $\cK_A$ categorifies the reduced Burau representation.
\end{proposition}
\begin{proof}

One can check that $H_1(\cD_{2n}, \Z)$ is a free module over $\Z[\Z] \cong \mathcal{Z}_A$ of rank $2n-1$, with basis $\{[\gamma_1], ..., [\gamma_{2n-1}] \}$ defined by:
\begin{equation} \label{eqn: gamma loops defn}
[\gamma_{j}] := 
\begin{cases}
[\lambda_{-1}] - [\lambda_{1}], &\text{ for } j = n; \\
(-1)^{n-j}([\lambda_{-n+j-1}] - [\lambda_{-n+j}]), &\text{ for } j \leq n-1; \\
(-q)^{n-j}([\lambda_{-n+j}] - [\lambda_{-n+j+1}], &\text{ for } j \geq n+1.
\end{cases}
\end{equation}
Similarly, $\{P^A_1, P^A_2, ..., P^A_{2n-1}\}$ is a $\cZ_A$-basis for $K_0(\cK_A)$. 

Under both of these bases, both $\rho_{KA}$ and $\rho_{RHA}$ are given by the same matrices for each generators of $\cA(A_{2n-1})$ as follows:
\begin{align*}
\sigma^A_1 &\mapsto 
\begin{bmatrix}
-q & -q & 0 \\
0  &  1 & 0 \\
0  &  0 & I_{n-2}
\end{bmatrix}, \qquad 
\sigma^A_{2n-1} \mapsto
\begin{bmatrix}
I_{n-2} & 0  & 0 \\
0       & 1  & 0 \\
0       & -q & -q
\end{bmatrix},
\\
\sigma^A_j &\mapsto
\begin{cases}
\vspace{3mm}
\begin{bmatrix}
I_{j-2} & 0  & 0  & 0  & 0 \\
0       & 1  & 0  & 0  & 0 \\
0       & -1 & -q & -q & 0 \\
0       & 0  & 0  & 1  & 0 \\
0       & 0  & 0  & 0  & I_{n-j-1}
\end{bmatrix}, \qquad \text{for } 2 \leq j \leq n-1; \\
\vspace{3mm}
\begin{bmatrix}
I_{j-2} & 0  & 0  & 0  & 0 \\
0       & 1  & 0  & 0  & 0 \\
0       & -1 & -q & -1 & 0 \\
0       & 0  & 0  & 1  & 0 \\
0       & 0  & 0  & 0  & I_{n-j-1}
\end{bmatrix}, \qquad \text{for }  j = n; \\
\begin{bmatrix}
I_{j-2} & 0  & 0  & 0  & 0 \\
0       & 1  & 0  & 0  & 0 \\
0       & -q& -q & -1 & 0 \\
0       & 0  & 0  & 1  & 0 \\
0       & 0  & 0  & 0  & I_{n-j-1}
\end{bmatrix}, \qquad \text{for } n+1 \leq j \leq 2n-2.
\end{cases} 
\end{align*}
It follows that $\rho_{KA}$ and $\rho_{RHA}$ are isomorphic representations.
\end{proof}

\begin{proposition}\label{Hom and cat rep B}
Under the identification $\mathcal{Z}_{B,s} \cong \mathcal{Z}_{B,r}$ given by $s \mapsto -r$ and $q \mapsto q$, the  $\cZ_{B,s}$-linear representation 
$
\rho_{KB}: \cA(B_n) \ra \text{Aut}_{\cZ_{B,s}}(K_0(\cK_B))
$ 
is isomorphic to a $\cZ_{B,s}$-linear subrepresentation of 
$
\wt{\rho_{RHB}}: \cA(B_n) \ra \text{Aut}_{\mathcal{Z}_{B,s}} ( H_1(\sD_{2n+1},\Z) )
$.
\end{proposition}
\begin{proof}
Consider the sub $\mathcal{Z}_{B,r}$-module $\sH_n \subseteq H_1(\sD_{2n+1}, \Z)$ generated by $\{ [\xi_1], ... , [\xi_n] \}$, where:
\[
[\xi_j] = \begin{cases}
(1-q)[l_0] - (1-r)[l_1] &\text{ for } j =1; \\
(-q)^{1-j}(1-r)([l_{j-1}] - [l_j]) &\text{ for } j \geq 2.
\end{cases}
\]
Note that $[\xi_j] = -r[\xi_j]$ for all $j \geq 2$, so that
\[
\sH_n \cong \mathcal{Z}_{B,r} \oplus \left( \mathcal{Z}_{B,r}/\<r+1\> \right)^{\oplus n-1}
\]
as $\mathcal{Z}_{B,r}$-modules.

It is easy to verify on the generators that $\sH_n \subseteq H_1(\sD_{2n+1}, \Z)$ is closed under the action of $\cA(B_n)$, so we obtain a $\cZ_{B,r}$-linear subrepresentation on $\sH_n$, which we denote by
$
\rho_{RHB}: \cA(B_n) \ra \text{Aut}_{\mathcal{Z}_{B,r}} ( \sH_n ).
$

Using the set of generators $\{P^B_1, P^B_2, ..., P^B_n\}$ for $\cK_B$, $\rho_{KB}$ is given by the following matrices for each generators of $\cA(B_n)$:
\begin{align*}
\rho_{KB}(\sigma^B_1) &= 
\begin{bmatrix}
-sq & -(1+s) & 0 \\
0  &  1     & 0 \\
0  &  0     & I_{n-2}
\end{bmatrix}; 
\qquad
\rho_{KB}(\sigma^B_n) = 
\begin{bmatrix}
I_{n-2} & 0  & 0 \\
0       & 1  & 0 \\
0       & -q & -q
\end{bmatrix};
\\
\rho_{KB}(\sigma^B_j) &=
\begin{bmatrix}
I_{j-2} & 0  & 0  & 0  & 0 \\
0       & 1  & 0  & 0  & 0 \\
0       & -q & -q & -1 & 0 \\
0       & 0  & 0  & 1  & 0 \\
0       & 0  & 0  & 0  & I_{n-j-1}
\end{bmatrix}, \qquad \text{for } j \neq 1, n.
\end{align*}
One can check that using the set of generators $\{ [\xi_1], ... , [\xi_n] \}$ for $\sH_n$, $\rho_{RHB}$ is defined by associating the above matrices to the generators of $\cA(B_n)$ with $s = -r$.
Thus, under the identification $\mathcal{Z}_{B,s} \cong \mathcal{Z}_{B,r}$ given by $s \mapsto -r$ and $q \mapsto q$, the two representations are isomorphic.
\end{proof}

Denote $ ev_{\pm 1} : \mathcal{Z}_{B,s} \ra \mathcal{Z}_A$ as the $\Z[q^{\pm 1}]$-linear evaluation maps defined by $s \mapsto \pm 1$.
Throughout the rest of this section, we shall view $H_1(\cD_{2n}, \Z)$ as a $\mathcal{Z}_{B,s}$-module through $ev_{-1}$ and $K_0(\cK_A)$ as a $\mathcal{Z}_{B,s}$-module through $ev_1$.

The functor $\Aa_{2n-1} \otimes_{\Ba_n} -$ as in \cref{B tensor to A} preserves cones, hence it induces a map on the Grothendieck groups $K_0(\Aa_{2n-1} \otimes_{\Ba_n} -): K_0(\cK_B) \ra K_0(\cK_A)$, given by
\[
K_0(\Aa_{2n-1} \otimes_{\Ba_n} -)([P^B_j]) = 
\begin{cases}
	[P^A_n], &\text{for } j = 1; \\
	[P^A_{n-(j-1)}] + [P^A_{n+(j-1)}], &\text{otherwise}.
\end{cases}
\]
On the other hand, the natural inclusion map $\iota : \DA \setminus \Delta_0 \ra \DA \setminus \Delta$ induces a map on the homology $\iota: H_1(\DA \setminus \Delta_0, \Z) \ra H_1(\DA \setminus \Delta,\Z)$ that sends
\[
[\lambda_j] \mapsto \begin{cases}
0, &\text{ for } j = 0; \\
[\lambda_j], &\text{ otherwise}.
\end{cases}
\]
Thus, $\iota$ lifts uniquely to $\wt{\iota} : \sD_{2n+1} \ra \cD_{2n}$, which induces a map on the homology $\wt{\iota}: H_1(\sD_{2n+1}, \Z) \ra H_1(\cD_{2n}, \Z)$.
We shall now show a ``decategorified'' analogue of \cref{fullmaintheorem}:
\begin{theorem}\label{decat main theorem}
Consider the action of $\cA(B_n)$ on $\sH_n$ and $K_0(\cK_B)$ given by $\rho_{RHB} = \wt{\rho_{RHB}}|_{\sH_n}$ and $\rho_{KB}$ respectively as in \cref{Hom and cat rep B}; similarly consider the action of $\cA(B_n)\overset{\Psi}{\hookrightarrow} \cA(A_{2n-1})$ on $H_1(\cD_{2n},\Z)$ and $K_0(\cK_A)$ given by $\rho_{RHA}$ and $\rho_{KA}$ respectively as in \cref{Hom and cat rep A}.
Then there exists $\cZ_{B,s}$-linear isomorphisms $\Theta_A$ and $\Theta_B$ such that following diagram is commutative with all four maps $\mathcal{Z}_{B,s}$-linear and $\cA(B_n)$-equivariant:
\begin{center}
\begin{tikzpicture} [scale= 0.8]
\node (tbB) at (-2,1.5) 
	{$\mathcal{A}(B_n)$};
\node (cbB) at (-2,-3.5) 
	{$\mathcal{A}(B_n)$};
\node (tbA) at (10.5,1.5) 
	{$\mathcal{A}(A_{2n-1}) \overset{\Psi}{\hookleftarrow} \mathcal{A}(B_n) $}; 
\node (cbA) at (10.5,-3.5) 
	{$\mathcal{A}(A_{2n-1}) \overset{\Psi}{\hookleftarrow} \mathcal{A}(B_n) $};

\node[align=center] (cB) at (0,0) 
	{$\sH_n$};
\node[align=center] (cA) at (7,0) 
	{$H_1(\cD_{2n}, \Z)$};
\node (KB) at (0,-2)
	{$K_0(\cK_B)$};
\node (KA) at (7,-2) 
	{$K_0(\cK_A)$};
	
\coordinate (tbB') at ($(tbB.east) + (0,-1)$);
\coordinate (cbB') at ($(cbB.east) + (0,1)$);
\coordinate (tbA') at ($(tbA.west) + (0,-1)$);
\coordinate (cbA') at ($(cbA.west) + (0,1)$);

\draw [->,shorten >=-1.5pt, dashed] (tbB') arc (245:-70:2.5ex);
\draw [->,shorten >=-1.5pt, dashed] (cbB') arc (-245:70:2.5ex);
\draw [->, shorten >=-1.5pt, dashed] (tbA') arc (-65:250:2.5ex);
\draw [->,shorten >=-1.5pt, dashed] (cbA') arc (65:-250:2.5ex);

\draw[->] (cB) -- (KB) node[midway, left]{$\cong$} node[midway, right]{$\Theta_B$};
\draw[->] (cB) -- (cA) node[midway,above]{$\wt{\iota}$}; 
\draw[->] (cA) -- (KA) node[midway,left]{$\cong$} node[midway, right]{$\Theta_A$};
\draw[->] (KB) -- (KA) node[midway,above]{$K_0(\Aa_{2n-1} \otimes_{\Ba_n} -)$};
\end{tikzpicture}.
\end{center}
\end{theorem}
\begin{proof}
We use isomorphisms $\Theta_A$ and $\Theta_B$ which identify the respective generators chosen in the proof of \cref{Hom and cat rep A} and \cref{Hom and cat rep B} respectively.

The functor $\Aa_{2n-1} \otimes_{\Ba_n} -$ commutes with the grading functor, so $K_0(\Aa_{2n-1} \otimes_{\Ba_n} -)$ is $\cZ_{B,s}$ linear.
The fact that it is $\cA(B_n)$-equivariant follows from \cref{tensor equivariant}.
On the other hand, it follows immediately from the construction of the covering spaces $\sD_{2n+1}$ and $\cD_{2n}$ that $\wt{\iota}$ is a $\mathcal{Z}_{B,s}$-linear map and is $\cA(B_n)$-equivariant.

Using \eqref{loops in B and A} and \eqref{eqn: gamma loops defn}, the restriction of $\wt{\iota}$ to $\sH_n$ is given by
\[
\wt{\iota}([\xi_j]) = \begin{cases}
[\gamma_n], &\text{for } j=1; \\
[\gamma_{n-(j-1)}] + [\gamma_{n+(j-1)}], &\text{for } j \geq 2.
\end{cases}
\]
It now follows immediately that $\Theta_A \circ \wt{\iota} = K_0(\Aa_{2n-1} \otimes_{\Ba_n} -) \circ \Theta_B$.
\end{proof}

\section{Trigraded Intersection Numbers, Graded Dimensions of Homomorphism Spaces}\label{int num and hom}
In this section, we shall relate the trigraded intersection number and the Hom spaces between the corresponding complexes.
Throughout this section, we will fix the following shorthand notations: $\cK_B := \Kom^b(\sB_n$-$\text{p$_{r}$g$_{r}$mod})$, $\cK_A := \Kom^b(\Aa_{2n-1}$-$\text{p$_{r}$g$_{r}$mod})$, $\Ba m := \Ba_n$-$\text{mod}$ and $\Aa m := \Aa_{2n-1}$-$\text{mod}$.

For $V = \oplus_{(r,s) \in \Z \times \Z/2\Z} V_{(r,s)}\{r\}\<s\>$ a $(\Z\times\Z/2\Z)$-graded $\R$-vector space, we denote its \emph{bigraded dimension} as
\[
\bigrdim(V) := \sum_{(r,s) \in \Z \times \Z/2\Z} \dim(V_{(r,s)})q_2^r q_3^s.
\]

Recall that for each pair of objects $(C^*, \partial_C), (D^*, \partial_D)$ in  $\cK_B$, one can consider the internal (bigraded) Hom complex $\HOM^*_{\cK_B}(C,D)$ defined as follows:
for each cohomological degree $s_1 \in \Z$,
\[
\HOM^{s_1}_{\cK_B}(C,D) := \bigoplus_{\substack{(s_2,s_3) \in \Z \times \Z /2\Z, \\ m + n = s_1}} \Hom_{\Ba m}(C^m,D^n\{s_2\}\<s_3\>)\{-s_2\}\<s_3\>
\]
is a $\Z \times \Z/2\Z$-graded $\R$-vector space and the differentials are given by 
\[
d(f) = \partial_D \circ f - (-1)^{s_1}f \circ \partial_C,
\]
for each $f \in \HOM^{s_1}_{\cK_B}(C,D)$.
It follows that each $H^{n}\left( \HOM^*_{\cK_B}(C,D) \right)$ is a ($\Z \times \Z/2\Z$)-graded $\R$-vector space.
We define the \emph{Poincar\'e polynomial} $\mathfrak{P}(C,D) \in \Z[q_1, q_1^{-1}, q_2, q_2^{-1}, q_3]/\<q_3^2 -1\>$ of $\HOM^*_{\cK_B}(C,D)$ as 
\[ 
\mathfrak{P}(C,D) := \sum_{s_1 \in \Z} q_1^{s_1}\bigrdim_\R\left( H^{s_1}\left( \HOM^*_{\cK_B}(C,D) \right) \right).
\]

\begin{lemma}\label{poly c g}
For any trigraded admissible curve $\check{c},$ the following internal Hom complexes are quasi-isomorphic:
$$ (\HOM^*_{\cK_B}(P^B_j,L_B(\check{c})), d_C^*) \cong  \bigoplus_{\check{g} \in st(\check{c}, j)} (\HOM^*_{\cK_B}(P^B_j,L_B(\check{g})), d_G^*), $$
for all $1 \leq j \leq n$ and  $(s_1,s_2,s_3) \in \Z \times \Z \times \Z/2.$  
\end{lemma}

\begin{proof}
To simplify notation, denote
$
(C^*, \partial_C^*) := L_B(\check{c})
$
and
$
(G^*, \partial_G^*) := \bigoplus_{\check{g} \in st(\check{c}, j)} L_B(\check{g}).
$
Note that $G^*$ can be obtained from $C^*$ by discarding the modules $P^B_k$ in $L_B(\check{c})$ for $| k - j| > 1$.
In particular, for all $m \in \Z$, 
$
C^m = G^m \oplus U^m
$
where $U^m$ consists of all indecomposable $P^B_k$ in $C^m$ with $|k - j|>1$.
Using the decomposition above, let us write $\partial^m_C: C^m = G^m \oplus U^m \ra G^{m+1} \oplus U^{m+1} = C^{m+1}$ as:
\[
\partial_C^m = \begin{bmatrix}
	\tau^m & * \\
	*      & *
	\end{bmatrix}
\]
so that $\tau^m : G^m \ra G^{m+1}$.
Also note that the differential $\partial_G^m: G^m \ra G^{m+1}$ can be obtained from $\tau^m$ by modifying the differentials
$
P^B(x) \xra{\partial_{yx}} P^B(y)
$
to $0$ whenever $x$ and $y$ are crossings of two different $j$-strings of $\check{c}$.
Since 
\begin{equation} \label{Hom zero}
\bigoplus_{(s_1, s_2, s_3) \in \Z \times \Z \times \Z/2\Z} \Hom_{\Ba m}(P^B_j, P^B_k[s_1]\{s_2\}\<s_3\>) = 0,
\end{equation}
for all $k$ such that $|j - k|>1$, it follows that
\begin{equation}\label{Hom zero U}
\bigoplus_{(s_1, s_2, s_3) \in \Z \times \Z \times \Z/2\Z} \Hom_{\Ba m}(P^B_j, U^{s_1}\{s_2\}\<s_3\>) = 0,
\end{equation}
and thus
\begin{equation} \label{Hom zero conc}
\HOM^m_{\cK_B}(P_j^B, C^*)= \HOM^m_{\cK_B}(P_j^B, G^*),
\end{equation}
for each $m \in \Z$ as underlying graded vector space.
Moreover, we know that $d^m_G = (\partial^m_G\circ -)$ and $d^m_C = (\partial_C^m \circ -)$ by definition of the HOM complex.
But \eqref{Hom zero U} allows us to conclude that $d_C^m = (\tau^m\circ -)$.
Therefore to prove the proposition, it is sufficient to show that $d_C^m = (\tau^m\circ -)$ and $d^m_G = (\partial^m_G\circ -)$ have isomorphic kernels and isomorphic images for each $m \in \Z$.
For the rest of the proof, let $m \in \Z$ be arbitrary.

Let us first consider the simple case when $j \neq 2$.
We claim that $d_C^m = d_G^m$.
Note that when $j\neq 2$, $\partial_{yx}$ in $\tau^m$ that are modifed to $0$ in $\partial_G^m$ are always right multiplication by loops $X_{j-1}$ or $X_{j+1}$.
But for such maps, the corresponding induced maps on the HOM complex $(\partial_{yx} \circ -)$ are always 0, so $(\tau^m\circ -) = (\partial^m_G \circ -)$ as required.

Now let us consider the case when $j = 2$.
The types of maps $\partial_{yx} : P(x) \ra P(y)$ in $\tau^m$ that are modifed to $0$ in $\partial_G^m$ are of the following types:
\begin{enumerate}[(i)]
\item $\partial_{yx}=X_1$ or $X_3$;
\item $\partial_{yx}=(1|2)i$ or $\partial_{yx}=-i(2|1)$.
\end{enumerate}
Moreover, $\partial_{yx}$ of type (ii) does not exist in $\partial_G^m$ by definition of $L_B$.
By the same argument in the case $j \neq 2$, the induced differential in the HOM complex by $\partial_{yx}$ of type (i) is 0.
So can relate $d_C^m$ and $d_G^m$ as follows:
\begin{equation} \label{d_C in d_G}
d_C^m = (\tau^m \circ -) = (\partial^m_G\circ -) + (\d\circ -) = d_G^m + (\d\circ -),
\end{equation}
where $\d := \sum \partial_{yx}$, summing over all $\partial_{yx}$ in $\tau^m$ that are of type (ii).

Before we analyse the kernel and image of both $d^m_C$ and $d^m_G$, we shall consider a decomposition of $G^m$ and $G^{m+1}$ using $\tau^m$.
Denote $\mathfrak{G}^m$ and $\mathfrak{G}^{m+1}$ as the subset of all crossings of $\check{c}$ such that $G^m = \bigoplus_{z \in \mathfrak{G}^m} P(z)$ and $G^{m+1} = \bigoplus_{z \in \mathfrak{G}^{m+1}} P(z)$.
We shall reorganise the direct summands of $G^m$ and $G^{m+1}$ in the following way:
\begin{enumerate}
\item Set 
\begin{itemize}
\item $\alpha = 1$,
\item $\epsilon := \d$,
\item $X := \mathfrak{G}^m$,
\item $H^m := G^m$, 
\item $Y = \mathfrak{G}^{m+1}$, and 
\item $H^{m+1} = G^{m+1}$.
\end{itemize}
\item If $\epsilon = 0$, then skip to step (3);
otherwise let $\partial_{yx}$ be one of the summands in $\d$. 
Consider the smallest subset $X' \subseteq X$ and $Y' \subseteq Y$ such that 
\begin{itemize}
\item $x \in X'$, 
\item $y \in Y'$, and 
\item $\partial_{zw} = 0, X_1$ or $X_3$ whenever $w \in (X')^c, z\in Y'$ or $w\in X', z\in (Y')^c$.
\end{itemize}
We organise the direct summands of $H^m$ in the following way:
\[
H^m = Q^m_{\alpha} \oplus \left( \bigoplus_{x\in (X')^c} P(x) \right),
\text{ and }
H^{m+1} = Q^{m+1}_{\alpha} \oplus \left( \bigoplus_{y \in (Y')^c} P(y) \right).
\]
where $Q^m_{\alpha} := \bigoplus_{x \in X'} P(x)$ and $Q^{m+1}_{\alpha} :=  \bigoplus_{y \in Y'} P(y)$.
Let $e = \sum \partial_{yx}$, summing over all $\partial_{yx} = (1|2)i, -i(2|1)$  with $x \in X'$ and $y \in Y'$. \\
Redefine 
\begin{itemize}
\item $\alpha := \alpha +1$,
\item $\epsilon:= \epsilon - \gamma$, 
\item $H^m := \bigoplus_{x\in (X')^c} P(x)$, and 
\item $H^{m+1} := \bigoplus_{y \in (Y')^c} P(y)$.
\end{itemize} 
Repeat step (2).
\item If $H^m \neq 0$, then set $Q^m_{\alpha} := H^m$, else if $H^{m+1} \neq 0$, then set $Q^{m+1}_{\alpha} := H^{m+1}$.
\item Output $G^m = \bigoplus_{s \in S} Q^m_s$ and $G^{m+1} = \bigoplus_{s' \in S'} Q^m_{s'}$ with the appropriate index sets $S = \{1, ..., M\}$ and $S' = \{1, ..., M'\}$.
\end{enumerate}
Now consider $\tau^m$ and $\partial_G^m$ as block matrices corresponding to the decomposition obtained above:
\[
\tau^m = [(\tau^m)_{s',s}]_{(s',s) \in S'\times S}, \quad
\partial_G^m = [(\partial_G^m)_{s',s}]_{(s',s) \in S'\times S}.
\]
Note that by the construction of the decomposition we have that the block $(\tau^m)_{s',s}$ only have entries $X_1, X_3$ or 0 for all $s \neq s'$.
On the HOM complexes, the decompositions also give us
\begin{align*}
\HOM^m_{\cK_B}(P_j^B, C^*) = \HOM^m_{\cK_B}(P_j^B, G^*) = \bigoplus_{s \in S} \Hom_{\Ba m} (P_j^B, Q^m_s), \text{ and }
\\
\HOM^{m+1}_{\cK_B}(P_j^B, C^*) = \HOM^{m+1}_{\cK_B}(P_j^B, G^*) = \bigoplus_{s' \in S'} \Hom_{\Ba m} (P_j^B, Q^{m+1}_{s'}).
\end{align*}
Similarly consider the two differentials $d^m_C$ and $d^m_G$ written as block matrices corresponding to the decompositions:
$
d^m_C = [(d^m_{C})_{s',s}]_{(s',s) \in S' \times S}, \quad
d^m_G = [(d^m_{G})_{s',s}]_{(s', s) \in S'\times S}.
$
The construction of the decomposition guarantees the property that $(d^m_{C})_{s',s} = (d^m_{G})_{s',s}= 0$ whenever $s \neq s'$ (recall that the induced maps $(X_1\circ -)$ and $(X_3\circ -)$ are 0).
So to show that $d^m_C = (\tau^m \circ -)$ and $d^m_G = (\partial_G^m \circ -)$ have isomorphic images and isomorphic kernels, it is now sufficient to show them for each block $(d^m_{C})_{s',s}= ((\tau^m)_{s',s}\circ -)$ and $(d^m_{G})_{s',s} = ((\partial_G^m)_{s',s} \circ -)$  where $s = s'$.

To simplify notation, for the rest of this proof we shall drop the subscript $s$ and denote: 
\[
Q^m := Q^m_s; \quad
Q^{m+1} := Q^{m+1}_s;\quad
d_C := (d^m_{C})_{s,s};\quad
d_G := (d^m_{G})_{s,s};\quad
\tau := \tau^m_{s,s} ; \text{ and }
\partial_G := (\partial^m_G)_{s,s}.
\]
We shall now look at the possible types of maps $\tau: Q^m \ra Q^{m+1}$ which gives us all possible $d_{C} = (\tau \circ -)$, where $d_{G} = (\partial_G \circ -)$ can be obtained by $d_G = (\tau\circ -) - (\d \circ -)$ (following from \eqref{d_C in d_G}).

If $d_C = d_G$, i.e. $\tau$ has no entry of type (ii) so that $\d = 0$, then there is nothing left to show.
Otherwise, $\tau$ contains at least one entry $\partial_{yx}$ of type (ii).
The two possibilities of $\partial_{yx}$ of type (ii) are $(1|2)i$ and $-i(2|1)$.
We will only explicitly show the classification method used to obtain all possible types of $Q^m \xra{\tau} Q^{m+1}$ when $\partial_{yx} = (1|2)i$, where the same method can be applied to the case when $\partial_{yx} = -i(2|1)$.

So let us consider the case when $\tau$ has an entry with $\partial_{yx} = (1|2)i$.
Recall that by the definition of $L_B$, for any $\partial_{yx}$ of type (ii), there must be a corresponding 1-crossing $x'$ of $\check{c}$ such that 
\begin{itemize}
\item $x'$ and $x$ are connected by an essential segment in $D_0$, and
\item $x'$ and $y$ are endpoints of an essential segment of $\check{c}$ in $D_1$.  
\end{itemize}
So in the case $\partial_{yx} = (1|2)i$, $x'$ and $y$ are connected through an essential segment in $D_1$ of type 1 (refer to \cref{sixtype}) and the map $\partial_{yx'}: P(x') \ra P(y)$ is the right multiplication by $(1|2)$.
By the construction of $Q^m$ and $Q^{m+1}$, $Q^m$ must then at least contain the direct summands $P(x)$ and $P(x')$ and $Q^{m+1}$ must at least contain the direct summand $P(y)$, so the differential $Q^m \xra{\tau} Q^{m+1}$ must contain at least two entries $\partial_{yx} = (1|2)i$ and $\partial_{yx'} = (1|2)$.
Thus the crossings $x,x'$ and $y$ must be contained in the corresponding partial curve of $\check{c}$ below:
\begin{figure}[H]
\begin{tikzpicture} [scale=.9]

\draw[thick]  (13.475,8)--(13.7,8);
\draw[thick]  (13.475,10)--(13.7,10); 
\draw[thick, blue]  (13.475,8)--(13.475,10);
\draw[thick,red]  (11.25,9.65)--(13.475,9.65);
\draw[thick]  (11.25,10)--(13.7,10);  
\draw[thick]  (11.25,8)--(13.7,8);
\draw[thick,green, dashed]  (11.25,10)--(11.25,9.5);
\draw[thick, green, dashed]   (11.25,8)--(11.25,8.5);
\draw[thick, green, dashed, ->]   (11.25,9)--(11.25,9.6);
\draw[thick,  green, dashed, ->]   (11.25,9)--(11.25,8.4);
\filldraw[color=black!, fill=yellow!, very thick]  (11.25,9) circle [radius=0.1]  ;
\draw[fill] (12.75,9) circle [radius=0.1]  ;

\draw[blue] (12,8) -- (12,10);
\draw[thick, red] (11.25,8.35) -- (12,8.35) ;

\node[below right] at (13.475,9.65) {y};
\node[below right] at (12,9.65) {x'};
\node[right] at (12,8.35) {x};
\end{tikzpicture}
\caption{{\small The crossings $x,x'$ and $y$ of the partial curve of $\check{c}$.}} \label{must have form}
\end{figure}
As seen from \cref{must have form} above, $x$ and $y$ are the only crossings that can be in a another distinct essential segment of $\check{c}$.

Let us now first consider the subcase when $x$ is not in another distinct essential segment.
If $y$ is also not in another distinct essential segment, then we have that $Q^m \xra{\tau} Q^{m+1}$ is of the form:
\begin{figure}[H]
\begin{tikzcd}
P(x') \ar[r, "(1|2)"  ] & P(y) \\
P(x ) \ar[ru,"(1|2)i",swap] \ar[u, phantom, "\oplus" ]
\end{tikzcd}.
\end{figure}
If instead $y$ is part of another essential segment of $\check{c}$ with its other endpoint some crossing $w$, then the essential segment must be in $D_2$.
Since $P(y)$ is a direct summand of $Q^{m+1}$, $w$ must have the property $w_1 = y_1 - 1$ so that $P(w)$ is an entry of $Q^m$ and $\partial_{yw}$ is an entry of $\tau$.
The only two such possibilities are:
\begin{figure}[H]
\begin{tikzpicture} [scale=.9]

\draw[thick]  (13.475,8)--(15,8);
\draw[thick]  (13.475,10)--(15,10); 
\draw[thick, blue]  (13.475,8)--(13.475,10);
\draw[thick,red]  (11.25,9.65)--(13.475,9.65);
\draw[thick]  (11.25,10)--(13.7,10);  
\draw[thick]  (11.25,8)--(13.7,8);
\draw[thick,green, dashed]  (11.25,10)--(11.25,9.5);
\draw[thick, green, dashed]   (11.25,8)--(11.25,8.5);
\draw[thick, green, dashed, ->]   (11.25,9)--(11.25,9.6);
\draw[thick,  green, dashed, ->]   (11.25,9)--(11.25,8.4);
\filldraw[color=black!, fill=yellow!, very thick]  (11.25,9) circle [radius=0.1]  ;
\draw[fill] (12.75,9) circle [radius=0.1]  ;
\draw[fill] (14.25,9) circle [radius=0.1]  ;

\draw[blue] (12,8) -- (12,10);
\draw[thick, red] (11.25,8.35) -- (12,8.35) ;

\node[below left] at (13.475,9.65) {y};
\node[below right] at (12,9.65) {x'};
\node[below right] at (12,8.45) {x};
\node[left] at (13.475,9.85)  {w};
\draw[thick,red] plot[smooth,tension=.8] coordinates { (13.475,9.65)  (13.9,9.45)  (14.25, 8.7)  ( 14.625, 9 )  (14.35,9.5) (13.475,9.85) };
\draw[thick,red] plot[smooth,tension=.8] coordinates { (12,8.35) (12.4, 8.5)   (12.75,9)  };

\draw[thick]  (18.475,8)--(20.2,8);
\draw[thick]  (18.475,10)--(20.2,10); 
\draw[thick, blue]  (18.475,8)--(18.475,10);
\draw[thick,red]  (16.25,9.65)--(18.475,9.65);
\draw[thick]  (16.25,10)--(18.7,10);  
\draw[thick]  (16.25,8)--(18.7,8);
\draw[thick,green, dashed]  (16.25,10)--(16.25,9.5);
\draw[thick, green, dashed]   (16.25,8)--(16.25,8.5);
\draw[thick, green, dashed, ->]   (16.25,9)--(16.25,9.6);
\draw[thick,  green, dashed, ->]   (16.25,9)--(16.25,8.4);
\filldraw[color=black!, fill=yellow!, very thick]  (16.25,9) circle [radius=0.1]  ;
\draw[fill] (17.75,9) circle [radius=0.1]  ;
\draw[fill] (19.25,9) circle [radius=0.1]  ;

\draw[blue] (17,8) -- (17,10);
\draw[blue] (20,8) -- (20,10);
\draw[thick, red] (16.25,8.35) -- (17,8.35) ;

\node[above right] at (18.475,9.5) {{ y}};
\node[below right] at (17,9.65) {x'};
\node[below right] at (17,8.45) {x};
\node[right] at (20,8.35) {w};
\draw[thick,red] plot[smooth,tension=.5] coordinates { (18.475,9.65)  (18.75, 9.5) (19.25,8.4) (20,8.35)   };
\draw[thick,red] plot[smooth,tension=1] coordinates { (17,8.35) (17.4, 8.5) (17.75,9)  };
\end{tikzpicture}
\caption{{\small The two possible essential segments from $y$.}} \label{two ess y}
\end{figure}

%
%
Now note that if $w$ is a 2-crossing (left picture in \cref{two ess y}, then one sees that $w$ can not be connected to any other crossing $z$ through another distinct essential segment in $\check{c}$ with $P(z)$ a direct summand of $Q^{m+1}$;
if instead $w$ a 3-crossing (right picture in \cref{two ess y}, then the only possibility for $P(z)$ to be a direct summand of $Q^{m+1}$ is when $w$ is also a 3-crossing, with $w$ and $z$ endpoints of an essential segment of $\check{c}$ in $D_3$ of type 2, giving us $\partial_{wz} = X_3$.
Recall the chosen decomposition of $G^m$ and $G^{m+1}$, where $Q^{m+1} \subseteq G^{m+1}$ corresponds to the smallest subset of crossings in $\mathfrak{G}^{m+1}$ which contains $y$, with the property that maps between the direct summands of the decompositions of $G^m$ and $G^{m+1}$ are either 0 or $X_1$ or $X_3$.
Thus $P(z)$ must be excluded from $Q^{m+1}$.
We can therefore conclude that for the subcase when $x$ is not connected to any other distinct essential segments, we have 3 possible forms for $Q^m \xra{\tau} Q^{m+1}$:
\begin{center}
\begin{tikzcd}
P(x') \ar[r, "(1|2)"  ] & P(y) \\
P(x ) \ar[ru,"(1|2)i",swap] \ar[u, phantom, "\oplus" ]
\end{tikzcd} 
 or
\begin{tikzcd}
P(w) \ar[rd, "X_2 \text{ or } (3|2)"] \ar[d, phantom, "\oplus"]\\
P(x') \ar[r, "(1|2)"  ] & P(y) \\
P(x ) \ar[ru,"(1|2)i",swap] \ar[u, phantom, "\oplus" ]
\end{tikzcd}.
\end{center}
To analyse the maps $d_C$ and $d_G$, let us identify the morphism spaces as
\begin{align*}
\HOM^{m}_{\cK_B}(P_2^B, G^*) = & 
\bigoplus_{(s, t) \in \Z \times \Z/\Z}\Hom_{\Ba m}(P_2^B, (P(x')\oplus P(x) \oplus P(w) )\{s\}\<t\>)\{s\}\<t\> \\ 
\cong & \left(\R\{(2|1)\}\<x_3\> \oplus \R\{i(2|1)\}\<x_3+1\>\right)\{x_2+1\} 
	\\ &\oplus \left(\R\{(2|1)\}\<x'_3\> \oplus \R\{i(2|1)\}\<x'_3+1\>\right)\{x'_2\}
	\\ &\oplus Z,
\end{align*}
where $Z = 0$ for the first type of $Q^m \xra{\tau} Q^{m+1}$, and
\begin{align*}
\HOM^{m+1}_{\cK_B}(P_2^B, G^*) = & 
\bigoplus_{(s, t) \in \Z \times \Z/\Z}\Hom_{\Ba m}(P_2^B, P(y)\{s\}\<t\>)\{s\}\<t\> \\ 
\cong & \left(\R\{X_2\}\<y_3\> \oplus \R\{X_2 i\}\<y_3+1\>\right)\{y_2+1\} 
	\\ &\oplus \left(\R\{\id\}\<y_3\> \oplus \R\{i\}\<y_3+1\>\right)\{y_2\}.
\end{align*}
Using this identification, we can write $d_C$ and $d_G$  as the corresponding matrices:

\[
d_C =
\begin{bmatrix}
1 &  0 & 0 & -1 & e \\
0 &  1 & 1 & 0  & f \\
0 &  0 & 0 & 0  & 0 \\
0 &  0 & 0 & 0  & 0 \\
\end{bmatrix}
\quad
\text{ and }
\quad
d_G =
\begin{bmatrix}
1 &  0 & 0 & 0  & e \\
0 &  1 & 0 & 0  & f \\
0 &  0 & 0 & 0  & 0 \\
0 &  0 & 0 & 0  & 0 \\
\end{bmatrix},
\]

\noindent where $d_G$ is obtained from $d_C$ by removing maps that were induced by $(1|2)i$.
It follows that $d_C$ and $d_G$ have the same image and have isomorphic kernels.

Now consider the other subcase where $x$ is in another essential segment of $\check{c}$ with its other endpoint some crossing $y'$.
Note that since $x$ is already part of an essential segment in $D_0$, the essential segment connecting $x$ and $y'$ can only be in $D_1$.
As before, we must have $y'_1 = x_1 - 1$ so that $P(y')$ is a direct summand of $Q^{m+1}$ and that $\partial_{y'x}$ is an entry of $\tau$.
Furthermore, if $x$ and $y'$ is connected by the essential segment of Type 2 in \cref{sixtype}, then $\partial_{y'x} = X_1$.
Therefore such $P(y')$ is excluded from $Q^{m+1}$.
Collecting the results, the only possible essential segment connecting $x$ and $y'$ with $y'_1 = x_1 - 1$ and $\partial_{y'x} \neq X_1$ is the essential segment of Type 1.
Thus the crossings $x,x',y$ and $y'$ must be contained in the corresponding partial curve of $\check{c}$ below:
\begin{figure}[H]
\begin{tikzpicture} [scale=1]

\draw[thick]  (13.475,8)--(13.7,8);
\draw[thick]  (13.475,10)--(13.7,10); 
\draw[thick, blue]  (13.475,8)--(13.475,10);
\draw[thick,red]  (11.25,9.65)--(13.475,9.65);
\draw[thick]  (11.25,10)--(13.7,10);  
\draw[thick]  (11.25,8)--(13.7,8);
\draw[thick,green, dashed]  (11.25,10)--(11.25,9.5);
\draw[thick, green, dashed]   (11.25,8)--(11.25,8.5);
\draw[thick, green, dashed, ->]   (11.25,9)--(11.25,9.6);
\draw[thick,  green, dashed, ->]   (11.25,9)--(11.25,8.4);
\filldraw[color=black!, fill=yellow!, very thick]  (11.25,9) circle [radius=0.1]  ;
\draw[fill] (12.75,9) circle [radius=0.1]  ;

\draw[blue] (12,8) -- (12,10);
\draw[thick, red] (11.25,8.35) -- (12,8.35) ;

\node[right] at (13.475,9.75) {y};
\node[below right] at (12,9.65) {x'};
\node[below right] at (12,8.45) {x};
\node[right] at (13.475,9.45) {y'};

\draw[thick,red] plot[smooth,tension=.8] coordinates { (12,8.35) (12.2,8.5)  (12.45, 9.1) (12.9, 9.35) (13.475,9.45)};
\end{tikzpicture}
\caption{The crossings $x,x',y$ and $y'$ when $x$ is in another essential segment.} 
\end{figure}
The same analysis in the previous subcase on the possible essential segments connected to $y$ can be applied similarly to the crossings $y$ and $y'$ here.
Thus we conclude that for this subcase, $Q^m \xra{\partial^m_C} Q^{m+1}$ is equal to one of the following 6 types (there are 3 possible combinations of $X_2$ and $(3|2)$ maps in the rightmost diagram):
\begin{center}
\begin{tikzcd}
P(x') \ar[rd," " description] \ar[r, "(1|2)"  ] & P(y) \\
P(x ) \ar[ru,"(1|2)i" description] \ar[r,"(1|2)", swap]
	\ar[u, phantom, "\oplus" ] 
	& P(y') \ar[u, phantom, "\oplus"] 
\end{tikzcd}
, \quad
\begin{tikzcd}
P(z) \ar[rd, "X_2 \text{ or } (3|2)"] \ar[d, phantom, "\oplus"]\\
P(x') \ar[rd," " description] \ar[r, "(1|2)"  ] & P(y) \\
P(x ) \ar[ru,"(1|2)i" description] \ar[r,"(1|2)", swap]
	\ar[u, phantom, "\oplus" ] 
	& P(y') \ar[u, phantom, "\oplus"] 
\end{tikzcd}
or
\begin{tikzcd}
P(z) \ar[rd, "X_2 \text{ or } (3|2)"] \ar[d, phantom, "\oplus"]\\
P(x') \ar[rd," " description] \ar[r, "(1|2)"  ] & P(y) \\
P(x ) \ar[ru,"(1|2)i" description] \ar[r,"(1|2)", swap]
	\ar[u, phantom, "\oplus" ] 
	& P(y') \ar[u, phantom, "\oplus"] \\
P(z') \ar[ru, "X_2 \text{ or } (3|2)", swap] \ar[u, phantom, "\oplus"]
\end{tikzcd}
\end{center}
swapping $x$ with $x'$ (and correspondingly $y$ with $y'$) if necessary.
Let us again identify the morphism spaces as:
\begin{align*}
\HOM^{m}_{\cK_B}(P_2^B, G^*) = & 
\bigoplus_{(s, t) \in \Z \times \Z/\Z}\Hom_{\Ba m}(P_2^B, (P(x')\oplus P(x) \oplus P(z) )\{s\}\<t\>)\{s\}\<t\> \\ 
\cong & \left(\R\{(2|1)\}\<x_3\> \oplus \R\{i(2|1)\}\<x_3+1\>\right)\{x_2+1\} 
	\\ &\oplus \left(\R\{(2|1)\}\<x'_3\> \oplus \R\{i(2|1)\}\<x'_3+1\>\right)\{x'_2\}
	\\ &\oplus Z,	
\end{align*}
where $Z = 0$ for the first type of $Q^m \xra{\tau} Q^{m+1}$, and
\begin{align*}
\HOM^{m+1}_{\cK_B}(P_2^B, G^*) = & 
\bigoplus_{(s, t) \in \Z \times \Z/\Z}\Hom_{\Ba m}(P_2^B, P(y)\{s\}\<t\>)\{s\}\<t\> \\ 
\cong & \left(\R\{X_2\}\<y_3\> \oplus \R\{X_2 i\}\<y_3+1\>\right)\{y_2+1\}
	\\ &\oplus \left(\R\{X_2\}\<y'_3\> \oplus \R\{X_2 i\}\<y'_3+1\>\right)\{y'_2+1\}
	\\ &\oplus \left(\R\{\id\}\<y_3\> \oplus \R\{i\}\<y_3+1\>\right)\{y_2\}
	\\ &\oplus \left(\R\{\id\}\<y'_3\> \oplus \R\{i\}\<y'_3+1\>\right)\{y'_2\}
	\\
	= & \left(\R\{X_2\}\<y_3\> \oplus \R\{X_2 i\}\<y_3+1\>\right)\{y_2+1\}
	\\ &\oplus \left(\R\{X_2\}\<y'_3\> \oplus \R\{X_2 i\}\<y'_3+1\>\right)\{y'_2+1\}
	\\ &\oplus V.
\end{align*}
Writing $d_C$ and $d_G$  as the corresponding matrix, we get
\[
d_C =
\begin{bmatrix}
1 &  0  & 0 & -1 & e' \\
0 &  1  & 1 & 0  & f' \\
0 &  -1 & 1 & 0  & g' \\
1 &  0  & 0 & 1  & h' \\
0 &  0  & 0 & 0  & 0 
\end{bmatrix}
\quad 
\text{ and }
\quad
d_G =
\begin{bmatrix}
1 &  0 & 0 & 0  & e' \\
0 &  1 & 0 & 0  & f' \\
0 &  0 & 1 & 0  & g' \\
0 &  0 & 0 & 1  & h' \\
0 &  0 & 0 & 0  & 0 
\end{bmatrix},
\]
which also have the same image and same kernel.
Thus for $\partial_{yx} = (1|2)i$, all possible cases of $d_C$ and $d_G$ have isomorphic images and isomorphic kernels as required.

Applying the same classification method to the case when $\partial_{yx} = -i(2|1)$, the posssible types of $Q^m \xra{\tau} Q^{m+1}$ is given by:
\begin{center}
\begin{tikzcd}
 & P(y) \\
P(x ) \ar[ru,"-i(2|1)"] \ar[r, "(2|1)"]& P(y') \ar[u, phantom, "\oplus" ]
\end{tikzcd} 
\quad
or
\quad
\begin{tikzcd}
 	& P(y) \\
P(x ) \ar[ru,"-i(2|1)"] \ar[r, "(2|1)"] \ar[rd, "X_2 \text{ or } (2|3)", swap]
	& P(y') \ar[u, phantom, "\oplus" ] \\
	& P(z) \ar[u, phantom, "\oplus"]
\end{tikzcd}
\end{center}
when $y$ is not part of another distinct essential segment of $\check{c}$, and
\begin{center}
\begin{tikzcd}
P(x')\ar[r, "(2|1)"] \ar[dr, " " description]
	& P(y) \\
P(x ) \ar[ru,"-i(2|1)" description] \ar[r, "(2|1)"] \ar[u, phantom, "\oplus"]
	& P(y') \ar[u, phantom, "\oplus" ]
\end{tikzcd} 
,
\begin{tikzcd}
P(x')\ar[r, "(2|1)"] \ar[dr, " " description]
 	& P(y) \\
P(x ) \ar[ru,"-i(2|1)" description] \ar[r, "(2|1)"] \ar[rd, "X_2 \text{ or } (2|3)", swap] \ar[u, phantom, "\oplus"]
	& P(y') \ar[u, phantom, "\oplus" ] \\
	& P(z) \ar[u, phantom, "\oplus"]
\end{tikzcd} 
or
\begin{tikzcd}
	& P(z') \ar[d, "\oplus"] \\
P(x')\ar[r, "(2|1)"] \ar[dr, " " description] \ar[ru, "X_2 \text{ or } (2|3)"]
 	& P(y) \\
P(x ) \ar[ru,"-i(2|1)" description] \ar[r, "(2|1)"] \ar[rd, "X_2 \text{ or } (2|3)", swap] \ar[u, phantom, "\oplus"]
	& P(y') \ar[u, phantom, "\oplus" ] \\
	& P(z) \ar[u, phantom, "\oplus"]
\end{tikzcd} 
\end{center}
when $y$ is part of another distinct essential segment of $\check{c}$, where as before we swap $y$ with $y'$ (correspondingly $x$ with $x'$) if neccessary.
By identifying the morphism spaces and comparing the corresponding matrices of $d_C$ and $d_G$ as before, it follows that $d_C$ and $d_G$ have isomorphic images and isomorphic kernels.
This covers all cases of $\d_C$ and $\d_G$, completing the proof.
\end{proof}

\begin{lemma}\label{poly tri int b g}
The Poincar\'e polynomial $\mathfrak{P}(P^B_j, L_B(\check{g}))$  of
$\HOM^*_{\cK_B}(P^B_j,L_B(\check{g}))$ is equal to the trigraded intersection number $I^{trigr}(\check{b}_j, \check{g})$ for any trigraded $j$-string $\check{g}$.
\end{lemma}
\begin{proof}
This follows exactly as in \cite[Lemma 4.11]{KhoSei}.
\end{proof}

\begin{proposition} \label{poin poly equals tri int}
For any $\sigma$ and  $\tau$ in $\cA(B_n),$ and any $1 \leq j,k \leq n,$ the Poincar\'e polynomial of 
$$HOM^*_{\cK_B}( \sigma(P^B_j), \tau(P^B_k)) $$
is equal to the trigraded intersection number $I^{trigr}(\check{\sigma}(\check{b}_j), \check{\tau}(\check{b}_k)).$
\end{proposition}
\begin{proof}
By \cref{compute tri int}, we get that
$
I^{trigr}(\check{b}_j, \check{c}) = \sum_{\check{g} \in st(\check{c},j)} I^{trigr}(\check{b}_j, \check{g}).
$
Using \cref{poly c g}, we instead get that
$
\mathfrak{P}(P^B_j, L_B(\check{c})) = \sum_{\check{g} \in st(\check{c},j)}\mathfrak{P}(P^B_j, L_B(\check{g})).
$
By \cref{poly tri int b g}, each $\mathfrak{P}(P^B_j, L_B(\check{g})) = I^{tri}(\check{b}_j, \check{g})$, thus we can conclude that
$
I^{trigr}(\check{b}_j, \check{c}) =  \mathfrak{P}(P^B_j, L_B(\check{c})).
$
The proposition now follows from the fact that the categorical action of $\cA(B_n)$ respects morphism spaces and similarly the topological action of $\cA(B_n)$ respects trigraded intersection number.
\end{proof}
\begin{remark}
Note that we can also use \cref{poin poly equals tri int} to prove the faithfulness of the $\cA(B_n)$ categorical action.
The proof is similar to \cite[the paragraph before Section 5]{KhoSei} modulo the center of $\cA(B_n)$, which is an easy check that elements of the centre act by shifting degrees and therefore are not isomorphic to the identity functor.
\end{remark}

\newpage

\printbibliography

\end{document}